\documentclass[10pt]{amsart}
\usepackage{hyphenat}
\usepackage{longtable}
\usepackage{amssymb,amsmath,amsthm,amscd,mathrsfs,graphicx, color}
\usepackage{lscape}
\usepackage[cmtip,all,matrix,arrow,tips,curve,rotate]{xy}
\usepackage[colorlinks=true,hypertexnames=false]{hyperref}
\usepackage{array}
\usepackage{enumerate}

\mathchardef\mhyphen="2D

\numberwithin{equation}{section}
\newtheorem{teo}{Theorem}[section]
\newtheorem{pro}[teo]{Proposition}
\newtheorem{lem}[teo]{Lemma}
\newtheorem{cor}[teo]{Corollary}

\theoremstyle{definition}
\newtheorem{dfn}[teo]{Definition}
\newtheorem{exa}[teo]{Example}

\newtheorem{wrn}[teo]{Warning}

\theoremstyle{remark}
\newtheorem{convention}[teo]{Convention}
\newtheorem{rem}[teo]{Remark}

\begin{document}
\bibliographystyle{alpha}

\title[Stacks and Higgs bundles]{An introduction to moduli stacks, with a view towards Higgs bundles on algebraic curves}
\author[Casalaina-Martin]{Sebastian Casalaina-Martin }
\address{University of Colorado at Boulder, Department of Mathematics,  
Boulder, CO,  USA}
\email{casa@math.colorado.edu}
\author[Wise]{Jonathan Wise}
\address{University of Colorado at Boulder, Department of Mathematics,  
Boulder, CO,  USA}
\email{wise@math.colorado.edu}

\thanks{The first author was partially supported by NSF grant DMS-1101333, a Simons Foundation Collaboration Grant for Mathematicians (317572), and NSA Grant H98230-16-1-005.  The second author was partially supported by an NSA Young Investigator's Grant  H98230-14-1-0107. }

\date{\today}

\begin{abstract}
This article is based in part on lecture notes prepared for  the summer school ``The Geometry, Topology and Physics of Moduli Spaces of Higgs Bundles''  at the Institute for Mathematical Sciences at the National University of Singapore in July of 2014. 
The aim is to provide a brief introduction to algebraic stacks, and then  to give  several constructions of the moduli stack of Higgs bundles on algebraic curves.  The first construction is via a ``bootstrap'' method  from the algebraic stack of vector bundles on an algebraic curve. 
This construction is motivated in part by  Nitsure's GIT construction of  a projective moduli space of semi-stable Higgs bundles, and 
we describe the relationship between Nitsure's moduli space and the algebraic stacks constructed here.
   The third approach is via deformation theory, where we directly construct the stack of Higgs bundles using Artin's criterion.
\end{abstract}
\maketitle

\tableofcontents



\section*{Introduction}

Stacks have been used  widely by algebraic geometers since the 1960s   for studying parameter spaces for algebraic and geometric objects \cite{DM}.  
Their popularity is growing in other areas as well (e.g., \cite{MM94}).  Nevertheless, despite their utility, and their having been around for many years, stacks still do not seem to be as  popular   as might be expected.   Of course, echoing the  introduction of \cite{fantechi},  this may  have something to do with the technical nature of the topic, which may dissuade the uninitiated, particularly when there may be less technical methods available that can be used instead.  
To capture a common sentiment, it is hard to improve on the following excerpt from  Harris--Morrison \cite[3.D]{HM},  which serves as an introduction to their section on stacks:
\emph{
\begin{quote}
\begin{quote}
`Of course, here I'm working with the moduli stack rather than with the moduli space.  For those of you who aren't familiar with stacks, don't worry: basically, all it means is that I'm allowed to pretend that the moduli space is smooth and there's a universal family over it.'
\end{quote}
 \vskip .1 in 
Who hasn't heard these words, or their equivalent spoken at a talk?  And who hasn't fantasized about grabbing the speaker by the lapels and shaking him until he says what -- exactly -- he means by them? 
\end{quote}}
At the same time, the reason stacks are used as widely as they are  is that   they  \emph{really} are a natural language for talking about parameterization.     While there is no doubt a great deal of technical background that goes into the set-up, the pay-off is that once this foundation has been laid,  stacks  provide a clean, unified  language for discussing   what otherwise may require many caveats and special cases.    

Our goal here is to give a brief introduction to algebraic stacks, with a view towards  defining the moduli stack of Higgs bundles. 
 We have tried to provide enough motivation for   stacks that the reader is inclined to proceed to the definitions, and a sufficiently streamlined  presentation of the definitions that the reader does not immediately stop at that point!  In the end, we hope the reader has a good sense of what the moduli stack of Higgs bundles is, why it is an algebraic stack, and how the stack relates to the quasi-projective variety of   Higgs bundles constructed by Nitsure \cite{N91}.

\vskip .1 in 
Due to the authors' backgrounds, for precise statements, the topic will be treated in the language algebraic geometry, i.e., schemes.   However, the aim is to have a presentation that is accessible to those in other fields, as well, particularly complex geometers,  and most of the presentation  can be made replacing the word ``scheme'' with the words ``complex analytic space'' (or even ``manifold'') and ``\'etale cover'' with ``open cover''.  This is certainly the case up though, and including, the definition of a stack in \S \ref{S:stack}.    
The one possible exception to this rule is the topic of algebraic stacks \S \ref{S:AlgStack}, for which definitions in the literature are really  geared towards the   category of schemes.
In order to make our presentation as accessible as possible, in \S \ref{S:AdaptStackPre} and  \S \ref{S:AlgStack} we provide  definitions of  algebraic stacks  that make sense for any presite; in particular, the definition  gives a notion of an ``algebraic'' stack in the category of complex analytic spaces.
  There are other notions of analytic stacks in the literature, but for concision,  we do not pursue the connection between the definitions.   
   
\vskip .1 in 

The final sections of this survey (\S\ref{S:def-thy} and \S\ref{S:Artin}) study algebraic stacks infinitesimally, with the double purpose of giving modular meaning to infinitesimal motion in an algebraic stack and introducing Artin's criterion as a means to prove stacks are algebraic.  Higgs bundles are treated as an extended example, and along they way we give cohomological interpretations of the tangent bundles of the moduli stacks of curves, vector bundles, and morphisms between them.  In the end we obtain a direct proof of the algebraicity of the stack of Higgs bundles, complementing the one based on established general algebraicity results given earlier.

\vskip .1 in 

While these notes are meant to be somewhat self contained, in the sense that essentially all relevant definitions are included, and all results are referenced to the literature, these notes are by no means comprehensive. 
There are now a number of introductions to the topic of algebraic stacks that have appeared recently (and even more are in preparation), and these notes are  inspired by  various parts of those  treatments.  While the following list is not exhaustive, it provides a brief summary of some of the other resources available.
 To begin, we direct the reader to the  now classic, concise introduction by Fantechi \cite{fantechi}.   The reader may very well want to read that ten page introduction before this one.  
 A more detailed introduction to  Deligne--Mumford stacks is given in  the book \cite{DMstacks}.  
  The book \cite{FGAe} provides another detailed introduction, with less emphasis on the algebraic property of algebraic stacks.  The book  \cite{LMB} provides a concise, detailed,  but perhaps less widely accessible introduction to the topic.   There is also the comprehensive treatise  \cite{stacks}, which has a complete treatment of all of the details.   The reader who would like to follow the early development of the subject might want to consult \cite{Giraud}, \cite{mumfordPic}, \cite{DM}, \cite[Exp.~II]{sga4-1}, \cite[Exp.~XVII,~XVIII]{sga4-3}, \cite{Knutson}, and \cite{Artin-versal}.

\subsection*{Notation} 

We denote by $\mathsf S$  the category of schemes over a fixed base, which the reader may feel free to assume is $\operatorname{Spec} \mathbb C$, where $\mathbb C$ is the field of complex numbers.  All schemes will be members of this category.  Until we discuss algebraic stacks, one could even take $\mathsf S$ to be the category of complex analytic spaces, or any other reasonable category of ``spaces'' with which one would like to work. 

Our convention is generally to use roman letters for schemes (e.g., $M$), script letters for functors to sets (e.g., $\mathscr M$), and calligraphic letters for categories fibered in groupoids; (e.g.,  $\mathcal M$).   We will typically use sans serif letters for various standard  categories (e.g., $\mathsf M$).


\section{Moduli problems  as functors}

When one wants to parameterize some kind of algebraic or geometric objects, one says one has a moduli problem.    The goal is to find another geometric object, called a moduli space,  that parameterizes the objects of interest.  
One of the most natural ways to phrase a moduli problem is in terms of the corresponding moduli functor.  From this perspective, the hope is that the moduli functor will be representable by a  geometric object, which will be the moduli space.   In fact, it is possible to say a great deal about the moduli space purely in terms of the moduli functor, without even knowing the moduli space exists!  For example, the tangent space can be computed by evaluating the moduli functor on the spectrum of the ring of dual numbers (\S\ref{S:TSMF}).

We make this precise in what follows.   The best way to get a feel for this is through examples.  For  this reason, in this section we start by considering the familiar   problems of parameterizing linear subspaces of a fixed vector space, and 
  of parameterizing Riemann surfaces of a fixed genus.


\subsection{Grassmannians} 
   We expect this example is familiar to most of the readers, and the aim will be to motivate an approach to these types of problems, which we will continue to use throughout.   We direct the reader to \cite[Ch.1,\S 5]{GH} for details.

In this section we explicitly take $\mathsf S$ to be the category of $\mathbb C$-schemes, or the category of analytic spaces, for simplicity; a ``scheme'' refers to a member of this category.  For an object $S$ of $ \mathsf S$, we adopt the notation $\mathbb C^n_S:=\mathbb C^n\times S$ for the trivial vector bundle of rank~$n$ over $S$.

    Recall that Grassmannians arise from the moduli problem that consists  of   parameterizing  $r$-dimensional complex subspaces of  $\mathbb C^n$.   
 
There is clearly a set of such spaces 
$$
G(r,n):=\{W\subseteq_{\text{lin}} \mathbb C^n:\dim W=r\}.
$$
In fact, one can easily put a ``natural'' complex structure on $G(r,n)$ that makes this set into a smooth complex projective variety of dimension $r(n-r)$, which we call the Grassmannian. 
Of key importance  from our perspective is the fact that  there is a rank $r$ vector bundle $\mathbb U$ over $G(r,n)$, and an inclusion of vector bundles:
\newcommand{\univsub}{{\mathbb U}}
$$
0\to \univsub \to \mathbb C^n_{G(r,n)}
$$
such that for any $\mathbb C$-scheme $S$ and any inclusion of vector bundles $\mathbb F\hookrightarrow \mathbb C^n_S$ over $S$ with $\operatorname{rank} \mathbb F=r$, there is a \emph{unique} $\mathbb C$-morphism
$$
f:S\to G(r,n)
$$
such that the inclusion $\mathbb F\hookrightarrow \mathbb C^n_S$ is isomorphic  to the pullback of the inclusion $\univsub \hookrightarrow \mathbb C^n_{G(r,n)}$ by $f$.  Diagramatically: 
$$
\xymatrix@C=.2cm @R=.2cm{
\mathbb F \ar@{->}[rdd]  \ar@{^(->}@<-2pt>[rrd] \ar@{}[rrr]^(0.1){}="a" \ar@<2pt> "a";[rrr] &&&\univsub \ar@{}[rrd]^(0.1){}="b" \ar@{^(->}@<-2pt> "b";[rrd] \ar@{->}[rdd]&&&\\
&&\mathbb C^n_S \ar@{->}[rrr] \ar@{->}[ld]&&&\mathbb C^n_{G(r,n)} \ar@{->}[ld]&\\
&S \ar@{->}[rrr]^f_{\exists !}&&&G(r,n)&&\\
}
$$
Set theoretically, we define $\univsub$ by setting the fiber of $\univsub$ over a point $[W]\in G(r,n)$ corresponding to a subspace $W\subseteq \mathbb C^n$ to be the subspace $W$ itself; i.e., we have  $\univsub_{[W]}=W\hookrightarrow \mathbb C^n=( \mathbb C^n_{G(r,n)})_{[W]}$.  The bundle $\univsub$ is called the universal subbundle.  

\label{p:GrCov}
In fact, the topology and geometry of $G(r,n)$, as well as  the vector bundle structure on $\univsub$, are all encoded in the universal property.  To see the topology, consider a linear projection $p : \mathbb C^n \rightarrow \mathbb C^r$.  If $S$ is any scheme and $\mathbb F \hookrightarrow  \mathbb C^n_S$ is a rank~$r$ vector subbundle, then for a general projection $p$,  denoting by  $p_S$ the induced morphism $p_S:\mathbb C^n_S\to \mathbb C^r_S$, then the composition $p_{\mathbb F} : \mathbb F \rightarrow \mathbb C^n_S\stackrel{p_S}{\to}  \mathbb C^r_S$ is an isomorphism over a nonempty  open subset $S_p\subseteq S$.  Choosing appropriate projections, these open subsets cover $S$.  This applies in particular to the Grassmannian itself, so that we obtain an open cover of $G(r,n)$ by subsets $U_p$ where $p_{\mathbb U} \big|_{U_p}:\mathbb U\big|_{U_p}\to \mathbb C^r_{U_p}$ is an isomorphism.  Moreover, we can easily see that locally, after trivializing $\mathbb U$,  a point of $U_p$ may be identified with a section of the projection $p : \mathbb C^n \rightarrow \mathbb C^r$; 
 or more precisely, a map $S \rightarrow U_p$ may be identified with a section of $p _S: \mathbb C_S^n \rightarrow \mathbb C_S^r$.  Consequently, one can check from the universal property that $U_p$ must be isomorphic to the space of sections of the projection, which is isomorphic to  $\operatorname{Hom}_{\mathbb C}(\mathbb C^r,\mathbb C^{n-r})= \mathbb C^{(n-r) \times r}$, the space  of $(n-r) \times r$ complex matrices. This of course agrees with the standard method of constructing charts for the Grassmanian by choosing bases for the subvector space, and performing row reduction.  
 We also have the identification $p_{\mathbb U} \big|_{U_p} : \univsub_{U_p} \rightarrow \mathbb C^r_{U_p}$, which gives the vector bundle structure on $\univsub$.

The take-away from this discussion is that there is an identification of families (up to  isomorphism)  of $r$-dimensional subspaces of $\mathbb C^n$ parameterized by $S$, with  morphisms of $\mathbb C$-schemes $S$ into $G(r,n)$ and, moreover, that \emph{we can understand everything about $G(r,n)$ in terms of this identification}.  As perverse as it might seem, we could throw away the topological space and sheaf of rings and consider  instead the Grassmannian \emph{functor}:
\begin{gather*}
\mathscr G(r,n):\mathsf {S}^{\text{op}}\to(\mathsf  {Set})  \\
\mathscr G(r,n)(S):= \Biggl\{ (\mathbb F,i) \: \Bigg| \: \raisebox{-8pt}{\vbox{\hbox{$\mathbb F$ is a vector bundle on $S$}\hbox{$i : \mathbb F \hookrightarrow \mathbb C_S^n$ is a linear inclusion}}} \Biggr\} \Bigg/ \sim
\end{gather*}
where two inclusions of vector bundles $\mathbb F\hookrightarrow S\times \mathbb C^n$ and $\mathbb F'\hookrightarrow S\times \mathbb C^n$ over $S$ are equivalent if there is a commutative diagram of the form
$$
\xymatrix@R=10pt{
\mathbb F \ar@{^(->}@<-1pt>[r] \ar@{->}[d]_-{\rotatebox{90}{$\scriptstyle \sim$}}& \mathbb C^n_S \ar@{=}[d]\\
\mathbb F'\ar@{^(->}@<-1pt>[r]&\mathbb C^n_S.
}
$$
The functor $\mathscr G(r,n)$ acts on morphisms in the following way.
Associated to a morphism $f:S'\to S$, we have
$$
f^*:\mathscr G(r,n)(S)\to \mathscr G(r,n)(S')
$$
taking the class of the  inclusion $\mathbb F\hookrightarrow S\times \mathbb C^n$ to the class of the inclusion $\mathbb F'\hookrightarrow S'\times \mathbb C^n$ 
defined via pullback diagrams:
$$
\xymatrix@C=.3cm @R=.3cm{
\mathbb  F': \ar@{=}[r] & f^\ast \mathbb F \ar@{->}[rdd]  \ar@{^(->}@<-1pt>[rrd] \ar@{->}@<1pt>[rrr]&&&\mathbb F\ar@{^(->}@<-1pt>[rrd] \ar@{->}[rdd]&&&\\
& &&\mathbb C^n_{S'} \ar@{->}[rrr]_<>(0.25){\mathbb C^n_f } \ar@{->}[ld]&&&\mathbb C^n_S. \ar@{->}[ld]&\\
& &S' \ar@{->}[rrr]_{f}&&&S&&\\
}
$$
Here $f^\ast \mathbb F$ is the pullback of $\mathbb F$ to $S'$ via $f$ and $\mathbb C^n_f$ is the canonical morphism $\mathbb C^n_{S'} = \mathbb C^n \times S' \xrightarrow{\mathrm{id} \times f} \mathbb C^n \times S = \mathbb C^n_S$.

The assertion that there is an identification of families (up to  isomorphism)  of $r$-dimensional subspaces of $\mathbb C^n$ parameterized by $S$, with  morphisms of schemes $S$ into $G(r,n)$, can be formulated precisely as an isomorphism of functors
\begin{gather}\label{E:GrRep}
\vcenter{\xymatrix@R=4pt{
 \operatorname{Hom}_{\mathsf S}\bigl(-,G(r,n)\bigr) \ar[r] & \mathscr G(r,n) \\
\bigl(f:S\to G(r,n)\bigr) \ar@{|->}[r] & f^*\bigl(\univsub\hookrightarrow \mathbb C^n_{G(r,n)}\bigr).
}} \end{gather}
In other words, the Grassmannian functor is representable by the Grassmannian scheme.  We also say that the Grassmannian  is a fine moduli space for the Grassmannian functor.    

Our previous extraction of the geometric structure of $G(r,n)$ from the universal property, in other words, from the functor $\mathscr G(r,n)$, is a manifestation of the general principle encoded in Yoneda's lemma (stated below).  Indeed, if there is another space $G(r,n)'$   such that 
$$
\operatorname{Hom}_{\mathsf S}\bigl(-,G(r,n)'\bigr)\cong \mathscr G(r,n),
$$
then 
$$
\operatorname{Hom}_{\mathsf S}\bigl(-,G(r,n)'\bigr)\cong \operatorname{Hom}_{\mathsf S}\bigl(-,G(r,n)\bigr),
$$
and one may use Yoneda's lemma to conclude that $G(r,n)\cong G(r,n)'$.  
\begin{lem}[Yoneda]
Let $\mathsf S$ be a category.   There is a fully faithful functor
\begin{gather*}
h:\mathsf S\to \mathsf {Fun}\bigl(\mathsf S^{\mathsf {op}},(\mathsf {Set})\bigr) \\
S\mapsto h_S:=\operatorname{Hom}_{\mathsf S}(-,S) \\
(S\stackrel{\phi}{\to} S')\mapsto (h_S\xrightarrow{\phi_\ast} h_{S'}),
\end{gather*}
which identifies $\mathsf S$ with a full subcategory of the category of functors  $\mathsf {Fun}\bigl(\mathsf S^{\mathsf {op}},(\mathsf {Set})\bigr)
$.  For each $S\in \operatorname{Ob}(\mathsf S)$, the map 
\begin{gather*}
\operatorname{Hom}_{\mathsf {Fun}(\mathsf S^{\mathsf {op}},(\mathsf {Set}))}(h_S, \mathscr F)\to \mathscr F(S) \\
(\eta:h_S\to \mathscr F)\mapsto \eta(S)(\operatorname{id}_S)
\end{gather*}
is a bijection. 
\end{lem}

This is best worked through on one's own, preferably in private.  Nevertheless, in the spirit of completeness that inspires these notes, we refer the reader to \cite[Ch.2, \S 2.1]{FGAe} or \cite[Ch.~0, \S1.1]{EGA1971} for more details.    

\begin{dfn}[Representable functor]
A functor isomorphic to one of the form  $\operatorname{Hom}_{\mathsf S}(-,S)$ for some $S$ in $\mathsf S$ is called a \emph{representable functor}.  
\end{dfn}

\begin{rem}\label{R:Yoneda}
Due to the Yoneda Lemma, we will often write $S$ in place of $h_S$ or  $\operatorname{Hom}_{\mathsf S}(-,S)$.  
\end{rem}

\begin{rem}
Note that under the identification of \eqref{E:GrRep}, the identity map  $\operatorname{id}_{G(r,n)}\in \operatorname{Hom}_{\mathsf S}\bigl(G(r,n),G(r,n)\bigr)$ is sent to the class of the inclusion of the universal subbundle  $\univsub \hookrightarrow \mathbb C^n_{G(r,n)}$ in $\mathscr G(r,n)\bigl(G(r,n)\bigr)$.  
\end{rem}

Summarizing the discussion above, the fact that the Grassmannian functor is representable is equivalent to the fact  that there is a universal family, $\univsub$, of dimension $r$-subspaces of $\mathbb C^n$, parameterized by the Grassmannian $G(r,n)$, such that  any other family of  dimension $r$-subspaces of $\mathbb C^n$ is obtained from this universal family via pullback along a (unique) map to the Grassmannian.

\begin{rem}
This whole construction could have been dualized, considering instead $r$-dimensional quotients of the dual vector space $(\mathbb C^n)^\vee$.  The problem is clearly equivalent, but this dual  formulation is more obviously related to the Quot scheme we will discuss later.  We leave it to the reader to make this comparison when the Quot scheme is introduced.    
\end{rem}

\subsection{Moduli of Riemann surfaces}   \label{S:mod-RS}
The example of the Grassmannian gives a  moduli space that is representable by a projective variety.  Another  accessible example of a moduli problem that on the other hand  captures the complications leading to the development of stacks is that of parameterizing Riemann surfaces.

We will work in this section with the algebraic analogues of  compact  Riemann surfaces of genus $g$; i.e.,  smooth proper complex algebraic curves of genus $g$.   Our moduli problem is that of parameterizing all such curves up to isomorphism.  Motivated by the discussion of the Grassmannian, we make this more precise by defining a moduli functor.  To do this, we define a \emph{relative curve of genus $g$} to be 
a surjective morphism of schemes  $\pi:X\to S$ such that 
\begin{itemize}
\item $\pi$ is smooth,
\item $\pi$ is proper,
\item every geometric fiber is a connected curve of genus $g$.
\end{itemize}
 Recall that if $X$ and $S$ are smooth, then $\pi:X\to S$ is smooth if and only the differential is everywhere surjective.  In other words, this includes the analogues of  surjective  morphisms of smooth complex manifolds, with surjective differential, where every (set theoretic)  fiber is a compact Riemann surface of genus $g$.   
 
    We now define the moduli functor of genus $g$ curves as:
\begin{gather*}
\mathscr M_g:\mathsf S^{\mathsf {op}}\to (\mathsf {Set}) \\
\mathscr M_g(S):=\{\pi:X\to S \text{, a relative curve of genus } g \}/\sim
\end{gather*}
where two relative curves $\pi:X\to S$ and $\pi':X'\to S$ are equivalent if there is a commutative  diagram~\eqref{E:EquivCurve}, in which the upper horizontal arrow is an isomorphism:
\begin{equation} \label{E:EquivCurve}
\vcenter{\xymatrix{
X' \ar[r]^{\sim} \ar[d]_{\pi'} &  X \ar[d]^\pi \\
S \ar@{=}[r] & S.
}}
\end{equation}
The functor $\mathscr M_g$ acts on morphisms in the following way.  Associated to a morphism $f:S'\to S$, we have
$$
f^*:\mathscr M_g(S)\to \mathscr M_g(S')
$$
defined via pullback diagrams; in other words, given $[\pi:X\to S]\in \mathscr M_g(S)$, we define $f^*[\pi:X\to S]$ as the class of the curve $\pi':X'\to S'$, defined  by the fiber product diagram:
\begin{equation*}
\xymatrix{
X' \ar[r] \ar[d]_{\pi'} & X \ar[d]^\pi \\
S' \ar[r]^f & S.
} \end{equation*}

\begin{exa}[Isotrivial family] \label{E:isotriv1}
Consider the relative curve $\pi:X\to S=\mathbb C^*$ of genus $1$ given by the equation
$$
y^2=x^3+t, \ t\ne 0.
$$
Here $X$ is the projective completion of $\{(x,y,t): y^2-x^3-t=0\}$ in $\mathbb P^2_{\mathbb C}$, and the map $\pi:X\to S$ is induced  by $(x,y,t)\mapsto t$.  
It turns out this family is \emph{isotrivial} (i.e., all of the fibers are abstractly isomorphic); this can be seen quickly by confirming that the $j$-invariant of each fiber is $0$, or simply writing down the isomorphism $(x,y,t) \mapsto (\lambda^{-2} x, \lambda^{-3} y, 1)$, where $\lambda^6 = t$, between $X_t$ and $X_1$.  One can also show that this family is not equivalent to a trivial family; i.e., not isomorphic to $S\times X_1$.
One way to see this is to check that the  monodromy action of $\pi_1(\mathbb C^*,1)$  on  $H^1(X_1,\mathbb C)$ is nontrivial (see e.g., \cite[\S 4.1.2]{casa13}), and therefore, that $X/S$ is not a trivial family.
  Another  way to see this is to observe that the explicit isomorphism given above trivializes $X$ over the pullback via $\mathbb C^\ast \rightarrow \mathbb C^\ast : \lambda \mapsto \lambda^6$.  We can therefore characterize $X$ as the quotient of $X_1 \times \mathbb C^\ast$ by the action $\zeta . (x,y,\lambda) = (\zeta^{-2} x, \zeta^{-3} y, \zeta \lambda)$ of a $6$-th root of unity $\zeta \in \mu_6$.  
The \'etale  sheaf on $S$ of isomorphisms between $X$ and the $X_1 \times S$ is a torsor under the automorphism group of $X_1$, which is $\mu_6$.  It is therefore classified by an element of $H^1(S, \mu_6) = \operatorname{Hom}(\mu_6,\mu_6)=\mathbb Z/6\mathbb Z$,
 which one can check is  $1$.  Since this corresponds to a nontrivial  torsor, $X/S$ is not a trivial family.  One can generalize this example  to show that for every $g$, there exist isotrivial, but nontrivial families of curves  (see Example \ref{E:isotriv}).
\end{exa}

\begin{pro} \label{P:MgNR}
The functor $\mathscr M_g$ is not representable.  
\end{pro}

\begin{proof}  See e.g.~\cite{HM} for more details.    The main point is that 
if $\mathscr M_g$ were representable, then every isotrivial family of curves would be equivalent to a  trivial family, which we have just seen is not the case.   

 To this end, suppose that $\mathscr M_g=\operatorname{Hom}_{\mathsf S}(-,M_g)$ for some scheme $M_g$.    Let $C_g\to M_g$ be the family of curves corresponding to the identity morphism $\operatorname{id}_{M_g}$.  
Now let $\pi:X\to S$ be an isotrivial family of curves over a connected base $S$, that is not equivalent  to $S\times X_s$ for any  $s\in S$, with $s=\operatorname{Spec} \mathbb C$ (here $X_s$ is the fiber of $X/S$; see e.g., Examples \ref{E:isotriv1} and \ref{E:isotriv}).  The curve $X_s\to s$ corresponds to an element of  $\mathscr M_g(\operatorname{Spec}\mathbb C)=\operatorname{Hom}_{\mathsf S}(\operatorname{Spec}\mathbb C,M_g)$;   i.e., the isomorphism class of the curve  corresponds to a point, say $[X_s]\in M_g$.    Similarly, the family $\pi:X\to S$ corresponds to a morphism $f:S\to M_g$.  Since every fiber of $\pi:X\to S$ is isomorphic to $[X_s]$, the image of $S$ under $f$ must be the point $[X_s]$.  By definition of the representability of the functor, we must then have that the pullback of the universal family $C_g\to M_g$ along $f$, i.e., $X_s\times S$, is equivalent to $\pi:X\to S$, which we assumed was not the case. 
\end{proof}

\begin{rem} \label{R:coarse}
While the functor $\mathscr M_g$ is not representable, it does admit a coarse moduli space.  In other words, 
there is a quasi-projective variety $M_g$ and a morphism  $\Phi:\mathscr M_g\to M_g$ (here we are using the convention mentioned in Remark \ref{R:Yoneda} of denoting a scheme and its associated functor with the same letter) such that:

\begin{enumerate}
\item  for any scheme $S$ and any  morphism $\Psi:\mathscr M_g\to S$, there is a unique morphism $\eta :M_g\to S$ such that the following diagram commutes:
$$
\xymatrix{
\mathscr M_g \ar@{->}[r]^\Phi \ar@{->}[d]_\Psi& M_g  \ar@{-->}[ld]^{\exists !}_\eta \\
S&\\
}
$$

\item  $\Phi$ is a bijection when evaluated on any algebraically closed field.   
\end{enumerate}
 A scheme representing a functor, i.e., a fine moduli space,  is clearly a coarse moduli space.  
 For many moduli functors that we consider, we will be able to find a scheme satisfying the first condition above; however, when the automorphisms of our objects are positive dimensional, it will not in general be possible to find a scheme satisfying the second condition, as well.
 
For brevity (or at least for lesser verbosity), we have not included a discussion of coarse moduli spaces.  The interested reader may consult \cite[Def.~1.3]{HM}, or   \cite[Def.~4.1]{alpergms}.
\end{rem}

There was not anything special about curves in Proposition~\ref{P:MgNR}:  all we needed was an isotrivial family that was nontrivial.  
In fact,    nontrivial automorphism groups almost always give rise to nontrivial isotrivial families (see Corollary \ref{C:IfAutNotRep}) so we conclude that moduli problems involving objects with nontrivial automorphisms will \emph{almost never} be representable by genuine spaces.  Since most moduli problems of interest involve objects with nontrivial automorphisms, this means 
\emph{fine moduli spaces almost never exist}; i.e.,  moduli problems are almost never representable by schemes.

There are a few ways one might try to go around this problem:  one could \emph{rigidify} the problem by imposing additional structure in the hopes of eliminating nontrivial automorphisms; one could give up on the idea of a fine moduli space and settle for a \emph{coarse moduli space}  that   at least gets the points right, even if it botches the universal property (see Remark~\ref{R:coarse}).

Stacks do not go around the problem of isomorphisms so much as they go through it.  By remembering how objects are equivalent, as opposed to merely that they are equivalent, we can eliminate the issue that is responsible for Proposition~\ref{P:MgNR}.  The price we must pay is to enlarge the class of objects we are willing to call spaces and 
sacrifice some of our geometric intuition. 
 What we hope to illustrate in these notes is that the cost in geometric intuition is less than one might first expect, and that the payoff in newly available moduli spaces more than compensates for it.  Indeed, stacks will allow us to bring our intuition about geometric families to bear on the geometry of their moduli spaces, giving us a new---and, we would argue, more powerful---sort of intuition to replace what we have sacrified.

\subsection{The tangent space of a moduli functor} \label{S:TSMF}
Because of Yoneda's lemma, a fine moduli space is uniquely characterized by the functor it represents.  In particular, the tangent space to a moduli space at a point can be determined directly from the moduli functor, without even knowing that the moduli space is representable!  This is not merely a formal convenience:  the tangent space is an essential tool in \emph{proving} that moduli functors are representable (see \S\ref{S:def-thy} and \S\ref{S:Artin}).

As an example of the definition, we compute the tangent space of the Grassmannian.  The same calculation actually computes the tangent space of the Quot scheme, as well (see also \cite[\S 6.4]{FGAe}).
  Notice that we do not make use anywhere of the representability of the Grassmannian in this calculation.  Further examples will appear in \S\ref{S:homogeneity}.

We have already seen on p.\pageref{p:GrCov}  how the cover of $G(r,n)$ by open subsets $U_p$ associated to projections $p : \mathbb C^n \rightarrow \mathbb C^r$ is visible from the functor $\mathscr{G}(r,n)$.  From this, near a point of the Grassmannian  $W\hookrightarrow \mathbb C^n$,  we get a local identification of $\mathscr{G}(r,n)$ with the functor represented by $\operatorname{Hom}_{\mathbb C}(\mathbb C^r, \mathbb C^{n-r}) = \operatorname{Hom}(W, \mathbb C^n / W)$.  Since this is a vector space, this gives a local identification of the tangent space $T_{\mathscr{G}(r,n),W}$ with $\operatorname{Hom}(W, \mathbb C^n / W)$.  As is well known, this identification can be made globally.  Rather than verify the compatibility of this identification with gluing, which would require us to use the fact that $\mathscr{G}(r,n)$ is a sheaf, we will arrive at it via a global construction working directly from  the definition of $\mathscr{G}(r,n)$.

\begin{dfn}[Ring of dual numbers]
The \emph{ring of dual numbers} (over $\mathbb{C}$) is $\mathbb D = \mathbb{C}[\epsilon] / (\epsilon^2)$.  For any scheme $S$ over $\mathbb C$, we write $S[\epsilon]$ for $S \times \operatorname{Spec} \mathbb D$.
\end{dfn}

\begin{dfn}[Tangent space to a moduli functor]
If $\mathscr{F} : \mathsf S^{\mathsf{op}} \rightarrow (\mathsf{Set})$ is a   functor, the \emph{tangent bundle} to $\mathscr{F}$ is the following functor:
\begin{gather*}
T_{\mathscr{F}} : \mathsf S^{\mathsf{op}} \rightarrow (\mathsf{Set}) \\
T_{\mathscr{F}}(S) = \mathscr{F}\bigl(S[\epsilon]\bigr)
\end{gather*}
\end{dfn}

Reduction modulo $\epsilon$ gives a map $\mathbb D \rightarrow \mathbb{C}$ and therefore a closed embedding $S \rightarrow S[\epsilon]$ for any scheme $S$.  This induces natural functions
\begin{equation*}
T_{\mathscr{F}}(S) = \mathscr{F}\bigl(S[\epsilon]\bigr) \rightarrow \mathscr{F}(S)
\end{equation*}
and therefore a natural transformation
\begin{equation*}
T_{\mathscr{F}} \rightarrow \mathscr{F} .
\end{equation*}
We allow ourselves the following notational shortcut, which some may feel is abusive:  When $\xi \in \mathscr{F}(S)$, in other words, when $(S, \xi)$ is an $S$-point of $\mathscr{F}$, we write $T_{\mathscr{F}}(S, \xi)$ or $T_{\mathscr{F}}(\xi)$ for the fiber of $T_{\mathscr{F}}(S)$ over $\xi \in \mathscr{F}(S)$.

For instance, if we have  $\mathscr F=\operatorname{Hom}_{\mathsf S}(-,M)$ for some complex manifold $M$, and $\xi:\operatorname{Spec}\mathbb C\to M$ is a point of $M$, then $T_{\mathscr F}(\xi)=T_M(\xi)$, the holomorphic tangent space to $M$ at $x$.  

\begin{wrn}
The tangent space to an arbitrary functor $\mathscr{F}$ may not be at all well-behaved!  For example $T_{\mathscr{F}}(\xi)$ might not even be a vector space.  When trying to characterize the functors that are representable by schemes (or algebraic spaces or algebraic stacks) one of the first axioms we impose is that $\mathscr{F}$ should behave infinitesimally like a scheme, and in particular that its tangent spaces should be vector spaces (see \S\ref{S:def-thy}).
\end{wrn}

\subsubsection{Tangent space to the Grassmannian}
Suppose that $S$ is a scheme and $\mathbb F \hookrightarrow  \mathbb{C}_S^n$ is an $S$-point of $\mathscr{G}(r,n)$.  Denote by $F$ the sheaf of regular  sections of $\mathbb F$.  
Note that $F$ is a  quasicoherent locally free sheaf of rank $r$, and recall that $\mathbb F$ can be recovered from $F$, either from transition functions, or directly as 
\begin{equation*}
\mathbb F=\underline{\raisebox{0pt}[0pt][-.5pt]{$\operatorname{Spec}$}}_S(\operatorname{Sym}^\bullet F^\vee),
\end{equation*}  
where $F^\vee=\underline{\operatorname{Hom}}_{\mathcal O_S}(F,\mathcal O_S)$ is the sheaf of sections of the dual vector bundle  $\mathbb F^\vee=\operatorname{Hom}_S(\mathbb F,\mathbb C_S)$.   
By definition, $T_{\mathscr{G}(r,n)}(S,\mathbb F)$ is the set of isomorphism classes of extensions of $\mathbb F$ to a vector subbundle $\mathbb F_1 \subseteq \mathbb{C}^n_{S[\epsilon]}$.  Because vector bundles are determined by their sheaves of sections, deforming $\mathbb F$ is the same as finding a locally free  deformation $ F$ to $S[\epsilon]$.
One can easily check by dualizing that this is equivalent to finding a locally free deformation of $ F^\vee$ to $S[\epsilon]$.

Since $\mathbb F \hookrightarrow  \mathbb C^n_S$, we have a quotient $\mathcal{O}_S^n \rightarrow  F^\vee$ by duality.
We write $ E^\vee$ for the kernel of this quotient.  To find a locally free  deformation of $ F^\vee$ to $S[\epsilon]$  is the same as to complete the diagram below with a locally free $\mathcal{O}_{S[\epsilon]}$-module $ F_1^\vee$:
\begin{equation}\label{E:TanSpDiag1}
 \xymatrix{
& & \mathcal{O}_{S[\epsilon]}^n \ar@{-->}[r] \ar[d] &  F_1^\vee \ar@{-->}[d] \ar@{-->}[r] & 0 \\
0 \ar[r] &  E^\vee \ar[r] & \mathcal{O}_S^n \ar[r] &  F^\vee \ar[r] & 0. 
} \end{equation}
Since $ F_1^\vee$ is locally free, tensoring the exact sequence 

\begin{equation}\label{E:SmallTanSp}
0\to \epsilon \mathcal O_S\to \mathcal O_S[\epsilon]\to \mathcal O_S\to 0
\end{equation}
with $ F_1^\vee$ we see that the kernel of $ F_1^\vee \rightarrow  F^\vee$ is $\epsilon  F_1^\vee \simeq  F^\vee$.  Note, moreover, that a short computation shows that any quasicoherent sheaf $ F_1^\vee$ fitting into the diagram \eqref{E:TanSpDiag1} above with $\epsilon  F_1^\vee \simeq  F^\vee$
 will be locally free, because a local basis for $ F^\vee$ can be lifted via the projection $ F_1^\vee \rightarrow  F^\vee$ to a basis for $ F_1^\vee$.  This can also be seen using the infinitesimal criterion for flatness~\cite[Ex.~6.5]{Eisenbud}.

Therefore our problem is, equivalently, to complete the diagram below with $\mathcal{O}_{S[\epsilon]}$-modules $ E^\vee_1$ and $ F^\vee_1$, so that the middle row is a short exact sequence, and the vertical arrows are those induced from \eqref{E:SmallTanSp}:
 \begin{equation} \label{E:BigTanSp}
 \begin{gathered}
 \xymatrix@R=1.5em{
 &0\ar[d]&0\ar[d]&0\ar[d]&\\
 0 \ar[r]  & \epsilon  E^\vee \ar[r] \ar@{-->}[d]& \epsilon \mathcal{O}_S^n \ar[r] \ar[d] & \epsilon  F^\vee \ar[r]  \ar@{-->}[d]& 0\\
0\ar@{-->}[r] & E_1^\vee \ar@{-->}[r] \ar@{-->}[d]& \mathcal{O}_{S[\epsilon]}^n \ar@{-->}[r] \ar[d] &  F_1^\vee \ar@{-->}[d] \ar@{-->}[r] & 0 \\
0 \ar[r] &  E^\vee \ar[r] & \mathcal{O}_S^n \ar[r] &  F^\vee \ar[r] & 0.\\
&0\ar@{<-}[u]&0\ar@{<-}[u]&0\ar@{<-}[u]&\\
}
\end{gathered}
 \end{equation}
In particular, all rows and columns above are short exact sequences.

Since $\epsilon  E^\vee$ necessarily maps to zero in $ F_1^\vee$, to produce a quotient $ F_1^\vee$ of $\mathcal{O}_{S[\epsilon]}^n$ lifting the quotient $ F^\vee$ of $\mathcal{O}_S^n$, it is equivalent to produce a quotient of $\mathcal{O}_{S[\epsilon]}^n / \epsilon  E^\vee$.  We are therefore trying to complete the diagram below:

\begin{equation*} \xymatrix{
0 \ar[r] &  E_1^\vee / \epsilon  E^\vee \ar@{=}[d] \ar@{-->}[r] & \mathcal{O}_{S[\epsilon]}^n / \epsilon  E^\vee \ar@{-->}[r] \ar[d] &  F_1^\vee \ar@{-->}[d] \ar@{-->}[r] & 0 \\
0 \ar[r] &  E^\vee \ar[r] & \mathcal{O}_S^n \ar[r] &  F^\vee \ar[r] & 0.
} \end{equation*}
But note that $ E_1^\vee / \epsilon  E^\vee$ projects isomorphically to $ E^\vee$. 
 Therefore $T_{\mathscr{G}(r,n)}(S,F)$ is isomorphic to the set of lifts of the following diagram:
\begin{equation*} \xymatrix{
& \mathcal{O}_{S[\epsilon]}^n / \epsilon  E^\vee \ar[d] \\
 E^\vee \ar@{-->}[ur] \ar[r] & \mathcal{O}_S^n.
} \end{equation*}
We have one canonical lift by composing the inclusion $ E^\vee \hookrightarrow  \mathcal{O}_S^n$ with the section $\mathcal{O}_S^n \rightarrow \mathcal{O}_{S[\epsilon]}^n$.  (Note that this is not an $\mathcal{O}_{S[\epsilon]}$ homomorphism until it is restricted to $ E^\vee$.)  An easy diagram chase in \eqref{E:BigTanSp} shows that the difference between any  two lifts is a homomorphism $ E^\vee \rightarrow  F^\vee$, so we get a canonical bijection:
\begin{equation*}
T_{\mathscr{G}(r,n)}(S,\mathbb F) = \operatorname{Hom}_{\mathcal O_S}( E^\vee, F^\vee) = \operatorname{Hom}_{\mathcal O_S}( F, E)=\operatorname{Hom}_S(\mathbb F,\mathbb E)
\end{equation*}
where $\mathbb E=\underline{\raisebox{0pt}[0pt][-.5pt]{$\operatorname{Spec}$}}_S(\operatorname{Sym}^\bullet  E^\vee)$ is the quotient of the trivial vector bundle $\mathbb{C}_S^n$ by the subbundle $\mathbb F$.


\section{Moduli functors as categories fibered in groupoids}

As nontrivial automorphisms tend to preclude the representability of a moduli problem (by a scheme), one plausible way to proceed is to account for these automorphisms by retaining them in the  definition of the moduli functor.  We are led to consider functors valued in groupoids (categories in which all morphisms are isomorphisms), rather than in sets, and immediately find ourselves in a morass of technicalities (see \S \ref{S:lax}).   In our judgment, a more elegant solution 
can be found in the notion of a category fibered in groupoids {(see Remark~\ref{R:MgCFG} for the reasoning that motivates our point of view)}.   In \S \ref{S:lax} we briefly discuss functors valued in groupoids, with the primary objective of convincing the reader that another approach would be preferable.  We then introduce categories fibered in groupoids in \S \ref{S:CFGs}.

\subsection{(Lax $2$-)Functors to groupoids}  \label{S:lax}
Instead of studying a moduli problem by defining a functor whose value on a scheme is the \emph{set} of isomorphism classes of families over that scheme, we try to define one whose value is the \emph{category} of families parameterized by that scheme, with morphisms being isomorphisms of families.  Contrary to the set of isomorphism classes, the category explicitly allows two families to be isomorphic in more than one way.

\begin{dfn}[Groupoid]  A \emph{groupoid} is a category in which all morphisms are isomorphisms.
\end{dfn}

\begin{rem}
	In the literature, in the definition of a groupoid, it is common to require the additional condition that the category be small (the class of objects is a set).  This is not required for our treatment, and so we drop the condition since almost every category we will consider will not be small.  For instance, the category of sets with one element is obviously not small (for every set $E$ there is a one element set $\{E\}$), and one can immediately generalize this to examples we consider here.   However, the groupoids we study will usually be \emph{essentially small}, meaning they are equivalent to small categories, so none of the pathologies that compel one to include a hypothesis of smallness will trouble us.

	One way to avoid worrying about small categories (or at least, to transfer the worry to somewhere else) is to introduce Grothendieck universes.  Essentially, one assumes axiomatically that there are very large sets, 
	called universes, that are large enough to `do set theory' within.  All objects of interest occur within the universe, but one can still use set-theoretic language to speak about the universe itself.  For example, the collection of all $1$-element sets \emph{within the universe} does form a set, namely {one in bijection with} the universe itself. 
		
	Ultimately, these considerations are technical from our perspective, and we will remain silent about them in the sequel.  One may consult \cite[Tag 0007]{stacks} for yet another way around these technical issues.
\end{rem}

Continuing on, at first glance, it seems that  we want  a ``functor'' to the category of groupoids:
$$
\mathcal M:\mathsf S^{\mathsf {op}}\to (\mathsf {Groupoid}).
$$
Making this precise leads to the morass of technicalities alluded to above.  
The issue here is the assignment for morphisms.  In our examples, we pulled back families along morphisms.    The fact that pullbacks are only defined up to isomorphism, albeit a canonical one, means that one does not have an equality of  $g^\ast f^\ast$ with  $(fg)^\ast$ but only an isomorphism between them.  One then has not one but two ways of identifying $h^\ast g^\ast f^\ast \simeq (fgh)^\ast$ and one must require these be the same.

\begin{exa} \label{E:MgL2}
Here is how this plays out for the  moduli of curves:
Let $$\mathscr{M}^{L2}_g:\mathsf S^{\operatorname{op}}\to \mathsf {(Groupoid)}$$ be the moduli functor in groupoids  for curves of genus $g$; by definition $\mathscr{M}^{L2}_g(S)$ is the \emph{category} of families of curves over $S$, with $S$-isomorphisms as the morphisms.  If $T \rightarrow S$ is a morphism of schemes we obtain a functor $\mathscr{M}^{L2}_g(S) \rightarrow \mathscr{M}^{L2}_g(T)$ sending $C$ to $C \mathop{\times}_S T$.  If we have a pair of morphisms $U \rightarrow T \rightarrow S$ then we obtain two maps $\mathscr{M}^{L2}_g(S) \rightarrow \mathscr{M}^{L2}_g(U)$, one sending $C$ to $(C \mathop{\times}_S T) \mathop{\times}_T U$ and the other sending $C$ to $C \mathop{\times}_S U$.  These are canonically isomorphic, but they are not equal!  Do we have to keep track of this canonical isomorphism, as well as the compatibilities it must satisfy when we encounter a sequence $V \rightarrow U \rightarrow T \rightarrow S$?  
\end{exa}

Pursuing this line of reasoning, one ultimately arrives at the definition of a \emph{pseudo-functor} or \emph{lax $2$-functor}.  However, just to define lax $2$-functors is an unpleasant task, with little to do with the geometry that ultimately motivates us.  (The reader who desires one may find a definition in \cite[\S 3.1.2]{FGAe}.)  Fortunately, we will not have to think too hard about lax $2$-functors because Grothendieck has supplied a more elegant solution:  categories fibered in groupoids.  The fundamental observation is that the pullbacks we need  are canonically isomorphic because they satisfy universal properties that are literally the same.  If we keep track of universal properties rather than the objects possessing them, we arrive at a more efficient definition.  

\begin{rem} \label{R:MgCFG}
According to the philosphy behind categories fibered in groupoids, the mistake in Example~\ref{E:MgL2} was to \emph{choose} a fiber product $C \mathop{\times}_S T$.  We were then forced to carry it around and keep track of the compatibilities that it obviously satisfies.  The category fibered in groupoids posits only that an object satisfying the universal property of the fiber product \emph{exists}, i.e., that there is \emph{some} cartesian diagram~\eqref{E:CartCurves},
\begin{equation} \label{E:CartCurves} \vcenter{\xymatrix{
D \ar[r] \ar[d] & C \ar[d] \\
T \ar[r] & S
}} \end{equation}
without actually specifying a construction of one.
\end{rem}


\subsection{Categories fibered in groupoids} \label{S:CFGs}
From our perspective, the motivation for a  category fibered in groupoids is to avoid the technical complications of the definition of a   lax $2$-functor by essentially clumping all of the groupiods of interest  into one large category  over the category  of schemes; the issues we ran into defining the lax $2$-functor on morphisms are avoided by using the universal properties of  fibered products in our category.   We now make this precise.  
Temporarily, we let $\mathsf S$ denote any category. 

\begin{dfn} \label{D:CatOver}
A \emph{category over $\mathsf S$} is a pair  $(\mathcal M,\pi)$ consisting of a category $\mathcal M$ together with a covariant functor $\pi : \mathcal M \to  \mathsf S$.   If $S$ is an object of  $\mathsf S$, the \emph{fiber of $\mathcal M$ over $\mathsf S$}, denoted $\mathcal M_S$ or $\mathcal M(S)$ is defined to be the subcategory consisting of all  objects over $S$, and all morphisms over the identity of $S$.
An object \emph{$X$ in $\mathcal M$ is said to lie over $S$} if it is in $\mathcal M_S$.  A \emph{morphism $\tilde  f:X'\to X$ in $\mathcal M$ is said to lie over  a morphism $f:S'\to S$} if $\pi(\tilde f)=f$.
A morphism between categories $(\mathcal M, \pi)$ and $(\mathcal M', \pi')$ over $\mathsf S$ is a functor $F : \mathcal M \rightarrow \mathcal M'$ such that $\pi' \circ F = \pi$.
\end{dfn} 
  
We indicate objects and morphisms lying above other objects and morphisms in diagrams like this:
\begin{equation}\label{E:CoSdiag}
\xymatrix{
\mathcal M \ar@{->}[d]^\pi&X'\ar@{->}[r]^{\tilde  f}  \ar@{|->}[d]& X \ar@{|->}[d]\\
\mathsf S&S' \ar@{->}[r]^{f}& S.\\
}
\end{equation}

A fibered category is essentially one for which ``pullback diagrams'' exist.  
Keeping in mind the definition of a fibered product, a quick glance at   \eqref{E:CMdiag} should make clear the meaning of a pullback diagram in this setting.

\begin{dfn}[Cartesian morphism] \label{D:CartMorph}
Let $(\mathcal M,\pi)$ be a category over $\mathsf S$.  A morphism $\tilde f:X'\to X$ in $\mathcal M$ is \emph{cartesian} if the following condition holds.  
Denote by    $f : S' \to S$ the morphism $\pi(\tilde f)$ in   $\mathsf S$ (as in  \eqref{E:CoSdiag}).  
 Given any morphism $g:S''\to S'$ in $\mathsf S$, and any morphism $\widetilde {f\circ g}:X''\to X$ in $\mathcal M$ lying over $f\circ g$, 
 there is a unique  morphism $\tilde g:X''\to X'$ in $\mathcal M$ lying over $g$ such that $\widetilde {f\circ g} =\tilde  f \circ \tilde g$.    Pictorially, every diagram~\eqref{E:CMdiag} has a unique completion:
\begin{equation}\label{E:CMdiag}
 \begin{gathered}
\xymatrix {
X'' \ar@{->}@/^.5pc/[rrd]^{\widetilde{f\circ g}} \ar@{|->}@/_.0pc/[d] \ar@{-->}[rd]_{\exists !}^{\tilde g}&&\\
S'' \ar@{->}@/_.3pc/[rd]_{g} \ar@{.>}@/^.5pc/[rrd]^<>(0.7){f\circ g}& X'\ar@{->}[r]^<>(0.4){\tilde  f}  \ar@{|->}[d]& X \ar@{|->}[d]\\
 &S' \ar@{->}[r]_{f}& S.\\
}
 \end{gathered}
\end{equation}
\end{dfn}

\begin{dfn}[Fibered category]
A \emph{category} $(\mathcal M,\pi)$ over $\mathsf S$ is said to be  \emph{fibered   over $\mathsf S$} if for any  morphism   $f : S' \to S$ in $\mathsf S$ and any object $X$ of $\mathcal M$ lying over  $S$, there exists a cartesian  morphism  $\tilde  f: X' \to  X$ in $\mathcal M$ lying over  $f$. 
\end{dfn}

\begin{dfn}[Category fibered in groupoids (CFG)] \label{D:CFG}
A category $(\mathcal M, \pi)$ over $\mathsf S$ is said to be \emph{fibered in groupoids} if it is fibered over $\mathsf S$, and for every $S$ in $\mathsf S$, we have that $\mathcal M(S)$ is a groupoid.  
\end{dfn}

\begin{rem}
Categories fibered in groupoids are typically introduced as fibered categories in which all morphisms are \emph{cartesian}~\cite[Exp.\ VI, \S6, Remarques]{SGA1}, \cite[Def.~3.21]{FGAe}, or equivalently, a fibered category   where the  fibers are all  groupoids.  
In examples, as we have seen here, where we define all morphisms via fibered product diagrams, and pullbacks, one is led naturally to this definition.  
The added generality of fibered categories (not neccesarily fibered in groupoids) is important in order to formulate faithfully flat descent efficiently~\cite[\S{}B.3]{FGA1}, \cite[Exp.~VIII]{SGA1}, \cite[\S4.2]{FGAe}.  However, it is not particularly relevant to the study of moduli problems that is our focus here.

Fibered categories were originally defined in~\cite[\S{}A.1.a, Def.\ 1.1]{FGA1} to be what we would call lax $2$-functors (albeit valued in categories rather than in groupoids), and what others call \emph{pseudo-functors}~\cite[Exp.~VI, \S8]{SGA1}, 
 \cite[Def.~3.10]{FGAe}.  The definition of a fibered category given in \cite[Exp.\ VI, Def.\ 6.1]{SGA1}, is equivalent to the slightly different formulation in~\cite[Def.~3.5]{FGAe}.  Lax $2$-functors are equivalent (we do not attempt to say precisely in what sense) to fibered categories with \emph{cleavage} \cite[Exp.\ VI, \S7--8]{SGA1}, \cite[Def.~3.9, \S3.1.3]{FGAe}.  
\end{rem}

\begin{rem}
The notation $\mathcal M(S)$ for $\pi^{-1} (S)$ is meant to be suggestive of the relationship between categories fibered in groupoids and functors valued in groupoids.  Indeed, one may construct an equivalence between the notions such that the groupoid-valued functor associated to $\mathcal M$ has value $\mathcal M(S)$ on $S \in \mathsf S$.
\end{rem}

\begin{exa}[CFG associated to a presheaf] \label{E:S-CFG}
Let $\mathscr F:\mathsf S^{\mathsf {op}}\to (\mathsf {Set})$ be a functor in sets.  One obtains a CFG $\pi:\mathcal F\to \mathsf S$ in the following way.  For each $S\in \operatorname{Ob}(\mathsf S)$, let $\mathcal F_S=\mathscr F(S)$.    Let $S,S'\in \operatorname{Ob}(\mathsf S)$ and suppose that $X_S\in \mathcal F_S$ and $X_{S'}\in \mathcal F_{S'}$.  Then we assign a morphism $X_{S'}\to X_S$ in $\mathcal F$ if there is a morphism $f:S'\to S$ in $\mathsf S$, and $\mathscr F(f)(X_S)=X_S'$.    The functor $\pi:\mathcal F\to \mathsf S$ is given by sending $X_S$ to $S$, and similarly for morphisms.   

It is not difficult to generalize this construction to yield a CFG associated to a lax $2$-functor~\cite[\S3.1.3]{FGAe}, provided that one has first given a precise definition of the latter.
\end{exa}

\begin{exa}[CFG associated to an object of $\mathsf S$]
Let $S$ be in $\mathsf S$.   The  \emph{slice category} $\mathsf S/S$ is the CFG defined to have objects that are pairs $(S', f)$ where $S' \in \mathsf S$ and $f : S' \rightarrow S$ is a morphism.  A morphism $(S', f) \rightarrow (S'', g)$ is a morphism $h : S' \rightarrow S''$ such that $gh = f$:
\begin{equation*} 
\xymatrix@R=1em{
S' \ar[rr]^h \ar[dr]_f & & S'' \ar[dl]^g \\ & S
} \end{equation*}
The projection $\pi : \mathsf S/S \rightarrow \mathsf S$ is $\pi(S',f) = S'$.  A CFG equivalent to one of the form $\mathsf S/S$ for some $S$ in $\mathsf S$ is said to be  \emph{representable by $S$}.
\end{exa}

\begin{rem}[Agreement of CFGs for $S$ and $h_S$]
For any $S\in \operatorname{Ob}(\mathsf S)$, we can assign the functor $h_S$, and the category fibered in groupoids associated to $h_S$. This agrees with the CFG $\mathsf S/S$.  
\end{rem}

\begin{exa}[The CFG (sieve) associated to a family of maps]\label{E:CFGSiS}  Given a collection of morphisms $\mathcal R=\{S_i\to S\}$ in $\mathsf S$, we define an associated full sub-CFG of $\mathsf S/S$, which will also be denoted by $\mathcal R$.  The objects of $\mathcal R$ are the objects  $S'\to S$ of $\mathsf S/S$ that factor through one of the $S_i\to S$.
Sub-CFGs of representable CFGs (i.e., a sub-CFGs of $\mathsf S/S$ for some $S$ in $\mathsf S$) are known as \emph{sieves}.
\end{exa}

To give the full statement of the Yoneda lemma for CFGs we need another definition.  The correct language for discussing this is that of the $2$-category of CFGs over $\mathsf S$.   We postpone this more technical discussion until later (\S \ref{S:2-cat}).  Here we give a working definition that suffices for our purposes.

\begin{dfn}
Let $p_{\mathcal M}:\mathcal M\to \mathsf S$ and $p_{\mathcal N}:\mathcal N\to \mathsf S$ be CFGs over $\mathsf S$.   There is a category $\operatorname{Hom}_{\mathsf {CFG/S}}(\mathcal M,\mathcal N)$ with objects being morphisms $\mathcal M\to \mathcal N$ of fibered categories over $\mathsf S$ and morphisms being  natural isomorphisms.
\end{dfn}

\begin{lem}[{2-Yoneda \cite[\S 3.6.2]{FGAe}}]  \label{L:2-Yon}  
Let $\mathsf S$ be a category, and let $S\in \operatorname{Ob}(\mathsf S)$.    For any fibered category $\pi:\mathcal M\to \mathsf S$ the natural transformation
\begin{gather*}
\operatorname{Hom}_{\mathsf {CFG/S}}(\mathsf S/S,\mathcal M)\to \mathcal M(S) \\
F \mapsto F(S \xrightarrow{\mathrm{id}_S}S) 
\end{gather*}
(defined similarly for morphisms) is an equivalence of categories.
\end{lem}

\begin{convention}
In view of the $2$-Yoneda lemma, we may introduce the following notation:  Suppose that $S \in \mathsf S$ and that $\mathcal M$ is a CFG over $\mathsf S$.  We write $X : S \rightarrow \mathcal M$ to mean $X \in \operatorname{Hom}_{\mathsf{CFG}/\mathsf{S}}(\mathsf S/S, \mathcal M)$.  By the Yoneda lemma, this is the same as specifying an object of $\mathcal M(S)$, and we frequently do not distinguish notationally between $X : S \rightarrow \mathcal M$ and $X \in \mathcal M(S)$.
\end{convention}

\begin{cor}[{\cite[\S 3.6]{FGAe}}]  \label{C:2-Yon}  
Let $\mathsf S$ be a  small category, and let $S,S'\in \operatorname{Ob}(\mathsf S)$.     The map
$$
\operatorname{Hom}_{\mathsf S}(S',S)\to  \operatorname{Ob}\left(\operatorname{Hom}_{\mathsf {CFG/S}}(\mathsf S/S',\mathsf S/S)\right)
$$
obtained by post-composition of arrows (e.g., a morphism $f:S'\to S$ is sent to the functor that assigns to an arrow $(g:S''\to S')\in \operatorname{Ob}(\mathsf S/S')$, the composition $(f\circ g:S''\to S)\in \operatorname{Ob}(\mathsf S/S)$) is a bijection.  
\end{cor}

\begin{rem}
Let $\mathscr M,\mathscr N:\mathsf S^{\mathsf {op}}\to (\mathsf {Set})$ be two functors.    Let $\mathcal M,\mathcal N$ be the associated categories fibered in groupoids over $\mathsf S$.  In a similar way, there is a bijection 
$$
\operatorname{Hom}_{\mathsf {Fun}(\mathsf S^{\mathsf {op}},(\mathsf {Set}))}(\mathscr M,\mathscr N)\to  \operatorname{Ob}\left( \operatorname{Hom}_{\mathsf {CFG}/\mathsf S}(\mathcal M,\mathcal N)\right).
$$
We delay introducing the notion of $2$-categories until later (\S \ref{S:2-cat}), but the consequence of the $2$-Yoneda Lemma, and this observation, is that the category $\mathsf S$, and the category of functors $\mathsf {Fun}(\mathsf S^{\mathsf {op}},(\mathsf {Set}))$ can be viewed as full $2$-subcategories of the $2$-category of CFGs over $\mathsf S$.  Consequently, we  will frequently identify objects $S$ of $\mathsf S$, and functors to sets $\mathscr M:\mathsf S^{\mathsf {op}}\to (\mathsf {Set})$ with their associated CFGs, $\mathsf S/S$, and $\mathcal M$, respectively.  
\end{rem}


\subsubsection{The Grassmannian as a category fibered in groupoids} \label{SS:Gr-CFG}
For this section, let $\mathsf S$ be the category of schemes over $\mathbb C$.  The Grassmannian CFG
$$
\pi:\mathcal G(r,n)\to \mathsf S
$$
is defined as follows.  Objects of $\mathcal G(r,n)$ are pairs $(S, \mathbb F)$ where $S$ is a scheme and $\mathbb F \hookrightarrow  \mathbb C^n_S$ is a vector subbundle of rank~$r$.  Morphisms in $\mathcal G(r,n)$ are defined via pullback.  More precisely, a morphism $(S',\mathbb F') \rightarrow (S,\mathbb F)$ is a diagram
$$
\xymatrix@C=.3cm @R=.3cm{
\mathbb F' \ar@{->}[rdd]  \ar@{^(->}@<-1pt>[rrd] \ar@<2pt>@{->}[rrr]&&&\mathbb F\ar@{^(->}[rrd] \ar@{->}[rdd]&&&\\
&& \mathbb C^n_{S'} \ar@{->}[rrr]_<>(0.25){\mathbb C^n_f} \ar@{->}[ld]&&& \mathbb C^n_S \ar@{->}[ld]&\\
&S' \ar@{->}[rrr]_{f}&&&S&&\\
}
$$
where $f:S'\to S$ is a morphism in $\mathsf S$ and the rectangles are all cartesian.   The map  $\pi:\mathcal G(r,n)\to \mathsf S$ is the forgetful functor $\pi(S,\mathbb F) = S$.

\begin{rem}\label{R:Gr-CFG}
To dispel any confusion that may arise from the notation in Example \ref{E:S-CFG}, we emphasize that the CFG associated to the functor $\mathscr G(r,n)$ is equivalent to  $\mathcal G(r,n)$. 
 The former is fibered in sets whereas the latter is fibered in groupoids that are equivalent, but not isomorphic, to sets.
 We use the two notations to emphasize that one is a CFG and one is a functor.
\end{rem}

\subsubsection{The CFG of curves of genus $g$} \label{S:CFG-curves}
   Again we take $\mathsf S$ to be the category of schemes over $\mathbb C$.  We define the CFG of genus $g$ curves
$$
\pi:\mathcal M_g\to \mathsf S
$$
in the following way.  Objects of $\mathcal M_g$ are pairs $(S, X)$ where $X\to S$ is a relative curve of genus $g$ (see \S\ref{S:mod-RS}).  Morphisms in $\mathcal M_g$ are defined via pullback.  More precisely, a morphism from $(S',X') \rightarrow (S,X)$ is a cartesian  diagram
\begin{equation*} \xymatrix{
X' \ar[r] \ar[d] & X \ar[d] \\
S' \ar[r]^f & S
} \end{equation*}
where $f:S'\to S$ is a morphism in $\mathsf S$.  The map $\pi:\mathcal M_g\to \mathsf S$ is $\pi(S,X) = S$.

\begin{rem}
We emphasize  that the CFG associated to the functor $\mathscr M_g$ is not equivalent to $\mathcal M_g$.  The former is fibered in sets whereas the latter is fibered in groupoids that are not equivalent to sets.
\end{rem}


\section{Stacks}\label{S:stack}

Stacks are the categories fibered in groupoids that respect topology, in the sense that compatible locally defined morphisms into a stack can be glued together into global morphisms.  This is the most basic requirement a category fibered in groupoids must satisfy in order to be studied geometrically.

As usual, algebraic geometry introduces a troublesome technicality here:  the Zariski topology is 
 much too coarse to do much interesting gluing.  Indeed, suppose that $\mathcal F$ is a CFG that `deserves to be studied geometrically' and consider a scheme $S$ with a free action of $G = \mathbb{Z} / 2 \mathbb{Z}$ and a $\xi \in \mathcal{F}(S)$ that is equivariant with respect to the $G$-action.  If $\mathcal{F}$ were a stack in the complex analytic topology, we could descend $\xi$ to an element of $\mathcal{F}(S/G)$ because $S \rightarrow S/G$ is a covering space, an in particular a local homeomorphism, hence a cover in the analytic topology.  There is no such luck in the Zariski topology, where $S \rightarrow S/G$ may fail to be a local homeomorphism and may therefore also fail to be a cover in the Zariski topology. 

It is important to be able to do this kind of descent, so one must introduce an abstract replacement for the concept of a topology, called a Grothendieck topology.  The essence of the definition is to isolate exactly the aspects of topology that are necessary to speak about gluing. 
It is possible to express the sheaf conditions without ever making reference to points, or even to open sets.  All that is needed is the concept of a cover.  Grothendieck's definition actually goes further, replacing even the concept of a cover with the abstract notion of a sieve.  Even though topologies afford a few pleasant properties that pretopologies do not (see \S \ref{S:top} for more details), we will primarily limit our discussion to pretopologies in this introduction.  

\begin{rem}
Lest it appear genuinely pointless to remove the points from topology, consider the vast expansion of settings that can be considered topologically by way of Grothendieck topologies.  To take just one, Quillen~\cite{Quillen-HA} and Rim~\cite[Exp.~VI]{sga7-1} were able to understand extensions of commutative rings---that is, deformations of affine schemes---by putting a Grothendieck topology on the category of \emph{commutative rings} (see \cite{Gaitsgory,def-rings,coh-rings} for further developments of this idea).
\end{rem}

We begin this section by reviewing the definition of a Grothendieck pretopology, using sheaves on a topological space as our motivation.   We then define the isomorphism presheaf, and descent data, and finally, we give the definition of a stack.


\subsection{Sheaves and pretopologies}
We take as  motivation for Grothendieck pretopologies the definition of a sheaf on a topological space.

\subsubsection{Sheaves on a topological space}

Let $X$ be a topological space.  For any open subsets $U' \subseteq U$ of $X$, let $\iota_{U',U}$ denote the inclusion of $U'$ inside of $U$. Define the  category of open sets on $X$, 
$
\mathsf O_X
$, 
with the following objects and morphisms:
\begin{gather*}
\operatorname{Obj}\mathsf O_X=\{U\subseteq X: U \text{ open}\} \\
\operatorname{Hom}_{\mathsf O_X}(U',U)= \begin{cases}
\{ \iota_{U',U} \} &   U'\subseteq U\\
\varnothing & \text {else} 
\end{cases}
\end{gather*}

A \emph{presheaf (of sets)} is a functor
$$
\mathscr F:\mathsf O_X^{\mathsf {op}}\to (\mathsf {Set}).
$$
Given $a\in \mathscr F(U)$, and a subset $\iota_{U',U}:U'\subseteq U$, we denote by $a|_{U'}$, or $\iota_{U',U}^*(a)$,  the image of $a$ under the map $\mathscr F(\iota_{U',U}):\mathscr F(U)\to \mathscr F(U')$.

A presheaf is: 
\begin{enumerate}
\item \emph{separated} if, given an open  cover $\{U_i\to U\}$ and two sections $a$ and $b$ in $\mathscr F(U)$ such that  $a|_{U_i}=b|_{U_{i}}$ in  $\mathscr F(U_i)$  for all $i$, one has  $a=b$.

\item a \emph{sheaf} if, given an open cover $\{ U_i \rightarrow U \}$ with intersections $U_{ij} = U_i \cap U_j$, and elements $a_i \in \mathscr F(U_i)$ for all $i$, satisfying $a_i \big|_{U_{ij}} = a_j \big|_{U_{ij}}$ for all $i$ and $j$, there is a unique $a \in \mathscr F(U)$ such that $a \big|_{U_i} = a_i$ for all $i$.
\end{enumerate}

A morphism of presheaves, separated presheaves, or sheaves is a natural transformation of functors.

Now, to motivate the definition of a pretopology, a presite,  and a sheaf on a presite, we rephrase this definition  in the language of equalizers.  Recall that if $U'$ and $U''$ are open subsets of $U$, then $U'\cap U''=U'\times_U U''$.  
Given any open cover $\mathcal U=\{U_i\to U\}$, denote by $\mathscr F(\mathcal U)$ the equalizer of diagram~\eqref{E:ShRMs}:
 \begin{equation}\label{E:ShRMs}
\xymatrix{
  \displaystyle \prod_i \mathscr F(U_i)  \ar@{->}@<-2pt>[r]_-{\mathrm{pr}_2^*} \ar@{->}@<2pt>[r]^<>(0.5){\mathrm{pr}_1^*}& \displaystyle \prod_{i,j} \mathscr F(U_i\mathop{\times}_U U_j).
}
\end{equation}
Recall that the equalizer is the categorical limit for morphisms into the diagram; in other words, 
we obtain a diagram
$$
\xymatrix{
\mathscr F(\mathcal U) \ar@{->}[r]& \prod_i \mathscr F(U_i)  \ar@{->}@<-2pt>[r]_-{\mathrm{pr}_2^*} \ar@{->}@<2pt>[r]^-{\mathrm{pr}_1^*}& \prod_{ij} \mathscr F(U_i\times_U U_j),
}
$$
and the arrow on the left is universal (terminal) for morphisms into \eqref{E:ShRMs}.  The natural map on the left in the diagram below is  induced by the restriction maps:
$$
\xymatrix{
\mathscr F(U) \ar@{->}[r]& \prod_i \mathscr F(U_i)  \ar@{->}@<-2pt>[r]_-{\mathrm{pr}_2^*} \ar@{->}@<2pt>[r]^-{\mathrm{pr}_1^*}& \prod_{ij} \mathscr F(U_i\times_U U_j).
}
$$
By the universal property of the equalizer, this induces a map:
$$
\mathscr F(U)\to \mathscr F(\mathcal U)
$$
The sheaf conditions have the following translations into the language of equalizers.

\begin{lem}[{\cite[Cor.~2.40]{FGAe}}]
\label{L:SheafEq}
 Let $\mathscr F:\mathsf O_X^{\mathsf {op}}\to (\mathsf {Set})$ be a presheaf.  
\begin{enumerate}
\item $\mathscr F$ is separated if and only if  $\mathscr F(U)\to \mathscr F(\mathcal U)$ is injective for all open $U$ in $X$ and all open covers $\mathcal U$ of $U$.

\item $\mathscr F$ is a sheaf if and only if $\mathscr F(U)\to \mathscr F(\mathcal U)$ is a bijection for all open $U$ in $X$ and all open covers $\mathcal U$ of $U$.  
\end{enumerate}
\end{lem}

The main takeaway from this discussion is that we can repackage the sheaf condition in terms of fibered products and equalizers.  This provides the motivation for a Grothendieck pretopology, and a sheaf on a presite.


\subsubsection{Pretopologies}  
Temporarily denote by $\mathsf S$ any category.  
A \emph{Grothendieck pretopology} $\mathscr T$ on  $ \mathsf S$ consists of the following data: for each object $S$ in $\mathsf S$, a collection of families of maps $\{S_\alpha\to S\}$, called \emph{covers} of $S$   in $\mathscr T$ (or \emph{covering families} in $\mathscr T$),
such that:

\begin{enumerate}
\item [(PT 0)]  For all objects $S$ in $\mathsf S $, and for all morphisms $S_\alpha \to S$ which appear in some   covering family of $S$, and for all morphisms $S' \to  S$, the fibered product $S_\alpha \times_ S S'$ exists.

\item [(PT 1)]  For all objects $S$ in $\mathsf S$, all morphisms $S' \to  S$, and all  covering families $\{S_\alpha \to S\}$, the family $\{S_\alpha \times_S S' \to S'\}$  is a   covering family.

\item [(PT 2)]  If $\{S_\alpha \to S\}$ is a   covering family, and if for all $\alpha $, $\{S_{\beta\alpha} \to S_\alpha\}$  is a   covering family, then the family of composites $\{S_{\beta\alpha}\to S_\alpha\to S\}$
 is a   covering family.
 
\item [(PT 3)]   If $S' \to S$ is an isomorphism, then it  is a   covering family.
\end{enumerate}

\begin{exa}
Let $X$ be a topological space.  The category $\mathsf O_X$ together with open covers is a Grothendieck pretopology.  
\end{exa}

\begin{exa}
   Let $X$ be a topological space, define a category $\mathsf S$ to have as objects $\mathscr P(X)$, the set of all subsets of $X$, and as morphisms the inclusions.  We give every subset $S\subseteq X$ the induced topology.  Then the collection of all  open covers of subsets of $X$ gives a Grothendieck pretopology on $\mathsf S$.  Indeed, (PT0) is satisfied, since the fibered product is given by intersection.  (PT1) holds since we are giving every subset the induced topology, so an open  cover of a superset gives an open cover of  a subset.  (PT2) holds since refinements of open covers are open covers.  (PT3) holds since isomorphisms are equalities.  

\end{exa}

\begin{exa} \label{E:top-on-CFG}
Let $\pi : \mathcal X \rightarrow \mathsf S$ be a CFG over a category $\mathsf S$ equipped with a Grothendieck pretopology.  Call a family of maps $\{ X_i \rightarrow X \}$ in $\mathcal X$ covering if $\{ \pi(X_i) \rightarrow \pi(X) \}$ is covering in $\mathsf S$.  Then this determines a Grothendieck pretopology on $\mathcal X$.
To verify this, it may be helpful to observe that the induced   morphism  of CFGs $\mathcal X / X\to \mathsf S/\pi(X)$ is an  equivalence of categories; indeed, by the definition, $\mathcal X$ being a CFG implies the morphism   is essentially surjective, and fully faithful.  
\end{exa}

\begin{dfn}[Presite]\label{D:Presite}
A pair $(\mathsf S,\mathscr T)$ consisting of a category $\mathsf S$ together with a Grothendieck pretopology $\mathscr T$  is called a \emph{presite}.  Often $\mathscr T$ is left tacit and one uses $\mathsf S$ to stand for both the presite and its underlying category of objects.
\end{dfn}

\begin{exa}[Covers in the \'etale pretopology] \label{E:etale}
  The primary example of a Grothendieck pretopology that we will use is the \'etale pretopology  on the category of schemes.  We denote the associated presite by $\mathsf S_{\operatorname{et}}$.   
   Covers in $\mathsf S_{\operatorname{et}}$ are collections of  jointly surjective \'etale morphisms.   Recall that \'etale morphisms are the algebro-geometric analogue of local isomorphisms in the complex analytic category.  

Given a pretopology on $\mathsf S$ and a scheme $S$ in $\mathsf S$, we obtain the category $\mathsf S/S$, and an induced pretopology defined in the obvious way.   For instance, $(\mathsf S/S)_{\operatorname{et}}$ has covers given by jointly surjective \'etale morphisms in that category.  
  \end{exa}

  \begin{rem}[Analytic category]
  Readers who prefer working in the analytic category, 
should feel free   to  take $\mathsf S$ to be the category of complex analytic spaces, and to work with the pretopology $\mathscr T$ generated by the  usual open covers of complex analytic spaces,  in the analytic topology.   In fact, one could as easily  take $\mathsf S$ to be the category of complex manifolds, with smooth morphisms (so that fibered products remain in the category), and work with the pretopology $\mathscr T$ generated by the  usual open covers of complex manifolds.  
\end{rem}

\begin{exa}[Standard pretopologies on schemes]
The most commonly used Grothendieck pretopologies on the category of schemes are the:
\begin{itemize}
\item \emph{Zariski pretopology},
\item \emph{\'etale pretopology},
\item \emph{faithfully flat finite presentation (fppf) pretopology},
\item \emph{faithfully flat quasicompact (fpqc) pretopology}.
\end{itemize}
  Each of these pretopologies is a refinement of the one preceding it.  We will write $\mathsf S_{_{\operatorname{Zar}}}$, $\mathsf S_{_{\operatorname{et}}}$, etc., for the respective presites.  The Zariski pretopology (covers by Zariski open sets) is too coarse for most of the applications we have in mind.  For simplicity, we will work almost exclusively with the \'etale pretopology.  
  \end{exa}

\subsubsection{Sheaves on a presite}

The definition of a pretopology is exactly set up to allow us to make a definition of  a sheaf following our discussion of sheaves on topological spaces.     A \emph{presheaf (of sets)} on $\mathsf S$ is just a functor
$$
\mathscr F:\mathsf S^{\mathsf {op}}\to (\mathsf {Set}).
$$

Given any cover $\mathcal R=\{S_i\to S\}$, denote by $\mathscr F(\mathcal R)=\mathscr F(\{S_i\to S\})$ the equalizer of the diagram
  $$
\xymatrix{
  \prod_i \mathscr F(S_i)  \ar@{->}@<-2pt>[r]_-{\mathrm{pr}_2^*} \ar@{->}@<2pt>[r]^-{\mathrm{pr}_1^*}& \prod_{ij} \mathscr F(S_i\times_S S_j).
}
$$
As before, there  is a natural map
$$
\mathscr F(S)\to \mathscr F(\mathcal R).
$$

Following Lemma \ref{L:SheafEq}, we make the following definition:

\begin{dfn}[Sheaf on a site] Let $\mathscr F$ be a  presheaf  on a  presite $(\mathsf S,\mathscr T)$.
\begin{enumerate}
\item $\mathscr F$ is \emph{separated}   if  $\mathscr F(S)\to \mathscr F(\mathcal R)$ is an injection for every covering family $\mathcal R$ of every object $S$ of $\mathsf S$.

\item $\mathscr F$ is a \emph{sheaf}  if $\mathscr F(S)\to \mathscr F(\mathcal R)$ is a bijection for every covering family $\mathcal R$ of every object $S$ of $\mathsf S$.  
\end{enumerate}
\end{dfn}

  \begin{rem}
On occasion we will be more precise about how fine a  pretopology can be used to obtain a given statement.  For instance, we may specify that a presheaf is a sheaf with respect to the fpqc pretopology, which implies it is also a sheaf with respect to all of the other pretopologies mentioned above. 
\end{rem}

\begin{dfn}[Subcanonical presite]
A pretopology $\mathscr T$ on a category $\mathsf S$ is called \emph{subcanonical} if every representable functor on $\mathsf S$ is a sheaf with respect to $\mathscr T$.   A presite $(\mathsf S,\mathscr T)$ is called subcanonical if $\mathscr T$ is subcanonical.  
\end{dfn}

\begin{teo}[{Grothendieck \cite[Thm.~2.55]{FGAe}}]
Let $S$ be a scheme.   The presite  \emph{$(\mathsf S/S)_{_{\operatorname{fpqc}}}$} is subcanonical; in particular  \emph{$(\mathsf S/S)_{_{\operatorname{et}}}$} is subcanonical.
\end{teo}


\subsection{The isomorphism presheaf} \label{S:isomPr}
Let $\pi:\mathcal M\to \mathsf S$ be a CFG.  
If we view $\mathcal M$ as a space, as is our intention, and view an object $X \in \mathcal M(S)$ as a map $S \rightarrow \mathcal M$ then we expect that morphisms into $\mathcal M$ that agree locally should agree globally.  The only proviso is that because two maps into $\mathcal M$ can agree with one another in more than one way, we must interpret local agreement of $X$ and $Y$ to include not only choices of local isomorphisms between $X$ and $Y$ over a cover, but also compatibility of these choices on the overlaps in the cover.

In order to state this condition precisely, we introduce the presheaf of witnesses to the agreement of objects of $\mathcal M$.  The condition we want to impose is that this presheaf be a sheaf.

\begin{dfn}[Isomorphism presheaf]\label{D:IsomPS}
Let $\pi:\mathcal M\to \mathsf S$ be a CFG.  
Given $S$ in $\mathsf S$, and $X,Y\in \mathcal M(S)$, we obtain a presheaf 
$$
\mathscr {I}\!\!som_{\mathcal M}(X,Y):(\mathsf S/S)^{\mathsf {op}} \to (\mathsf {Set})
$$
in the following way.  For every object  $(S'\to S)$ in $\mathsf S/S$, we set
$$
\mathscr {I}\!\!som_{\mathcal M}(X,Y)(S'\to S):= \operatorname{Isom}_{\mathcal{M}(S')}(X \big|_{S'}, Y \big|_{S'})
$$
Thus $\mathscr {I}\!\!som_{\mathcal M}(X,Y)(S'\to S)$ consists of all isomorphisms $\alpha : X \big|_{S'} \rightarrow Y \big|_{S'}$ in $\mathcal M$ that lie over the identity $\mathrm{id}_{S'}$ in $\mathsf S$.  
The assignment for morphisms is left to the reader (see \cite[p.62]{FGAe}).
\end{dfn}

\begin{rem}
An observant reader will note that the restrictions in the definition of $\mathscr {I}\!\!som_{\mathcal M}(X,Y)$ depend, albeit only up to a canonical isomorphism, on a choice of inverse to the functor $\operatorname{Hom}(S, \mathcal M) \rightarrow \mathcal M(S)$, as guaranteed by the Yoneda lemma.
\end{rem}

\begin{rem}
Concretely, to say that isomorphisms form a sheaf means the following.
 Given a cover $\{S_i\to S\}$ in the pretopology on $\mathsf S$, and any collection of isomorphisms $\alpha_i:X|_{S_i}\to Y|_{S_i}$ over the identity on $S_i$ such that $\alpha_{i}|_{S_{ij}}=\alpha_j|_{S_{ij}}$, there is a unique isomorphism $\alpha :X\to Y$ such that $\alpha|_{S_i}=\alpha_i$.    Here we are using the shorthand $S_{ij}:=S_i\times_S S_j$.  
\end{rem}

\begin{dfn}[Prestack]\label{D:Prestack}
A CFG $\mathcal M\to \mathsf S$ such that for  every $S$ in $\mathsf S$, and every $X,Y$ in $ \mathcal M(S)$,  the presheaf 
$\mathscr {I}\!\!som_{\mathcal M}(X,Y)$ is a sheaf, is called a \emph{prestack}.  We will also say that isomorphisms are a sheaf.
\end{dfn}

\begin{rem}
The notation would be more consistent with sheaf notation  (see Proposition \ref{P:StCond}) if we called categories fibered in groupoids (or fibered categories) `prestacks' and called prestacks `separated prestacks'.  But we will keep to tradition.  
\end{rem}

It is typically very easy to prove that a category fibered in groupoids arising from a moduli problem is a prestack in a subcanonical topology.  This is because an object $Y \in \mathcal M(S)$ is usually representable by a scheme, perhaps with some extra structure or properties, and descending local isomorphisms between $X \in \mathcal M(S)$ and $Y$ amounts to descending locally defined morphisms from $X$ to $Y$, which is automatic because $Y$ represents a sheaf in any subcanonical topology!

\begin{exa} \label{E:MgPS}
Let us see how this works concretely in the example of $\mathcal M_g$.  Let $S$ be a scheme in $\mathsf S$, and let $X\to S$ and $Y\to S$ be relative curves of genus $g$.  Suppose that there exists an \'etale cover $\{S_i\to S\}$ so that for each $S_i$ there are $S_i$-isomorphisms $\alpha_i:X_{S_i}\to Y_{S_i}$   such that $\alpha_{i}|_{S_{ij}}=\alpha_j|_{S_{ij}}$ (using the shorthand $S_{ij} = S_i \mathop{\times}_S S_j$).  These correspond by the universal property of fiber product to morphisms $X_{S_i} \rightarrow Y$ and $X_{S_{ij}} \rightarrow Y$ satisfying the same compatibility condition.  As $Y$ represents a sheaf and $\{X_{S_i}\to X\}$ is  a cover of $X$ in the pretopology, these glue to a morphism $\alpha : X \rightarrow Y$ such that $\alpha \big|_{S_i} = \alpha_{S_i}$.  To check that this is in fact a morphism over $S$, note that the commutativity of the diagrams~\eqref{E:local-comm} 
\begin{equation} \label{E:local-comm} \vcenter{\xymatrix{
X_{S_i} \ar[r] \ar[d] & Y \ar[d] \\
S_i \ar[r] & S
}} \end{equation}
for all $i$ implies the commutativity of diagram~\eqref{E:global-comm},
\begin{equation} \label{E:global-comm} \vcenter{ \xymatrix@R=1.5em{
X \ar[rr]^\alpha \ar[dr] & & Y \ar[dl] \\
& S
}} \end{equation}
this time because morphisms into $S$ are a sheaf.
\end{exa}

In order to make a precise general statement along these lines, we make a very general definition:

\begin{dfn}[{Stable class of arrows \cite[Def.~3.16, p.48]{FGAe}}]\label{D:StableClass}
A class of arrows $\mathbf P$ in a category $\mathsf S$ is \emph{stable (under base change)} if morphisms in $\mathbf P$ can be pulled back via arbitrary morphisms in $\mathsf S$, and the result of any such pullback is also in $\mathbf P$.  
 \end{dfn}

\begin{exa}[CFG associated to a stable class of arrows] \label{E:CFG-for-stable-class}
Given {any stable class of  arrows} $\mathbf P$ in a category $\mathsf S$, one may make $\mathbf P$ into a category $\mathcal P$ by setting  objects to be arrows in $\mathbf P$ and setting morphisms to be cartesian squares.   There is a morphism  $\mathcal P\to \mathsf S$ given by sending an object $X\to S$ to the target $S$ (and similarly for morphisms).    
Then one can check that $\mathcal P\to \mathsf S$ is a CFG if and only if $\mathbf P$ is a stable class of arrows.
\end{exa}

\begin{teo}[{\cite[Prop.~4.31, p.88]{FGAe}}] \label{T:SCPS}  Let $(\mathsf S,\mathscr T)$ be a subcanonical presite, and $\mathbf P$ a stable class of arrows.  Let $\pi : \mathcal P \to \mathsf S$ be the associated CFG (Example~\ref{E:CFG-for-stable-class}).  Then $(\mathcal P, \pi)$ is a prestack.
\end{teo}

\begin{rem}
The discussion in Example~\ref{E:MgPS} adapts in a straightforward way to a proof of Theorem~\ref{T:SCPS}.
\end{rem}

\begin{cor}
The CFGs $\mathcal G(r,n)$ and $\mathcal M_g$ are prestacks in the \'etale topology on schemes.
\end{cor}

\begin{proof}
The case of $\mathcal M_g$ follows directly from the theorem with $(\mathsf S,\mathscr T)=\mathsf S_{\operatorname{et}}$ and $\mathcal P$ the class of relative curves of genus $g$.  The case of $\mathcal G(r,n)$ requires a slight modification (since there is more than one arrow in the definition of the objects), but is essentially the same.  
\end{proof}


\subsection{Descent for categories fibered in groupoids} 
\label{S:Descent}

Recall that we can view vector bundles either as global geometric objects admitting local trivializations, or alternatively, as collections of locally trivial objects together with  transition functions, which satisfy the cocycle condition.   We now use this motivation to define the notion of descent datum in the setting of CFGs.   This is  the last definition we will need before defining a stack.

\subsubsection{Vector bundles on open subsets of a complex manifold} \label{S:vector-bundle-descent}
Let $X$ be a complex manifold.  Let $\mathsf O_X$ be the category of open sets.  Give $\mathsf O_X$  the structure of a presite in the usual way, by taking covering families to be  open covers.    Define a CFG of vector bundles 
$$
\pi:\mathcal V^r_X\to \mathsf O_X
$$
in the following way.  For each $U\subseteq X$ open, the objects in the  fiber $\mathcal V_U^r$ are the  rank $r$, holomorphic vector bundles on $U$.  Morphisms in $\mathcal V^r$ are given by pullback diagrams.  
Consider an open covering $\{U_i \subseteq  U\}$.  We use the notation $U_{ij} = U_i \cap U_j$ and $U_{ijk} = U_i \cap U_j \cap U_k$ for the double and triple overlaps.   Suppose we are given a vector bundle $E_i$  over  $U_i$ for every $i$,  and an isomorphism $\alpha_{ij} : E_i | _{U_{ij}} \to  E_j | _{U_{ij}}$ for every $i, j$,  all of which satisfy the cocycle condition $\alpha_{ik} = \alpha_{jk} \circ \alpha_{ij}$ over $U_{ijk}$.  Then there exists a vector bundle $E$ lying over  $U$,  together with isomorphisms $\alpha_i: E |_{U_i} \to  E_i$  such that $\alpha_{ij} = \alpha_j |_{U_{ij}} \circ (\alpha_i |_{U_{ij}})^{-1}$. 

This means precisely that the category fibered in groupoids of vector bundles on a topological space satisfies descent and therefore is a stack in the usual topology.  The formulation of descent and the definition of a stack axiomatize this familiar gluing process.

\subsubsection{Intuitive definition of descent} \label{S:descent}  In this section we will give a direct translation of the gluing condition for vector bundles encountered in \S\ref{S:vector-bundle-descent} in the context of categories fibered in groupoids.  While intuitive, this formulation has both technical and practical definiencies.  We correct these in \S\ref{S:descent-gluing} and \S\ref{S:descent-sieves}, but a reader looking to develop intuition about stacks may safely ignore these matters, especially in a first reading.

\begin{dfn}[Descent using gluing data] \label{D:ConcDS}
Let $\mathcal M$ be a category fibered in groupoids over a presite $(\mathsf S,\mathscr T)$ with a cleavage (Definition~\ref{D:CFG}).  A \emph{descent datum for $\mathcal M$ over a space $S$} is the following: a covering $\{S_i \to  S\}$; for every $i$, an object $X_i$  over  $S_i$; for every $i, j$ an isomorphism $\alpha_{ij} : X_i | _{S_{ij}} \to  X_j | _{S_{ij}}$  in the fiber $\mathcal M(S_{ij})$, which satisfies the cocycle condition $\alpha_{ik} = \alpha_{jk} \circ \alpha_{ij}$ over $S_{ijk}$.
The descent datum is said to be  \emph{effective} if there exists an $X$ lying over  $S$,  together with isomorphisms $\alpha_i: X |_{S_i} \to  X_i$ in the fiber such that $\alpha_{ij} = \alpha_j |_{S_{ij}} \circ (\alpha_i |_{S_{ij}})^{-1}$.
\end{dfn}

\begin{rem}
From the example in \S\ref{S:vector-bundle-descent}, we see that every descent datum for the CFG $\mathcal V^r_X\to \mathsf O_X$ is effective.  
\end{rem}

We now discuss the category of descent data, and the meaning of effective descent data in this context.  

\begin{dfn}[The category of descent data] \label{D:dd}
Let $(\mathsf S,\mathscr T)$ be a presite, and let $\pi:\mathcal M\to \mathsf S$ be a CFG.  
Let $\mathcal R=\{S_i\to S\}$ be a covering in $\mathsf S$.  An \emph{object with descent data on $\mathcal R$}, or \emph{descent datum on $\mathcal R$}, is a collection $(\{X_i\},\{\alpha_{ij}\})$ of objects $X_i\in \mathcal M(S_i)$, together with isomorphisms $\alpha_{ij}:\operatorname{pr}_2^*X_j \cong \operatorname{pr}_1^* X_i$ in $\mathcal M(S_i\times_S S_j)$, such that the following cocycle condition is satisfied:
For any triple of indices $i,j,k$, we have the equality
\begin{equation} \label{E:cc}
\operatorname{pr}_{13}^*\alpha_{ik}=\operatorname{pr}_{12}^*\alpha_{ij}\circ \operatorname{pr}_{23}^*\alpha_{jk}:\operatorname{pr}_3^*X_k\to \operatorname{pr}_1^* X_i.
\end{equation}
An \emph{arrow between objects with descent data} 
$$
\{\phi_i\}:(\{X_i\},\{\alpha_{ij}\})\to (\{X_i'\},\{\alpha_{ij}'\})
$$
is a collection of arrows $\phi_i:X_i\to X_i'$ with the property that for each pair of indices, $i,j$, the diagram 
$$
\begin{CD}
\operatorname{pr}_2^*X_j@>\operatorname{pr}_2^*\phi_j>> \operatorname{pr}_2^*X_j'\\
@V\alpha_{ij}VV @V\alpha_{ij}'VV\\
\operatorname{pr}_1^*X_i@>\operatorname{pr}_1^*\phi_i>> \operatorname{pr}_1^*X_i'\\
\end{CD}
$$
commutes.
 The objects and morphisms above determine a \emph{category of descent data}  $\mathcal M(\{ S_i \rightarrow S \})$.
\end{dfn}

\begin{rem}
We have deliberately omitted a number of canonical isomorphism from the notation here.  This obscures some technical issues (which can, of course, be resolved: see \S\ref{S:descent-gluing}).  For example, different choices of fiber products $S_{ij}$ and $S_{ijk}$ will lead to different but equivalent categories $\mathcal M(\{ S_i \rightarrow S \})$.
\end{rem}

\begin{rem} \label{R:dd}
If $\mathcal R$ is a covering, as in Definition~\ref{D:dd}, it is technically convenient to indicate the category of descent data with respect to $\mathcal R$ as a map $\mathcal R \rightarrow \mathcal M$.  This notation is consistent with the Yoneda identification between $\mathcal M(S)$ and the category of morphisms $S \rightarrow \mathcal M$, which is the special case where $\mathcal R$ is the trivial covering $\{ S \rightarrow S \}$.  In general, the notation is justified by replacing the covering with the associated sieve or simplicial object (see Section~\ref{S:MoreDesc}).

Note that whenever a cover $\mathcal R'$ refines a cover $\mathcal R$ there is an induced morphisms $\mathcal M(\mathcal R) \rightarrow \mathcal M(\mathcal R')$, which we notate as a composition:
\begin{equation*}
\mathcal R' \rightarrow \mathcal R \rightarrow \mathcal M
\end{equation*}
We identify the object $S$ with the trivial cover of itself, so that it makes sense to write $\mathcal R \rightarrow S$ for any cover $\mathcal R$ of $S$.
\end{rem}

Given $X\in \mathcal M(S)$, we can construct an object with descent data on $\{\sigma_i:S_i\to S\}$ as follows.  The objects are the pullbacks $\sigma_i^*X$; the isomorphisms $\alpha_{ij}:\operatorname{pr}_2^*\sigma_j^*X \cong \operatorname{pr}_1^* \sigma_i^*X$ are the isomorphisms that come from the fact that both  $\operatorname{pr}_2^*\sigma_j^*X$ and  $\operatorname{pr}_1^* \sigma_i^*X$ are pullbacks of $X$ to $S_i\times_S S_j$ (they are both equipped with canonical isomorphisms with $(\sigma_i\circ \operatorname{pr}_1)^*X=(\sigma_j\circ \operatorname{pr}_2)^*X$).  (If we identify  $\operatorname{pr}_2^*\sigma_j^*=(\sigma_j\circ \operatorname{pr}_2)^*$, etc., as is common, then the $\alpha_{ij}$ are identity morphisms.)
If $\phi:X\to X'$ is an arrow in $\mathcal M(S)$, then we get arrows $\sigma_i^*X \xrightarrow{\sigma_i^*(\phi)}\sigma_i^*X'$ yielding an arrow from the object with descent data associated to $X$ with the object with descent data associated to $X'$.  In short, we have defined a functor
$
\mathcal M(S)\to \mathcal M(\{ S_i \rightarrow S \})
$.

\begin{dfn}[Effectivity of descent data]  
A descent datum $( \{ X_i \}, \{ \alpha_{ij} \}) \in \mathcal M(\{ S_i \rightarrow S \})$ is said to be \emph{effective} if it is isomorphic to the image of an object of $\mathcal M(S)$.  
\end{dfn}


\subsection{The long-awaited definition of a stack} \label{S:DefStack}

We are finally ready for the definition of a stack:

\begin{dfn}[Stack] \label{D:Stack}
 A \emph{stack}  is a category fibered in groupoids (Definition \ref{D:CFG}) over a presite  $(\mathsf S,\mathscr T)$ (Definition \ref{D:Presite}) such that isomorphisms are a sheaf (Definition \ref{D:Prestack}) and every descent datum is effective (Definition \ref{D:ConcDS}).    A \emph{prestack} is a category fibered in groupoids over $(\mathsf S,\mathscr T)$ such that isomorphisms are a sheaf.
A \emph{morphism of (pre)stacks} over $(\mathsf S,\mathscr T)$  is a morphism of the underlying CFGs over $\mathsf S$.
\end{dfn}

The definition can be rephrased  in the following way, emphasizing the connection with sheaves.  

\begin{pro}[{\cite[Prop.~4.7, p.73]{FGAe}}]\label{P:StCond}
Let $\mathcal M\to \mathsf S$ be a category fibered in groupoids over a presite $(\mathsf S,\mathscr T)$.
\begin{enumerate}
\item $\mathcal M$ is a prestack over $\mathsf S$ if and only if  for each covering $\mathcal R=\{S_i\to S\}$, the functor $\mathcal M(S)\to \mathcal M(\mathcal R)$ is fully faithful.

\item  $\mathcal M$ is a stack over $\mathsf S$ if and only if for each covering $\mathcal R=\{S_i\to S\}$, the functor $\mathcal M(S)\to \mathcal M(\mathcal R)$ is an equivalence of categories.
\end{enumerate}

\end{pro}

This essentially leads driectly to the following proposition.  

\begin{pro}[{\cite[Prop.~4.9, p.73]{FGAe}}]\label{P:ShSt}
Let $(\mathsf S,\mathscr T)$ be a presite, and let $\mathscr M:\mathsf S^{\mathsf {op}}\to (\mathsf {Sets})$ be a presheaf.  Denote by $\mathcal M\to \mathsf S$ the associated category fibered in groupoids.  

\begin{enumerate}
\item $\mathcal M$ is a prestack if and only if $\mathscr M$  is a separated presheaf.

\item $\mathcal M$ is a stack if and only if $\mathscr M$ is a sheaf.
\end{enumerate}
\end{pro}

\begin{cor}
If $(\mathsf S,\mathscr T)$ is subcanonical, then every object in $\mathsf S$ represents a stack.  
More precisely, given an object $S$, we  associate to it the category fibered in groupoids  $\mathsf S/S$, which  is a stack over $\mathsf S$.     
\end{cor}

\begin{rem}
In fact,  using the language of $2$-categories, for a subcanonical presite $(\mathsf S,\mathscr T)$, the category $\mathsf S$, and the category of sheaves  on $(\mathsf S, \mathscr T)$ can be viewed as full $2$-subcategories of the $2$-category of stacks  over $\mathsf S$.
\end{rem}

\begin{dfn}[Local class of arrows] \label{D:LocalClass} 
A class $\mathbf P$ of arrows in a presite $(\mathsf S,\mathscr T)$ is \emph{local (on the target)} if it is stable  and has the 	``converse'' property that for any cover $\{S_i\to S\}$ and any arrow $X\to S$, if  the pullbacks $S_i\times_ SX\to S_i$ are in $\mathbf P$ for all $i$, then $X\to S$ is also in $\mathbf P$.  
\end{dfn}

\begin{rem}
Fix a morphism $X \rightarrow Y$ in $\mathsf S$ and consider the collection of all $S \rightarrow Y$ such that $X \mathop{\times}_Y S \rightarrow S$ has property $\mathbf P$.  This is a category fibered in groupoids, in fact a sieve of $Y$, that we denote $\mathcal P(X \rightarrow Y)$.  Then $\mathbf P$ is a local property if and only if $\mathcal P(X \rightarrow Y)$ is a stack for all morphisms $X \rightarrow Y$. 
\end{rem}
This observation can be used to prove the following theorem:

\begin{teo}[{\cite[Thm.~4.38, p.93]{FGAe}}] \label{T:polarized-descent}   Let $S$ be a scheme.    Let $\mathbf P$ be a class of flat projective  canonically polarized morphisms of finite presentation in \emph{$(\mathsf S/S)_{_{\operatorname{fpqc}}}$}, which is local.  
Then the associated CFG  \emph{$\mathcal P\to (\mathsf S/S)_{_{\operatorname{fpqc}}}$}  (from Definition \ref{E:CFG-for-stable-class}) is a stack.  
\end{teo}

\begin{rem}
Canonically polarized morphisms include smooth morphisms such  that the determinant of the relative cotangent bundle is relatively ample; for instance, families of smooth curves of genus $g\ge 2$.    The theorem in fact holds in more generality for  polarized morphisms, but then one must add a compatibility condition  for the polarizations, which lengthens the statement (see \cite[Thm.~4.38, p.93]{FGAe}).  The above will be sufficient for many of our applications.  
\end{rem}

\begin{cor} \label{C:Mg-stack}
The CFG $\mathcal M_g$ is  a stack in the fpqc topology (and therefore in the \'etale topology) for $g \geq 2$.
\end{cor}

\begin{exa}
Let us make the descent condition concrete in the case of $\mathcal M_g$.  Suppose we are given an \'etale cover $\{S_i\to S\}$ and for each $S_i$ a relative curve $X_i\to S_i$.  Suppose moreover that for each $i,j$, we are given an $S_{ij}$-isomorphism 
 $\alpha_{ij} : X_i | _{S_{ij}} \to  X_j | _{S_{ij}}$, which satisfies the cocycle condition $\alpha_{ik} = \alpha_{jk} \circ \alpha_{ij}$ over $S_{ijk}$.
Then there exists a relative curve $X\to S$,  together with $S_i$-isomorphisms $\alpha_i: X |_{S_i} \to  X_i$  such that $\alpha_{ij} = \alpha_j |_{S_{ij}} \circ (\alpha_i |_{S_{ij}})^{-1}$.
\end{exa}

\begin{rem}
A similar argument can be used to show that $\mathcal G(r,n)$ is a stack.  For brevity, we omit the details  as we will also give references for Quot stacks, of which $\mathcal G(r,n)$ is a special case.  
\end{rem}

\begin{rem}
One can also show that $\mathcal M_g$ is a stack for $g=0$;   it does not follow immediately from the theorem above since the canonical bundle of such curves is not ample.  However, for $g=0$, the anti-canonical bundle is ample, and this gives a polarization that can be used in a more general formulation of the theorem \cite[Thm.~4.38]{FGAe}. 
However, $\mathcal M_1$ is not a stack!  See Section~\ref{S:Raynaud} for a demonstration and a discussion of the fix.
\end{rem}


\section{Fibered products of stacks}\label{S:2-cat}

Because of the ubiquity of base change in algebraic geometry, it is essential to know that one can take fiber products of stacks.  In this section we present a construction of the fiber product of categories fibered in groupoids, which yields a fiber product of stacks when applied to CFGs that are stacks.   

A reader who wants to get to algebraic stacks as quickly as possible may prefer to look briefly at \S \ref{S:WorkFP} and then skip the remainder of \S\ref{S:2-cat}, referring back as necessary.  
A more detailed discussion of $2$-categories and universal properties within them can be found in the Stacks Project \cite[Tags~003G, 02X8, 003O]{stacks}.

\subsection{A working definition  of $2$-fibered products}\label{S:WorkFP}  Here we give a working definition of a $2$-fibered product of CFGs.  This should suffice for understanding the definitions of an algebraic stack.  Again, on a first pass, the reader is encouraged to  look at this section, and then skip the remainder of \S \ref{S:2-cat}.  Our presentation here is taken largely from \cite{stacks}.

Recall that a fiber product of morphisms of sheaves $f : \mathscr X \rightarrow \mathscr Z$ and $g : \mathscr Y \rightarrow \mathscr Z$ is defined by setting $(\mathscr X \mathop\times_{\mathscr Z} \mathscr Y)(S) = \mathscr X(S) \mathop\times_{\mathscr Z(S)} \mathscr Y(S)$ for all $S$.  The same definition could be used for lax $2$-functors valued in groupoids, except one must first define a fiber product of groupoids.  Here there is a subtle, but crucial point, since objects of $\mathscr Z(S)$ can be `equal' to each other in more than one way.  The fiber product must therefore keep track of all of the different ways $f(X)$ and $g(Y)$ are equal to each other.  One defines $\mathscr X(S) \mathop\times_{\mathscr Z(S)} \mathscr Y(S)$ to be the category of triples $(X,Y,\alpha)$ where $X \in \mathscr X(S)$, $Y \in \mathscr Y(S)$, and $\alpha : f(X) \simeq g(Y)$ is an isomorphism.  Morphisms $(X,Y,\alpha) \rightarrow (X',Y',\alpha')$ in this groupoid are pairs $(u,v)$ with $u : X \rightarrow X'$ and $v : Y \rightarrow Y'$ are morphisms such that $\alpha' f(u) = f(v) \alpha$ as morphisms $f(X) \rightarrow g(Y')$.

\begin{rem}
This construction is analogous to one construction of the homotopy fiber product in algebraic topology, with isomorphism playing the role of homotopy.  This is not an accident, as homotopy fiber products are intended to be invariant under replacement of the spaces involved with homotopy equivalent spaces; the fiber product of groupoids is intended to be invariant under equivalence of categories.  Even more directly, groupoids can be realized as topological spaces by way of the geometric representation of the simplicial nerve, 
under which the fiber product of groupoids is transformed into the homotopy fiber product.
\end{rem}

Since we are not defining stacks in terms of lax $2$-functors here, we make the straightforward translation of the above idea to categories fibered in groupoids:

\begin{dfn}[{\cite[Tag~0040]{stacks}}]\label{D:2-product-categories-over-C} 
Let $f : \mathcal X \to \mathcal Z$ and $g : \mathcal Y \to \mathcal Z$ be morphisms of CFGs over $\mathsf S$.  Then the \emph{fiber product} of $\mathcal X$ and $\mathcal Y$ over $\mathcal Z$, 
denoted $\mathcal X \mathop\times_{\mathcal Z} \mathcal Y$, is the category over $\mathsf S$ whose objects are quadruples $(S, X, Y, \alpha)$ where $S \in \mathsf S$, $X \in \mathcal X(S)$, $Y \in \mathcal Y(S)$, and $\alpha : f(X) \simeq g(Y)$ is an isomorphism in $\mathcal Z(S)$.  A morphism $(S,X,Y,\alpha) \rightarrow (S',X',Y',\alpha')$ is triple $(\varphi, u, v)$ where $\varphi : S \rightarrow S'$ is a morphism in $\mathsf S$, $u : X \rightarrow X'$ is a morphism in $\mathcal X$ lying above $\varphi$, and $v : Y \rightarrow Y'$ is a morphism in $\mathcal Y$ lying above $\varphi$, and $\alpha' f(u) = g(v) \alpha$. 
\end{dfn}

\begin{lem}[{\cite[Tag~0040]{stacks}}]
The fiber product of CFGs is a CFG.
\end{lem}

\begin{lem}[{\cite[(3.3), p.16]{LMB}}]\label{L:FPS}
The fibered product of stacks is a stack.  
\end{lem}

The $2$-fibered product of CFGs $\mathcal X\times_{\mathcal Z}\mathcal Y$ has a universal property similar to that satisfied by a fiber product of sheaves.  Indeed, there are forgetful  morphisms $p:\mathcal X\times_{\mathcal Z}\mathcal Y\to \mathcal X$ and $q:\mathcal X\times_{\mathcal Z}\mathcal Y\to \mathcal Y$, respectively sending $(S,X,Y,\alpha)$ to $X$ and to $Y$.  Then $\alpha$ gives an isomorphism $f p(S,X,Y,\alpha) = f(X) \simeq g(Y) = g q(S,X,Y,\alpha)$.  By the definition of morphisms in $\mathcal X \mathop\times_{\mathcal Z} \mathcal Y$, this isomorphism is natural in $(S,X,Y,\alpha)$.  We denote this natural isomorphism $\psi : f p \simeq g q$.  In standard terminology, the following diagram is \emph{$2$-commutative}:
$$
\xymatrix{
& \mathcal X \times_{\mathcal Z} \mathcal Y \ar[r]^{p} \ar[d]_q & \mathcal X \ar[d]^{f} \ar@{=>}[ld]_\psi \\
& \mathcal Y \ar[r]_{g} & \mathcal Z. }
$$
The universal property of the $2$-fibered product is that $(\mathcal X \mathop\times_{\mathcal Z} \mathcal Y, p, q, \psi)$ is the universal completion of the diagram of solid arrows below to a $2$-commutative diagram
\begin{equation} \label{E:2-comm} \vcenter{
\xymatrix{
& \mathcal W \ar@{-->}[r]^{p} \ar@{-->}[d]_q & \mathcal X \ar[d]^{f} \ar@{==>}[ld]_\psi \\
& \mathcal Y \ar[r]_{g} & \mathcal Z. }
} \end{equation}
In other words, given a $2$-commutative diagram~\eqref{E:2-comm}, there is a 
 $2$-commutative diagram:
$$
\xymatrix{
\mathcal W\ar@{->}@/^1 pc/ [rrd]^a \ar@{-->}[rd]_\gamma \ar@/_1 pc/ [rdd]_b  & \\
 & \mathcal X \times_{\mathcal Z} \mathcal Y \ar[r]_-p \ar[d]_q & \mathcal X \ar[d]^{f} \\
 & \mathcal Y \ar[r]^{g} & \mathcal Z }
$$
Note that the $2$-commutativity of this diagram includes the tacit specification of natural isomorphisms $p \gamma \simeq a$ and $q \gamma \simeq b$.  The functor $\gamma$ is determined by these data uniquely up to a unique natural transformation.  Using this one can show  that if $\mathcal W$ also satisfies the universal property of $\mathcal X \mathop\times_{\mathcal Z} \mathcal Y$ then $\gamma : \mathcal W \rightarrow \mathcal X \mathop\times_{\mathcal Z} \mathcal Y$ is an equivalence.

\begin{exa}
Suppose that $X,Y,Z\in \mathsf S$ and $\mathsf S$ has fibered products.  Then it follows from the $2$-Yoneda Lemma that  $\mathsf S/X\times_{\mathsf S/Z} \mathsf S/Y$ is equivalent to  $\mathsf S/(X\times_ZY)$.  
Similarly, suppose that $\mathscr X,\mathscr Y,\mathscr Z$ are pre-sheaves on $\mathsf S$ with associated CFGs $\mathcal X,\mathcal  Y,\mathcal Z$.  Then $\mathcal X\times_{\mathcal Z} \mathcal Y$ is equivalent  to the CFG associated to $\mathscr  X\times_{\mathscr Z}\mathscr Y$.   
\end{exa}

\subsection{The diagonal}
\label{S:diagonal}

The following example is used repeatedly.  

\begin{exa}[{The diagonal and the sheaf of isomorphisms}] \label{E:diagonal}
Let $\mathcal M$ be an $S$-stack over $\mathsf S$, let $S$ be in $\mathsf S$, and let $X,Y$ in $\mathcal M(S)$ be two objects corresponding under the $2$-Yoneda Lemma  to $S$-morphisms $X,Y:S\to \mathcal M$.    Then, using the notation from \S \ref{S:isomPr},  there is $2$-cartesian diagram
$$
\xymatrix{
\mathscr I\!\!\mathit{som}_{\mathcal M}(X,Y) \ar[r]\ar[d]^{} & S \ar[d]^{(X, Y)} \\
\mathcal M\ar[r]^<>(0.5){\Delta} & \displaystyle \mathcal M\mathop\times\mathcal M.
}
$$
We will verify this by way of the universal property \eqref{E:2-comm} (see also \cite[Prop.~5.12]{DMstacks}).  Suppose that we have a map $f : T \rightarrow S$, an object $Z \in \mathcal M(T)$, and an isomorphism
\begin{equation*}
(\varphi, \psi) : (Z,Z) = \Delta(Z) \simeq f^\ast (X,Y) = (f^\ast X, f^\ast Y),  \  \ \text{in} \  \mathcal (M\times \mathcal M)(Z).
\end{equation*}
Then the composition $\psi \circ \varphi^{-1}$ is an isomorphism $f^\ast X \simeq f^\ast Y$, hence yields a section of $\mathscr I\!\!\mathit{som}_{\mathcal M}(X,Y)$ over $T$.  Conversely, given such an isomorphism $\alpha : f^\ast X \simeq f^\ast Y$, we obtain $2$-commutative diagram by taking $Z = (f^\ast X, f^\ast X)$ and $(\varphi, \psi) = (\mathrm{id}_{f^\ast X}, \alpha)$.

Intuitively, $\mathscr I\!\!\mathit{som}_{\mathcal M}(X,Y)$ is the sheaf of witnesses to the equality of $X$ and $Y$, in the same way that the diagonal is the moduli space of pairs of objects of $\mathcal M$ that are equal to one another.  Notably, $\mathscr I\!\!\mathit{som}_{\mathcal M}(X,Y)$ is not a subobject of $S$, reflecting the fact that $\Delta$ is not an embedding.  This is because a pair of objects of a groupoid can be equal---that is, isomorphic---to each other in more than one way.
\end{exa}

From the perspective of moduli, the isomorphisms of the objects of study were of central importance.  On the other hand, 
the diagonal map is central in the definition of many properties of schemes.
  The diagram above  relates the two notions.

\begin{dfn}[Injective morphism of  stacks] \label{D:MonIso}
A morphism $f:\mathcal X\to \mathcal Y$ of  stacks over $\mathsf S$ is called \emph{injective} (resp.~an \emph{isomorphism}) 
 if for each $S\in \mathsf S$, the functor $f(S):\mathcal X(S)\to \mathcal Y(S)$ is fully faithful (resp.~an equivalence of categories).    A \emph{substack} is an injective morphism of stacks.
\end{dfn}

\begin{lem}\label{L:InjDiag}
	A stack $\mathcal X$ has injective  diagonal if and only if $\mathcal X$ is representable by a sheaf (i.e., equivalent to the stack associated to a sheaf). 
	
\end{lem}
\begin{proof}
	Injective diagonal means that $\mathscr I\!\!\mathit{som}_{\mathcal X}(x,y)$ is a subobject of $S$, or, equivalently, that for any $x,y \in \mathcal X(S)$ there is at most one isomorphism between $x$ and $y$.  Thus $\mathcal X(S)$ is equivalent to a set.
\end{proof}

\subsection{Fibered products and the stack condition}

As an application of the formalism of fiber products introduced above, we give a simple but often useful criterion for a CFG over a presite to be stack, which is simply a translation of Proposition \ref{P:StCond} (see also Definition \ref{D:descent-sieves}).

\begin{lem} \label{L:StCond} 

 Let $\pi:\mathcal M\to \mathsf S$ be a CFG over a presite $(\mathsf S,\mathscr T)$.

\begin{enumerate}

\item $\mathcal M$ is a prestack  if and only for every covering family $\mathcal R = \{S_i\to S\}$,  every morphism $f:\mathcal R\to \mathcal M$ (see Remark~\ref{R:dd}), and every pair of morphisms $f_1,f_2:S\to \mathcal M$ making diagram~\eqref{E:prestack} $2$-commutative,  there is a unique $2$-isomorphism $f_1\Rightarrow f_2$. 
\begin{equation} \label{E:prestack}
\begin{gathered}
\xymatrix{
\mathcal R \ar[r]^f \ar[d] & \mathcal M \ar[d]^\pi\\
 \mathsf S/S \ar@{-->}@/^.5 pc/[ru]^<>(0.5){}^{f_1} \ar@{-->}@/_.5 pc/[ru]^<>(0.5){}_{f_2} \ar[r] &\mathsf S
}
\end{gathered}
\end{equation}

\item $\mathcal M$ is a stack  if and only for every sieve $\mathcal R$ associated to a covering family $\{S_i\to S\}$, with natural map $\mathcal R\to S$,  any morphism $f:\mathcal R\to \mathcal M$, there exists a morphism $f_1:S\to \mathcal M$ making the diagram below $2$-commutative, which is unique up to $2$-isomorphism (as explained in (1)).
\begin{equation} \label{E:stack}
\begin{gathered}
\xymatrix{
\mathcal R \ar[r]^f \ar[d] & \mathcal M \ar[d]^\pi\\
 \mathsf S/S \ar@{-->}[ru]^{f_1}_{\exists !} \ar[r] &\mathsf S
}
\end{gathered}
\end{equation}
\end{enumerate}
\end{lem}

The lemma can be summarized with the following efficient characterization of stacks:
\begin{cor} \label{C:sieve-stack}
Let $\mathcal M$ be a category fibered in groupoids over $\mathsf S$.  Then $\mathcal M$ is a stack over $\mathsf S$ if and only if for every object $S$ of $\mathsf S$ and every cover $\mathcal R$ of $S$, the functor
\begin{equation*}
\operatorname{Hom}(S, \mathcal M) \rightarrow \operatorname{Hom}(\mathcal R, \mathcal M)
\end{equation*}
is an equivalence of categories.
\end{cor}
\begin{proof}
The first part of Lemma~\ref{L:StCond} is the full faithfulness, and the second part is the essential surjectivity.
\end{proof}

\begin{rem}
Recall that the notation in the corollary is based on the abuse explained in Remark~\ref{R:dd}.  The statement requires no such abuse when formulated with sieves (see~\ref{S:descent-sieves}), the efficiency of which is one reason we like using sieves to think about Grothendieck topologies.  For example, Corollary~\ref{C:sieve-stack} generalizes immediately to give a definition of higher stacks, while the other formulations of descent from \S\ref{S:Descent} become even more combinatorially complicated.
\end{rem}

The following proposition is quite useful for making boot-strap arguments.  It allows us to show that a CFG is a stack by showing it is a stack \emph{relative} to a CFG already known to be a stack.

\begin{pro} \label{P:DescRel}
Suppose that $p : \mathcal X \rightarrow \mathcal Y$ is a morphism of CFGs over $\mathsf S$ and that $\mathcal Y$ is a stack over $\mathsf S$.  Then $\mathcal X$ is a stack over $\mathsf S$ if and only if it is a stack over $\mathcal Y$, where $\mathcal Y$ is given the pretopology of Example~\ref{E:top-on-CFG}.
\end{pro}
\begin{proof}
We will verify the second part of the criterion in Lemma~\ref{L:StCond}; the first part is very similar, so we omit it.  Assume first that $\mathcal X$ is a stack over $\mathcal Y$.  Consider the following diagram:
\begin{equation*} 
\xymatrix@C=4em@R=1.3em{
\mathcal R \ar[r]^f \ar[dd] & \mathcal X \ar[d]^p \\
& \mathcal Y \ar[d]\\
S \ar@{-->}[uur]^{f_1} \ar@{-->}[ru]_h  \ar[r]& \mathsf S\\
} \end{equation*}
We would like to find a morphism $f_1$ rendering the outer square $2$-commutative.  Since $\mathcal Y$ is a stack over $\mathsf S$, we can find  an appropriate  lift $h$ as in the diagram.  But then the assumption that $\mathcal X$ is a stack over $\mathcal Y$ guarantees the existence of the desired lifting  $f_1$ of  $h$. Thus $\mathcal X$ is a stack over $\mathsf S$.   The converse is similar.
\end{proof}

\begin{cor} \label{C:DescRel}
Suppose that $p : \mathcal X \rightarrow \mathcal Y$ is a morphism of CFGs over a presite $\mathsf S$.  Assume that $\mathcal Y$ is a stack over $\mathsf S$.  Then $\mathcal X$ is a stack over $\mathsf S$ if and only if for $S \in \mathsf S$ and every $y: S \rightarrow \mathcal Y$, the fiber product $\mathcal X_S = \mathcal X \mathop\times_{\mathcal Y} S$ is a stack on $\mathsf S/S$.
\end{cor}

\begin{proof}
We show the harder direction,  that $\mathcal X$ is a stack over $\mathsf S$.    From Proposition \ref{P:DescRel}, it is enough to show that $\mathcal X$ is a stack over $\mathcal Y$.  From Lemma \ref{L:StCond} and the definition of the fibered product, one can easily deduce  that 
$\mathcal X$ is a stack over $\mathcal Y$ if and only if $\mathcal X \mathop\times_{\mathcal Y} (\mathcal Y/Y)$ is a stack over $\mathcal Y/Y$ for all $Y \in \mathcal Y$. 
Now using the  
 equivalence of categories $\mathcal Y/Y\cong \mathsf S/S$ (Example \ref{E:top-on-CFG}), we are done.
\end{proof}

\begin{exa}
As an example of the utility of Proposition~\ref{P:DescRel} and Corollary~\ref{C:DescRel}, consider the CFG $\mathcal X$ whose $S$-points are triples $(C, C', f)$ where $C$ and $C'$ are families of smooth curves over $S$ of genera $g$ and $h$, both $\geq 2$, and $f : C \rightarrow C'$ an $S$-morphism.  Then there is a projection $p:\mathcal X \rightarrow \mathcal M_g \times \mathcal M_h$, where $\mathcal M_g$ denotes the CFG of families of smooth curves of genus $g$.  We know that $\mathcal M_g$ and $\mathcal M_h$  are stacks in the \'etale topology by Corollary~\ref{C:Mg-stack}, and therefore the product $\mathcal M_g \times \mathcal M_h$ is a stack (Lemma \ref{L:FPS}).   Applying Corollary \ref{C:DescRel} to $p$, to verify that $\mathcal X$ is a stack it is therefore sufficient to show that, for any \emph{fixed} pair of smooth curves $C$ and $C'$ over a scheme $S$, the functor $\mathscr Mor(C,C')(-):(\mathsf S/S)^{\mathsf {op}}\to (\mathsf {Set})$ that assigns to an $S$-scheme $T$ the set of $T$-morphisms $C_T \rightarrow C'_T$, is a sheaf.  This is easily verified using the fact that the \'etale topology is subcanonical (see Example \ref{E:MgPS}).
\end{exa}


\section{Stacks adapted to a presite}\label{S:AdaptStackPre}

 In this section, we take the most naive approach to defining an algebraic stack on a presite, namely, we define what we call a stack adapted to a presite (Definition \ref{D:SAdaptStack}).  This  is simply a stack with a representable cover  by an object in the presite (in the sense of Definition \ref{D:AdaptStack}).      While there are technical reasons (see Example \ref{E:LogAb})  that in total generality  this is not really the `right' definition, it nevertheless  provides a  quick definition, that immediately suffices for many of the standard examples one sees (e.g., algebraic spaces,  Fantechi  Deligne--Mumford stacks, and Laument--Moret-Bailly Deligne--Mumford stacks).     In fact, by iteratively enlarging the presite, and then reducing back down to the original presite, one can obtain all of the definitions of algebraic stacks we discuss in this paper (see \S \ref{S:AdaptPersp}).

 We hope that in this generality,  and brevity, the salient points of an algebraic stack will be apparent to readers who prefer to work in categories other than schemes.  
  The reader interested in moving quickly to the definition of the algebraic stack of smooth curves, or the algebraic stack of  Higgs bundles, may prefer to read this section, and then skip directly to \S  \ref{S:MSHB}.    Technically, the stack of Higgs bundles  is  adapted   to the smooth presite of algebraic spaces, rather than  to the   \'etale presite of schemes; nevertheless, the main aspects of the formalism should already be apparent to the reader after this section.

\subsection{Definiton of stacks adapted to a presite} \label{S:AdaptStacksPre} 
In order to define a stack adapted to a presite, we want the definition of a representable morphism:

\begin{dfn}[$\mathsf S$-representable morphism] \label{D:RepS}
Let  $\mathsf S$ be a subcanonical  presite that admits fibered products.    We say  a morphism $\mathcal X \rightarrow \mathcal Y$ of CFGs  is $\mathsf S$-representable  if  for every $S$ in $ \mathsf S$, the fiber product $\mathcal X \mathop\times_{\mathcal Y} S$ is in $\mathsf S$ (i.e., equivalent to a stack $\mathsf S/S'$  for some $S'$ in $\mathsf S$).
$$
\xymatrix@R=1em{
\mathcal X \mathop{\times}_{\mathcal Y} S  \ar[r] \ar[d]& S \ar[d]\\
\mathcal X \ar[r] & \mathcal Y.
}
$$
\end{dfn}

With this definition, we define  an $\mathsf S$-adapted stack  by requiring that the stack admit  an `$\mathsf S$-representable cover' in $\mathsf S$.  More precisely:

\begin{dfn}[$\mathsf S$-adapted stack] \label{D:SAdaptStack} Let 
$\mathsf S$ be a subcanonical presite that admits fibered products.   Then an \emph{$\mathsf S$-adapted stack} (or a stack adapted to the presite $\mathsf S$) is a stack $\mathcal M$ over $\mathsf S$ admitting an  $\mathsf S$-representable cover in $\mathsf S$  of the following form: there exists a $U$ in $\mathsf S$ and an $\mathsf S$-representable  morphism 
$$
\begin{CD}
U@>p>>\mathcal M
\end{CD}
$$ 
such  that for every $S$ in $\mathsf S$ and every $S\to \mathcal M$, the morphism $U\times_{\mathcal M}S\to S$ in $\mathsf S$ obtained from base change is a cover in the pretopology on $\mathsf S$.     
  Such a morphism  $p$ is called a \emph{presentation} of  $\mathcal M$.    
   A \emph{morphism of  $\mathsf S$-adapted stacks} is  a morphism of stacks.  
    \end{dfn}

\begin{rem}
We will define more generally a covering morphism of CFGs in Definition \ref{D:CovStack}.  The cover defined in Definition \ref{D:SAdaptStack} above is a cover in that sense as well (Remark \ref{R:SAdaptStackCov}).
\end{rem}

\begin{exa}[Algebraic spaces and adapted Deligne--Mumford  stacks] \label{E:SDM-AS}
If $S$ is a scheme and $\mathsf S$ is the \'etale presite on schemes over $S$, then the $\mathsf S$-adapted stacks are  called Fantechi Deligne--Mumford (F DM) stacks (over $S$).      
The $\mathsf S$-adapted stacks that are representable by sheaves are   called algebraic  spaces.   
\end{exa}

\begin{exa}
If $\mathsf S$ is the presite of complex analytic spaces with the pretopology induced from  usual open covers, then one obtains a notion of an adapted complex analytic stack.  
Similarly, if $\mathsf S$ is the presite of topological spaces with the presite induced from usual open covers, then one obtains a notion of an adapted topological  stack.  There are other definitions of complex analytic and topological  stacks appearing in the literature; we do not pursue the relationship among the definitions.  
\end{exa}

\begin{rem}[Conditions on the diagonal]  Other notions of algebraic stacks are determined by requiring the diagonal $\Delta :\mathcal M\to \mathcal M\times_{\mathsf S}\mathcal M$ to be $\mathsf S$-representable, and putting further geometric conditions on the diagonal.     For instance, an Laument--Moret-Bailly   Deligne--Mumford (LMB DM) stack over the \'etale presite  $\mathsf S$ of schemes over a fixed scheme $S$ is  an $\mathsf S$-adapted stack, with $\mathsf S$-representable, separated, and quasicompact diagonal.    
\end{rem}

To see how all the other algebraic stacks discussed in this paper can be defined using stacks adapted to a presite,  see \S \ref{S:AdaptPersp}.

\subsection{Bootstrapping stacks adapted to a presite}

It can often be useful to show that a stack is $\mathsf S$-adapted using  bootstrap methods; in other words, it will often be the case that one can exhibit a morphism from a  stack of interest to a well-known $\mathsf S$-adapted stack, so that it is easy to check the $\mathsf S$-adapted  condition  on the fibers.  The following proposition  states that in this situation the stack is an $\mathsf S$-adapted stack.

\begin{pro} \label{P:SAdaptRel}
Let $f:\mathcal X\to \mathcal Y$ be a morphism of stacks over a subcanonical presite $\mathsf S$ that admits fibered products.  Assume that $\mathcal Y$ is an  $\mathsf S$-adapted stack and that for all $S$ in $\mathsf S$ and all morphisms $S\to \mathcal Y$ we have that $\mathcal X\times_{\mathcal Y}S$ is  $\mathsf S$-adapted.
 Then $\mathcal X$ is an $\mathsf S$-adapted  stack.\end{pro}

\begin{proof}
Later we will prove Proposition \ref{P:AdaptRel}, whose proof can be readily modified  to fit this situation.  We sketch the argument here.  
Choose a presentation  $Y \rightarrow \mathcal Y$.  Then  set $\mathcal Z =\mathcal X \times_{\mathcal Y} Y$.  
One can check that $\mathsf S$-representable morphisms are stable under base change (see e.g., Lemma \ref{L:Rep}; in that notation, take $\mathsf C=\mathsf S$).   Thus the projection $\mathcal Z \rightarrow \mathcal X$ is  $\mathsf S$-representable.  One can then check directly from the definition that for every $S$ in $\mathsf S$ and every morphism $S\to \mathcal X$, the morphism $\mathcal Z\times_{\mathcal X}S\to S$ is a cover in $\mathsf S$ (use $Y\times _{\mathcal Y}S=Y\times_{\mathcal Y}\mathcal X\times_{\mathcal X}S=\mathcal Z\times_{\mathcal X}S$).   
 Furthermore, $\mathcal Z$ is $\mathsf S$-adapted, by assumption, so there is a presentation   $Z \rightarrow \mathcal Z$.   One can check that $\mathsf S$-representable morphisms are stable under composition (see e.g., Lemma \ref{L:Rep}; in that notation, take $\mathsf C=\mathsf S$).  Thus the composition $Z\to \mathcal Z\to \mathcal X$ is $\mathsf S$-representable.  
 One can then check directly from the definition that this morphism is a cover in the sense of Definition \ref{D:SAdaptStack}, and thus provides the desired presentation for $\mathcal X$.
\end{proof}

\begin{rem}
A result similar to Proposition \ref{P:SAdaptRel} holds when   additional conditions are placed on the diagonal of the stacks (see Corollary  \ref{C:AlgRel}). 
\end{rem}


\section{Algebraic stacks}\label{S:AlgStack}

In the previous section we introduced the notion of a stack adapted to a presite. This provided a  quick definition, that suffices in many cases.  However, in general, that approach is a little too naive, particularly if one does not enlarge the presite iteratively.  In this section, we take a slightly more lengthy approach, which considers stacks that are adapted to larger classes of morphisms in the the presite.  After iterating this process, we arrive at the definition of an algebraic stack (adapted to a class of morphisms in the presite); stacks adapted to the presite  are algebraic stacks.  

\vskip .1 cm 
We now also provide a lengthier motivation to the study of algebraic stacks than we provided in the previous section.  
Suppose that $\mathsf S$ is a subcanonical presite, meaning that every $S \in \mathsf S$ represents a sheaf, or equivalently that $\mathsf S/S$ is a stack on $\mathsf S$.  A stack on $\mathsf S$ is fundamentally a topological object, and the category of stacks on $\mathsf S$ is therefore too inclusive a milieu for our geometric purposes.  Even when $\mathsf S$ is a category of geometric objects, the stacks (and even the sheaves) on $\mathsf S$ need not behave at all geometrically.  To take just one example, a stack on the category of schemes always has a tangent space at a point, but this tangent space may not have the structure of a vector space (see \S\ref{S:homogeneity}).  Nevertheless, some sheaves and stacks that are not representable \emph{do} behave geometrically, and our goal will be to identify those that do: algebraic spaces and algebraic stacks.

\begin{rem}
Granting our post hoc reasoning for the definition of algebraic stacks, it might have been more sensible to call them `geometric stacks'.  This terminology has caught on in some places~\cite{TV-II}, but we stick to the traditional nomenclature here.
\end{rem}

The essential idea in the definition of algebraic stacks is that a stack that resembles geometric objects locally can itself be studied geometrically, provided that the meaning of `local' is sufficiently geometric.  In the algebraic category, `locally' is interpreted `over a smooth cover' and the `geometric objects' are taken to be schemes.  

Here a technicality arises:  there are \'etale \emph{sheaves}, known as algebraic spaces, that resemble schemes \'etale-locally, but are not themselves schemes.  More worryingly, the class of algebraic stacks modelled locally by algebraic spaces is strictly larger than the class of those modelled locally by schemes.  In order to make the whole theory satisfyingly formal, one takes the smallest class of stacks that includes affine schemes and is stable under disjoint unions and groupoid quotients.  Thankfully, this turns out to be the same as the class of algebraic stacks with smooth covers by algebraic spaces.

This section will be quite formal, and applies to essentially any presite.  The main idea is that of a \emph{$\mathbf P$-adapted stack} in Section~\ref{S:AdaptStacks} over a presite $\mathsf S$, which is precisely a stack modelled locally, according to a suitable property $\mathbf P$, on morphisms in  $\mathsf S$.  In Section~\ref{S:P-adapted}, we iterate the construction to arrive at the class of algebraic stacks. 

In fact, most stacks we encounter in practice (e.g., quasiseparated Deligne--Mumford stacks) are adapted to the class of morphisms of schemes (or even to affine schemes), and the iteration in Section~\ref{S:P-adapted} is useful only to have a category with good formal properties.  A reader interested in seeing stacks in geometric action may prefer to skip Section~\ref{S:P-adapted}.

The key point is an elementary notion of a cover, and to this end, we must first discuss representable morphisms (\S \ref{S:RepMorph}).    We also present  Proposition \ref{P:AdaptRel}, which is a useful tool for bootstrapping from one adapted stack to another, analogous to Corollary \ref{C:DescRel} for stacks.

\subsection{Covers}
\label{S:covers}

We explain what it means for a morphism of stacks over a presite $\mathsf S$ to be a cover:

\begin{dfn}[Covering morphism of CFGs] \label{D:CovStack}
	A morphism of CFGs $\mathcal X \rightarrow \mathcal Y$ on a presite $\mathsf S$ is said to be \emph{covering} if, for any morphism $S \rightarrow \mathcal Y$, there is a cover $\{ S_i \rightarrow S \}$ such that the induced maps $S_i \rightarrow \mathcal Y$ lift to $\mathcal X$:
	\begin{equation*} \xymatrix@R=1em{
	&S_i\ar[d] \ar@/_10pt/@{-->}[ld]\\
			 \mathcal X\times_{\mathcal Y}S \ar[d] \ar[r]& S\ar[d] \\
			\mathcal X \ar[r]& \mathcal Y.
	} \end{equation*}
\end{dfn}

\begin{rem}
Note that a morphism of objects of $\mathsf S$ is covering according to Definition~\ref{D:CovStack} if and only if it is covering in the topology associated to the pretopology of $\mathsf S$ (Definition~\ref{D:top}).
\end{rem}

\begin{rem}\label{R:SAdaptStackCov}  The presentation   $P:U\to \mathcal M$ in Definition \ref{D:SAdaptStack} is covering the sense of 
Definition \ref{D:CovStack}.   Indeed, choose an $S$ in $ \mathsf S$ and a morphism $S\to \mathcal M$,  and consider the fibered product  
	\begin{equation*}\vcenter{ \xymatrix@R=1em{
			U\times_{\mathcal M} S \ar[r] \ar[d] & S \ar[d] \\
			U \ar[r]^P & \mathcal M.
	}} \end{equation*}
  In Definition \ref{D:CovStack}  we  require that there is a cover $\{S_i\to S\}$  so that for each $i$, the map $S_i\to S$ factors through $U\times_{\mathcal M}S$.  Of course, since  $U\times_{\mathcal M}S\to S$ is a cover by Definition \ref{D:SAdaptStack},  this condition is automatically satisfied taking $\{S_i\to S\}=\{U\times_{\mathcal M}S\to S\}$.\end{rem}

\begin{lem}\label{L:FPCompCov}
Coverings are stable under composition, base change, and fibered  products of morphisms. 
\end{lem}
\begin{proof}
	We prove stability under composition: Suppose that $\mathcal X \rightarrow \mathcal Y \rightarrow \mathcal Z$ is a composition of covering morphism and $W$ is a scheme over $\mathcal Z$.  Since $\mathcal Y$ covers $\mathcal Z$,  there is a cover $\{ W_i \rightarrow W \}$ such that the maps $W_i \rightarrow \mathcal Z$ lift to $\mathcal Y$.  Now, since $\mathcal X$ covers $\mathcal Y$ there are covers $\{ W_{ij} \rightarrow W_i \}$ such that $W_{ij} \rightarrow \mathcal Y$ lifts to $\mathcal X$.  The family $\{ W_{ij} \rightarrow W \}$ is a cover of a cover, hence is covering.

	We prove stability under base change:  Suppose that $\mathcal X \rightarrow \mathcal Y$ is covering and $\mathcal Z \rightarrow \mathcal Y$ is an arbitrary morphism.  Let $W$ be a scheme over $Z$.  Then there is a cover $\{ W_i \rightarrow W \}$ such that the maps $W_i \rightarrow \mathcal Y$ lift to $\mathcal X$, as $\mathcal X$ covers $\mathcal Y$.  But then by definition of the fiber product, the maps $W_i \rightarrow \mathcal Z$ lift to $\mathcal Z \mathop{\times}_{\mathcal Y} \mathcal X$.

	Stability under fiber product now follows formally from the parts already demonstrated (e.g., the proof of Lemma~\ref{L:Rep} can easily be adapted to this purpose).
\end{proof}

\subsection{Representable morphisms} \label{S:RepMorph}

A representable morphism of stacks on $\mathsf S$ is roughly one whose fibers are representable by some prescribed class of stacks:

\begin{dfn}[Representable morphism] \label{D:Rep}
Let $\mathsf C$ be a class of CFGs over a subcanonical presite $\mathsf S$ such that $\mathsf C$ is stable under isomorphism and fiber product and every CFG over $\mathsf S$ can be covered (Defintion \ref{D:CovStack}) by objects of $\mathsf C$.  We say that a morphism $\mathcal X \rightarrow \mathcal Y$ of CFGs is \emph{representable by objects of $\mathsf C$} or \emph{$\mathsf C$-representable} (or \emph{representable}, when $\mathsf C$ is clear from context) if, for every $\mathcal Z \in \mathsf C$, the fiber product $\mathcal X \mathop\times_{\mathcal Y} \mathcal Z$ is in $\mathsf C$.
$$
\xymatrix@R=1em{
\mathcal X \mathop{\times}_{\mathcal Y} \mathcal Z  \ar[r] \ar[d]& \mathcal Z \ar[d]\\
\mathcal X \ar[r] & \mathcal Y.
}
$$
When $\mathsf S$ is the category of schemes and $\mathsf C$ is also the category of schemes, i.e., the collection of stacks on the \'etale site of schemes representable by schemes, we call a $\mathsf C$-representable morphism \emph{schematic} (i.e., we recover Definition \ref{D:RepS}). 
\end{dfn}

\begin{rem}
For obvious reasons, schematic morphisms are sometimes also called representable morphisms, without qualification.  However, it is also common to use the term representable morphism for morphisms representable by algebraic spaces, so for clarity we use `schematic' to refer to morphisms representable by schemes.
\end{rem}

\begin{lem}  \label{L:Rep}
For $\mathsf C$  a class of CFGs over a category $\mathsf S$ that is stable under isomorphism and fiber product and such that every CFG over $\mathsf S$ can be covered (Definition~\ref{D:CovStack}) by objects of $\mathsf C$: 
\begin{enumerate}
	\item The composition of $\mathsf C$-representable morphisms is $\mathsf C$-representable. \label{I:rep-comp}
	\item The base change of a $\mathsf C$-representable morphism is $\mathsf C$-representable. \label{I:rep-bc}
	\item The fibered product of $\mathsf C$-representable morphisms is $\mathsf C$-representable. \label{I:rep-prod}
	\item If $f : \mathcal X \rightarrow \mathcal Y$ and $g : \mathcal Y \rightarrow \mathcal Z$ are morphisms such that $gf$ is $\mathsf C$-representable and the diagonal of $g$ is $\mathsf C$-representable then $f$ is $\mathsf C$-representable. \label{I:rep-rel}
\end{enumerate}
\end{lem}
\begin{proof}
	The verification is formal, so the proof in \cite[Prop.~5.8, p.87]{DMstacks} for schematic morphisms applies here.  Variants appear in \cite[Lem.~3.11 and~3.12]{LMB}, \cite[Tags 0300, 0301, 0302]{stacks}.  For completeness, we give a proof.

	\eqref{I:rep-comp} Suppose $\mathcal X \rightarrow \mathcal Y$ and $\mathcal Y \rightarrow \mathcal Z$ are $\mathsf C$-representable.  To see that the composition is $\mathsf C$-representable, consider  a $\mathcal W\in  \mathsf C$ and a morphism $\mathcal W\to \mathcal Z$. Then $\mathcal W \mathop{\times}_{\mathcal Z} \mathcal X = (\mathcal W \mathop{\times}_{\mathcal Z} \mathcal Y) \mathop{\times}_{\mathcal Y} \mathcal X$ is $\mathsf C$-representable  using first the $\mathsf C$-representability of $\mathcal Y \rightarrow \mathcal Z$, and then the $\mathsf C$-representability of  $\mathcal X \rightarrow \mathcal Y$.

	\eqref{I:rep-bc} If $\mathcal X \rightarrow \mathcal Y$ is a $\mathsf C$-representable morphism, $\mathcal W \rightarrow \mathcal Z \rightarrow \mathcal Y$ are morphisms of CFGs with $\mathcal W \in \mathsf C$, then $\mathcal W \mathop{\times}_{\mathcal Z} (\mathcal Z \mathop{\times}_{\mathcal Y} \mathcal X) = \mathcal W \mathop{\times}_{\mathcal Y} \mathcal X$, hence is in $\mathsf C$.

\eqref{I:rep-prod} 	
is essentially \cite[Rem.~(1.3.9) p.33]{EGAI}, which observes that the conclusion follows from (1) and (2), together with  the fact that given morphisms $f:\mathcal X\to \mathcal X'$ and $g:\mathcal Y\to \mathcal Y'$ over a stack $\mathcal Z$, the fibered  product   $\mathcal X\times_{\mathcal Z}\mathcal Y\xrightarrow{f\times_{\operatorname{id}_{\mathcal Z}}g} \mathcal X'\times_{\mathcal Z}\mathcal Y'$ is given by the composition of morphisms obtained from fibered product diagrams:
$$
\begin{CD}
\mathcal X\times_{\mathcal Z}\mathcal Y @>f\times_{\operatorname{id}_{\mathcal Z}} \operatorname{id}_{\mathcal Y}>> \mathcal X'\times_{\mathcal Z}\mathcal Y @>\operatorname{id}_{\mathcal X'}\times_{\operatorname{id}_{\mathcal Z}} g>> \mathcal X'\times_{\mathcal Z}\mathcal Y'.
\end{CD}
$$
	
\eqref{I:rep-rel}    For any $\mathcal W \rightarrow \mathcal Y$ with $\mathcal W$ in $\mathsf C$, we have a cartesian diagram
$$
\xymatrix@R=1em{
\mathcal W \mathop{\times}_{\mathcal Y} \mathcal X  \ar[r] \ar[d]& \mathcal W \mathop{\times}_{\mathcal Z} \mathcal X \ar[d]\\
\mathcal Y  \ar[r]^<>(0.5){\Delta_g} & \mathcal Y \mathop{\times}_{\mathcal Z} \mathcal Y.
}
$$ 
	We know $\mathcal W \mathop{\times}_{\mathcal Z} \mathcal X$ is {in $\mathsf C$} since $gf$ is representable, so $\mathcal W \mathop{\times}_{\mathcal Y} \mathcal X$ is {in $\mathsf C$} by representability of the diagonal of $g$.
\end{proof}

\begin{dfn}[Locality to the target for representable morphisms] \label{D:LocTargCFG}
Let $\mathbf P$ be a property of morphisms between CFGs in  $\mathsf C$ that is stable under base change (Definition \ref{D:StableClass}).  We say $\mathbf P$ is local to the target if, for any morphism $\mathcal X \rightarrow Y$ and any cover $\mathcal Z \rightarrow \mathcal Y$, the morphism $\mathcal X \mathbin\times_{\mathcal Z} \mathcal Y \rightarrow \mathcal Z$ has property $\mathbf P$ if and only if $\mathcal X \rightarrow \mathcal Y$ does.
\end{dfn}

Now that we have these definitions, we can translate  properties of morphisms in $\mathsf S$ into properties of representable morphisms of stacks:

\begin{dfn}[$\mathbf P$-representable morphism]   \label{D:RepMorphProp}
Let $\mathsf C$ be a class of CFGs over a subcanonical presite $\mathsf S$ such that $\mathsf C$ is stable under isomorphism and fiber product and every CFG over $\mathsf S$ can be covered by objects of $\mathsf C$.
Let $\mathbf P$ be a property of morphisms between CFGs in  $\mathsf C$ that is stable under base change (Definition \ref{D:StableClass}), stable under composition,  and local on the target (Definition \ref{D:LocTargCFG}).    A $\mathsf C$-representable morphism $\mathcal X\to \mathcal Y$ of stacks over  $\mathsf S$ (Definition \ref{D:Rep})  is said to have property $\mathbf P$ if for every $\mathcal Z$ in $\mathsf C$ and every $\mathcal Z\to \mathcal Y$, the morphism $\mathcal X\times_{\mathcal Y}\mathcal Z\to \mathcal Z$ of CFGs between objects of $\mathsf C$ has property $\mathbf P$.   We also call this a $\mathbf P$-representable morphism.
\end{dfn}

\begin{exa} \label{Exa:StabLoc}
Here are some examples of classes of morphisms of schemes that are stable under base change and local to the target for the \'etale topology on schemes:
\begin{itemize}
\item quasiseparated \cite[Prop.~(2.7.1)~(ii)]{EGAIV2};
\item quasicompact \cite[Cor.~(2.6.4)~(v)]{EGAIV2};
\item flat \cite[Prop.~(17.7.4)~(iii)]{EGAIV4};
\item smooth \cite[Prop.~(17.7.4)~(v)]{EGAIV4};
\item \'etale \cite[Prop.~(17.7.4)~(vi)]{EGAIV4};
\item unramified \cite[Prop.~(17.7.4)~(iv)]{EGAIV4};
\item separated \cite[Prop.~(2.7.1)~(i)]{EGAIV2};
\item proper \cite[Prop.~(2.7.1)~(vii)]{EGAIV2};
\item finite type \cite[Prop.~(2.7.1)~(v)]{EGAIV2};
\item locally of finite type \cite[Prop.~(2.7.1)~(iii)]{EGAIV2};
\item finite presentation \cite[Prop.~(2.7.1)~(vi)]{EGAIV2};
\item locally of finite presentation \cite[Prop.~(2.7.1)~(iv)]{EGAIV2};
\item locally of finite type and pure relative dimension $d$ \cite[Cor.~(4.1.4)]{EGAIV2};
\item surjective \cite[Props.~(3.6.1) and~(3.6.2)~(i)]{EGA1971} or \cite[Prop.~(2.6.1)]{EGAIV2}.
\end{itemize}
For more references, see \cite[(3.10)]{LMB} or \cite[Prop.~5.5]{DMstacks}.
\end{exa}

\subsection{Locality to the source}
\label{S:lts}

\begin{dfn}[Local to the source]\label{D:Loc-Src}
Let $\mathbf P$ be a property of morphisms in $\mathsf S$ that is local (on the target) and stable under base change and composition.  We call $\mathbf P$ \emph{local to the source} if a morphism $X \rightarrow Y$  in $\mathsf S$ has the property $\mathbf P$ if and only if, for any covering family $U_i \rightarrow X$ in the pretopology $\mathsf S$, all of the composed morphisms $U_i \rightarrow Y$ have property $\mathbf P$.
\end{dfn}

\begin{exa}
Let $\mathbf P$ be the class of \emph{\'etale morphisms} of schemes.  Then $\mathbf P$ is local to the source in the \'etale topology.
\end{exa}

\begin{exa}
Let $\mathbf P$ be the class of \emph{local isomorphisms} of complex analytic spaces.  Then $\mathbf P$ is local to the source in the  analytic  topology.  
\end{exa}

\begin{exa}
Recall that smooth surjections are covering (in the sense of Definition \ref{D:CovStack})  in the \'etale topology (Example~\ref{E:smooth-etale-cover}). 
 The class of smooth morphisms is local to the source in the \'etale topology~\cite[Tag~06F2]{stacks}.  We will prove a more general version of this statement in Lemma~\ref{L:intrinsic}.
\end{exa}

\begin{exa}
The same proof shows that smooth morphisms in the category of complex analytic spaces, and submersions in the category of $C^\infty$-manifolds, are local to the source in the usual topologies.
\end{exa}
 
\subsection{Adapted stacks} \label{S:AdaptStacks} 

Let $\mathsf C$ be a class of CFGs over a subcanonical presite $\mathsf S$ such that $\mathsf C$ is stable under isomorphism and fiber product and every CFG over $\mathsf S$ can be covered by objects of $\mathsf C$. 
Let $\mathbf P$ be a property of morphisms between CFGs in  $\mathsf C$ that is stable under base change (Definition \ref{D:StableClass}),  stable under composition, local to the source  (Definition \ref{D:Loc-Src}), and local to the target (Definition \ref{D:LocTargCFG}).

In this section we will introduce the class of stacks that admit $\mathbf P$-representable covers by objects of $\mathsf C$, calling these (Definition \ref{D:AdaptStack}) stacks \emph{$\mathbf P$-adapted (to $\mathsf C$)}, or just \emph{adapted stacks}, if the context is clear (in particular, if it will not be confused with a stack adapted to the presite).  In the next section, we will show that once one has the class of $\mathbf P$-adapted stacks,  the property $\mathbf P$ can always be defined  for morphisms between $\mathbf P$-adapted stacks, which will allow us to iterate this procedure in the next section and arrive at the definition of an algebraic stack (with respect to a class of morphisms $\mathbf P$ in $\mathsf S$).

Perhaps remarkably, many of the algebraic stacks we consider in this paper are smooth-adapted to schemes (the class $\mathsf C$ is taken to be $\mathsf S$ and the class of morphisms $\mathbf P$ is taken to be the class of smooth morphisms), so the reader who so desires may safely ignore the question of iteration and proceed after this section to Section~\ref{S:AlgStack2}.   In fact, for the purposes of this paper, we will only need to iterate once, to obtain the class of algebraic spaces, and smooth morphisms between algebraic spaces, as all of the stacks we will work with are smooth-adapted to algebraic spaces.

\begin{dfn}[Stacks $\mathbf P$-adapted to $\mathsf C$] \label{D:AdaptStack} 
Let $\mathsf C\supseteq \mathsf S$ be a class of CFGs over a subcanonical presite $\mathsf S$ such that $\mathsf C$ is stable under isomorphism and fiber product and every CFG over $\mathsf S$ can be covered by objects of $\mathsf C$. 
Let $\mathbf P$ be a property of morphisms between CFGs in  $\mathsf C$ that is stable under base change (Definition \ref{D:StableClass}), stable under composition,   local to the source  (Definition \ref{D:Loc-Src}), and local to the target (Definition \ref{D:LocTargCFG}).
   Then a stack on $\mathsf S$ is \emph{$\mathbf P$-adapted to $\mathsf C$} if it admits a $\mathbf P$-representable (Definition \ref{D:RepMorphProp}) cover (Definition \ref{D:CovStack}) by an object of $\mathsf C$.  

In other words, $\mathcal M$ is $\mathbf P$-adapted to $\mathsf C$ if there exists a $U$ in $\mathsf C$ and a    morphism 
$$
p :  U \longrightarrow \mathcal M
$$ 
such  that for every $\mathcal Z$ in $\mathsf C$ and every $\mathcal Z \to \mathcal M$, we have that $U\times_{\mathcal M}\mathcal Z$ is in $\mathsf C$,  and the morphism $U\times_{\mathcal M}\mathcal Z \to \mathcal Z $ has property  $\mathbf P$, and is covering.     Such a morphism  $p$ is called a \emph{presentation} of  $\mathcal M$.    

When $\mathbf P$, $\mathsf S$, and $\mathsf C$ are clear, we abbreviate the terminology to \emph{adapted stacks}.  A \emph{morphism} of adapted stacks is  a morphism of stacks.  
\end{dfn}
 
\begin{rem}
The presence of both $\mathsf S$ and $\mathsf C$ in the definition is for technical reasons.  We will later wish to iterate this construction, and we do not wish to undertake a definition of Grothendieck topologies on $2$-categories.  Had we given such a definition, we would have simply taken $\mathsf C$ to be a $2$-site and defined $\mathbf P$-adapted stacks to $\mathsf C$ without any reference to $\mathsf S$.  This would be similar to the approach taken in \S\ref{S:AdaptStackPre}; note that in that approach one can avoid $2$-cagetories by simply working with  presites on the category of algebraic spaces (rather than with all the stacks that arise from the adaption process to the presite).  
\end{rem}

\begin{exa}\label{E:DM-AS}
If $\mathsf C=\mathsf S$ (as categories)  and $\mathbf P$ is the class of coverings in the presite $(S,\mathscr T)$, then then 
 $\mathbf P$-adapted stacks are the same as the $\mathsf S$-adapted stacks of Definition \ref{D:SAdaptStack}.  
In particular, 
if $S$ is a scheme and $\mathsf S$ is the \'etale presite on schemes over $S$  (i.e., $(\mathsf S/S)_{\operatorname{et}}$), we take $\mathsf C=\mathsf S$ and $\mathbf P$ to be the class of \'etale covers, then the $\mathbf P$-adapted stacks are  called adapted Deligne--Mumford (F DM) stacks.      
The $\mathbf P$-adapted stacks that are representable by sheaves are   called algebraic  spaces (see \S \ref{S:AlgStk}).
 \end{exa}

\begin{exa}
If $\mathsf S = \mathsf C$ is the presite of complex analytic spaces, and $\mathbf P$ is the class of smooth morphisms, with the pretopology induced from  usual open covers, then one obtains a notion of an adapted complex analytic stack.  
Similarly, if $\mathsf S = \mathsf C$ is the presite of topological spaces with the presite induced from usual open covers, then one obtains a notion of an adapted topological  stack.  
There are other definitions of complex analytic and topological  stacks appearing in the literature; we do not pursue the relationship among the definitions.  
\end{exa}

\begin{lem} \label{L:fp-adapted}
Suppose that $\mathcal X \rightarrow \mathcal Z$ and $\mathcal Y \rightarrow \mathcal Z$ are morphisms of $\mathbf P$-adapted stacks.  Then $\mathcal X \mathbin\times_{\mathcal Z} \mathcal Y$ is a $\mathbf P$-adapted stack.
\end{lem}
\begin{proof}
Choose $\mathbf P$ covers $X_0 \rightarrow \mathcal X$, $Y_0 \rightarrow \mathcal Y$, and $Z_0 \rightarrow \mathcal Z$.  Then the map 
\begin{equation*}
X_0 \mathbin\times_{\mathcal Z} Y_0 \mathbin\times_{\mathcal Z} Z_0 \rightarrow \mathcal X \mathbin\times_{\mathcal Z} \mathcal Y \mathbin\times_{\mathcal Z} \mathcal Z = \mathcal X \mathbin\times_{\mathcal Z} \mathcal Y
\end{equation*}
is a composition of changes of base of representable $\mathbf P$ covers, hence is a representable $\mathbf P$ cover.
\end{proof}

\subsection{Iterated adaptation}
\label{S:P-adapted}

In order to iterate the definition of $\mathbf P$-adapted stacks, we must extend the property $\mathbf P$ to morphisms representable by adapted stacks.  The main content of this section is that there is a unique way to do this so that $\mathbf P$ remains stable under base change and composition, and local to the source and target.

\begin{rem}
Note that Lemma~\ref{L:intrinsic} gives a canonical way of extending the definition of smooth morphisms to all morphisms of CFGs over the category of schemes in a way that is still stable under composition and base change and local to source and target.  This definition necessarily agrees with Definition~\ref{D:P-adapted}, below, which is valid for an arbitrary class of morphisms that is stable under composition to base change and local to source and target.  The reader should feel free to skip to Definition~\ref{D:alg-stack} if he or she is only interested in algebraic stacks over schemes and is willing to rely on Lemma~\ref{L:intrinsic}.
\end{rem}

\begin{dfn}[Bootstrapping property $\mathbf P$] \label{D:P-adapted}
Let $\mathsf C\supseteq \mathsf S$ be a class of CFGs over a subcanonical presite $\mathsf S$ such that $\mathsf C$ is stable under isomorphism and fiber product and every CFG over $\mathsf S$ can be covered by objects of $\mathsf C$. 
Let $\mathbf P$ be a property of morphisms in $\mathsf C$ that is stable under base change (Definition \ref{D:StableClass}), stable under composition,   local to the source  (Definition \ref{D:Loc-Src}), and local to the target (Definition \ref{D:LocTargCFG}).
Suppose that $\mathcal X$ and $\mathcal Y$ are $\mathbf P$-adapted stacks, so that there are  presentations $X_0 \rightarrow \mathcal X$ and $Y_0 \rightarrow \mathcal Y$; i.e., $\mathbf P$-representable covers from objects in $\mathsf C$.  We say that \emph{$\mathcal X \rightarrow \mathcal Y$ is in class $\mathbf P$} if the map $X_0 \mathbin\times_{\mathcal Y} Y_0 \rightarrow Y_0$ is in class $\mathbf P$.
\end{dfn}

\noindent For brevity in what follows, we will sometimes write $\mathcal X\to \mathcal Y$ is $\mathbf P$ if it is in class $\mathbf P$. 

\begin{lem}
The definition of the class $\mathbf P$ of morphisms between $\mathbf P$-adapted stacks given above does not depend on the choices of presentations  $X_0 \rightarrow \mathcal X$ and $Y_0 \rightarrow \mathcal Y$.
\end{lem}
\begin{proof}
Suppose that $X'_0 \rightarrow \mathcal X$ is a different presentation.  Then let $X''_0 = X'_0 \mathbin\times_{\mathcal X} X_0$.  The projection $X''_0 \rightarrow X'_0$ is a $\mathbf P$ cover, since $\mathbf P$ covers are stable under base change, and $X_0 \rightarrow \mathcal X$ is a $\mathbf P$ cover.  Pulling back to $Y_0$ over $\mathcal Y$, we have a  diagram:
\begin{equation*} \xymatrix{
X''_0 \mathbin\times_{\mathcal Y} Y_0 \ar[r]^f \ar[d]_g & X'_0 \mathbin\times_{\mathcal Y} Y_0 \ar[d]^p \\
X_0 \mathbin\times_{\mathcal Y} Y_0 \ar[r]^-q & Y_0
} \end{equation*}
By assumption $q$ is $\mathbf P$, and $g$ is the base change of a $\mathbf P$ morphism, so it is $\mathbf P$.  Therefore $qg = pf$ is $\mathbf P$.  But $\mathbf P$ is local to the source, and $f$ is the base change of the $\mathbf P$ cover, $X''_0 \rightarrow X'_0$, so $p$ is $\mathbf P$, as required.  This proves the independence of the choice of $X_0$.

Now suppose that $Y'_0 \rightarrow \mathcal Y$ is another presentation.  Again we form $Y''_0 = Y'_0 \mathbin\times_{\mathcal Y} Y_0$ and note that the two projections are $\mathbf P$ covers, as is the map $Y''_0 \rightarrow \mathcal Y$.  Now consider the diagram with cartesian squares:
\begin{equation*} \xymatrix{
X_0 \mathbin\times_{\mathcal Y} Y_0 \ar[r]^-f & Y_0 \\
X_0 \mathbin\times_{\mathcal Y} Y''_0 \ar[r]^-h \ar[u] \ar[d]_p & Y''_0 \ar[d]^q \ar[u] \\
X_0 \mathbin\times_{\mathcal Y} Y'_0 \ar[r]^-g & Y'_0
} \end{equation*}
The map $f$ is $\mathbf P$, by assumption, so $h$ is $\mathbf P$ as well.  As $q$ is $\mathbf P$, so is $qh = gp$.  But $p$ is a $\mathbf P$ cover (being the base change of $q$), so it follows that $g$ is $\mathbf P$, as required.
\end{proof}

\begin{lem}
The property $\mathbf P$ of morphisms of stacks admitting $\mathbf P$-representable covers is stable under composition, stable under base change, local to the target, and local to the source.
\end{lem}
\begin{proof}
Consider a composition $\mathcal X \rightarrow \mathcal Y \rightarrow \mathcal Z$ where $X_0 \rightarrow X$, $Y_0 \rightarrow Y$, and $Z_0 \rightarrow Z$ are all $\mathbf P$-representable covering maps, where $\mathcal X \rightarrow \mathcal Y$ is $\mathbf P$.  Form the following diagram:
\begin{equation*} \xymatrix@R=20pt@!C=4pc{
& & X_0 \mathbin\times_{\mathcal Y} Y_0 \mathbin\times_{\mathcal Z} Z_0 \ar[dr]^(.6)f_(.3){\mathbf P}_(.6){(\mathbf{Cov})} \ar[dl]_(.5){\mathbf P}^{\mathbf{Cov}} \\
& X_0 \mathbin\times_{\mathcal Y} Y_0 \ar[dr]^{\mathbf P}_{(\mathbf{Cov})} \ar[dl]^{\mathbf{Cov}}_{\mathbf P} & & Y_0 \mathbin\times_{\mathcal Z} Z_0 \ar[dl]_(.6){\mathbf{Cov}}^{\mathbf P} \ar[dr]^g \\
X_0 \ar[d]_{\mathbf P}^{\mathbf{Cov}} & & Y_0 \ar[d]_{\mathbf P}^{\mathbf{Cov}} & & Z_0 \ar[d]_{\mathbf P}^{\mathbf{Cov}} \\
\mathcal X \ar[rr]^{(\mathbf{Cov})} & & \mathcal Y \ar[rr] & & \mathcal Z
} \end{equation*}
Morphisms that are $\mathbf P$ or covering without any further assumptions are labelled with a $\mathbf P$ or $\mathbf{Cov}$.  All of these morphisms are $\mathsf C$-representable.  If $\mathcal X \rightarrow \mathcal Y$ is covering then the other arrows labelled $(\mathbf{Cov})$ are covering as well.  

To prove stability under composition, suppose that $\mathcal Y \rightarrow \mathcal Z$ is $\mathbf P$.  Then $g$ is $\mathbf P$, so $gf$ is $\mathbf P$. 

On the other hand, we have a cartesian diagram:
\begin{equation*} \xymatrix@!C=3pc@R=1pc{
X_0 \mathbin\times_{\mathcal Y} Y_0 \mathbin\times_{\mathcal Z} Z_0 \ar[rr]_(.65){\mathbf{Cov}}_(.4){\mathbf{P}}^(.525){p} \ar[d] & & X_0 \mathbin\times_{\mathcal Z} Z_0 \ar[d] \ar[r]_-{\mathbf{Cov}}^-q & Z_0 \\
Y_0 \ar[rr]^(.35){\mathbf P}^(.65){\mathbf{Cov}}  & & \mathcal Y
} \end{equation*}
Since $Y_0$ is $\mathbf P$ over $\mathcal Y$, we deduce that $p$ is $\mathbf P$ and therefore that $qp = gf$ is $\mathbf P$.  But $p$ is covering, so by locality to the source, $q$ is $\mathbf P$.  This means that $\mathcal X \rightarrow \mathcal Z$ is $\mathbf P$, by definition.  This proves stability under composition.

To prove locality to the source, we use the same diagrams.  Suppose that $\mathcal X \rightarrow \mathcal Z$ is $\mathbf P$.  Then $q$ is $\mathbf P$, so $qp = gf$ is $\mathbf P$.  But then $f$ is covering and $\mathbf P$, so we deduce that $g$ is $\mathbf P$, as required.

Now we prove stability under base change.  Consider a cartesian diagram
\begin{equation} \label{E:base-change-local} \vcenter{\xymatrix{
\mathcal X' \ar[r] \ar[d] & \mathcal X \ar[d] \\ 
\mathcal Y' \ar[r] & \mathcal Y
}} \end{equation}
where $\mathcal X$ is $\mathbf P$ over $\mathcal Y$.  Let $\mathcal X' \rightarrow \mathcal X$ be the base change.  Suppose that $X_0 \rightarrow \mathcal X$ and $Y_0 \rightarrow \mathcal Y$ are $\mathbf P$ covers.  Then $p : X_0 \mathbin\times_{\mathcal Y} Y_0 \rightarrow Y_0$ is $\mathbf P$, by assumption.  Changing base to $\mathcal Y'$, we get $\mathbf P$ covers $X'_0 \rightarrow \mathcal X$ and $Y'_0 \rightarrow \mathcal Y$.  The map $q : X'_0 \mathbin\times_{\mathcal Y'} Y'_0 \rightarrow Y'_0$ is the base change of $p$.  But both $X_0 \mathbin\times_{\mathcal Y} Y_0$ and $Y_0$ are in $\mathsf C$, by assumption, and $\mathbf P$ is stable under base change for morphisms in $\mathsf C$.  Therefore $q$ is $\mathbf P$, from which it follows that $\mathcal X' \rightarrow \mathcal Y'$ is $\mathbf P$, by definition.

Finally, we prove locality on the target.  Suppose that we have a cartesian diagram~\eqref{E:base-change-local}.  Assume that $\mathcal Y'$ is a cover of $\mathcal Y$ and that $\mathcal X'$ is $\mathbf P$ over $\mathcal Y$.  Let $X_0 \rightarrow \mathcal X$ and $Y_0 \rightarrow \mathcal Y$ be $\mathbf P$ covers.  Assume that $\mathcal X'$ is $\mathbf P$ over $\mathcal Y'$.  Set $Y'_0 = Y_0 \mathbin\times_{\mathcal Y} \mathcal Y'$ and $X'_0 = X_0 \mathbin\times_{\mathcal X} \mathcal X'$.  Then $X'_0 \mathbin\times_{\mathcal Y'} Y'_0 \rightarrow Y'_0$ is the base change of $X_0 \mathbin\times_{\mathcal Y} Y_0$.  By assumption $Y'_0 \rightarrow Y_0$ is a cover, so by locality to the target for morphisms between objects of $\mathsf C$, we deduce that $X_0 \mathbin\times_{\mathcal Y} Y_0$ is $\mathbf P$ over $Y_0$.  This means that $\mathcal X$ is $\mathbf P$ over $\mathcal Y$, by definition.
\end{proof}

\begin{lem}
Given a class of morphisms $\mathbf P'$ of $\mathbf P$-adapted stacks that extends the class of morphisms $\mathbf P$ to the class of $\mathbf P$-adapted stacks, and which is 
 stable under base change and composition, and local to the source and target,  $\mathbf P'=\mathbf P$. 
\end{lem}

\begin{proof}
We leave the proof to the reader.
\end{proof}

\subsection{Algebraic stacks}

The lemmas in the previous section imply that the property $\mathbf P$ makes sense for stacks that are $\mathbf P$-adapted to $\mathsf C$.  Replacing $\mathsf C$ with the category of stacks on $\mathsf S$ that are $\mathbf P$-adapted to $\mathsf C$, we may iterate the procedure, enlarging the class of stacks under consideration and ensuring that the property $\mathbf P$ makes sense for them at each step.  Taking the union of all of these classes, we arrive at the class of algebraic stacks:

\begin{dfn}[Algebraic stack] \label{D:alg-stack}
Let $\mathbf P$ be a property of morphisms of $\mathsf S$ that is  stable under base change, stable under composition, local to the target, and local to the source.  An \emph{algebraic stack} (with respect to the property $\mathbf P$) is a member of the smallest class $\mathsf C$ of stacks on $\mathsf S$  such that $\mathsf C\supseteq \mathsf S$, $\mathsf C$ has  property $\mathbf P$ defined on it,  $\mathbf P$ is 
 stable under base change and composition, and local to the source and target,  and the class of stacks on $\mathsf S$ that are   $\mathbf P$-adapted to $\mathsf C$  (Definition \ref{D:AdaptStack}) is     $\mathsf C$.
\end{dfn}

\begin{dfn}[Deligne--Mumford stacks and algebraic spaces]\label{D:AlgDMAlgSp}
A stack on the category of schemes is called an \emph{algebraic stack} if it is algebraic with respect to the \'etale topology and the class of smooth maps.  It is called a \emph{Deligne--Mumford stack} if it is algebraic with respect to the \'etale topology and the class of \'etale maps.  It is called a \emph{algebraic space} if it is an algebraic stack and it is representable by a sheaf.
\end{dfn}

\begin{rem}
One can also use the usual topology on complex analytic spaces and take $\mathbf P$ again to be smooth morphisms.  Much of what we do here in the algebraic category works analytically as well, but we will not address the analytic category directly in what follows.
Likewise, there is a class of `algebraic stacks' on the category of manifolds, obtained by taking the usual topology on the category of manifolds and taking $\mathbf P$ to be the class of submersions.
\end{rem}

\subsection{Fiber products and bootstrapping}
\label{S:RepAlgStack}

\begin{cor} \label{C:fp-alg}
Fiber products of algebraic stacks are algebraic.
\end{cor}
\begin{proof}
This is immediate by iteration of Lemma~\ref{L:fp-adapted}.
\end{proof}

It can often be useful to show that a stack is $\mathbf P$-adapted to $\mathsf C$ using  bootstrap methods; in other words, it will often be the case that one can exhibit a morphism from a  stack of interest to a well-known adapted stack, so that it is easy to check the $\mathbf P$-adapted  condition  on the fibers.  The following proposition  states that in this situation the stack is a $\mathbf P$-adapted stack.   
We make this precise in the following remark and proposition.

The class of $\mathbf P$-adapted stacks is stable under isomorphism and fibered products (Lemma~\ref{L:fp-adapted}).  Therefore, from Definition \ref{D:Rep}, we have the notion of a morphism of stacks representable by adapted stacks. 

\begin{pro} \label{P:AdaptRel}
Let $f:\mathcal X\to \mathcal Y$ be a morphism of stacks over a subcanonical presite $\mathsf S$.  Assume that $\mathcal Y$ is a $\mathbf P$-adapted stack and that $f$ is  representable by $\mathbf P$-adapted stacks.  Then $\mathcal X$ is a $\mathbf P$-adapted  stack.
\end{pro}
\begin{proof}
Choose a $\mathbf P$-representable cover $Y_0 \rightarrow \mathcal Y$, with $Y_0 \in \mathsf C$.  Then then set $\mathcal Z = Y_0 \mathbin\times_{\mathcal Y} \mathcal X$.  The projection $\mathcal Z \rightarrow \mathcal X$ is a $\mathbf P$-representable cover.  Furthermore, $\mathcal Z$ is $\mathbf P$-adapted, by assumption, so there is a $\mathbf P$-representable cover $X_0 \rightarrow \mathcal Z$, with $X_0 \in \mathsf C$.  Then $X_0 \rightarrow \mathcal X$ is a composition of $\mathbf P$-representable covers, hence is a $\mathbf P$-representable cover, as required.
\end{proof}

\begin{cor}\label{L:AlgRel}
	Suppose that $\mathcal Y$ is an     algebraic stack and $f : \mathcal X \rightarrow \mathcal Y$ is a morphism of stacks   that is representable by     algebraic stacks.  Then $\mathcal X$ is an     algebraic stack.
\end{cor}
\begin{proof}
This is immediate by iteration of  Proposition~\ref{P:AdaptRel}.
\end{proof}


\section{Moduli stacks of Higgs bundles} \label{S:MSHB}

In this section we construct the moduli stack of Higgs bundles on a smooth complex projective curve.  We also construct several related moduli stacks of interest.    
In what follows, the reader should always feel free to assume that the morphism $\pi:X\to S$ is a projective morphism between schemes, and that all sheaves are coherent (or even vector bundles).  For instance, one case of special interest that will always satisfy the given hypotheses  will be the case where $S=\operatorname{Spec} \mathbb C$, $X$ is  a smooth complex projective curve (or compact Riemann surface), and the  sheaves are taken  to be the sheaves of sections of holomorphic vector bundles on $X$.  

In order not to avoid repetition below when defining various categories fibered in groupoids, we fix the following notation.  
Suppose that $f : T \rightarrow S$ is an $S$-scheme.  If $\xi$ is an object of some CFG on $\mathsf S/S$, we denote by $\xi_{T}$ a pullback of $\xi$ to $T$.  For example, if $X$ is an $S$-scheme, we denote by $X_{T}$ a $T$-scheme making the the following diagram cartesian:
\begin{equation*} \xymatrix{
X_T \ar[r] \ar[d] & X \ar[d] \\
T \ar[r]^f & S.
}
\end{equation*}
If $L$ is a line bundle on $X$, we denote by $L_{T} = f^\ast L$ a pullback of that line bundle to $X_T$.  Note that we have somewhat abusively used $f^*L$ to denote the pullback of $L$ via the morphism $X_T \rightarrow X$ induced from $f$.  We shall make this abuse repeatedly in what follows.

When $T$ is denoted by decorating $S$, we abbreviate our notation further and denote $X_T$ by decorating $X$ the same way as $S$; i.e., if $T = S'$ then we allow ourselves to write $X'$ instead of $X_{S'}$.

\subsection{The moduli space of curves}

We will show that the stack of smooth curves is algebraic  using the representability of the Hilbert scheme.  We give another proof in Section~\ref{S:def-thy}, using Artin's criteria (see Theorem~\ref{T:cvbm-stack}).

The following is a useful criterion to check properties of the diagonal of an algebraic  moduli stack of schemes.  

\begin{teo} [{\cite[Thm.~5.23, p.133]{FGAe}}]  Let $S$ be a noetherian scheme, let $X\to S$ be a flat, projective scheme over $S$, and let $Y\to S$ be a quasi-projective scheme over $S$.  Then the functor
$
\mathscr I\!\!\mathit{som}_{\mathsf S/S}(X,Y)
$ on schemes over $S$ is representable by a scheme.
\end{teo}

\begin{rem}
	In fact $\mathscr I\!\!\mathit{som}_{\mathsf S/S}(X,Y)$ is representable  by an open subscheme of the Hilbert scheme  $\operatorname{Hilb}_{X\times_SY/S}$.
\end{rem}

\begin{rem}
A special case of a theorem of Olsson  (\cite[Thm.~1.1]{olssonHom}) allows one to drop the noetherian and projective hypotheses, in exchange for requiring finite presentation and settling for an algebraic space rather than scheme.   \emph{Let $S$ be a  scheme, let $X$ be a flat, proper scheme of finite presentation over $S$, and let $Y$ be  a  separated scheme of finite presentation over $S$.  Then the functor
$
\mathscr I\!\!\mathit{som}_{\mathsf S/S}(X,Y)
$
is representable by an algebraic space over $S$.}
\end{rem}

\begin{teo}[{\cite[Prop.~(5.1)]{DM}}]
For $g\ge 2$, the stack  $\mathcal M_g\to (\mathsf S/\mathbb C)_{\mathsf {et}}$ is a Deligne--Mumford stack.  
\end{teo}

\begin{proof}
The previous theorem asserts that $\Delta:\mathcal M_g\to \mathcal M_g\times_{\mathbb C}\mathcal M_g$ is schematic. 
More detailed analysis of the isomorphisms of algebraic curves establishes that the the diagonal is unramified.    Since we are in characterstic $0$, this essentially follows from the fact that  the automorphism group  of a smooth curve of genus $g\ge 2$  is finite, and group schemes in characteristic $0$ are smooth.    

Now let $H_g$ be the open subset of the Hilbert scheme parameterizing smooth, $\nu$-canonically embedded curves for $\nu\gg 0$. 
   The universal family $C_g\to H_g$ determines a morphism $P:H_g\to \mathcal M_g$, that is schematic by virtue of the fact that the diagonal is schematic (see Lemma \ref{L:DiagIsomRep}).   One can check that $P$ is smooth, and therefore provides a presentation of the stack.   Therefore $\mathcal M_g$ is DM, since the diagonal is unramified (see Lemma \ref{L:UnRamDiagDM}).
\end{proof}

\subsection{The $\mathsf {Quot}$ stack and the $\operatorname{Quot}$ scheme}   

Let $\pi:X\to S$ be a proper  morphism of finite presentation between schemes.  
Let $E$ be a quasicoherent sheaf of finite presentation on  $X$.    We define a category fibered in groupoids 
$$
\mathsf {Quot}_{E/X/S}\to \mathsf S/S
$$ 
as follows. 
 For each $S$-scheme  $f:S'\to S$, we set the objects of $\mathsf {Quot}_{E/X/S}(S')$ to be surjections 
$$
\begin{CD}
E'@>q'>> F'@>>> 0
\end{CD}
$$
where, as was introduced at the beginning of Section~\ref{S:MSHB}, $E'$ denotes the pullback of $E$ to $X' = X\mathbin\times_S S'$.  The sheaf $F'$ is required to be an $S'$-flat quasicoherent sheaf of finite presentation on $X'$.  We will denote such an object by the pair $(F',q')$.

Given an $S$-morphism  $g:S''\to S'$  and object $(F'',q'') \in \mathsf {Quot}_{E/X/S}(S'')$, a morphism over $g$ from $(F'',q'')$ to $(F',q')$ is a commutative diagram of quasicoherent sheaves on $X''$:
\begin{equation*} \xymatrix@R=15pt{
E'' \ar[r]^{q''} \ar@{=}[d] & F'' \ar[r] \ar[d]^<>(0.5){\wr} & 0 \\
g^\ast E' \ar[r]^{g^\ast q'} & g^\ast F' \ar[r] & 0
} \end{equation*}
(The equality on the left indicates the canonical isomorphism induced by pullback.)

As quotients of a coherent sheaf cannot have any nontrivial automorphisms respecting the quotient map, the $\mathsf{Quot}_{E/X/S}$ CFG is equivalent to the CFG associated to the presheaf 
\begin{equation*}
\mathscr Quot_{E/X/S}:(\mathsf S/S)^{\mathsf {op}}\to (\mathsf {Set})
\end{equation*}
wherein $\mathscr Quot_{E/X/S}(T)$ is defined to be the set of isomorphism classes of objects of $\mathsf{Quot}_{E/X/S}(T)$.  Of course, $\mathsf{Quot}_{E/X/S}$ is not literally equal to $\mathscr Quot_{E/X/S}$, since one is a CFG and the other is a presheaf.

\begin{rem}
By dualizing, one can check that the  Grassmannian $\mathcal G(r,n)$ is a substack  of  $\mathsf {Quot}_{\mathscr O_{\operatorname{Spec}\mathbb C}^{\oplus n}/\operatorname{Spec}\mathbb C/\operatorname{Spec}\mathbb C}$.
\end{rem}

The following is a special case of a result of Lieblich:  
\begin{teo}[{\cite[Prop.~2.7]{lieblich062}, \cite[Tag 08KA]{stacks}}] \label{T:LP2.7}
Let $\pi:X\to S$ be a proper  morphism of finite presentation between schemes.  
 The CFG $\mathsf {Quot}_{E/X/S}$ is an algebraic space 
 that is locally of finite presentation over $S$.
\end{teo}

More can be said for projective morphisms $\pi:X\to S$.   Let $L$ be a relatively very ample line bundle on $X/S$.  If $S'=\operatorname{Spec}k$ for a field $k$, and $F'$ is a coherent sheaf on $X_{S'}$, then the Hilbert function of $F'$ with respect to $L'$ is defined to be 
$$
\Phi(m) : = \chi(X',F'(m))=\sum_{i=0}^{\dim X}(-1)^i\dim_k H^i(X',F'\otimes (L')^{\otimes m}).
$$
This is in fact a polynomial in $\mathbb Q[m]$ (see \cite[\S 5.1.4]{FGAe} for more details).    There is a decomposition 
$$
\mathsf{Quot}_{E/X/S} =\coprod_{\Phi\in \mathbb Q[m]} \mathsf{Quot}^{\Phi,L}_{E/X/S}
$$
where $\mathsf{Quot}_{E/X/S}^{\Phi,L}(S')$ is the set of equivalence classes $\langle F',q' \rangle$ over $S'$ such that for each $s'\in S'$, the Hilbert polynomial of $F'_t$ with respect to $L'_t$ is equal to $\Phi$.    Crucially, cohomology and base change implies that the inclusions of the subfunctors $\mathsf{Quot}^{\Phi,L}_{E/X/S} \subseteq \mathsf{Quot}_{E/X/S}$ are representable by \emph{open and closed} subfunctors of $\mathsf{Quot}_{E/X/S}$.  This reduces the study of $\mathsf{Quot}_{E/X/S}$ to that of the subfunctors $\mathsf{Quot}_{E/X/S}^{\Phi,L}$.

The main theorem is due to Grothendieck \cite{FGA} (see also \cite{AK80}, and \cite[Thm.~5.14, p.127]{FGAe}). 

\begin{teo}[{\cite{FGA}}] \label{T:Grass}
Let $S$ be a  noetherian scheme, $\pi:X\to S$ a projective morphism, and $L$ a relatively very ample line bundle on $X/S$.  Then for any coherent sheaf $E$ and any polynomial $\Phi\in \mathbb Q[m]$, the CFG $\mathsf{Quot}_{E/X/S}^{\Phi,L}$ is representable by a projective $S$-scheme $\operatorname{Quot}_{E/X/S}^{\Phi,L}$.
\end{teo}

\begin{rem}
By dualizing one can check  that the Grassman CFG $\mathcal{G}(r,n)$ is isomorphic to the CFG $\mathsf{Quot}_{\mathcal O_{\operatorname{Spec}\mathbb C}^{\oplus n}/\operatorname{Spec}\mathbb C/\operatorname{Spec} \mathbb C}^{r,\mathcal O_{\operatorname{Spec}\mathbb C}}$, and  the Grassmannian $G(r,n)$ is isomorphic to the Quot scheme  $\operatorname{Quot}_{\mathcal O_{\operatorname{Spec}\mathbb C}^{\oplus n}/\operatorname{Spec}\mathbb C/\operatorname{Spec} \mathbb C}^{r,\mathcal O_{\operatorname{Spec}\mathbb C}}$.  Of course, one cannot use this observation to construct the Grassmannian, as Grothendieck's proof of Theorem~\ref{T:Grass} relies on the representability of the Grassmannian by a projective scheme.
\end{rem}

\subsection{Stacks of quasicoherent sheaves}  

Let $S$ be a scheme.  
Let $\pi:X\to S$  be a proper morphism of  finite presentation between schemes. We define a category fibered in groupoids  
$$
\mathsf {QCoh}_{X/S}\to \mathsf S/S
$$
in the following way.  For an $S$-scheme $f:S'\to S$ in $\mathsf S/S$, we take $\mathsf {QCoh}_{X/S}(S')$ to consist (in the notation introduced at the beginning of Section~\ref{S:MSHB}) of the $S'$-flat quasicoherent  sheaves \cite[Tag 01BE]{stacks} 
on $X_{S'}$. 
 Morphisms in $\mathsf {QCoh}_{X/S}$ are defined by pullback; i.e., if $g:S''\to S'$ is a morphism, $F''$ is an object of $\mathsf {QCoh}_{X/S}(S'')$   and $F'$ is an object of $\mathsf {QCoh}_{X/S}(S')$, we define a morhpism $F''\to F'$ to be an isomorphism $F'' \rightarrow g^*F'$.
Following \cite{lieblich062}, we use $\mathsf{Coh}_{X/S}$ to denote the substack of $\mathsf{QCoh}_{X/S}$ consisting of quasicoherent sheaves \emph{of finite presentation}  \cite[Tag 01BN]{stacks}  (when the structure sheaf is coherent, finitely presented sheaves coincide with coherent sheaves \cite[Tag 01BZ]{stacks}).  We similarly define $\mathsf {Fib}_{X/S}$ and $\mathsf {Fib}(r)_{X/S}$ by restricting to locally free sheaves of finite rank and locally free sheaves of rank $r$, respectively.

The first statement we want is descent for quasicoherent sheaves.  This is \cite[VIII, Thm.1.1, Cor.1.2, p.196]{SGA1}.  See also \cite[(3.4.4)]{LMB}, and \cite[Thm.~4.23, p.82]{FGAe}.    The same  argument shows that  the    statements hold for coherent sheaves, and for vector bundles.  

\begin{teo}[{\cite[Exp.~VIII, Thm.~1.1, p.196]{SGA1}}] \label{T:DQC}
Let $S$ be a scheme, and let $X$ be a scheme over $S$.  The CFG $\mathsf {QCoh}_{X/S}\to \mathsf S/S$ is a stack with respect to the fpqc topology.  The same is true for  $\mathsf {Coh}_{X/S}$, $\mathsf {Fib}_{X/S}$ and $\mathsf {Fib}(r)_{X/S}$.
\end{teo}

Moreover, these are algebraic stacks.  The following is  a special case of a result of Lieblich:

\begin{teo}[{\cite[Thm.~2.1]{lieblich062}}]\label{T:CohMax}
 Let $\pi:X\to S$  be a proper morphism of  finite presentation between schemes, with $S$ an excellent scheme (see e.g., \cite[Tag 07QS]{stacks}; for instance $S$ can be a scheme of finite type over a field).   The stack $\mathsf {Coh}_{X/S}$   is an algebraic  stack, 
 locally of finite presentation over S.  The same is true for $\mathsf {Fib}_{X/S}$ and $\mathsf {Fib}(r)_{X/S}$.
\end{teo}

For a projective morphism of noetherian schemes, the $\operatorname{Quot}$ scheme can be used to establish the presentation of the Artin stack (see Remark \ref{R:QuotPres}).     For our purposes, the benefit of this perspective will be in comparing the moduli stack of Higgs bundles to the moduli scheme of semistable Higgs bundles constructed via GIT.

\begin{teo}[{\cite[Thm.~4.6.2.1, p.29]{LMB}}]\label{T:CohP}
Let $S$ be a noetherian scheme, and let $\pi:X\to S$ be a projective morphism.  Assume that $\pi_*\mathscr O_X=\mathscr O_S$ universally (i.e., $\pi'_*\mathscr O_{X'}=\mathscr O_S'$ for all $S'\to S$).     Then the $S$-stack $\mathsf {Coh}_{X/S}$ is an algebraic $S$-stack locally
 of finite type.     The same is true for $\mathsf {Fib}_{X/S}$ and $\mathsf {Fib}(r)_{X/S}$.
\end{teo}

\begin{rem}\label{R:CD0} The hypothesis that $\pi_\ast \mathcal O_X = \mathcal O_S$ universally is satisfied whenever $X$ is projective and flat over $S$ with reduced, connected geometric fibers \cite[Ex.~9.3.11, p.303]{FGAe}.
\end{rem}

\begin{rem}\label{R:QuotPres}
For each coherent sheaf $E$ on $X$, 
a universal quotient over the scheme  $\operatorname{Quot}_{E/X/S}$   induces a morphism $\operatorname{Quot}_{E/X/S}\to \mathsf {Coh}_{X/S}$.  Taking appropriate open subsets of these $\operatorname{Quot}$ schemes gives a presentation of the algebraic  stack (see \cite[p.30]{LMB} for more details).  
\end{rem}

\begin{rem}
Suppose that $S=\operatorname{Spec}k$ for an algebraically closed field $k$, and that $X$ is  a smooth projective curve.    We can define the CFG $\mathsf {Fib}_{X/S}(r,d)$ by restricting to vector bundles of degree $d$.  All of the statements above hold in this setting; i.e., the CFG $\mathsf {Fib}_{X/S}(r,d)$ of vector bundles of rank $r$ and degree $d$ on $X$ is an algebraic stack, of finite type over $k$, that admits a presentation from open subsets of $\operatorname{Quot}$ schemes. 
\end{rem}


\subsection{The stack of Higgs bundles over a smooth projective curve} \label{S:HigSmCur}

Let $X$ be a smooth, projective curve over $\mathbb C$, of genus $g$.   Fix $r\ge 1$ and $d\in \mathbb Z$.    A \emph{Higgs bundle} on $X$ of rank $r$ and degree $d$ consists of a pair $$(E,\phi)$$ where  $E$ is a vector bundle (locally free sheaf of finite rank) on $X$ with $\operatorname{rank}E=r$ and $\deg E=d$, and $$\phi\in \operatorname{Hom}_{\mathscr O_X}(E,E\otimes K_X),$$ where $K_X=\Omega^1_X=\omega_X$ is the canonical bundle on $X$.   The aim of this section is to construct  an  \emph{algebraic stack $\mathcal {H}_{X/\operatorname{Spec}\mathbb C}(r,d)$ of Higgs bundles on $X$ of rank $r$ and degree $d$,  over the \'etale site $\mathsf S/\operatorname{Spec}\mathbb C$}.

\subsubsection{Higgs bundles as a category fibered in groupoids}
  We begin by defining the CFG underlying the stack    $\mathcal {H}_{X/\operatorname{Spec}\mathbb C}(r,d)$.
Given a $\mathbb C$-scheme $f:S'\to \operatorname{Spec}\mathbb C$, 
we start by defining $\mathcal {H}_{X/\operatorname{Spec}\mathbb C}(r,d)(S')$ to have objects 
consisting  of  pairs $$(E',\phi')$$ where $E' \in \mathsf {Fib}_{X/\operatorname{Spec} \mathbb C}(r,d)(S')$ 
 is a relative  vector bundle of rank $r$ and degree $d$ on $X'/S'$,  and 
$$
\phi' \in   
\operatorname{Hom}_{\mathcal O_{X'}}(  E', E'\otimes f^*K_X).
$$    
 Given a $\mathbb C$-morphism  $g:S''\to S'$, and   $(E',\phi')$ in $\mathcal {H}_{X/\operatorname{Spec}\mathbb C}(r,d)(S')$,  we obtain a pulled-back family 
$$( E'',\phi''):=g^*( E',\phi')$$
in $\mathcal {H}_{X/\operatorname{Spec}\mathbb C}(r,d)(S'')$, 
where (in the notation from the beginning of Section~\ref{S:MSHB}) $  E''= g^* E'$ and  $\phi''= g^*\phi'$. 

We now define the morphisms of $\mathcal {H}_{X/\operatorname{Spec}\mathbb C}(r,d)$ in the following way.  Given $( E'',\phi'')$ in $\mathcal {H}_{X/\operatorname{Spec}\mathbb C}( F)(S'')$, and $( E',\phi')$ in $\mathcal {H}_{X/\operatorname{Spec}\mathbb C}( F)(S')$, then a morphism $( E'',\phi'')\to ( E',\phi')$ over $g:S''\to S'$ consists of an isomorphism 
\begin{equation*}
\alpha : E'' \rightarrow g^\ast E'
\end{equation*}
such that the following diagram commutes:
\begin{equation*} \xymatrix@R=15pt{
E'' \ar[r]^-{\phi''} \ar[dd]_{\alpha} & E'' \otimes (fg)^\ast K_X \ar[d]^{\alpha \otimes \operatorname{id}} \\
& g^\ast E' \otimes (fg)^\ast K_X \ar@{=}[d] \\
g^\ast E' \ar[r]^-{g^\ast \phi'} & g^\ast(E' \otimes f^\ast K_X).
} \end{equation*}

We now have a category fibered in groupoids:
$$
\mathcal {H}_{X/\operatorname{Spec}\mathbb C}(r,d)\to \mathsf S/\operatorname{Spec}\mathbb C.
$$

\subsubsection{Higgs bundles as an algebraic stack}

The key point is the following well-known lemma:

\begin{lem} \label{L:Ray}
Let $\pi:X\to S$ be a flat, proper morphism of finite presentation between schemes.  
Let $G$ be an $S$-flat quasicoherent sheaf of $\mathscr O_X$-modules that is of finite presentation.  There exists a quasicoherent $\mathscr O_S$-module $ M$ of finite presentation, such that the linear scheme  $V=\underline{\raisebox{0pt}[0pt][-.5pt]{$\operatorname{Spec}$}}_S\left(\operatorname{Sym}_{\mathscr O_S}^\bullet M\right)$  defined by $M$  represents the functor
$$
(S' \xrightarrow{f} S) \mapsto \Gamma(X',f^*G),
$$
(see the beginning of Section~\ref{S:MSHB} for notation).   The formation of $V$ commutes with all base changes $S'\to S$.   This functor will  be denoted by $\mathscr Hom_{X/S}(\mathscr O_X,G)$.   
\end{lem}

Lemma \ref{L:Ray} is  essentially contained in \cite[Cor.~7.7.8, Rem.~7.7.9, p.202--3]{EGAIII2}.  This can also be found in \cite[p.28]{raynaud70}, \cite[Lem.~3.5]{N91}, and detailed proofs are given in \cite[p.206--207]{BLR} and \cite[Thm.~5.8, p.120]{FGAe}.

\begin{cor} \label{C:SchMor1}
The forgetful morphism of CFGs
$$
\mathcal {H}_{X/\operatorname{Spec} \mathbb C}(r,d)\to \mathsf {Fib}_{X/\operatorname{Spec} \mathbb C}(r,d)
$$
is schematic.
\end{cor}

\begin{proof}
An $S$-morphism $S'\to \mathsf {Fib}_{X/\operatorname{Spec}\mathbb C}(r,d)$  corresponds to a vector bundle $E'$ over $X'$.  Using the construction of the fibered product $\mathcal {H}_{X/\operatorname{Spec} \mathbb C}(r,d)\times_{\mathsf {Fib}_{X/S}(r,d)} S'$ in Definition  \ref{D:2-product-categories-over-C}, one can establish that this is equivalent to the  {functor on $\mathsf S/S'$}
$$\mathscr Hom_{X'/S'}(\mathscr O_{X'},{E'}^\vee\otimes E'\otimes f^*K_X).$$ 
This  is representable by a scheme, by virtue of the lemma above.  
\end{proof}

\begin{cor}\label{C:HrdS}  The CFG 
$\mathcal {H}_{X/\operatorname{Spec} \mathbb C}(r,d)\to \mathsf S/\operatorname{Spec} \mathbb C$ is a stack.
\end{cor}

\begin{proof}  We have seen in Theorem \ref{T:DQC} that $\mathsf {Fib}_{X/\operatorname{Spec} \mathbb C}(r,d)$ is a stack.  Therefore, the corollary follows from Corollary \ref{C:SchMor1} and Corollary \ref{C:DescRel}.
\end{proof}

\begin{teo}  The CFG of Higgs bundles 
$\mathcal {H}_{X/\operatorname{Spec} \mathbb C}(r,d)\to \mathsf S/\operatorname{Spec} \mathbb C$ is an algebraic stack, locally of finite type over $\operatorname{Spec} \mathbb C$.  
\end{teo}

\begin{proof}  We have seen in Corollary \ref{C:HrdS},  that $\mathcal {H}_{X/\operatorname{Spec}\mathbb C}(r,d)$ is a stack.  Under the hypotheses here, it follows from Theorem \ref{T:CohP} (see also Remark \ref{R:CD0}) that $\mathsf {Fib}_{X/\operatorname{Spec}\mathbb C}(r,d)$ is an algebraic stack locally of finite type over $\mathbb C$.  Thus from  Corollary 
 \ref{C:SchMor1} and Corollary \ref{L:AlgRel} we have that   $\mathcal {H}_{X/\operatorname{Spec}\mathbb C}(r,d)$ is an algebraic stack, locally  of finite type over $\mathbb C$. 
 \end{proof}

\subsection{Meromorphic Higgs bundles, and the stack of sheaves  and endomorphisms}\label{S:MeroHiggs}

There is also  interest in considering so-called meromorphic Higgs bundles on a smooth projective curve $X$; that is pairs $(E,\phi)$ where $E$ is a vector bundle and $\phi:E\to E\otimes K_X(D)$ is a morphism of sheaves, for some fixed divisor $D$ on $X$.   The construction of such a stack can be made in essentially the same way as for Higgs bundles.  Since it is not much more work, we provide here a more general construction.

In this setting, the input data   are the following:
\begin{enumerate}
\item $\pi:X\to S$  a proper morphism of finite presentation between schemes, such that either\begin{enumerate}
\item $S$ is excellent, or,

\item $S$ is noetherian, $\pi$ is projective, and $\pi_*\mathscr O_X=\mathscr O_S$ universally.  
\end{enumerate}

\item   $ F$  an $S$-flat  quasicoherent  sheaf of $\mathscr O_X$-modules of finite presentation.
\end{enumerate}
The output from this data will be an  \emph{algebraic stack $\mathcal {E}_{X/S}( F)$ of endomorphisms of coherent sheaves with values in $ F$,  over the site $(\mathsf S/S)_{\operatorname{et}}$}.
Taking $X/S$ to be a smooth complex projective curve over $S=\operatorname{Spec}\mathbb C$, $F=K_X(D)$, and restricting to vector bundles, one obtains the stack of meromorphic Higgs bundles associated to the divisor $D$.

\begin{rem}
While the construction above generalizes the construction of Higgs bundles for curves, it seems that for many applications these  stacks are  not   the correct generalizations.   First in higher dimension (even in the smooth case over $S=\operatorname{Spec}\mathbb C$ and  with $F=\Omega_X^1$), one should at least include the integrability condition $\phi\wedge \phi =0$ (in $\underline{\operatorname{End}}(E)\otimes \Omega_X^2$).   Second,  in any dimension, it is often  better to put 
derived structures on $X$ and $S$ so that the relative cotangent complex
becomes perfect, and then consider the moduli problem on the derived
scheme $X$.  The result will be a derived stack (a notion which we will not introduce here). 
One case where the generalized construction we take here does suffice is for families of stable curves, where one can consider endomorphisms  with values in the relative dualizing sheaf $\omega_{X/S}$.   This will be discussed further below.  
\end{rem}

\subsubsection{The CFG   $\mathcal {E}_{X/S}( F)$}
  We begin by defining the category fibered in groupoids underlying the stack    $\mathcal{E}_{X/S}( F)$.
Given a morphism $f:S'\to S$, 
we start by defining $\mathcal{E}_{X/S}( F)(S')$ to have objects 
those  pairs $$(E',\phi')$$ where  $ E'$ 
 is an $S'$-flat coherent sheaf  on $X'$   and 
$$
\phi' \in  
\operatorname{Hom}_{\mathcal O_{X'}}(  E', E'\otimes F'),
$$ 
where $F'= f^*F$ and $X' = X \mathbin\times_S S'$ (as always, we use the notation introduced at the beginning of Section~\ref{S:MSHB}).
 
 Given an $S$-morphism  $g:S''\to S'$, we again  denote the composition $f\circ g=f'$, and will  use the notation introduced at the beginning of Section~\ref{S:MSHB}.  
Given  $(E',\phi')$ in $\mathcal{E}_{X/S}( F)(S')$,  we obtain a pulled-back family 
$$g^*( E',\phi') = (g^\ast E', g^\ast \phi')$$
in $\mathcal{E}_{X/S}( F)(S'')$, 

We now define morphisms in the following way.  Given $( E'',\phi'')$ in $\mathcal{E}_{X/S}( F)(S'')$, and $( E',\phi')$ in $\mathcal{E}_{X/S}( F)(S')$, a morphism $( E'',\phi'')\to ( E',\phi')$ over $g:S''\to S'$ consists of an isomorphism 
\begin{equation*}
\alpha : E'' \rightarrow g^\ast E'
\end{equation*}
such that the following diagram commutes:
\begin{equation*} \xymatrix@R=15pt{
E'' \ar[r]^-{\phi''} \ar[dd]_{\alpha} & E'' \otimes F'' \ar[d]^{\alpha \otimes \operatorname{id}} \\
& g^\ast E' \otimes F'' \ar@{=}[d] \\
g^\ast E' \ar[r]^-{g^\ast \phi'} & g^\ast(E' \otimes f^\ast F')
} \end{equation*}

Now we have the category fibered in groupoids  
$$
\mathcal{E}_{X/S}( F)\to \mathsf S/S.
$$

\subsubsection{$\mathcal{E}_{X/S}( F)$ is a stack} \label{S:EX/S(F)}
In this setting, we replace Lemma \ref{L:Ray} with a special case of a result due to Lieblich.    Let $X \to S$  be a proper morphism of finite presentation between schemes.  Let $E$ and $G$ be finitely presented, quasicoherent sheaves on $X$ such that $G$ is $S$-flat.  We define a CFG $$\mathscr Hom_{X/S}(E , G ) \to \mathsf S/S$$ by assigning  to each morphism $f:S'\to S$  the set of homomorphisms $$ \operatorname{Hom}_{\mathscr O_{X'}}(E_{S'} , G_{S'} ).$$ Morphisms are defined by pullback.  

\begin{pro}[{Leiblich \cite[Prop.~2.3]{lieblich062}}] \label{P:LiebP2.3}
Let $X \to S$  be a proper morphism of finite presentation between schemes.  Let $E$ and $G$ be finitely presented, quasicoherent sheaves on $X$ such that $G$ is $S$-flat. The CFG $\mathscr Hom_{X/S}(E , G )\to \mathsf S/S$ is representable by an algebraic space   over $S$, locally of finite type.
\end{pro}

\begin{cor}\label{C:RepMor1}
The forgetful morphism of CFGs
$$
\mathcal{E}_{X/S}(F)\to \mathsf {Coh}_{X/S}
$$
is representable by algebraic spaces,
 locally of finite type over $S$.  
\end{cor}

\begin{proof}
An $S$-morphism $S'\to \mathsf {Coh}_{X/S}$  corresponds to   $E'$ in $\mathsf {Coh}_{X/S}(S')$.  Therefore the fibered product $\mathcal{E}_{X/S}(F)\times_{\mathsf {Coh}_{X/S}} S'$ is the CFG 
$\mathscr Hom_{X/S}(E' , E'\otimes F')$, which is representable by virtue of Lieblich's result above.  
\end{proof}

\begin{cor}\label{C:HSS}  The CFG 
$\mathcal{E}_{X/S}(F)$ is a stack.
\end{cor}

\begin{proof}  We have seen in Theorem \ref{T:DQC} that $\mathsf {Coh}_{X/S}$ is a stack.  Therefore, the corollary follows from Corollary \ref{C:RepMor1} and Corollary  \ref{C:DescRel}.
\end{proof}

\subsubsection{$\mathcal{E}_{X/S}( F)$ is an algebraic  stack}

\begin{teo}
$\mathcal{E}_{X/S}(F)$ is an algebraic stack,
 locally of finite presentation over~$S$.
\end{teo}

\begin{proof}  We have seen in Corollary \ref{C:HSS},  that $\mathcal{E}_{X/S}(F)$ is a stack.  Under the hypotheses here, it follows from Theorem \ref{T:CohMax}   that $\mathsf{Coh}_{X/S}$ is an algebraic stack, 
locally of finite presentation over $S$.  Thus from  Corollary 
 \ref{C:RepMor1} and Corollary \ref{L:AlgRel} we have that   $\mathcal{E}_{X/S}(F)$ is an algebraic stack,
  locally of finite presentation  over $S$.  
  \end{proof}

\subsection{The stack of Higgs bundles over the moduli of stable curves}\label{S:HiggMg}

We now construct a moduli stack of Higgs bundles over the moduli stack of stable curves:
$$
\mathcal {HS}h_{\overline{\mathcal M}_g}\to \overline{\mathcal M}_g.
$$

 We start by fixing a base $S$ in $\mathsf S$ (for instance $S=\operatorname{Spec}\mathbb C$, or even $S=\operatorname{Spec}\mathbb Z$).    

\begin{dfn}[Stable curve] \label{D:stable-curve}
A \emph{stable curve} of genus $g\ge 2$ over an algebraically closed field $k$ is a connected curve of arithmetic genus $g$,  with no singularities apart from nodes, and with finite automorphism group (every component isomorphic to $\mathbb P^1_k$ meets the rest of the curve in at least $3$ points).   

For an $S$-scheme $S'\to S$,  a \emph{relative stable  curve of genus $g$} over $S'$  is  
a surjective morphism of schemes  $\pi:X'\to S'$  that is flat, proper and whose every geometric fiber is a stable curve of genus $g$.
\end{dfn}

We define the CFG $\overline{\mathcal M}_g$ over $\mathsf S/S$.  For every $S$-scheme $S'\to S$ we take the objects of $\overline{\mathcal M}_g(S')$ to be the relative stable curves over $S'$.  The morphisms in $\overline{\mathcal M}_g$ are given by pullback diagrams, exactly as in the definition of $\mathcal M_g$.   Using the fact  that for every relative stable curve of genus $g\ge 2$ the relative dualizing sheaf $\omega_{X/S}$ is relatively ample, one can show  that $\overline{\mathcal M}_g$ is an algebraic stack (in fact DM)  exactly as was done for $\mathcal M_g$.    In fact $\mathcal M_g$ is an open substack of $\overline{\mathcal M}_g$, since smoothness is an open condition.

Now we define the CFG $\mathsf {Coh}_{\overline{\mathcal M}_g/S}$.  Over an $S$-scheme $S'\to S$, the objects of $\mathsf {Coh}_{\overline{\mathcal M}_g/S}(S')$ are pairs $(X'/S',E')$ where $X'\to S'$ is a relative stable curve, and $E'$ is an $S'$-flat finitely, presented quasicoherent sheaf on $X'$.    Morphism are defined by pullback in both entries.    There is a natural morphism of CFGs 
$$
\mathsf {Coh}_{\overline{\mathcal M}_g/S}\to \overline{\mathcal M}_g
$$
given by forgetting the sheaf.  
For every $S$-scheme $S'\to S$, and every morphism $S'\to \overline{\mathcal M}_g$ induced by a relative stable curve $X'\to S'$, the fibered product
$\mathsf {Coh}_{\overline{\mathcal M}_g/S}\times_{\overline{\mathcal M}_g}S'$ is equivalent to $\mathsf {Coh}_{X'/S'}$.    Since this is a  algebraic stack, we have immediately from Corollary \ref{C:DescRel} and Corollary \ref{L:AlgRel} that $\mathsf {Coh}_{\overline{\mathcal M}_g/S}$ is a  algebraic stack (this is also a special case of \cite[Thm.~2.1]{lieblich062}).   Similarly, we have stacks $\mathsf {Coh}_{\overline{\mathcal M}_g/S}(r,d)$ (resp.~$\mathsf {Fib}_{\overline{\mathcal M}_g/S}(r,d)$) where we restrict to sheaves (resp.~bundles) of relative rank $r$ and degree $d$.  
Recall that the rank and degree for sheaves on a reducible curve are defined via the Hilbert polynomial with respect to  the relative dualizing sheaf.  

Finally we define the CFG $\mathcal {HS}h_{\overline{\mathcal{M}}_g/S}$.   Over an $S$-scheme $S'\to S$, let the objects of $\mathcal {HS}h_{\overline{\mathcal{M}}_g/S}(S')$ be triples $(X'/S',E',\phi')$ where $X'\to S'$ is a relative stable curve, and $E'$ is an $S'$-flat finitely, presented quasicoherent sheaf on $X'$, and $\phi'\in \operatorname{Hom}_{\mathscr O_{X'}}(E',E'\otimes \omega_{X/S})$.     Morphism are defined by pullback in the first two entries, and in the same way as in the definition of Higgs bundles in the last entry.    There is a natural morphism of CFGs 
$$
\mathcal {HS}h_{\overline{\mathcal{M}}_g/S}\to \mathsf {Coh}_{\overline{\mathcal M}_g/S}
$$
given by forgetting $\phi'$.  
For every $S$-scheme $S'\to S$, and every morphism $S'\to \mathsf {Coh}_{\overline{\mathcal M}_g/S}$ induced by a pair $(X'/S',E')$, the fibered product
$\mathcal {HS}h_{\overline{\mathcal{M}}_g/S}\times_{\mathsf {Coh}_{\overline{\mathcal M}_g/S}}S' $ is equivalent to $\mathcal E_{X'/S'}(\omega_{X'/S'})$.    Since this is an algebraic stack,  
we have immediately from Corollary \ref{C:DescRel} and Corollary  \ref{C:AlgRel} that $\mathcal {HS}h_{\overline{\mathcal M}_g/S}$ is an   algebraic stack. 
    Similarly, we have stacks $\mathcal {HS}h_{\overline{\mathcal M}_g/S}(r,d)$ (resp.~$\mathcal H_{\overline{\mathcal M}_g/S}(r,d)$) where we restrict to sheaves (resp.~bundles) of relative rank $r$ and degree $d$.

\begin{rem} 
In the discussion above we omitted some noetherian hypotheses on the test $S$-schemes $S'$.  
This is possible via the standard noetherian reduction arguments of \cite[\S 8]{EGAIV3}.    Similar arguments are made in \S \ref{S:lfp}, and we direct the reader there for more details on these types of arguments.
\end{rem}


\subsection{Semi-stable Higgs bundles and the quotient stack}

In \cite{N91} Nitsure constructs a moduli scheme of semi-stable Higgs bundles on a smooth projective curve taking values in a line bundle $L$.  The construction of the moduli stack of Higgs bundles on an algebraic curve above essentially follows Nitsure's construction in the setting of stacks.  Here we introduce Nitsure's space, and compare it to the moduli stack.   Nitsure's construction uses geometric invariant theory, so our stack-oriented perspective will rely on quotient stacks (Section~\ref{S:TorsG}).

\subsubsection{Nitsure's construction}

Recall that the slope of a bundle $E$ on a smooth curve $X$ is defined to be $\mu(E):=\deg (E)/\operatorname{rank}(E)$.  We will say that $(E',\phi')$ is a sub-Higgs bundle of $(E,\phi)$ if $E'\stackrel{\iota}{\hookrightarrow} E$ and there is a commutative diagram
$$
 \xymatrix{
E\ar@{->}[r]^-\phi & E\otimes L\\
E' \ar@{->}[r]^-{\phi'} \ar@{^(->}[u]^\iota & E'\otimes L \ar@{^(->}[u]_{\iota\otimes \operatorname{id}_L}.
}
$$
A Higgs bundle $(E,\phi)$ is said to be slope stable (resp.\ semi-stable) if 
for every sub-Higgs bundle $(E',\phi')\subseteq (E,\phi)$, one has $\mu(E')< \mu(E)$ (resp.\ $\mu(E') \leq \mu(E)$).  

Fix positive integers $r$ and $d$.  By \cite[Prop.~2.3.1]{HL}, semistability is an open condition on a flat family of coherent sheaves.  Therefore there is an \emph{open} substack $\mathcal H_{X/\operatorname{Spec}\mathbb C}^{ss}(L)(r,d) \subseteq \mathcal H_{X/\operatorname{Spec}C}(r,d)$ of semi-stable Higgs bundles over $X$ of rank $r$ and degree $d$, taking values in $L$.  Here we aim to relate this stack to a quasi-projective variety constructed by Nitsure.   We review Nitsure's construction briefly.

 Let $\mathscr O_X(1)$ be an ample line bundle on $X$.  Let $N\in \mathbb Z$ be the minimal positive integer such that 
$$
\frac{d}{r}+ N\deg \mathscr O_X(1)>\max\left\{2g-1, \ 2g-1+\frac{(r-1)^2}{r}\deg L\right\}.
$$
Set $p=(d+rN\deg \mathscr O_X(1))+r(g-1)$.  Let $Q$ be the component of the Quot scheme containing  quotients $\mathscr O_X(-N)^{\oplus p}\to E\to 0$ where $E$ is a rank $r$, degree $d$ vector bundle on $X$.    Let $Q^\circ$ be the locus in $Q$ where $E$ is locally free, the quotient $\mathscr O_X^p\to E(N)\to 0$ obtained by twisting by $\mathscr O_X(N)$ induces an isomorphism on the space of global sections, and $H^1(X,E(N))=0$.  It is known that $Q$ is projective, and   that $Q^\circ$ is reduced and open in $Q$  (see \cite[p.~281]{N91}).     Let $\mathscr E$ be the universal  bundle on $X\times Q^\circ$ obtained from the universal quotient.  Lemma \ref{L:Ray} implies there is a linear scheme  $F\to Q^\circ$ parameterizing pairs $
(\mathscr O_X(-N)^{\oplus p}\to E\to 0,\phi)$, where $\phi:E\to E\otimes L$.   \cite[Cor.~3.4]{N91} implies that every slope semi-stable Higgs bundle of rank $r$ and degree $d$ (together with a presentation as a quotient) shows up in this family.    

   Let $F^{ss}$ be the locus of pairs where the bundle and the endomorphism form a slope semi-stable Higgs bundle.    The universal family over $F$ induces a morphism
$$
F^{ss}\to \mathcal H^{ss}_{X/\operatorname{Spec}\mathbb C}(L)(r,d).
$$
It is essentially the content of \cite[Prop.~3.6]{N91} that this morphism is smooth.  
Moreover, the obvious action of $\mathbb P\operatorname{GL}_p$ on $Q$ given by changing coordinates for the choice of generators of the  bundle lifts to an action of $\mathbb P\operatorname{GL}_p$ on $F$ \cite[p.~281]{N91}.  
The closed orbits in $F^{ss}$ are given by $S$-equivalence classes of slope semi-stable Higgs bundles \cite[\S 4]{N91}.  
 
  Nitsure constructs a quotient of $F$ in the following way.
 There is a quasi-projective variety $H$ equipped with a $\mathbb P\operatorname{GL}_p$-linearized  ample line bundle $\mathcal L$, and a 
$\mathbb P\operatorname{GL}_p$-equivariant  morphism $\hat \tau:F\to H$ such that  slope (semi-)stability on $F$ corresponds to GIT (semi-)stability on $H$    \cite[\S 5]{N91}.   Via $\hat \tau$, the quasiprojective GIT quotient $H/\!\!/_{\mathcal L}\mathbb P\operatorname{GL}_p$ induces a quasiprojective scheme structure on the set $F^{ss}/\mathbb P\operatorname{GL}_p$ of $S$-equivalence classes of slope semi-stable Higgs bundles \cite[p.~290]{N91}, which  we will denote by $ H_{X/\operatorname{Spec}\mathbb C}^{ss}(L)(r,d)$.  (We expect one can also obtain the quotient directly via GIT on $F$ with respect to the semiample line bundle obtained from $\mathcal L$ by pullback via $\hat\tau$.)
This is the moduli scheme of Higgs bundles constructed by Nitsure.    It is shown in \cite[p.290, Thm.~5.10]{N91} that $H_{X/\operatorname{Spec}\mathbb C}^{ss}(L)(r,d)
$ is a good quotient for $F^{ss}$ by the action of $\mathbb P\operatorname{GL}_p$, and is a categorical moduli scheme  for the associated moduli functor $\mathscr H^{ss}_{X/\operatorname{Spec}\mathbb C}(L)(r,d) $ of Higgs bundles (defined in the obvious way).  
Consequently there is a diagram
$$
\xymatrix@C=2em@R=.5em{
& \mathcal H^{ss}_{X/\operatorname{Spec}\mathbb C}(L)(r,d)  \ar[r]&\mathscr H^{ss}_{X/\operatorname{Spec}\mathbb C}(L)(r,d) \ar[dd]\\
F^{ss} \ar[ru] \ar[rd]&&\\
& [F^{ss}/\mathbb P\operatorname{GL}_p] \ar[r]&H_{X/\operatorname{Spec}\mathbb C}^{ss}(L)(r,d).\\
}
$$
The composition  $\mathcal H^{ss}_{X/\operatorname{Spec}\mathbb C}(L)(r,d)\to H^{ss}_{X/\operatorname{Spec}\mathbb C}(L)(r,d)$  is a 
 categorical moduli scheme for the stack $\mathcal H^{ss}_{X/\operatorname{Spec}\mathbb C}(L)(r,d)$, as well, in the sense that it is initial among all morphisms to schemes.

\subsection{The stack of principal $G$-Higgs bundles} \label{S:G-Higgs-1}
The focus of our presentation has been on Higgs vector bundles.  For completeness we include a  brief section of principal Higgs bundles on smooth complex projective curves. We direct the reader to 
\cite{DPHitch12} for more on the topic.We will give a deformation-theoretic perspective on $G$-Higgs bundles in Section~\ref{S:def-tors}.

Let $X$ be a smooth complex projective curve, and let  $G$ be a   complex semisimple Lie group.  Let $\mathfrak g$ be the complex Lie algebra associated to $G$, and let 
$$
{\operatorname  {Ad}}:G\to {\mathrm  {Aut}}({\mathfrak  g})
$$
be the adjoint representation of $G$. Given an algebraic  principal $G$-bundle $P$ on $X$, the adjoint bundle of $P$ is the associated algebraic   vector bundle
$$
{\operatorname  {ad}}P:=P\times _{G}{\mathfrak  g}:=(P\times \mathfrak g)/G
$$
where the action of $G$ is the product action via the natural action of $G$ on $P$ and the adjoint action of $G$ on $\mathfrak g$.

\begin{dfn}[$G$-Higgs bundle] \label{D:GHiggsCC}
Let $X$ be a smooth complex projective curve, and let  $G$ be a   complex semisimple Lie group. 
A \emph{$G$-Higgs bundle} on $X$ is a pair $(P,\Phi)$ where $P$ is a principal $G$-bundle over $X$  and $\Phi\in H^0(X,\operatorname{ad}P\otimes K_X)$.
\end{dfn}

\begin{rem}
If $G\hookrightarrow \operatorname{GL}_n$, then the data of a $G$-Higgs bundle $(P,\Phi)$ gives rise to a Higgs bundle $(E,\phi)$, where $E := P \times_{G} \mathbb C^n$, and $\phi:E\to E\otimes K_C$ is described as follows.   
First observe that $\operatorname{End}(E)=P\times_{G}\operatorname{End}(\mathbb C^n)$. 
The embedding $G\hookrightarrow \operatorname{GL}_n$ induces on tangent spaces a $G$-equivariant map $\mathfrak g \to \operatorname{End}(\mathbb C^n)$, where by definition   $G$ acts by the adjoint representation  on $\mathfrak g$, and by conjugation on $\operatorname{End}(\mathbb C^n)$ via the embedding  $G\hookrightarrow \operatorname{GL}_n$.  Thus we obtain a morphism 
$$\operatorname{ad} P=P\times_{G}\mathfrak g\longrightarrow P\times_{G}\operatorname{End}(\mathbb C^n)=\operatorname{End}(E).
$$
Tensoring by $K_X$ and taking global sections, one obtains the morphism $\phi$ induced by $\Phi$.  
\end{rem}

Following the construction  in \S \ref{S:HigSmCur}, one can define the CFG $\mathcal H^{G}_{X/\operatorname{Spec}\mathbb C}(r,d)$ over the \'etale site  $\mathsf S/\operatorname{Spec}\mathbb C$  consisting of principal $G$-Higgs bundles whose associated vector bundle is of rank $r$ and degree $d$.

\begin{teo}
The category fibered in groupoids  of principal $G$-Higgs bundles, $\mathcal H^{G}_{X/\operatorname{Spec}\mathbb C}(r,d)\to \mathsf S/\operatorname{Spec} \mathbb C$, is an  algebraic stack locally of finite type over~$\operatorname{Spec}\mathbb C$.  
\end{teo}

\begin{proof}
The CFG  $\mathsf {Prin}^{G}_{X/\operatorname{Spec}\mathbb C}(r,d)$,
of principal $G$-bundles over $X$ with associated vector bundle of rank $r$ and degree $d$, is an  algebraic stack locally of finite type over $\mathbb C$   (see e.g., \cite[Exa.~(4.6)]{LMB}).   
The forgetful functor $\mathcal H^{G}_{X/\operatorname{Spec}\mathbb C}(r,d) \to \mathsf {Prin}^{G}_{X/\operatorname{Spec}\mathbb C}(r,d)$ is schematic  (this is similar to Corollary \ref{C:SchMor1}).
It follows from  Corollary \ref{L:AlgRel} that   $\mathcal {H}^{G}_{X/\operatorname{Spec}\mathbb C}(r,d)$ is an algebraic stack locally of finite type over~$\mathbb C$.
\end{proof}

\begin{rem}
As in \S \ref{S:MeroHiggs}, one can easily generalize this construction to the case of a proper morphism $\pi:X\to S$ of finite presentation between schemes over $\mathbb C$, and pairs $(P,\Phi)$, where $P$ is again  a principal $G$-bundle, but $\Phi$ is a global section of $\operatorname{ad}P\otimes F$, where $F$ is some fixed $S$-flat quasicoherent sheaf of $\mathscr O_X$-modules of finite presentation.  As in the case of Higgs bundles, for many applications this is not the correct generalization.   In the smooth case, the following is standard: Let $X$ be a smooth projective manifold.  
A \emph{$G$-Higgs bundle} is a pair $(P,\Phi)$ where $P$ is a principal $G$-bundle over $X$  and $\Phi\in H^0(X,\operatorname{ad}P\otimes \Omega_X^1)$  is such that $[\Phi, \Phi]=0\in H^0(X,\operatorname{ad} P\otimes \bigwedge^2\Omega_X^1)$.
\end{rem}


\section{The Hitchin fibration}

In this section we describe the Hitchin fibration at the level of stacks.  

\subsection{Characteristic polynomials and the Hitchin morphism}

In this section we introduce characteristic polynomials and the Hitchin morphism.

\subsubsection{Characteristic polynomials}

Let $X$ be a scheme, let $E$ be a locally free sheaf of rank $n$ on $X$, let $L$ be a locally free sheaf of rank $1$ on $X$, and let 
$$
\phi:E\to E\otimes L
$$
be a morphism of sheaves.  
Let  $\mathbb L=\underline{\raisebox{0pt}[0pt][-.5pt]{$\operatorname{Spec}$}}_X\left(\operatorname{Sym}^\bullet L^\vee\right)$ be the geometric line bundle associated to $L$, and let  
$$
p:\mathbb L \to X
$$
be the structure map.  Let 
\begin{equation}\label{E:Taut-S-1}
\begin{CD}
\mathcal O_{\mathbb L}@>T>> p^* L
\end{CD}
\end{equation}
be the tautological section (this is the section corresponding to the tautological map of geometric line bundles  $\mathbb L\times_X\mathbb A^1_{X}\to \mathbb L\times_X\mathbb L$, given heuristically  by  ``$(v_x,\lambda_x)\mapsto (v_x,\lambda_x v_x)$ for $v_x\in \mathbb L_x$ and $\lambda_x \in \mathbb A^1_{X,x}$''; see \S \ref{S:TautSec} below).

Tensoring by $p^*E$, we obtain 
$$
\begin{CD}
p^*E@>T>> p^*E\otimes p^*L.
\end{CD}
$$
Consequently we obtain the endomorphism
$$
\begin{CD}
p^*E@>T-p^*\phi>> p^*E\otimes p^*L.
\end{CD}
$$
Taking the determinant we obtain
$$
\begin{CD}
p^*\det E@>\det (T-p^*\phi)>> p^*\det E\otimes p^*L^{\otimes n}.
\end{CD}
$$
We can view this as a global section 
$$
\begin{CD}
\mathcal O_{\mathbb L}@>\det (T-p^*\phi)>>  p^*L^{\otimes n}.
\end{CD}
$$
Therefore we have
\begin{equation}\label{E:HiggsGamma}
\det (T-p^*\phi)\in \Gamma(\mathbb L,p^*L^{\otimes n})=\Gamma(X,  p_*\mathcal O_{\mathbb L}\otimes L^{\otimes n})=\bigoplus_{m\ge 0}  T^{m}\Gamma(X, L^{\otimes (n-m)}).
\end{equation}
On the right, $T$ is a formal variable, which we introduce to make the bookkeeping and some local computations easier to follow.   

We call $\det (T-p^*\phi)$ the \emph{characteristic polynomial of $\phi$} (it is an element of the graded ring of global  sections of tensor powers of $L$, which we formally view as a polynomial by introducing the formal variable $T$).

The component of $\det(T-p^*\phi)$ in $\Gamma(X,L^{\otimes i})$ can be obtained in the following way.  The morphism
$$
\phi:E\to E\otimes L
$$
determines a global section of $L$ via the composition 
$$
\mathcal O_X\to E^\vee \otimes E\to L,
$$
where the first map takes $1$ to $\operatorname{id}_E$ and the second is induced by $\phi$.  We define the trace $\operatorname{Tr}(\phi)\in H^0(X,L)$ to be  this global section.  The component of $\det(T-p^*\phi)$ in $H^0(X,L^{\otimes i})$ is given by $(-1)^i\operatorname{Tr}(\wedge^i\phi)$.  
In particular, $\det (T-p^*\phi)\in \bigoplus_{i= 0}^n T^{n-i}\Gamma(X, L^{\otimes i})$.  Moreover, the component in $\Gamma(X,\mathcal O_X)$ is always equal to $1$, and so we will drop this term in what follows.  

\subsubsection{The Hitchin morphism}
From the discussion above  we can define  the Hitchin morphism 
\begin{eqnarray*}
h:\mathcal H_{X/S}(L)(n)&\to&\bigoplus_{i = 1}^n T^{n-i} \Gamma(X, L^{\otimes i}) \\
(E,\phi)&\mapsto& \det(T-p^*\phi).\\
\end{eqnarray*}
Here we are denoting by $\mathcal H_{X/S}(L)(n)$ the sub-algebraic stack of $\mathcal E_{X/S}(L)(n)$  consisting of  endomorphisms of locally free coherent sheaves of rank $n$ with values in $ L$,  over the site $(\mathsf S/S)_{\operatorname{et}}$.
On the right, we suggestively indicate the functor that on an $S$-scheme $S'\to S$ takes values $\bigoplus_{i = 1}^n T^{n-i} \Gamma(X', L'^{\otimes i})$, where $X'=S'\times_SX$ and $L'$ is the pullback to $X'$.  In other words it is the functor $\mathscr H\!\!\mathit{om}_{X/S}(\mathcal O_X,\bigoplus_{i=1}^nT^{n-i}L^{\otimes i})$ of Lemma \ref{L:Ray}. 
By virtue of Lemma \ref{L:Ray}, 
this functor is representable by a scheme  $\mathbb A_{X/S}(L)(n)$ over $S$, .    We call $\mathbb A_{X/S}(L)(n)$ the Hitchin base for $\mathcal H_{X/S}(L)(n)$.  We obtain the Hitchin morphism
\begin{eqnarray*}
h:\mathcal H_{X/S}(L)(n)&\to&\mathbb A_{X/S}(L)(n).\\
\end{eqnarray*}

\begin{rem}
When $X/S$ is a smooth projective curve over $S=\operatorname{Spec}\mathbb C$ (and we consider $\mathbb C$-points of the moduli problem) the Hitchin map  takes a rank $n$ Higgs bundle $(E,\phi)$  on $X$ with values in a line bundle $L$, and sends it to the corresponding $n$-tuple of ``coefficients'' of the characteristic polynomial of $\phi$ in the complex  vector space $\bigoplus_{i= 1}^n \Gamma(X, L^{\otimes i})$.
\end{rem}

\subsubsection{The tautological section of $p^*L$}  \label{S:TautSec}
The main goal of this subsection is to define the tautological section \eqref{E:Taut-S-1}, and give a local description of the map.   
 There are several equivalent ways to define it.  
 One  can use adjunction to identify the groups  $
\operatorname{Hom}_{\mathcal O_{\mathbb L}}(\mathcal O_{\mathbb L},p^*L)
=\operatorname{Hom}_{\mathcal O_X}(L^\vee,\operatorname{Sym}^\bullet L^\vee)$ and then use the tautological morphism of sheaves on the right (see Remark \ref{R:Taut-Adj}).    Alternatively, one could consider the global section of the geometric line bundle $p^*\mathbb L$ on $\mathbb L$ given pointwise by assigning to $v\in \mathbb L$ the point $v$ in the fiber of $p^*\mathbb L$ over $v$ (see Remark \ref{R:Taut-Geom}).  
Finally, one can describe it from a morphism of geometric line bundles $T:\mathbb A^1_X\times_X\mathbb L \to \mathbb L\times_X\mathbb L$ over $\mathbb L$; since this is the how we will use the tautological section, we consider this approach in detail.  

The $\mathcal O_X$-module structure map for the rank $1$, locally free sheaf $L$
$$
\mathcal O_X\times L\to   L 
$$
induces a multiplication map
$$
\begin{CD}
\mathbb A^1_X\times_X\mathbb L @>\mu >> \mathbb L,
\end{CD}
$$
as is seen easily from the following diagram, where $U$ is an open set, and the dashed arrow indicates a section of $\mathcal O_X\times L$ over $U$:
$$
\xymatrix@C=.5cm @R=.5cm{
&\mathbb A^1_X\times_X \mathbb L \ar@{->}[r]^<>(0.5)\mu  \ar@{->}[d]& \mathbb L \ar@{->}[d]\\
U\ar@{-->}[ru] \ar@{^(->}[r]& X\ar@{=}[r]& X.
}
$$
This in turn induces a diagram
\begin{equation}\label{E:TautSecDiag}
\xymatrix@C=.5cm @R=.5cm{
\mathbb A^1_X\times_X\mathbb L \ar@{-->}[rd]^T \ar@{->}@/^1pc/[rrd]^\mu \ar@{->}@/_1pc/[rdd]_{pr_2} &&\\
&\mathbb L\times_X\mathbb L\ar@{->}[r] \ar@{->}[d]&\mathbb L \ar@{->}[d]^p\\
& \mathbb L\ar@{->}[r]_p&X\\
}
\end{equation}
The map $T$ is the geometric version of the tautological global section.  More precisely, $\mathbb A^1_X\times_X\mathbb L$ is the pullback of the  trivial geometric line bundle on $X$ (i.e., $p^*\mathbb A^1_X$), and is hence the trivial geometric line bundle on $\mathbb L$ (i.e., $\mathbb A^1_{\mathbb L}$), and $\mathbb L\times_X\mathbb L$ is $p^*\mathbb L$.  
The associated morphism of sheaves is a morphism
$$
\begin{CD}
\mathcal O_{\mathbb L} @>T >> p^*L,
\end{CD}
$$
which  corresponds to a global section $T\in H^0(\mathbb L,p^*L)$.

\vskip .2 cm 
It can be useful to describe the tautological section locally. 
Let  $U=\operatorname{Spec}R\subseteq X$ be an affine open subset.  Assume that   $L$ is trivialized over $U$, corresponding to the trivial $R$-module $R$. 
  We can  identify the $R$-algebra $\operatorname{Sym}^\bullet L^\vee(U)$ with $R[T]$, so that $p:\mathbb L\to X$ is identified over $U$ as $p:L|_U=\mathbb A^1_U=\operatorname{Spec}R[T]\to U=\operatorname{Spec}R$.  
  
  The multiplication map $\mu:\mathbb A^1_U\times_U\mathbb L|_U\to \mathbb L|_U$ is then identified with the $R$-algebra map
  $$
  \begin{CD}
R[T,T']=R[T]\otimes_R R[T']@<\mu<< R[T']
\end{CD}
$$
given by $T'\mapsto TT'$.   The diagram \eqref{E:TautSecDiag} defining the geometric tautological section, i.e., the map
$T:\mathbb A^1_U\times_U \mathbb L|_U\to \mathbb L|_U\times_U\mathbb L|_U$, is given 
at the level of $R$-algebras by the daigram:
$$
\xymatrix@C=.5cm @R=.5cm{
R[T]\otimes_R R[ T'] \ar@{<--}[rd]^T \ar@{<-}@/^1pc/[rrd]^\mu \ar@{<-}@/_1pc/[rdd]_{} &&\\
&R[\widetilde T']\otimes_R R[ T']\ar@{<-^)}[r] \ar@{<-_)}[d]&R[ T'] \ar@{<-_)}[d]\\
& R[\widetilde T']\ar@{<-^)}[r]&R \\
}
$$
The tautological map $T$ is given here by  $T'\mapsto TT'$ and $\widetilde T'\mapsto T'$.  
  
Now $p^{-1}(U)=\mathbb L|_U=\operatorname{Spec}R[T]$.
We have   that $\mathbb A^1_U\times_U \mathbb L|_U$ and  $\mathbb L|_U\times_U\mathbb L|_U$ are geometric line bundles, which on $\mathbb L|_U$, where they are trivialized, correspond to the trivial rank $1$, locally free module $R[T]$.  
  The tautological section $T:\mathcal O_{\mathbb L}\to p^*L$ is given under these identifications  by the multiplication map 
\begin{equation}\label{E:Taut-S-local}
\begin{CD}
R[T]@>\cdot T>> R[T].
\end{CD}
\end{equation}

\begin{rem}\label{R:Taut-Adj}
Alternatively, one can describe the tautological section as follows.  We have identifications
$$
\operatorname{Hom}_{\mathcal O_{\mathbb L}}(\mathcal O_{\mathbb L},p^*L)=\operatorname{Hom}_{\mathcal O_{\mathbb L}}(p^*\mathcal O_X,p^*L)=\operatorname{Hom}_{\mathcal O_X}(\mathcal O_X,p_*(\mathcal O_{\mathbb L}\otimes p^*L))
$$
$$
=\operatorname{Hom}_{\mathcal O_X}(\mathcal O_X,(\operatorname{Sym}^\bullet L^\vee)\otimes L)
=\operatorname{Hom}_{\mathcal O_X}(L^\vee,\operatorname{Sym}^\bullet L^\vee).
$$
There is a natural inclusion $L^\vee \hookrightarrow \mathcal O_X\oplus L^\vee \oplus (L^\vee)^{\otimes 2}\oplus \cdots$ onto the second factor, and this corresponds to the tautological section.  Locally, if we set $U=\operatorname{Spec}R\subseteq X$ to be an open affine over which $\mathbb L$ is trivial, then $\operatorname{Sym}^\bullet L^\vee$  is identified with $R[T]$, and  the natural map $L^\vee \to \operatorname{Sym}^\bullet L^\vee$ is associated to the map  $R\to R[T]$ by multiplication by $T$.  
\end{rem}

\begin{rem}\label{R:Taut-Geom}
One can also describe the tautological section geometrically as follows.  The sheaf $p^*L$ is the sheaf of sections of the line bundle $p^*\mathbb L:=\mathbb L\times_X\mathbb L\to \mathbb L$, where the structure morphism to $\mathbb L$ is the first projection.     There is a global section of the structure morphism given by the diagonal map $\mathbb L\to \mathbb L\times_X\mathbb L$.  This is the tautological global section.   
\end{rem}

\subsubsection{The characteristic polynomial locally}

Let $(E,\phi)$ be a Higgs bundle on $X$ with values in  $L$.  
Let  $U=\operatorname{Spec}R\subseteq X$ be an affine open subset.  Assume that $E$ is trivial over $U$, corresponding to $R^n$, and that  $L$ is also trivialized, corresponding to $R$.  Then $\phi:E\to E\otimes L$ can be identified as a map $R^n\to R^n$, and thus with an $n\times n$ matrix $(\phi_{ij})$ over $R$.  We can also identify $\operatorname{Sym}^\bullet L^\vee$ with $R[T]$, so that $p:\mathbb L\to X$ is identified over $U$ as $p:\mathbb A^1_U\to U$, and the tautological section $T:\mathcal O_{\mathbb L}\to p^*L$ is given under these identifications  by the multiplication map 
$$
\begin{CD}
R[T]@>\cdot T>> R[T].
\end{CD}
$$
The map $(T-p^*\phi):p^*E\to p^*E\otimes L$ can then be identified over $\mathbb L|_U$ with the map $R[T]^n\to R[T]^n$  given by the matrix 
$$
T-p^*\phi=\left(
\begin{array}{ccc}
T-\phi_{11} & \cdots & -\phi_{1n}\\
\vdots & \ddots & \vdots \\
-\phi_{n1}& \cdots & T-\phi_{nn}\\
\end{array}
\right).
$$
We then have 
\begin{equation}\label{E:detphi}
\det(T-p^*\phi)|_U=T^n-\operatorname{Tr}(\phi)T^{n-1}+\cdots +(-1)^n\det\phi\in R[T].
\end{equation}
Recall that globally we had 
$
\det (T-p^*\phi)\in \Gamma(\mathbb L,p^*L^{\otimes n})=\Gamma(X,  p_*\mathcal O_{\mathbb L}\otimes L^{\otimes n})=\bigoplus_{m\ge 0} T^{m}  \Gamma(X, L^{\otimes (n-m)})
$; the coefficients of the powers of $T$ in this description and  in \eqref{E:detphi} agree.

\subsection{Spectral covers and fibers of the Hitchin morphism}

Here we describe spectral covers, and the connection with the fibers of the Hitchin morphism.  The main point for Higgs bundles on smooth curves are the results of \cite{BNR,schaub} reviewed in Remark \ref{R:BNRSchaub}.   We give a weaker statement that holds in more generality in Proposition \ref{P:BNR} and Lemma  \ref{L:BNR-lemma}.

\subsubsection{Spectral covers}\label{S:SpectCov}
  Every   $\sigma:\mathcal O_X\to \bigoplus_{i=1}^nL^{\otimes i}$ (corresponding to a map $\sigma:S\to \mathbb A_{X/S}(L)(n)$) determines via \eqref{E:HiggsGamma} a global section $\sigma:\mathcal O_{\mathbb L}\to p^*L^{\otimes n}$ of the line bundle $p^*L^{\otimes n}$ on $\mathbb L$, and consequently a zero set of the section:
$$
\widetilde X(\sigma):=V(\sigma)\subseteq \mathbb L.
$$
The map $\widetilde X(\sigma)\to X$ (obtained by composition from the map $p:\mathbb L\to X$) is called the \emph{spectral cover associated to $\sigma$}.  
One can check (see the local computation below) that the spectral cover is a finite morphism of degree $n$.
Associated to the scheme $\widetilde X(\sigma)=V(\sigma)$ is an ideal sheaf $\mathcal I_\sigma$, defined via the short exact sequence:
\begin{equation}\label{E:SES-L*}
0\to \mathcal I_\sigma \to \operatorname{Sym}^\bullet L^\vee \to \mathcal O_{V(\sigma)}\to 0. 
\end{equation}

More generally, the computation shows  that for any $\sigma':S'\to \mathbb A_{X/S}(L)(n)$, corresponding to   $\sigma' :\mathcal O_{X'}\to \bigoplus_{i=1}^nL^{\otimes i}|_{X'}$, where $X'=X\times_SS'$, there is a corresponding spectral cover  $p':\widetilde X'(\sigma')\to X'$.   In particular, there is a universal spectral cover 
$$
\widetilde X(\sigma_{\operatorname{id}_{\mathbb A_{X/S}(L)(n)}})\to X\times_S\mathbb A_{X/S}(L)(n)
$$
determined by the identity morphism on $\mathbb A_{X/S}(L)(n)$, from 
which all the spectral covers are obtained by pullback.

\subsubsection{Local description of the spectral cover} 
 To describe $\widetilde X(\sigma)$, it can be useful to consider the construction locally on $X$.   Let  $U=\operatorname{Spec}R\subseteq X$ be an affine open subset.  Assume that   $L$ is  trivialized over $U$, corresponding to the $R$-module $R$.
Given $\sigma\in \bigoplus_{i=1}^n T^{n-i}\Gamma(X,L^{\otimes i})$, with components $\sigma_i\in \Gamma(X,L^{\otimes i})$, then we can write  $$\sigma|_U= T^n-\sigma_1|_UT^{n-1}+\cdots +(-1)^n\sigma_n|_U.$$
Since $L$, and hence $L^{\otimes i}$ is trivalized over $U$, we may view the $\sigma_i|_U$ as elements of $R$. 
The short exact sequence \eqref{E:SES-L*} is then written locally as 
$$
\begin{CD}
0@>>> R[T] @>\sigma|_U>> R[T] @>>> R[T]/\mathcal I_\sigma|_U @>>>0 .
\end{CD}
$$
In other words, $\mathcal I_{\sigma}|_U$ is given by $(\sigma|_U)\subseteq R[T]$, and $\widetilde X(\sigma)|_U=\operatorname{Spec}R[T]/(\sigma|_U)$.

\subsubsection{Minimal   ideals for endomorphisms}

Let $X$ be a scheme, $L$ a rank $1$ locally free sheaf on $X$, and $E$ a quasicoherent sheaf on $X$.  A morphism of $\mathcal O_X$-modules  
$$
\phi:L^\vee\to \operatorname{ End}(E)
$$
is equivalent to a morphism of $\mathcal O_X$-algebras
$$
\phi^\bullet :\operatorname{Sym}^\bullet L^\vee \to \operatorname{End}(E).
$$
The \emph{minimal ideal of $\phi$}, denoted $\mathcal I_\phi$,  is defined to be the kernel of $\phi^\bullet$.   The \emph{minimal cover} $\widetilde X(\phi)\to X$ associated to  $\phi$ is defined to be the subscheme of $\mathbb L$ defined by $\mathcal I_{\phi}$.   Note that $E$ induces a quasicoherent sheaf $M$ on $\widetilde X(\phi)$ such that the  pushforward of $M$ to $X$ is equal to  $E$. If $E$ is locally finitely  generated, the support of $M$ is exactly $\widetilde X(\phi)$ (i.e., locally, the support is  the set of  primes containing the annihilator of the finitely generated module).

More generally, for any ideal sheaf $\mathcal I\subseteq \operatorname{Sym}^\bullet L^\vee$, a quasicoherent sheaf $M$ on the scheme  $V(\mathcal I)\subseteq \mathbb L$ is equivalent to a quasicoherent sheaf $E$  on $X$ together with a morphism $\phi:L^\vee \to \operatorname{End}(E)$ such that $\mathcal I \subseteq \mathcal I_\phi$; the  identification is made via pushforward; i.e., $E$ is the pushforward of  $M$.  

Locally on $X$, for locally free sheaves this is described as follows.  Suppose that $E$ is locally free of rank $n$, and is trivial over $U=\operatorname{Spec}R\subseteq X$.  Then $\phi:L^\vee \to \operatorname{End}(E)$ 
induces an evaluation  morphism $\phi^\bullet|_U:R[T]\to \operatorname{End}_R(R^n)$ sending $T$ to $\phi|_U$.  The kernel   $\ker \phi^\bullet|_U$ is the restriction of the  minimal ideal $\mathcal I_\phi|_U$.  We have that $R^n$ is an $R[T]/\mathcal I_\phi|_U$-module, with support $\operatorname{Spec}R[T]/\mathcal I_\phi|_U$.  In addition, so long as $\sigma|_U\in \mathcal I_{\phi}|_U$, we can view $R^n$ as an $R[T]/(\sigma|_U)$-module, as well.  In other words, $E$ is obtained by push forward from a sheaf on the spectral cover associated to $\sigma$ (although it may only be supported on the possibly smaller minimal cover associated to $\phi$).     

This discussion allows us to identify   Higgs bundles with given minimal ideal with certain   sheaves on the spectral cover.

\begin{pro}[{\cite[Prop.~3.6]{BNR}}] \label {P:BNR} 
Let $X\to S$ be a proper morphism of finite presentation between schemes with $S$ excellent, and let $L$ be a locally free sheaf of rank $1$ on  $X$.  
Given an $S$-morphism $$\sigma':S'\to \mathbb A_{X/S}(L)(n),$$ to give a pair  $(E',\phi')$ in $\mathcal H_{X/S}(L)(n)(S')$ such that $\phi'$ has minimal ideal $\mathcal I_{\phi'}$ equal to  $\mathcal I_{\sigma'}$, it is equivalent to give a coherent sheaf $M'$  on $\widetilde X'(\sigma')$ (where $X' = X \mathbin\times_S S'$) such that  the pushforward of $M'$ to $X'$ is a rank $n$ locally free sheaf and the support of $M'$ is $\widetilde X'(\sigma')$.
\end{pro}

\begin{proof}
As we have seen above, $(E', \phi') \in \mathcal H_{X/S}(L)(n)(S')$ corresponds to a quasicoherent sheaf $M'$ on $\mathbb L'$ (the total space of $L'$) whose support is $\widetilde X'(\phi')$.  But $\widetilde X'(\phi')$ coincides with $\widetilde X'(\sigma')$ if and only if $\mathcal I_{\sigma'} = \mathcal I_{\phi'}$.
\end{proof}

\begin{rem}
It is easy to see that every locally free sheaf $M'$ of rank $1$ on $\widetilde X'(\sigma)$ satisfies the conditions in Proposition \ref{P:BNR}.  Note that one can  easily construct examples (even with $\widetilde X'(\sigma')$ a union of smooth complex projective curves) where there are sheaves $M'$ on $\widetilde X'(\sigma')$ that push forward to rank $n$ locally free sheaves $E'$ on $X$ such that the induced endomorphism $\phi':E'\to E'\otimes L'$ does not have minimal polynomial equal to $\sigma'$; this is the reason for the hypothesis on the support of $M'$. 
\end{rem}

\subsubsection{Fibers of the Hitchin map}

The following lemma asserts that the category of sheaves in Proposition \ref{P:BNR} above is contained in the fiber of the Hitchin map.

\begin{lem}\label{L:BNR-lemma}
A pair $(E',\phi')$ in  $\mathcal H_{X/S}(L)(n)(S')$ such that the minimal ideal $\mathcal I_{\phi'}$ is equal to   $\mathcal I_{\sigma'}$ has characteristic polynomial given by $\sigma'$.  Moreover, if $X'$ is integral and the fiber of $\widetilde X'(\sigma')$ over the generic point of $X$ is geometrically reduced (e.g., $X'$ is a variety over $\mathbb C$ and $\widetilde X'(\sigma')$ is reduced), then the converse holds. 

 In other words, under these hypotheses, Proposition \ref{P:BNR} describes the fiber of the Hitchin morphism in terms of sheaves on the spectral cover.
\end{lem}

\begin{proof}
Let $(E',\phi')$ be a Higgs bundle as in the statement of the  lemma, and let $p_{\phi'}(T)$ be the characteristic polynomial of $\phi'$.  

Assume first that $\mathcal I_{\phi'}=\mathcal I_{\sigma'}$.
Viewing the characteristic polynomial as a global section $p_{\phi'}(T):\mathcal O_{X'}\to \bigoplus_{i=1}^nL^{\otimes i}|_{X'}$  defines a locally principal  ideal $\mathcal I_{p_{\phi'}(T)}$ of $\operatorname{Sym}^\bullet L^\vee$ (\S \ref{S:SpectCov}).  The Cayley--Hamilton theorem says that $\mathcal I_{p_{\phi'}(T)}\subseteq \mathcal I_{\phi'}$. But we are assuming that $\mathcal I_{\phi'}=\mathcal I_{\sigma'}$.   So we have
$\mathcal I_{p_{\phi'}(T)}\subseteq \mathcal I_{\sigma'}$, with both being locally principal generated by monic polynomials of degree $n$.  So we have reduced to the local statement:  Given a ring $R$, and two principal ideals $(f(T))$ and $(g(T))$ in $R[T]$ with both $f(T)$ and $g(T)$ \emph{monic} of degree $n$, if $(f(T))\subseteq (g(T))$, then $f(T)=g(T)$.    Thus the characteristic polynomial $p_{\phi'}(T)$ is given by $\sigma$.

Conversely, suppose $X$ is integral, the fiber of $\widetilde X'(\sigma')$ over the generic point of $X$ is geometrically  reduced, and that the characteristic polynomial  $p_{\phi'}(T)$ is given by $\sigma'$.  We want to show that  $\mathcal I_{\phi'}=\mathcal I_{\sigma'}$.  It is enough to do this locally.  In other words, we assume that $R$ is an integral domain with field of fractions $K$ contained in an algebraic closure $\overline K$, that $\sigma(T)\in R[T]$ is a monic polynomial of degree $n$, that $\overline  K[T]/(\sigma(T))$ is reduced,
and that $\phi:R^n\to R^n$ is an endomorphism with characteristic polynomial $\sigma(T)$.  We want to show the containment $\mathcal I_\phi\supseteq (\sigma(T))$ is an equality.  So consider  $0\ne f(T)\in \mathcal I_\phi$.     The claim is that $f(T)$ is divisible by $\sigma(T)$ in $R[T]$.  To see this, note that with $\phi_K:K^n\to K^n$ the induced morphism, we have that $f(\phi_K)=0$.   Thus $f(T)$ is divisible in $K[T]$ by the minimal polynomial for $\phi_K$.  Note also that from the local definition of the characteristic polynomial, it is clear that the the characteristic polynomial of $\phi_K$ is the same as for $\phi$, namely $\sigma(T)$.   Our assumption that $\overline K[T]/(\sigma(T))$ is reduced, i.e., that the characteristic polynomial of $\phi_K$ has distinct roots, implies that the characteristic polynomial for $\phi_K$ agrees with the minimal polynomial.  Thus $f(T)$ is divisible by $\sigma(T)$ in $K[T]$.  But since $\sigma(T)$ is monic, one can then conclude (see the remark below) that $f(T)$ is divisible by $\sigma(T)$ in $R[T]$.
\end{proof}

\begin{rem}
\emph{Let $R$ be an integral domain and let $K$ be its field of fractions.  Suppose that $f(T),g(T)\in R[T]$ with $g(T)$ \emph{monic}, and there exists $h(T)\in K[T]$ such that $f(T)=g(T)h(T)\in K[T]$.  Then $h(T)\in R[T]$}.  Indeed, let us write 
$$
f(T)=\sum_{i=0}^{n+m}a_iT^i,  \ \ g(T)=\sum_{i=0}^nb_iT^i, \ \ h(T)=\sum_{i=0}^mc_iT^i.
$$
Then we have
\begin{align*}
a_{n+m}&=b_nc_m\\
a_{n+m-1}&=b_nc_{m-1}+b_{n-1}c_m\\
 &\vdots \\
a_n&= b_nc_0+b_{n-1}c_1+\cdots +b_0c_n.
\end{align*}
(Here $c_j$ is taken to be $0$ if $j>m$.)
Since $b_n$ is assumed to be $1$, the first equality shows that $c_m\in R$.  The second then shows that $c_{m-1}\in R$, and so on until we have established that $c_0\in R$.   
\end{rem}

\begin{exa}
It can be instructive to consider Proposition \ref{P:BNR} and Lemma \ref{L:BNR-lemma} the case where $X=S=\operatorname{Spec}\mathbb C$ and $L=\mathcal O_X$.
\end{exa}

\begin{rem}\label{R:BNRSchaub}
In the case where $\pi:X\to S$ is a smooth relative curve; i.e., a smooth proper morphism such that every geometric fiber is connected, and dimension $1$, more can be said.  If $\sigma':S'=\operatorname{Spec} (k)\to \mathbb A_{X/S}(L)(n)$ is a geometric point inducing a reduced spectral curve  $\widetilde X'(\sigma')$  over $\operatorname{Spec}k$, then $h^{-1}(\sigma')$ is the the compactified Picard stack, parameterizing  rank $1$, torsion-free sheaves on $\widetilde X'(\sigma')$ \cite[Prop.~5.1]{schaub}.  
If   $\widetilde X'(\sigma')$ is assumed further to be an  integral curve, which is the case considered in  \cite[Prop.~3.6]{BNR},
then the key point is that in the notation of Proposition \ref{P:BNR},  for $\pi'_*M$ to be locally free, and thus torsion-free, it must be that $M$ is torsion-free.  Then,  since $X'$ is a smooth curve over $k$, any torsion-free sheaf is locally free, and thus for any torsion-free sheaf on $X'$, the push-forward is locally free.    The case of a general spectral curve $\widetilde X'(\sigma')$ is considered in \cite{schaub}; the connection between slope stability of Higgs bundles on $X'$  and slope stability of rank $1$ torsion free sheaves on $\widetilde X'(\sigma')$  in the sense of Oda--Seshadri is also considered there.
See also \cite[Thm.~6.1]{N91}.
\end{rem}

\section{Morphisms of  stacks in algebraic geometry}
\label{S:AlgStack2}

As explained in Definition \ref{D:RepMorphProp}, many reasonable properties of morphisms in our presite $\mathsf S$ may be extended to $\mathsf S$-representable morphisms of stacks.  However, not every morphism of algebraic stacks that deserves to be called quasicompact or \'etale or smooth (to name just a few) is necessarily schematic.

In Section~\ref{S:P-adapted}, we saw how to extend the notion of smoothness to morphisms between algebraic stacks, but this definition required the choice of a presentation of the stack in question.  Since stacks representing moduli problems rarely come with an easily described presentation, it would be better to have an intrinsic definition of smoothness.

To give such a definition, as well as definitions of other geometric properties or morphisms between algebraic stacks, will be the purpose of this section.

\begin{wrn}
	We caution that although many of the definitions given here make sense for arbitrary morphisms of CFGs, or for arbitrary morphisms of stacks, they should not necessarily be regarded as reasonable generalizations from schemes without a further algebraicity assumption.

The characterization of smoothness we give in Section~\ref{S:sm-et-nr}, for example, is only reasonable for morphisms of algebraic stacks (or at least morphisms representable by algebraic stacks) and not necessarily for all stacks.
\end{wrn}

\subsection{Injections, isomorphisms, and substacks}

In Definition \ref{D:MonIso} we gave the definition of an injection and of an  isomorphism   of stacks.
We have not given a special name to morphisms of algebraic stacks that are only faithful (rather than being fully faithful or equivalences, respectively)  when viewed as functors,  because we already have one:

\begin{lem}
A morphism of algebraic stacks $f : \mathcal X \rightarrow \mathcal Y$ that is objectwise a faithful functor is representable by algebraic spaces.
\end{lem}

By saying $f : \mathcal X \rightarrow \mathcal Y$  is objectwise a faithful functor we mean that for each scheme $S$ the induced functor of groupoids $f(X):\mathcal X(S)\to \mathcal Y(S)$ is  faithful.  

\begin{proof}
Suppose that $Z$ is an algebraic space and $Z \rightarrow \mathcal Y$ is a morphism.  Let $\mathcal X_Z$ be the base change, which is algebraic by Corollary~\ref{L:AlgRel}.  The projection $\mathcal X_Z \rightarrow Z$ is also faithful.  But if $Z$ is viewed as a stack then, for any scheme $S$, the fiber $Z(S)$ is equivalent to a set.  Since  $\mathcal X_Z(S)$ is equivalent to a subcategory of $Z(S)$, this means that $\mathcal X_Z(S)$ is equivalent to a set, which means that $\mathcal X_Z$ is an algebraic space.  
\end{proof}

\begin{dfn}[Open and closed substacks]
	Let $\mathcal X$ be a stack over $\mathsf S$.  A substack $\mathcal U \subseteq \mathcal X$ (resp.~$\mathcal Z \subseteq \mathcal X$) is called an \emph{open substack} (resp.~ \emph{closed substack})   if for every scheme $W$, the fiber product $\mathcal U \mathop{\times}_{\mathcal X} W$  (resp.~$\mathcal Z \mathop{\times}_{\mathcal X} W$)  is representable by an open (resp.~closed)  subscheme of $W$.  
\end{dfn}

\subsection{The underlying topological space}

\begin{dfn}[Topological space of a stack] \label{D:under-top}
	Let $\mathcal X$ be a stack on the category $\mathsf S$ of schemes.  For each scheme $S$, write $|S|$ for the underlying topological space of $S$.  The \emph{underlying topological space} of $\mathcal X$ is the universal topological space $|\mathcal X|$ that receives a continuous map $|S| \rightarrow |\mathcal X|$ for each map $S \rightarrow \mathcal X$.
\end{dfn}

In other words, $\displaystyle |\mathcal X|=\varinjlim_{S\to \mathcal X}|S|$ is the colimit of the spaces $|S|$, taken over all maps from schemes $S$ to $\mathcal X$.

\begin{rem}
	In \cite[Tag 04XE]{stacks}, the underlying topological space was defined only for algebraic stacks, but \cite[Tag 04XG]{stacks} makes sense for arbitrary categories fibered in groupoids and agrees with the definition given here because the underlying set of a colimit of topological spaces is the colimit of the underlying sets.

	The topology on the underlying set was only defined in \cite[Tag 04XL]{stacks} for algebraic stacks.  However, the topology of \emph{loc.~cit.}~agrees with Definition~\ref{D:under-top}.  Recall that a subset of $|\mathcal X|$ is called open if its preimage under a fixed flat, finite presentation, surjective map $U \rightarrow \mathcal X$ is open.  This topology is clearly at least as fine than the topology we have defined, so we verify that every subset of $|\mathcal X|$ that is open in the sense of \cite[Tag~04XL]{stacks} is open in the sense of Definition~\ref{D:under-top}.

	Indeed, if $U \rightarrow \mathcal X$ is flat and locally of finite presentation then for any $V \rightarrow \mathcal X$, the map $U \mathop{\times}_{\mathcal X} V \rightarrow V$ is also flat and of finite presentation.  The preimage in $|V|$ of the image of $|U|$ in $| \mathcal X|$ is the same as the image of $|U \mathbin\times_{\mathcal X} V|$ in $|V|$.  Since $U \mathbin\times_{\mathcal X} V$ is flat of finite presentation over $V$, its image in $|V|$ is open, so the image of $|U|$ in $|\mathcal X|$ pulls back to an open subset of $|V|$.  This holds for any $V \rightarrow \mathcal X$ so the image of $|U|$ in $|\mathcal X|$ is open, by definition of the colimit topology.  As the topology in \emph{loc.~cit.}~is uniquely characterized by this property, it must agree with $|\mathcal X|$, as defined in Definition \ref{D:under-top}.
\end{rem}

\subsection{Quasicompact and quasiseparated morphisms}

\begin{dfn}[Quasicompact morphisms] \label{D:qcpt-qsep}
	We call a stack $\mathcal X$  over $\mathsf S$ \emph{quasicompact} if every covering of $\mathcal X$  (Definition \ref{D:CovStack})   by open substacks has a finite subcover.
	A morphism of stacks $\mathcal X \rightarrow \mathcal Y$ is \emph{quasicompact} if $\mathcal X \mathop{\times}_{\mathcal Y} Z$ is quasicompact for all quasicompact schemes $Z$ and all morphisms $Z \rightarrow \mathcal Y$.
	A morphism of stacks $\mathcal X \rightarrow \mathcal Y$ is \emph{quasiseparated} if the diagonal morphism $\mathcal X \rightarrow \mathcal X \mathop{\times}_{\mathcal Y} \mathcal X$ is quasicompact.
\end{dfn}

\subsection{Separation and properness}

We will not actually discuss the separatedness or properness of morphisms of algebraic stacks in any examples in this survey, but we include the definitions for the sake of completeness.

Since smooth morphisms of schemes are always locally of finite type, Section~\ref{S:P-adapted} shows that there is a unique way to make sense of locally finite type morphisms of algebraic stacks that is stable under base change and composition and local to the source and target.  Technically, this breaks our promise to give only \emph{intrinsic} definitions in this section.  However, in the noetherian situation, finite type coincides with finite presentation, which is characterized intrinsically in Section~\ref{S:lfp2}.

\begin{dfn}[Proper and separated morphisms of algebraic stacks] \label{D:PropSepAlgStk}
	A morphism of \emph{algebraic} stacks $\mathcal X \rightarrow \mathcal Y$ is said to be \emph{proper} if it is an isomorphism or it is separated, of finite type, and universally closed.  It is said to be separated if its diagonal $\mathcal X \rightarrow \mathcal X \mathbin\times_{\mathcal Y} \mathcal X$ is proper.
\end{dfn}

\begin{rem}
Definition \ref{D:PropSepAlgStk} is not as circular as it appears.  The diagonal morphism of a morphism of algebraic stacks is representable by algebraic spaces (Lemma \ref{L:DiagAlgStack},   \cite[Tag 04XS]{stacks}),
 so the definition of separatedness for algebraic stacks depends only on the definition of properness for algebraic spaces.  Iterating Definition \ref{D:PropSepAlgStk}, we see that definition of properness for algebraic spaces depends on separatedness for morphisms of algebraic spaces,
 and therefore,  by iterating again,   the definition depends on the definition of properness for diagonals of algebraic spaces.  Continuing further, we see that  we must define separatedness for diagonals of algebraic spaces.  
  But the diagonal of a morphism of algebraic spaces is injective, so the definition ultimately depends on the definition of separatedness for injections of algebraic spaces.  But the diagonal of an injection is an isomorphism, so is automatically proper. 
  \end{rem}

\subsection{Formal infinitesimal properties}

\begin{dfn}[{Infinitesimal extension \cite[Def.~(17.1.1)]{EGAIV4}}] \label{D:inf-ext}
	An \emph{infinitesimal extension} or \emph{nilpotent extension} of a scheme $S$ is a closed embedding $S \subseteq S'$ such that the ideal of $S$ in $S'$ is nilpotent.  
\end{dfn}

\begin{lem} \label{L:et-inf-lift}
Suppose that $S \subseteq S'$ is an infinitesimal extension and that $T \rightarrow S$ is \'etale.  Then there is an infinitesimal extension $T \subseteq T'$ and \'etale map $T' \rightarrow S'$ inducing $T$ as the fiber product $T' \mathbin\times_{S'} S$.  Moreover, $T'$ is unique up to unique isomorphism.
\end{lem}
\begin{proof}
Since $T'$ will be unique when it is constructed, this is a local problem in the Zariski topology on $S'$:  if $T'$ has been constructed over a suitable open cover, the uniquenes will imply that the various $T'$s can be glued together.  We can therefore work Zariski-locally in $S'$, or equivalently in $S$, since $S$ and $S'$ have the same Zariski topology.  

The same reasoning shows that we can work Zariski-locally in $T$ as well.  This permits us to assume that $S$, $S'$, and $T$ are all affine.  By the Jacobian criterion (e.g., \cite[Tag 00TA, 00T6]{stacks}), we can assume that $S = \operatorname{Spec} A$ and that $T = \operatorname{Spec} B$ where $B = A[x_1, \ldots, x_n] / (f_1, \ldots, f_n)$ and the determinant $\det \frac{\partial f_i}{\partial x_j}$ is a unit of $B$.  If $S' = \operatorname{Spec} A'$, we take $T' = \operatorname{Spec} B'$ where $B' = A'[x_1, \ldots, x_n] / (g_1, \ldots, g_n)$ where $g_i$ is an arbitrary lift of $f_i$ to a polynomial with coefficients in $A'$.  Then $\det \frac{\partial g_i}{\partial x_j}$ reduces to $\det \frac{\partial f_i}{\partial x_j}$ in $B$.  Since $B'$ is an infinitesimal extension of $B$, this means that $\det \frac{\partial g_i}{\partial x_j}$ is a unit in $B'$, so by the Jacobian criterion, $B'$ is \'etale over $A'$.

This establishes that $T'$ exists locally in the Zariski topology of $S'$.  It remains to prove the uniqueness of $T'$.  Suppose that $T'$ and $T''$ are two such extensions.  Then by the infinitesimal lifting criterion for \'etale maps, applied to the diagram
\begin{equation*} \xymatrix{
T \ar[r] \ar[d] & T' \ar[d] \ar@{-->}@/_5pt/[dl]_-u \\
T'' \ar[r] \ar@/_5pt/@{-->}[ur]_-v & S'
} \end{equation*}
and there are unique lifts $u : T' \rightarrow T''$ and $v : T'' \rightarrow T'$.  These must compose to the identity maps, again by the infinitesimal criterion for \'etale maps, this time applied to the diagrams below:
\begin{equation*} \xymatrix{
T \ar[r] \ar[d] & T' \ar[d] \\
T'' \ar[r] \ar@/^5pt/[ur]^{vu} \ar@/_5pt/[ur]_{\mathrm{id}} & S'
} \qquad 
\xymatrix{
T \ar[r] \ar[d] & T' \ar@/^5pt/[dl]^-{uv} \ar@/_5pt/[dl]_-{\mathrm{id}} \ar[d] \\
T'' \ar[r]  & S'.
} \end{equation*}
\end{proof}

\begin{cor} \label{C:equiv-sites}
Suppose that $S \subseteq S'$ is an infintisimal extension of schemes.  Then the \'etale sites of $S$ and $S'$ are equivalent.
\end{cor}

\begin{dfn}\label{D:formal-sm-et-nr}
A morphism of CFGs $f : \mathcal X \rightarrow \mathcal Y$ is said, respectively, to be \emph{formally unramified}, \emph{formally \'etale}, or \emph{formally smooth} if every commutative diagram~\eqref{E:formal} below, in which $S \subseteq S'$ is an infinitesimal extension of schemes, 
	admits at most one (up to unique isomorphism), exactly one (up to unique isomorphism), or at least one lift \emph{\'etale-locally in $S$}.
	\begin{equation} \label{E:formal} \vcenter{ \xymatrix{
				S \ar[r] \ar[d] & \mathcal X \ar[d] \\
				S' \ar[r] \ar@{-->}[ur] & \mathcal Y.
	}} \end{equation}
\end{dfn}

We explicate a bit the meaning of  `\'etale-locally'  in the definition.  To say that $\mathcal X \rightarrow \mathcal Y$ is formally smooth means that, given any lifting problem~\eqref{E:formal}, there is an \'etale cover $\{ S_\alpha \rightarrow S \}$ such that, denoting by $S'_\alpha$ the unique infintesimal extension of $S_\alpha$ lifting $S$ (by Lemma~\ref{L:et-inf-lift}), the diagram
\begin{equation} \label{E:formal2} \vcenter{\xymatrix{
				S_\alpha \ar[r] \ar[d] & S \ar[r] \ar[d] & \mathcal X \ar[d] \\
				S'_\alpha \ar@{-->}[urr] \ar[r] & S' \ar[r]  & \mathcal Y
}} \end{equation}
admits a lift for every $\alpha$.

To say that $\mathcal X \rightarrow \mathcal Y$ is formally unramified means, first, that that given any two lifts of~\eqref{E:formal}, there is a cover of $S$ by $S_\alpha$ such that the induced lifts of~\eqref{E:formal2} are isomorphic, and, second, that any two isomorphisms between lifts of~\eqref{E:formal} agree after passage to a suitable \'etale cover of $S$.

To be formally \'etale is the conjunction of these properties.

\subsection{Local finite presentation}
\label{S:lfp2}

It was observed in \cite[Prop.~8.14.2]{EGAIV3} that a scheme $X$ is locally of finite presentation if and only if whenever $A = \varinjlim A_i$ is a filtered colimit of commutative rings, the natural map
\begin{equation} \label{E:lfp}
\varinjlim X(A_i) \rightarrow X(A)
\end{equation}
is a bijection.  The same formula characterizes algebraic stacks that are locally of finite presentation, provided one interprets a filtered colimit of groupoids correctly.  It is therefore reasonable to use~\eqref{E:lfp} as the \emph{definition} of local finite presentation for stacks that are not known to be algebraic.

To make sense of the filtered colimit of groupoids in equation~\eqref{E:lfp}, one can take
\begin{equation*}
\mathsf{Obj} \varinjlim \mathcal X(A_i) = \bigcup_i \mathsf{Obj} \mathcal X(A_i)
\end{equation*}
and, for any objects $\xi \in \mathcal X(A_i)$ and $\eta \in \mathcal X(A_j)$,
\begin{equation*}
\operatorname{Hom}(\xi, \eta) = \varinjlim_{k \geq i,j} \operatorname{Hom}_{\mathcal X(A_k)}(\xi_k, \eta_k)
\end{equation*}
where $\xi_k$ and $\eta_k$ denote pullbacks of $\xi$ and $\eta$, respectively, to $\mathcal X(A_k)$.

In order to formulate local finite presentation for morphisms of stacks, it is useful to introduce the pro-object $``\varprojlim\!" \operatorname{Spec} A_i$ associated to a filtered system of commutative rings $A_i$.  By definition, $``\varprojlim\!" \operatorname{Spec} A_i$ is the covariant functor on schemes (and stacks) obtained by taking the filtered colimit of the functors represented by the $\operatorname{Spec} A_i$.  What this actually means is that one should interpret a morphism
\begin{equation*}
``\varprojlim\!"\operatorname{Spec} A_i \rightarrow  X
\end{equation*}
to a scheme (or stack) 
as an object of
\begin{equation*}
\varinjlim \operatorname{Hom}(\operatorname{Spec} A_i,  X) = \varinjlim   X(A_i);
\end{equation*}
i.e., as compatible systems of morphisms $\operatorname{Spec}A_i\to   X$.  

\begin{rem}
 In more technical terms, the quotation marks indicate that one is taking a colimit in the category of covariant functors valued in sets (or groupoids).  We absolve ourselves of responsibility for the notation~\cite[p.~81]{sga4-1}.
\end{rem}

\begin{dfn} \label{D:lfp}
	A morphism $\mathcal X \rightarrow \mathcal Y$ of stacks in the \'etale topology on schemes is said to be \emph{locally of finite presentation} 
	if whenever $A = \varinjlim A_i$ is a filtered colimit of commutative rings, 
	then 
	 every commutative diagram of solid lines~\eqref{E:lfp-lifting} can be completed uniquely by a dashed arrow.
\begin{equation} \label{E:lfp-lifting} \vcenter{\xymatrix{
\operatorname{Spec} A \ar[r] \ar[d] & \mathcal X \ar[d] \\
``\varprojlim\!" \operatorname{Spec} A_i \ar[r] \ar@{-->}[ur] & \mathcal Y
}} \end{equation}
\end{dfn}

\begin{rem}
Observe the resemblance between the lifting diagram~\eqref{E:lfp-lifting} and the lifting diagram~\eqref{E:formal}.  This allows us to reason formally about \'etale maps and local finite presentation maps at the same time.

Under this analogy, Lemma~\ref{L:et-lim} below is the analogue of Lemma~\ref{L:et-inf-lift}.
\end{rem}

\begin{lem} \label{L:et-lim}
Suppose that a ring $A$ is the filtered colimit of rings $A_\ell$.  Set $S = \operatorname{Spec} A$ and $S_\ell = \operatorname{Spec} A_\ell$.  Then for any affine\footnote{In fact, quasicompact and quasiseparated would suffice.} \'etale map $T \rightarrow S$ there is for some $\ell$ an affine scheme $T_\ell$ admitting an  \'etale map $T_\ell \rightarrow S_\ell$,
and this map is unique up to unique isomorphism and enlargement of $\ell$.
\end{lem}
\begin{proof}
Suppose that $T = \operatorname{Spec} B$ is affine and \'etale over $S$.  By the Jacobian criterion, there is an open cover of $T$ by finitely many subsets $U = \operatorname{Spec} C$ where $C = A[x_1, \ldots, x_n] / (f_1, \ldots, f_n)$ and $\det \frac{\partial f_i}{\partial x_j}$ is a unit of $C$.  Since we are proving a uniqueness statement and only finitely many such subsets are involved, we may treat the subsets individually.  We may therefore assume that $T = U$.

For $k$ sufficiently large, the coefficients of the $f_j$ all appear in the image of $A_k$ in $A$, so that we may form the ring $B_k = A[x_1, \ldots, x_n] / (g_1, \ldots, g_n)$ for some lifts $g_1, \ldots, g_n$ of $f_1, \ldots, f_n$ to $A_k$.  The determinant $\det \frac{\partial g_i}{\partial x_j}$ maps to the unit $\det \frac{\partial f_i}{\partial x_j}$ in $B$.  Pick $t \in B$ inverse to $\det \frac{\partial f_i}{\partial x_j}$.  As $B = \varinjlim_{\ell \geq k} B_\ell$, the element $t$ must be in the image of some $B_\ell$ with $\ell \geq k$.  Furthermore $t \frac{\partial g_i}{\partial x_j}$ must equal $1$ in all sufficiently large $B_\ell$.  This implies that by taking $\ell$ sufficiently large, the image of $\det \frac{\partial g_i}{\partial x_j}$ is a unit.  Thus $B_\ell$ is \'etale over $A_\ell$ for all sufficiently large $\ell$, and $B_\ell$ induces $B$ over $A$.

We must still prove that $B_\ell$ is unique up to enlargement of $\ell$.  Suppose that $C_k$ were \'etale over $A_k$, also inducing $B$ over $A$.  Since $B_\ell$ is of finite presentation over $A_\ell$ and $C_k$ is of finite presentation over $A_k$, the isomorphism $\varprojlim B_\ell \simeq B \simeq \varprojlim C_k$ must come from an isomorphism $B_m \simeq C_m$ defined over some $m$ that is $\geq k$ and $\geq \ell$.  This map is necessarily unique up to further enlargement of $m$ (see \cite[Thm.~(8.8.2)~(i)]{EGAIV3} for more details).
\end{proof}

\begin{dfn}
By an \'etale map $``\varprojlim\!" U_i \rightarrow ``\varprojlim\!" \operatorname{Spec} A_i$, we mean a family of maps $U_i \rightarrow \operatorname{Spec} A_i$ that are \'etale for all sufficiently large $i$.
\end{dfn}

In terms of this definition, Lemma~\ref{L:et-lim} says  any \'etale map $U \rightarrow ``\varprojlim\!" \operatorname{Spec} A_i  $ is induced from an \'etale map $``\varprojlim\!" U_i \rightarrow ``\varprojlim\!" \operatorname{Spec} A_i$, and that this map is unique up to a unique, suitably defined, isomorphism.

\begin{rem}
Lemma~\ref{L:et-lim} says, in a sense that we do not attempt to make precise, that the \'etale site of $\operatorname{Spec} A$ is the colimit of the \'etale sites of the $\operatorname{Spec} A_i$, whenever $A$ is the filtered colimit of the $A_i$.  See \cite[Thm.~8.3.13]{sga4-2}.
\end{rem}

\subsection{Smooth, \'etale, and unramified morphisms} \label{S:sm-et-nr}

\begin{dfn}[{Unramified, \'etale, smooth]\cite[Def.~(17.3.1)]{EGAIV4}}] \label{D:sm-et-nr} 
	A morphism  of  algebraic stacks $f : \mathcal X \rightarrow \mathcal Y$ is said, respectively, to be \emph{unramified}, \emph{\'etale}, or \emph{smooth} if it is locally of finite presentation and formally unramified, formally \'etale, or formally smooth.
\end{dfn}

We have now defined smooth morphisms between algebraic stacks in two ways.  On one hand, we have defined smooth morphisms in Definition~\ref{D:P-adapted} as the unique extension of smoothness from morphisms of schemes in a way that is stable under composition and base change and local to the source and target.  On the other hand, we have defined smoothness intrinsically for morphisms of algebraic stacks in Definition~\ref{D:sm-et-nr}.

We are obliged to verify that they are equivalent.  It is sufficient to show that local finite presentation and smoothness are stable under composition and base change and local to the source and target, as these definitions clearly agree with the usual definitions in the category of schemes and the extension in Definition~\ref{D:P-adapted} was uniquely characterized by these properties.

\begin{lem} \label{L:intrinsic}
Let $\mathbf P$ be one of the following properties of morphisms of stacks in the \'etale topology on schemes:
\begin{enumerate}[(i)]
\item formal smoothness,
\item formal unramifiedness,
\item formal \'etaleness, or
\item local finite presentation.
\end{enumerate}
Then $\mathbf P$ is stable under composition and base change and is local to the source and target.
\end{lem}
\begin{proof}
We omit the verification for composition and base change, since these are formal and straightforward.

All of these properties may be phrased in terms of existence or uniqueness (or both) of lifts of a diagram
\begin{equation}\label{E:FundLiftDiag} \xymatrix{
S \ar[r] \ar[d] & \mathcal X \ar[d] \\
S' \ar@{-->}[ur] \ar[r] & \mathcal Y
} \end{equation}
In the case of formal unramifiedness, formal \'etaleness, or formal smoothness, $S'$ is an infinitesimal extension of $S$.  In the case of local finite presentation, we have a ring $A$ that is the filtered colimit of rings $A_i$, and $S = \operatorname{Spec} A$ and $S' = ``\varprojlim\!" \operatorname{Spec} A_i$.

The only fact we will use here is that if $S \rightarrow S'$ is the left side of one of these diagrams and $T \rightarrow S$ is an \'etale map, then there is a unique (up to unique isomorphism) extension of $T$ to an \'etale map $T' \rightarrow S'$.  In the case where $S \rightarrow S'$ is an infinitesimal extension, this is Lemma~\ref{L:et-inf-lift}; in the case where $S = \operatorname{Spec} A$ and $S' = ``\varprojlim\!" \operatorname{Spec} A_i$ with the $A_i$ filtered and $\varinjlim A_i = A$, this is Lemma~\ref{L:et-lim}.

We prove locality to the target.  Consider a lifting problem~\eqref{E:FundLiftDiag}, and assume that $\mathcal Y' \rightarrow \mathcal Y$ is a covering map such that the base change $\mathcal X' \rightarrow \mathcal Y'$ is $\mathbf P$.  We may freely replace $S$ by an \'etale cover, so we may assume that the map $S' \rightarrow \mathcal Y$ factors through $\mathcal Y'$.  This induces a factorization of $S \rightarrow \mathcal X$ through $\mathcal X'$.  Property $\mathbf P$ for $\mathcal X'$ over $\mathcal Y'$ gives a lift $f$ in diagram~\eqref{E:formal3}, which yields $g$ by composition.
\begin{equation} \label{E:formal3} \vcenter{\xymatrix{
S \ar[r] \ar[d] & \mathcal X' \ar[r] \ar[d] & \mathcal X \ar[d] \\
S' \ar[r] \ar@{-->}[ur]^-f \ar@{-->}[urr]_(0.7)g & \mathcal Y' \ar[r] & \mathcal Y 
}} \end{equation}

Now we prove locality to the source.  Again, consider a lifting problem~\eqref{E:FundLiftDiag} and suppose that $X_0 \rightarrow X$ is a cover that is formally smooth over $\mathcal Y$.  Since we can replace $S$ by an \'etale cover, we can assume that $S \rightarrow \mathcal X$ factors through $X_0$.  Formal smoothness of $X_0$ over $\mathcal Y$ gives a lift $f$ in diagram~\eqref{E:formal4}, which gives us $g$ by composition.
\begin{equation} \label{E:formal4} \vcenter{\xymatrix{
S \ar[r] \ar[dd] & X_0 \ar[d] \\
& \mathcal X \ar[d] \\
S' \ar@{-->}[uur]^(0.6)f \ar@{-->}[ur]_-g \ar[r] & \mathcal Y
}} \end{equation}
\end{proof}


\section{Infinitesimal deformation theory}
\label{S:def-thy}

In this section we introduce some deformation theory, with a view towards Artin's criterion for a stack to be algebraic.   Roughly speaking, deformation theory is the study of families of objects over Artin rings. To get a concrete idea, consider  the Kodaira--Spencer approach to deforming a complex manifold $X$  over the unit disk by extending  the given transition functions $\tau_{ij}$ of the manifold to transition functions $\widetilde \tau_{ij}=\tau_{ij}+\tau_{ij}^{(1)}t+\tau_{ij}^{(2)}t^2+\cdots$, where $t$ is the parameter on the disk (for each fixed $t$, one obtains a manifold with the given transition functions). Deformation theory would then be  the problem of  iteratively accomplishing this formally, by first extending up to $t$ (e.g., $\tau_{ij}+\tau_{ij}^{(1)}t \mod t^2$), and then extending up to $t^2$, and so on.    In the process, one would observe that the first order deformations are governed by $H^1(X,T_X)$ (called the tangent space), and that once one has extended to first order, the ability to extend to second order would be goverened by $H^2(X,T_X)$ (called an obstruction space).  

In general, deformation theory aims to abstract this to the setting of any stack over schemes (and to deforming over bases other than the unit disk).
The heart of the matter turns out to be defining   the tangent space to a stack, and an obstruction theory to a stack.  We take an abstract point of view and identify precisely the condition, homogeneity, on a CFG that gives it a well-behaved tangent space.  The advantage of this abstraction is that the existence of a well-behaved tangent space, and obstruction theory,  can be used to prove that a stack is algebraic, as we will discuss in Section~\ref{S:Artin}.

\subsection{Homogeneous categories fibered in groupoids}
\label{S:homogeneity}

After introducing homogeneity and demonstrating its basic properties, our goal will be to show that the stack of Higgs bundles is homogeneous, without relying on its algebraicity.  In this first section we stick to the abstract setting, and also introduce the tangent space to a stack.  We discuss the pertinent deformation theory in the following section, \S \ref{S:def}, where we establish that the stack of Higgs bundles in homogeneous.

\subsubsection{The tangent bundle of a stack}
\label{sec:tangent-space}

The following definition reprises Definition~\ref{D:inf-ext}:

\begin{dfn}[Square-zero extension] \label{D:InfExt}
An \emph{infinitesimal extension} of schemes is a closed embedding $S \subseteq S'$ such that the ideal $I_{S/S'}$ is nilpotent.  A \emph{square-zero extension} is an infinitesimal extension $S \subseteq S'$ such that $I_{S/S'}^2 = 0$.
\end{dfn}

The utility of square-zero extensions is twofold:  every infinitesimal extension can be factored as a sequence of square-zero extensions, and square-zero extensions behave `linearly', in the sense that deformations and obstructions over square-zero extensions are classified by linear-algebraic data.   These observations may be viewed as a functorial perspective on Taylor series.

The most important examples of infinitesimal extensions are the trivial ones:

\begin{dfn}[Dual numbers] \label{D:RDN}
The \emph{ring of dual numbers} is $\mathbb D = \mathbb{Z}[\epsilon] / (\epsilon^2)$.
  For any scheme $S$, we write $S[\epsilon] = S \times_{\operatorname{Spec} \mathbb Z} \operatorname{Spec} \mathbb D$.  More generally, if $V$ is a quasicoherent sheaf of $\mathcal{O}_S$-modules, we write $\mathbb D(V)$ for the sheaf of $\mathcal{O}_S$-algebras,
\begin{equation*}
\mathcal{O}_S + \epsilon V = \operatorname{Sym}^\bullet (\epsilon V) / (\epsilon^2 \operatorname{Sym}^2 V),
\end{equation*}
whose underlying sheaf of abelian groups consists locally of symbols $f + \epsilon v$ with $f \in \mathcal{O}_S$ and $v \in V$ and has multiplication law
\begin{equation*}
(f + \epsilon v)(g + \epsilon w) = fg + \epsilon (fw + gv).
\end{equation*}
We write $S[\epsilon V]$ for the scheme whose underlying topological space is $S$ and whose sheaf of rings is $\mathbb D(V)$; i.e., $S[\epsilon V]=\underline{\operatorname{Spec}}_S \mathbb D(V)$.

There is a canonical morphism $S[\epsilon V] \rightarrow S$ corresponding to the following homomorphism of sheaves of rings:
\begin{gather*}
\mathcal{O}_S \rightarrow \mathcal{O}_S + \epsilon V \\
f \mapsto f + 0 \epsilon
\end{gather*}
Unless otherwise specified, when it is necessary to equip $S[\epsilon V]$ with the structure of a scheme over $S$, we do so with this morphism.
\end{dfn}

\begin{rem} \label{R:SDN}
The construction of $S[\epsilon V] \rightarrow S$ commutes with base change in $S$.  It can therefore be extended to apply to any CFG, as follows.  If $\mathcal X$ is a CFG and $V$ is a quasicoherent sheaf on $\mathcal X$, then we define $\mathcal X[\epsilon V](S)$ to the category of pairs $(\xi, \delta)$ where $\xi \in \mathcal X(S)$ and $\delta$ is a section of $S[\epsilon \xi^\ast V]$ over $S$.
When $V = \mathcal O_{\mathcal X}$, we write $\mathcal X[\epsilon]$ rather than $\mathcal X[\epsilon \mathcal O_{\mathcal X}]$; in this case $$\mathcal X[\epsilon] = \mathcal X \times_{\operatorname{Spec}\mathbb Z} \operatorname{Spec} \mathbb D.$$
\end{rem}

\begin{dfn}[Tangent bundle] \label{D:TangSp}  
Let  $\mathcal X$ be  a CFG over $\mathsf S/S$.  We give $R[\epsilon]$ the structure of an $S$-scheme via the canonical projection $R[\epsilon] \rightarrow R \rightarrow S$, and  the \emph{relative tangent bundle} of $\mathcal X$ over $S$  is  the category fibered in groupoids $T_{\mathcal X/S}$ over $\mathsf S/S$ with   
\begin{equation*}
T_{\mathcal X / S}(R) = {\mathcal X}\bigl(R[\epsilon]\bigr) 
\end{equation*}
for each $S$-scheme $R$.  
\end{dfn}

\begin{rem}
If $\mathcal X$ is a presheaf, then $T_{\mathcal X/S}$ is also a presheaf.  In the case of $S=\operatorname{Spec}\mathbb Z$  we write $T_{\mathcal X}=T_{\mathcal X/\operatorname{Spec}\mathbb Z}$.   
\end{rem}

The idea to study the tangent space this way, and to think of the spectrum of the ring of dual numbers as a pair of infinitesimally nearby points, goes back at least to Weil~\cite[\S2]{Weil}.

\begin{pro}
For any morphism of schemes $X \rightarrow S$, we have
\begin{equation*}
T_{X/S} = \underline{\operatorname{Spec}}_X ( \operatorname{Sym}^\bullet \Omega_{X/S} )
\end{equation*}
\end{pro}
\begin{proof}
This reduces to the affine situation, where it comes down  to the following identities: given a ring $k$ and algebras  $k\to A\stackrel{\phi}{\to} B$, then 
$$
\operatorname{Hom}_{k\operatorname{-alg},\phi}(A,B[\epsilon])=\operatorname{Der}_k(A,B)=\operatorname{Hom}_{A\operatorname{-mod}}(\Omega_{A/k},B)
$$
$$
=\operatorname{Hom}_{A\operatorname{-alg}}(\operatorname{Sym}^\bullet \Omega_{A/k},B),
$$
where the first group of homomorphisms consists of  the $k$-algebra homomorphisms that reduce to $\phi$ modulo $\epsilon$.  
\end{proof}

\begin{cor}
When $X$ is a smooth scheme over $\mathbb C$, then $T_{X/\operatorname{Spec}\mathbb C}$ is a vector bundle over $X$ and coincides with any familiar definition of the tangent bundle.
\end{cor}

\begin{cor}
When $\mathcal X$ is an algebraic stack over $S$, so is $T_{\mathcal X /S}$.  
\end{cor}

\begin{proof}
It is almost immediate that $T_{\mathcal X/S}$ is a stack in the \'etale topology:  If we have an \'etale cover of $R$ by $U_i$ then the $U_i[\epsilon]$ form an \'etale cover of $R[\epsilon]$.  \'Etale descent for the maps $U_i[\epsilon] \rightarrow {\mathcal X}$ to $R[\epsilon] \rightarrow {\mathcal X}$ yields \'etale descent for $U_i \rightarrow T_{{\mathcal X}/S}$ to $R \rightarrow T_{{\mathcal X}/S}$.  To see that $T_{{\mathcal X}/S}$ is an \emph{algebraic} stack, note that if $X_0 \rightarrow {\mathcal X}$ is a smooth cover of ${\mathcal X}$ by a scheme then then $T_{X_0/S} \rightarrow T_{\mathcal X/S}$ is also a smooth cover.
\end{proof}

\begin{wrn}
We will see in a moment (item~(\ref{I:proj}) on p.~\pageref{I:proj}) that there is a projection $T_{\mathcal X /S} \rightarrow {\mathcal X}$, as one expects, but that it is not necessarily representable by algebraic spaces!  This is because the objects parameterized by $\mathcal X$ may possess \emph{infintisimal automorphisms}.  Nevertheless we will see that the fibers of $T_{\mathcal X/S}$ over ${\mathcal X}$ behave like `groupoids with vector space structure'.  This is why we insist on calling $T_{\mathcal X/S}$ the tangent \emph{bundle}:  it has all the features of a vector bundle except for being a  set!
\end{wrn}

\subsubsection{Homogeneity}

\begin{teo}[{\cite[Prop.~2.1]{obs}}] \label{T:Homog} Let $S$ be a scheme. 
Suppose that $Q \subseteq Q'$ is an infinitesimal extension of $S$-schemes and $f : Q \rightarrow R$ is an \emph{affine} $S$-morphism.  Then there is a universal (initial) $S$-scheme $R'$ completing diagram~\eqref{E:Homog}:
\begin{equation} \label{E:Homog} \vcenter{\xymatrix{
Q \ar@{^(->}[r]^{\text{inf.}} \ar[d]^f_{\text{affine}} & Q' \ar@{-->}[d] \\
R \ar@{-->}[r] & R'
}} \end{equation}
\noindent Furthermore, $R'$ is also universal (initial) among algebraic $S$-stacks completing the diagram.  The underlying topological space of $R'$ is the same as that of $R$ and viewing $\mathcal O_{Q'}$ and $\mathcal O_{R'}$ as sheaves on $Q$ and $R$ respectively, 
\begin{equation*}
\mathcal{O}_{R'} =\mathcal{O}_R   \mathop{\times}_{f_\ast \mathcal{O}_Q} f_\ast \mathcal{O}_{Q'}.
\end{equation*}
In particular, if $R = \operatorname{Spec} A$ (hence $Q$ and $Q'$ are also affine, say $Q = \operatorname{Spec} B$ and $Q' = \operatorname{Spec} B'$) then
\begin{equation*}
R' = \operatorname{Spec} (A \mathop{\times}_B B') .
\end{equation*}
\end{teo}

The universality here means that for any \emph{algebraic stack} ${\mathcal X}$ the natural map displayed in equation~\eqref{D:HomogMap} is an equivalence of groupoids:
\begin{equation} \label{D:HomogMap}
{\mathcal X}(R') \rightarrow {\mathcal X}(R) \mathop{\times}_{{\mathcal X}(Q)} {\mathcal X}(Q')
\end{equation}

\begin{proof}
We will prove this in the case where $R$ is also an infinitesimal extension of $Q$, which implies that the \'etale sites of $Q$, $Q'$, $R$, and $R'$ are all equivalent (Corollary~\ref{C:equiv-sites}).

First, note that schemes are homogeneous.  Indeed, suppose we have a commutative diagram
\begin{equation*} \xymatrix{
Q \ar@{^(->}[r] \ar[d] & Q' \ar[d] \\
R \ar[r] & X
} \end{equation*}
where $X$ is a scheme.  All of $Q$, $Q'$, and $R$ have the same underlying topological space, so we denote by $f$ the continuous map from that space to the underlying topological space of $X$.  The maps of schemes give homomorphisms of sheaves of rings:
\begin{equation*} \xymatrix{
f^{-1} \mathcal O_X \ar[r] \ar[d] & \mathcal O_{Q'} \ar[d] \\
\mathcal O_R \ar[r] & \mathcal O_Q 
} \end{equation*}
By the universal property of $\mathcal O_{R'} = \mathcal O_R \mathbin\times_{\mathcal O_{Q}} \mathcal O_{Q'}$, there is a unique factorization $f^{-1} \mathcal O_X \rightarrow \mathcal O_{R'}$, which gives the desired, uniquely determined, map $R' \rightarrow X$.

To extend this to algebraic stacks, it is sufficient by Definition~\ref{D:alg-stack} to show that \emph{whenever $\mathcal X$ is a stack admitting a smooth cover by a homogeneous stack, $\mathcal X$ is homogeneous}.  

Suppose that we have a commutative diagram
\begin{equation*} \xymatrix{
Q \ar@{^(->}[r] \ar[d] & Q' \ar[d] \\
R \ar[r] & \mathcal X
} \end{equation*}
where $\mathcal X$ has a smooth cover $X_0 \rightarrow \mathcal X$ and $X_0$ is homogeneous.  Since we are trying to prove a uniqueness assertion, it is sufficient to work \'etale-locally on $Q$.  Replace $Q$ by an \'etale cover such that the map $Q \rightarrow \mathcal X$ factors through $X_0$.  Since $X_0$ is formally smooth over $\mathcal X$, we can assume, after refining the cover further, that the maps $Q' \rightarrow \mathcal X$ and $R \rightarrow \mathcal X$ also lift to $X_0$, extending the lift already constructed over $Q$.  By the homogeneity of $X_0$, there is a unique map $R' \rightarrow X_0$ compatible with the maps from $Q$, $Q'$, and $R$.  Composing with the projection to $\mathcal X$ gives the desired map $R' \rightarrow \mathcal X$.
\end{proof}

\begin{dfn}[Pushout] \label{D:InfPush}
Suppose that $Q \subseteq Q'$ is an infinitesimal extension of $S$-schemes and $f : Q \rightarrow R$ is an affine  $S$-morphism.
We refer to the scheme $R'$ in Theorem 
\ref{T:Homog} as the \emph{pushout} of $Q \subseteq Q'$ along the map $f:Q \rightarrow R$.
\end{dfn}

A CFG is said to be homogeneous if it behaves like an algebraic stack with respect to \eqref{D:HomogMap}: 

\begin{dfn}[Homogeneous CFG] \label{D:Homog}
A CFG (not necessarily an algebraic stack) is said to be \emph{homogeneous} if for every pushout as in Definition \ref{D:InfPush},  we have~\eqref{D:HomogMap} is an equivalence of categories.  
\end{dfn}

\begin{rem}
This definition is a natural extension of \cite[Def.~VI.2.5]{sga7-1}, which was stated only in the context of artinian   algebras.  For much of what we have to say, we will only be interested in homogeneity with respect to morphisms $Q \rightarrow R$ that are isomorphisms on topological spaces.  However, in practice, it is little more difficult to verify homogeneity in the generality we have formulated it.
\end{rem}

As we have seen in Theorem~\ref{T:Homog}, homogeneity is a necessary condition for a stack over schemes to be algebraic.  It is not sufficient, but in the presence of a reasonable finiteness condition, it does guarantee \emph{formal} representability by an algebraic stack (see Theorem~\ref{T:Rim}).

\vskip .2 cm 
The following theorem is very helpful for checking homogeneity. 
 It was asserted by Schlessinger in \cite[Lem.~3.4]{Schlessinger} but proved there only for free modules; Milnor gave a proof for projective modules~\cite[Thms.~2.1--2.3]{Milnor}; a proof in the general case can be found in \cite[Thm.~2.2]{Ferrand}.

\begin{teo}[Schlessinger, Milnor, Ferrand] \label{T:flat-homog}
	Let $\mathcal Q$ be the CFG whose fiber over a scheme $S$ is the category of flat, quasicoherent $\mathcal O_S$-modules.  Then $\mathcal Q$ is homogeneous.
\end{teo}

In language closer to the original statements, this says that if one has a cartesian diagram of commutative rings
\begin{equation*} \xymatrix{
		B  & B'  \ar[l]_{\text{inf.}} \\
		A  \ar[u]& A' \ar[u] \ar[l]
} \end{equation*}
in which $B' \rightarrow B$ (and therefore also $A' \rightarrow A$) is a nilpotent extension then it is equivalent to specify either of the following data:
\begin{enumerate}[(i)]
	\item a flat $A'$-module $M'$, or
	\item a flat $A$-module $M$, a flat $B'$-module $N'$, a flat $B$-module $N$, and an $A$-module map $M \rightarrow N$ and a $B'$-module map $N' \rightarrow N$ 
	inducing by adjunction isomorphisms 
	$M\otimes_AB \rightarrow N$ and   $N'\otimes_{B'}B \rightarrow N$ 
	of $B$-modules.
\end{enumerate}
One direction of the correspondence sends $M'$ to $M = M' \mathop{\otimes}_{A'} A$, $N' = M' \mathop{\otimes}_{A'} B'$, $N = M' \mathop{\otimes}_{A'} B$ with the canonical maps $M' \mathop{\otimes}_{A'} A \rightarrow M' \mathop{\otimes}_{A'} B$ and $M' \mathop{\otimes}_{A'} B' \rightarrow M' \mathop{\otimes}_{A'} B$.  The other direction has $M' = M \mathop{\times}_N N'$.

\subsubsection{Tangent bundle to a homogeneous stack}
\label{S:tangent-bundle}

To begin to appreciate the significance of homogeneity, consider that in general, the tangent bundle of a stack, and even the tangent bundle of a presheaf, need not be a very well behaved object:  it might not even be a vector space.  However, the tangent space of a \emph{homogeneous} stack has all of the familiar structure one expects:
\begin{enumerate}[(i)]

\item \label{I:proj} (projection to base) The map $\mathbb{Z}[\epsilon] / (\epsilon^2) \rightarrow \mathbb{Z}$ sending $\epsilon$ to $0$ induces closed embeddings $R \rightarrow R[\epsilon]$ for all schemes $R$.  For any $S$-scheme $R$, we use this to obtain a morphism  $T_{\mathcal X/S}(R) = {\mathcal X}(R[\epsilon]) \rightarrow {\mathcal X}(R)$, whence a projection $T_{\mathcal X/S} \rightarrow {\mathcal X}$. (This does not require homogeneity.)

\item (zero section) The map $\mathbb{Z} \rightarrow \mathbb{Z}[\epsilon] / (\epsilon^2)$ induces $R[\epsilon] \rightarrow R$.  This gives a map ${\mathcal X}(R) \rightarrow {\mathcal X}(R[\epsilon]) = T_{\mathcal X/S}(R)$, whence a section ${\mathcal X} \rightarrow T_{\mathcal X/S}$ of the aforementioned projection.  This is the zero section in the vector bundle when $T_{\mathcal X/S}$ is a vector bundle.
(This does not require homogeneity.)

\item (addition law) We have $ \mathbb{Z}[\epsilon_1,\epsilon_2]/(\epsilon_1^2, \epsilon_1 \epsilon_2, \epsilon_2^2) \stackrel{\sim}{\longrightarrow} \mathbb D \mathbin{\times}_{\mathbb{Z}} {\mathbb D}$ given by $$x+y\epsilon_1+z\epsilon_2\mapsto (x+y\epsilon_1,x+z\epsilon_2).$$  In addition to the two projections ${\mathbb D} \mathbin{\times}_{\mathbb{Z}} {\mathbb D} \rightarrow {\mathbb D}$, there is a third map  sending both $\epsilon_1$ and $\epsilon_2$ to $\epsilon$.  
We obtain a commutative diagram whose outer square is cartesian (in fact the upper right and lower left are cartesian as well):
$$
\xymatrix@!C=5em@R=1em{
\mathbb Z&&\mathbb Z[\epsilon_1] \ar[ll]\\
& \mathbb Z[\epsilon] \ar[lu]&\\
\mathbb Z[\epsilon_2] \ar[uu] && \mathbb Z[\epsilon_1,\epsilon_2]/(\epsilon_1,\epsilon_2)^2 \ar[ll]_<>(0.5){\epsilon_1\mapsto 0} \ar[uu]_<>(0.65){\epsilon_2 \mapsto 0} \ar[lu]_{\epsilon_1,\epsilon_2\mapsto\epsilon} \\
}
$$
Applying this to any $S$-scheme $R$, we obtain a pushout diagram with an extra morphism 
  $\sigma: R[\epsilon] \rightarrow R[\epsilon_1, \epsilon_2]$, 
where  $R[\epsilon_1, \epsilon_2] = R \times \operatorname{Spec} ({\mathbb D} \mathbin{\times}_{\mathbb{Z}} {\mathbb D})$ (it is also the vanishing locus of $\epsilon_1 \epsilon_2$ in $R[\epsilon_1][\epsilon_2]$):
$$
\xymatrix@!C=5em@R=1em{
R&&R[\epsilon_1] \ar@{<-}[ll]\\
&R[\epsilon]\ar@{<-}[lu]&\\
R[\epsilon_2] \ar@{<-}[uu] &&R[\epsilon_1,\epsilon_2] \ar@{<-}[ll] \ar@{<-}[uu]\ar@{<-}[lu]. \\
}
$$
Now consider maps from the diagram above into our homogeneous stack ${\mathcal X}$; the map ${\mathcal X}(R[\epsilon_1, \epsilon_2]) \rightarrow {\mathcal X}(R[\epsilon_1]) \mathbin{\times}_{{\mathcal X}(R)} {\mathcal X}(R[\epsilon_2])$ is an equivalence of categories.

Combining these facts, and abbreviating $T_{\mathcal X/S}$ to $T$, we have a diagram:
\scalebox{.9}{\parbox{\linewidth}{
\begin{equation*}
T(R) \!\mathop{\times}_{\mathcal X(R)} \!T(R) = \mathcal X(R[\epsilon_1]) \!\!\mathop{\times}_{\mathcal X(R)}\!\! \mathcal X(R[\epsilon_2]) \xleftarrow{\sim} \mathcal X(R[\epsilon_1, \epsilon_2]) \xrightarrow{\mathcal X(\sigma)} \mathcal X(R[\epsilon]) = T(R)
\end{equation*}
}}

\noindent This gives the fibers of $T_{\mathcal X/S}(R) \rightarrow {\mathcal X}(R)$ an addition law; i.e., given an $R$-point $R\to {\mathcal X}$, the collection of all extensions $R[\epsilon]\to {\mathcal X}$ has an induced addition law.    This coincides with the addition law of the vector space structure in the familiar situation where ${\mathcal X}$ is a smooth scheme and $R$ is the spectrum of a field.
\item (scalar multiplication) If $\lambda \in \Gamma(R, \mathcal{O}_R)$, there is an induced map:
\begin{equation*} \xymatrix@R=0pt{
\mathcal{O}_R + \epsilon \mathcal{O}_R \ar[r] & \mathcal{O}_R + \epsilon \mathcal{O}_R \\
f + \epsilon v \ar@{|->}[r] &  f + \epsilon \lambda v
} \end{equation*}
giving a map $R[\epsilon]\to R[\epsilon]$.   Applying $\mathcal X$, 
this induces a map $$T_{\mathcal X/S}(R) \rightarrow T_{\mathcal X/S}(R)$$ commuting with the projection to ${\mathcal X}(R)$.  This is the action of scalars in the vector space structure when $R$ is a point and ${\mathcal X}$ is smooth.
\item (vector space structure) The axioms of a vector space can be verified in similar ways, occasionally making use of ${\mathbb D} \mathop{\times}_{\mathbb{Z}} {\mathbb D} \mathop{\times}_{\mathbb{Z}} {\mathbb D} = \mathbb{Z}[\epsilon_1, \epsilon_2, \epsilon_3] / (\epsilon_1, \epsilon_2, \epsilon_3)^2$.  We leave these verifications to the reader.
\item (Lie bracket) The Lie bracket will not be used in the rest of the paper, and may be skipped. There is a map
\begin{equation*}
{\mathbb D} \mathop{\otimes}_{\mathbb{Z}} {\mathbb D} = \mathbb{Z}[\epsilon_1, \epsilon_2] / (\epsilon_1^2, \epsilon_2^2) \rightarrow \mathbb{Z}[\epsilon_1, \epsilon_2] / (\epsilon_1, \epsilon_2)^2 = {\mathbb D} \mathop{\times}_{\mathbb{Z}} {\mathbb D}
\end{equation*}
sending $\epsilon_1 \epsilon_2$ to $0$.  
There is also an isomorphism  
\begin{gather*}
(\mathbb D \mathop{\otimes}_{\mathbb{Z}} \mathbb D) \mathop{\times}_{\mathbb{Z}[\epsilon_1, \epsilon_2] / (\epsilon_1, \epsilon_2)^2} (\mathbb D \mathop{\otimes}_{\mathbb{Z}} \mathbb D) \cong (\mathbb D \mathop{\otimes}_{\mathbb{Z}} \mathbb D) \mathop{\times}_{\mathbb{Z}} \mathbb{Z}[\epsilon_1 \epsilon_2] / (\epsilon_1 \epsilon_2)^2 
\end{gather*}
sending $(a,b)$ on the left to $(a,(a \bmod{(\epsilon_1, \epsilon_2)}) + b-a)$ on the right (this choice of isomorphism will give rise to the choice of a sign for the Lie bracket).  This all leads to a commutative diagram, in which both the small square and the outer rectangle are cartesian:
\vskip1ex
\begin{equation}\label{E:LieBrackZZ}
\xymatrix{
\mathbb Z[\epsilon_1, \epsilon_2] & \mathbb Z[\epsilon_1][\epsilon_2] \ar[l] & \mathbb Z[\epsilon_1 \epsilon_2] \ar@/_20pt/[ll] \\
\mathbb Z[\epsilon_1][\epsilon_2] \ar[u] & \mathbb Z[\epsilon_1][\epsilon_2] \mathbin\times_{\mathbb Z[\epsilon_1, \epsilon_2]} \mathbb Z[\epsilon_1][\epsilon_2] \ar@{-}^-{\sim}[r] \ar[l] \ar[u] & \mathbb Z[\epsilon_1][\epsilon_2] \mathbin\times_{\mathbb Z} \mathbb Z[\epsilon_1 \epsilon_2] \ar[u]
} \end{equation}

Suppose $v, w \in T_{\mathcal X}({\mathcal X})$ are vector fields on ${\mathcal X}$.  We view them as maps (see Definition~\ref{D:RDN} and Remark~\ref{R:SDN} for notation):
\begin{gather*}
v : \mathcal X[\epsilon_1] \rightarrow \mathcal X \\
w : \mathcal X[\epsilon_2] \rightarrow \mathcal X .
\end{gather*}
We obtain two morphisms
\begin{gather*}
\operatorname{id}\times v : \mathcal X[\epsilon_1]\times_{\mathcal X}\mathcal X[\epsilon_2] \rightarrow \mathcal X[\epsilon_2] \\
w\times \operatorname{id} : \mathcal X[\epsilon_1]\times_{\mathcal X}\mathcal X[\epsilon_2] \rightarrow \mathcal X[\epsilon_1];
\end{gather*}
note that the fibered product is over the standard projections on both factors.  
Note that  $\mathcal X[\epsilon_1] \mathbin\times_{\mathcal X} \mathcal X[\epsilon_2] = \mathcal X[\epsilon_1][\epsilon_2] = \mathcal X \mathbin\times \operatorname{Spec} (\mathbb D \otimes \mathbb D)=\mathcal X\times \operatorname{Spec}\mathbb Z[\epsilon_1,\epsilon_2]$.

The two compositions $v  (\mathrm{id} \times w)$, $w  (v \times \mathrm{id}): \mathcal X[\epsilon_1] \mathop{\times}_{\mathcal X} {\mathcal X}[\epsilon_2] \rightarrow {\mathcal X}$ are retractions  (they agree with the identity when restricted to the canonical inclusion of $\mathcal X$ in $\mathcal X[\epsilon_1] \mathop\times_{\mathcal X} \mathcal X[\epsilon_2]$) that agree with one another when restricted to ${\mathcal X}[\epsilon_1, \epsilon_2] \subseteq {\mathcal X}[\epsilon_1] \mathop{\times}_{\mathcal X} {\mathcal X}[\epsilon_2]$.
In other words, they give us a commutative diagram, dual to~\eqref{E:LieBrackZZ}:
\begin{equation}\label{E:LieBrackXX} \vcenter{
\xymatrix{
\mathcal X[\epsilon_1, \epsilon_2] \ar[r] \ar[d] & \mathcal X[\epsilon_1][\epsilon_2] \ar[d] \ar@/^25pt/[ddr] \\
\mathcal X[\epsilon_1][\epsilon_2] \ar[r] \ar@/_15pt/[drr] & \mathcal X[\epsilon_1][\epsilon_2] \amalg_{\mathcal X[\epsilon_1, \epsilon_2]} \mathcal X[\epsilon_1][\epsilon_2]  \ar@{-->}[dr]  \\
& & \mathcal X
}} \end{equation}
All of the morphisms restrict to the identity on $\mathcal X$.  Now we restrict under the inclusion $\mathcal X[\epsilon_1 \epsilon_2] \subset \mathcal X[\epsilon_1][ \epsilon_2] \mathbin\times_{\mathcal X[\epsilon_1, \epsilon_2]} \mathcal X[\epsilon_1][\epsilon_2]$ to get
\begin{equation*}
[v,w] : \mathcal X[\epsilon_1 \epsilon_2] \subset \mathcal X[\epsilon_1][\epsilon_2] \mathbin\amalg_{\mathcal X} \mathcal X[\epsilon_1 \epsilon_2] \xrightarrow{\sim} \mathcal X[\epsilon_1][ \epsilon_2] \mathbin\amalg_{\mathcal X[\epsilon_1, \epsilon_2]} \mathcal X[\epsilon_1][\epsilon_2] \rightarrow \mathcal X
\end{equation*}
The point is that the homogeneity of $\mathcal X$ ensures the existence of the dashed arrow above, and the composition of the dashed arrow with the canonical inclusion of $\mathcal X[\epsilon_1 \epsilon_2]$.

For the reader's convenience, we include a verification that this does indeed compute the Lie bracket when $\mathcal X$ is representable by an affine scheme $\operatorname{Spec} A$.  In that case, $v$ and $w$ correspond to derivations $\delta_1$ and $\delta_2$ from $A$ to itself.  We consider these as ring homomorphisms:
\begin{equation*}
\mathrm{id} + \epsilon_i \delta_i : A \rightarrow A + \epsilon_i A
\end{equation*}
The map $v(\mathrm{id} \times w)$ is dual to the composition
\begin{equation*}
A \xrightarrow{\mathrm{id} + \epsilon_1 \delta_1} A + \epsilon_1 A \xrightarrow{\mathrm{id} \otimes (\mathrm{id} + \epsilon_2 \delta_2)} A + \epsilon_1 A + \epsilon_2 A + \epsilon_1 \epsilon_2 A
\end{equation*}
sending $f \in A$ to $f + \epsilon_1 \delta_1(f) + \epsilon_2 \delta_2(f) + \epsilon_1 \epsilon_2 \delta_2 \circ \delta_1(f)$.  To be clear, the second map is obtained from $\mathrm{id} + \epsilon_2 \delta_2$ by application of $A[\epsilon_1] \mathbin\otimes_A (-) = \mathbb Z[\epsilon_1] \mathbin\otimes_{\mathbb Z} (-) = (-)[\epsilon_1]$ and carries $f_0 + \epsilon_1 f_1$ to $(f_0 + \epsilon_2 \delta_2(f_0)) + \epsilon_1(f_1 + \epsilon_2 \delta_2(f_1))$.
The map $w(v \times \mathrm{id})$ sends $f$ to $f + \epsilon_1 \delta_1(f) + \epsilon_2 \delta_2(f) + \epsilon_1 \epsilon_2 \delta_1 \circ \delta_2(f)$.  Taking the difference of these recovers the derivation $\delta_2 \delta_1 - \delta_1 \delta_2$.
\end{enumerate}

Suppose $\xi$ is a $k$-point of ${\mathcal X}$. 
 The discussion above shows that $T_{\mathcal X}(\xi)$ has all the trappings of a $k$-vector space structure, except that $T_{\mathcal X}(\xi)$ is a groupoid that may not be equivalent to a set.  Nevertheless, it is useful to think of $T_{\mathcal X}(\xi)$ as a `$2$-vector space'.  In fact, given a scheme $R$ and a morphism $\xi:R\to \mathcal X$, we can extract two important invariants from this groupoid:
\begin{gather*}
T_{\mathcal X}^{-1}(\xi) := \operatorname{Aut}_{T_{\mathcal X}(\xi)}(0)\\
T_{\mathcal X}^0(\xi) := T_{\mathcal X}(\xi) / \mathrm{isom}
\end{gather*}
Here $0$, the zero section $0:R\to T_{\mathcal X}(\xi)$,  corresponds to an object of the groupoid  $T_{\mathcal X}(\xi)$, and $\operatorname{Aut}_{T_{\mathcal X}(\xi)}(0)$ is the automorphism group of that object.  Likewise, $T_{\mathcal X}(\xi) / \mathrm{isom}$ is the set of isomorphism classes of objects of the groupoid.
The following lemma is a formal consequence of the discussion above:

\begin{lem}
When $k$ is a field and $\xi:\operatorname{Spec} k \to \mathcal X$, the sets $T_{\mathcal X}^{-1}(\xi)$ and $T_{\mathcal X}^0(\xi)$ are $k$-vector spaces.
\end{lem}

It will be important to have a relative variant of the tangent bundle:

\begin{dfn}[Relative tangent bundle]
Let $f : {\mathcal X} \rightarrow \mathcal Y$ be a morphism of CFGs over $\mathsf S$.  We define the \emph{relative tangent bundle} $T_f = T_{{\mathcal X}/\mathcal Y}$ to be the kernel of the morphism of stacks $T_{\mathcal X/S} \rightarrow T_{\mathcal Y/S}$.  In other words, the relative tangent bundle is the fiber product of stacks
\begin{equation*}
T_{{\mathcal X}/\mathcal Y} = T_{\mathcal X/S} \mathop{\times}_{T_{\mathcal Y/S}} \mathcal Y = T_{\mathcal X/S} \mathop{\times}_{f^{-1} T_{\mathcal Y/S}} {\mathcal X}
\end{equation*}
where the morphisms $\mathcal Y \rightarrow T_{\mathcal Y/S}$ and ${\mathcal X} \rightarrow f^{-1} T_{\mathcal Y/S}$ are the zero sections.   In slightly more concrete terms, a section of $T_{\mathcal X / \mathcal Y}$ is a section of $T_{\mathcal X/S}$, together with an isomorphism between its image in $T_{\mathcal Y}$ and the zero section.
\end{dfn}

\subsubsection{Relative homogeneity}

\begin{dfn}[Relative homogeneity]
Let $f : \mathcal X \rightarrow \mathcal Y$ be a morphism of CFGs over $\mathsf{S}$.  We say that $f$ is \emph{homogeneous}, or that $\mathcal X$ is \emph{homogeneous} over $\mathcal Y$, if for any scheme $S$ and any morphism $S \rightarrow \mathcal Y$, the CFG $\mathcal X_S=\mathcal X\times_{\mathcal Y}S$ is homogeneous.
\end{dfn}

Another way of formulating the definition is that for any cocartesian diagram \eqref{E:Homog} and any compatible objects (we leave it to the reader to formulate compatibility precisely)
\begin{gather*}
\overline{\eta}' \in \mathcal Y(R') \qquad \eta \in \mathcal X(R) \qquad \xi' \in \mathcal X(Q') \qquad \xi \in \mathcal X(Q)
\end{gather*}
there is a $\eta \in \mathcal X(R')$ inducing all of them, and this $\eta$ is unique up to unique isomorphism.

The following lemma is proved by a standard formal argument:
\begin{lem} \label{L:RelHomog}
\begin{enumerate}[(i)]
\item Let $f : \mathcal X \rightarrow \mathcal Y$ and $g : \mathcal Y \rightarrow \mathcal Z$ be morphisms of CFGs over $\mathsf{S}$.  If $g$ is homogeneous then $f$ is homogeneous if and only if $gf$ is.
\item The base change of a homogeneous morphism of CFGs is homogeneous.
\end{enumerate}
\end{lem}

\subsection{Deformation theory}
\label{S:def}

When we speak of the deformation theory of a stack $\mathcal X$ we mean extending morphisms $R\to \mathcal X$ to morphisms $R'\to \mathcal X$, where $R'$ is an infinitesimal extension of $R$. 
The definition of the tangent space of a stack in \S\ref{sec:tangent-space} connects the tangent space of a stack ${\mathcal X}$ to the deformation theory of the objects it parameterizes.  Indeed, if $\xi \in {\mathcal X}(k)$ then $T_{\mathcal X}(\xi)$ over $\xi$ is precisely the groupoid of extensions of $\xi$ to $\xi' \in {\mathcal X}\bigl(k[\epsilon] / (\epsilon^2)\bigr)$. 

 We now work out several examples that will be used in the construction of the moduli of Higgs bundles.  Our main focus is on using the deformation theory to show that various stacks are homogeneous.

\subsubsection{Deformations of morphisms of vector bundles}
Here we discuss the deformation theory for morphisms of vector bundles. From our perspective, the main point is 
Lemma \ref{L:HomEFhmg} (see also Remark \ref{R:HomEFhmg}) on homogeneity.

Let $\pi : X \rightarrow S$ be a flat family of schemes over $S$ and let $E$ and $F$ be two vector bundles on $X$.  Let $\mathscr H = \mathscr Hom_{X/S}(E, F)$ 
be the $S$-sheaf of morphisms from $E$ to $F$ (see \S \ref{S:EX/S(F)}).

 We compute the tangent space $T_{\mathscr H/S}(\xi)$  of ${\mathscr H}$ over $S$ at an $R$-point $\xi : E_R \rightarrow F_R$, where $R\to S$ is an $S$-scheme, and $E_R$ and $F_R$ are the pullbacks of $E$ and $F$ to $R$.  An element of $T_{{\mathscr H}/S}(\xi)$ is an extension of the $S$-map $R \rightarrow {\mathscr H}$ to $R[\epsilon]$ such that the composition $R[\epsilon] \rightarrow {\mathscr H} \rightarrow S$ is the same as the composition $R[\epsilon] \rightarrow R \rightarrow S$ of the canonical retraction $R[\epsilon] \rightarrow R$ and the fixed map $R \rightarrow S$.  In other words, it is a morphism of vector bundles $E_{R[\epsilon]} \rightarrow F_{R[\epsilon]}$ that reduces to $u$ modulo $\epsilon$.

The vector bundle $E_{R[\epsilon]}$ is identified with $E_R+ \epsilon E_R$, and similarly for $F_{R[\epsilon]}$.  
Notice that $\xi(e_0 + \epsilon e_1) = \xi(e_0) + \epsilon \xi(e_1)$ gives one such morphism  $E_{R[\epsilon]} \rightarrow F_{R[\epsilon]}$ that reduces to $\xi$ modulo $\epsilon$.  If $\xi'$ is any other then the difference $\xi' - \xi$ is a map $E_R \rightarrow \epsilon F_R \cong F_R$,  so we get an identification $T_{{\mathscr H}/S}(\xi) \simeq {\mathscr H}(R)$.  As this identification is natural, we get
\begin{equation*}
T_{{\mathscr H}/S} = {\mathscr H} \mathop{\times}_S {\mathscr H} .
\end{equation*}
Of course, it is no surprise that the tangent space of ${\mathscr H}$ is ${\mathscr H}$ itself, since ${\mathscr H}$ already has a linear structure.  We have 
\begin{align*}
T^{-1}_{{\mathscr H}/S}(\xi)& = 0 \\
T^{0}_{{\mathscr H}/S}(\xi) &= \operatorname{Hom}(E_R, F_R)
\end{align*}

This gives the tangent space to ${\mathscr H}$ over $S$, but does not explain the higher order infinitesimal structure.  For that we must pose a more general lifting problem:  given $\xi \in {\mathscr H}(R)$ and an arbitrary square-zero extension $R'$ of $R$ as an $S$-scheme, can $\xi$ be lifted to $\xi' \in {\mathscr H}(R')$?  
To answer this question, we consider it locally in ${X}_R$, that is, we cover $X$ by open subsets $U$ with corresponding extensions $U'$ over $R'$ and ask for extensions of $E_R \big|_U \rightarrow F_R \big|_U$ to $E_{R'} \big|_{U'} \rightarrow F_{R'} \big|_{U'}$.  We make two observations:
\begin{enumerate}[(i)]

\item there is a cover of ${X}_{R'}$ by open sets $U'$ such that $\xi \big|_{X_R\cap U'}$ extends to a morphism $E_{U'} \rightarrow F_{U'}$.

\item if $\xi'$ and $\xi''$ are any two extensions, then $\xi' - \xi''$ may be viewed as a morphism $E_R \rightarrow F_R \otimes \pi^\ast J$ 
where $J$ is the ideal of $R$ in $R'$ and, by abuse of notation,  $\pi:X_R\to R$ is the morphism obtained from $\pi:X\to S$ by pullback.

\end{enumerate}
These observations combine to imply that there is a $\underline{\operatorname{Hom}}(E_R,F_R \otimes \pi^\ast J)$-\emph{torsor}  
$P$ on ${ X}$ (in the Zariski topology) whose sections are in bijection with the lifts of $\xi$ to $H(R')$ (those who would rather avoid the language of torsors may obtain the following lemma by a \v{C}ech cohomology calculation).  This yields a deformation-obstruction theory:

\begin{lem}[{cf.\ \cite[Thm.~8.5.3~(a)]{FGAe}}] \label{L:Sdef}
Let ${X}$ be an $S$-scheme, let $E$ and $F$ be vector bundles on ${X}$, and let $R \subseteq R'$ be a square-zero extension of $S$-schemes with ideal $J$.  Associated to any homomorphism $\xi : E_R \rightarrow F_R$ there is an obstruction $\omega \in H^1({ X}_R, \underline{\operatorname{Hom}}(E_R,F_R \otimes \pi^\ast J))$ whose vanishing is equivalent to the existence of an extension of $\xi$ to some $\xi' : E_{R'} \rightarrow F_{R'}$.  If there is at least one extension then the set of all extensions possesses a simply transitive action of $H^0({X}_R, \underline{\operatorname{Hom}}(E_R,F_R \otimes \pi^\ast J)) = \operatorname{Hom}(E_R,F_R \otimes \pi^\ast J)$. 
\end{lem}

Beyond its finite dimensionality, the particular deformation-obstruction theory is not actually necessary for the proof of algebraicity.  What is important is homogeneity:

\begin{lem}\label{L:HomEFhmg}
The functor $\mathscr Hom_{X/S}(E,F)$ is homogeneous.
\end{lem}
\begin{proof}
	If $f : Q \rightarrow R$ is an affine morphism of $S$-schemes and $Q \subseteq Q'$ is a square-zero extension of $S$-schemes, let $R'$ be the pushout (as in Theorem \ref{T:Homog}), with $f' : Q' \rightarrow R'$ denoting the tautological morphism.  (Since $Q$ and $Q'$ have the same underlying topological space, we do not bother to introduce notation for the inclusion $Q \subseteq Q'$.) 
Since $Q$ and $Q'$, as well as $R$ and $R'$ have the same underlying topological space, and we have a natural morphism  $F_{Q'}\to F_Q$, we can push forward to obtain a morphism $f'_*F_{Q'}\to f_*F_Q$.  Using adjunction we have have that the identity  $f^*F_R=F_Q$ induces a morphism $F_R\to f_*F_Q$.  Thus we obtain a fibered product $f'_*F_{Q'}\times_{f_*F_Q}F_R$.
	
	 Assume we have $v \in H(R)$ and $u' \in H(Q')$, both extending $u = f^\ast v$.  Then $u'$ gives us
\begin{equation*}
	{f'}^\ast E_{R'} = E_{Q'} \xrightarrow{u'} F_{Q'} = {f'}^\ast F_{R'},
\end{equation*}
whence $E_{R'} \rightarrow f'_\ast {f'}^\ast F_{R'} = f'_\ast F_{Q'}$ by adjunction.  We also have a map $E_{R'} \rightarrow E_R \xrightarrow{v} F_R$.  These induce the same map $E_{R'} \rightarrow f_\ast F_Q$  so we get a map
\begin{equation*}
	\eta' : E_{R'} \rightarrow f'_\ast F_{Q'} \mathop{\times}_{f_\ast F_Q} F_R
\end{equation*}
by the universal property of the fiber product.  

On the other hand, the canonical map $\gamma : F_{R'} \rightarrow f'_\ast F_{Q'} \mathop{\times}_{f_\ast F_Q} F_R$ is an isomorphism, by Theorem~\ref{T:flat-homog}, or directly:   $\gamma$ is an isomorphism modulo $J$, since it reduces to the identity $F_{R} = f_\ast F_Q \mathop{\times}_{f_\ast F_Q} F_R$, and the kernel of reduction is the fiber product of the kernels, namely $(f_\ast F_Q \otimes \pi^\ast J) \mathop{\times}_{0} 0 = F_R \otimes \pi^\ast J$.  Therefore $\gamma$ is an isomorphism by the $5$-lemma.  To conclude we can view $\eta'$ as a map $E_{R'} \rightarrow F_{R'}$ by composition with $\gamma^{-1}$.
\end{proof}

\begin{rem}\label{R:HomEFhmg}
	As was pointed out above (see Proposition \ref{P:LiebP2.3}), under mild hypotheses,  a result of Lieblich implies that $\mathscr Hom_{X/S}(E,F)$ is an algebraic space over $S$, locally of finite type.  In particular, this  also shows  that $\mathscr Hom_{X/S}(E,F)$ is homogeneous.  However, Lieblich's proof relies on the homogeneity, so this reasoning is actually circular.
\end{rem}

\subsubsection{Deformations of vector bundles} \label{S:def-vect}
Here we discuss the deformation theory for vector bundles; see also \cite[Prop.~6.5.1]{FGAe}.  The main points are Corollary~\ref{C:TFib} and Lemma \ref{L:Fib-hmg}
  (see also Remark \ref{R:DefEhmg}).

Let $\pi : {X} \rightarrow S$ be a family of schemes over $S$.  Suppose that $R$ is an $S$-scheme and $E \in \mathsf{Fib}_{{X}/S}(R)$ is a vector bundle over ${X}_R$.   Let $R'$ be a square-zero $S$-scheme extension with ideal $J$.  We ask  whether $E$ can be extended to $E' \in \mathsf{Fib}_{X/S}(R')$, and if so, in how many ways.  Again we make several observations:
\begin{enumerate}[(i)]
\item there is a cover of $X_{R'}$ by open subsets $U'$ such that $E \big|_{X_R \cap U'}$ can be extended in at least one way to a vector bundle on $U'$;
\item if $E'$ and $E''$ are two extensions of $E \big|_{X_R \cap U'}$ to $U'$ then there is a cover of $U'$ by open subsets $V'$ such that $E' \big|_{V'} \simeq E'' \big|_{V'}$ as extensions of $E \big|_{X_R \cap V'}$;
\item if $u, v : E' \big|_{V'} \rightarrow E'' \big|_{V'}$ are two isomorphisms of extensions of $E \big|_{X_R \cap V'}$ then $u$ and $v$ differ by a homomorphism $E \big|_{X_R \cap V'} \rightarrow  E \big|_{X_R \cap V'} \otimes \pi^\ast J$. 
\end{enumerate}
Choosing a suitable cover, we therefore obtain a \v{C}ech $2$-cocycle for the sheaf of groups $\underline{\operatorname{Hom}}(E, E\otimes \pi^\ast J)$.  This is a coboundary if and only if a deformation exists.  A more careful analysis shows that the isomorphism classes of all deformations then correspond to $1$-cocycles modulo coboundaries.  Here is the statement in its usual form:

\begin{lem}[{cf.\ \cite[Thm.~8.5.3~(b)]{FGAe}}] \label{L:VBdef}
Fix a flat family of schemes $\pi : X \rightarrow S$ and a square-zero $S$-extension $R \subseteq R'$ with ideal $J$.  A vector bundle $E$ 
 on $X_R$ induces an obstruction $\omega \in H^2\bigl(X_R, \underline{\operatorname{Hom}}(E, E \otimes \pi^* J)\bigr)$ whose vanishing is equivalent to the existence of an extension of $E$ to $X_{R'}$.  If $\omega = 0$, the set of isomorphism classes of extensions is a principal homogeneous set under a natural action of $H^1\bigl(X_R, \underline{\operatorname{Hom}}(E, E \otimes \pi^* J)\bigr)$.  The automorphisms (as an extension) of any given extension are canonically $H^0\bigl(X_R, \underline{\operatorname{Hom}}(E, E \otimes \pi^* J)\bigr)$.
\end{lem}

In particular, this lemma gives the tangent space of $\mathsf{Fib}_{X/S}$:

\begin{cor}\label{C:TFib}
	Suppose that $R$ is an $S$-scheme and $E$ is a vector bundle over $X_R$.  We have $T_{\mathsf{Fib}_{X/S}}(E) = 
	\mathrm{B} \underline{\operatorname{Hom}}(E, E)(X_R)$ 
where $E$ is the universal vector bundle on $\mathsf{Fib}_{X/S}$ 
and the prefix $\mathrm{B}$ denotes the classifying stack (Section~\ref{S:TorsG}). 
 In particular,
\begin{gather*}
T^{-1}_{\mathsf{Fib}_{X/S}}(E) = \underline{\operatorname{Hom}}(E,E) \\
T^0_{\mathsf{Fib}_{X/S}}(E) = \underline{\operatorname{Ext}}^1(E,E) .
\end{gather*}
\end{cor}
\begin{proof}
	We have just seen that sections of $T_{\mathsf{Fib}_{X/S}}$ correspond to $1$-cocycles for $\underline{\operatorname{Hom}}(E,E)$ modulo coboundaries.  But $1$-cocycles modulo coboundaries also classify torsors.
\end{proof}

The \v{C}ech calculations can be abstracted into the observation that properties (i) -- (iii) above imply there is a gerbe $\mathscr{G}$ over $X_R$  (see e.g., \cite[(3.15), p.22]{LMB} for the definition of a gerbe) whose sections correspond to extensions of $E$ to $E' \in \mathsf{Fib}_{X/R}(R')$, and that this gerbe is banded by the sheaf of abelian groups $\underline{\operatorname{Hom}}(E_R, E_R \otimes \pi^\ast J)$.  Giraud classifies banded gerbes cohomologically:
\begin{teo}[{\cite[Thm.~IV.3.4.2]{Giraud}}]
Let $\mathscr{G}$ be a gerbe on $X$, banded by an abelian group $A$.  There is an obstruction $\omega \in H^2(X, A)$ to the existence of a global section of $\mathscr{G}$.  Should this obstruction vanish, global sections up to isomorphism form a principal homogeneous set under the action of $H^1(X,A)$.  Automorphisms of any given section are in canonical bijection with $H^0(X,A)$.
\end{teo}

We also record the homogeneity of $\mathsf{Fib}_{X/S}$:

\begin{lem} \label{L:Fib-hmg}
For any scheme $X$, the CFG $\mathsf{Fib}_{X/S}$ is homogeneous.
\end{lem}
\begin{proof}
Since $S$ is representable, it is homogeneous by Theorem~\ref{T:Homog}.  Therefore by Lemma~\ref{L:RelHomog}, it is sufficient to show that $\mathsf{Fib}_{X/S}$ is homogeneous over $S$.  We will content ourselves to sketch the construction of the inverse to the functor
\begin{equation} \label{E:FibHomog}
\mathsf{Fib}_{X/S}(R') \rightarrow \mathsf{Fib}_{X/S}(Q') \mathop{\times}_{\mathsf{Fib}_{X/S}(Q)} \mathsf{Fib}_{X/S}(R)
\end{equation}
associated to a cocartesian diagram
\begin{equation*} \xymatrix{
Q \ar@{^(->}[r] \ar[d] & Q' \ar[d] \\
R \ar[r] & R'
} \end{equation*}
of $S$-schemes, where $Q \rightarrow R$ is affine and $Q \hookrightarrow Q'$ is an infinitesimal extension.  An object of the right side of~\eqref{E:FibHomog} consists of vector bundles $E'$ on $X_{Q'}$, $F$ on $X_R$, and $E$ on $X_Q$, along with identifications $E' \big|_{X_Q} = E = F \big|_{X_Q}$.  By applying \cite[Thm.~2.2]{Ferrand} to a cover of $X_{R'}$ by open affines and then to covers of the intersections by open affines, we obtain a vector bundle $F'$ on $X_{R'}$ restricting to $E'$ on $X_{Q'}$, to $E$ on $X_Q$, and to $F$ on $X_R$, as required.
\end{proof}

\begin{rem}\label{R:DefEhmg}
This also follows from the fact that $\mathsf {Fib}_{X/S}$ is an algebraic stack, but, as in Remark~\ref{R:HomEFhmg}, this reasoning is circular.
\end{rem}

\begin{cor} \label{C:VBMapHomog}
Let $\mathcal{F}$ be the stack of triples $(E,F,\sigma)$ where $E$ and $F$ are vector bundles on $X$ and $\sigma : E \rightarrow F$ is a morphism of vector bundles.  Then $\mathcal{F}$ is homogeneous.
\end{cor}
\begin{proof}
We have a morphism $\mathcal{F} \rightarrow \mathsf{Fib}_{X/S} \mathop{\times}_S \mathsf{Fib}_{X/S}$.  We have just seen that $\mathsf{Fib}_{X/S}$ is homogeneous, so by two applications of Lemma~\ref{L:RelHomog}, it follows that $\mathsf{Fib}_{X/S} \mathop{\times}_S \mathsf{Fib}_{X/S}$ is homogeneous.  Lemma~\ref{L:HomEFhmg} says that $\mathcal{F}$ is relatively homogeneous over $\mathsf{Fib}_{X/S} \mathop{\times}_S \mathsf{Fib}_{X/S}$, so we may  conclude by another application of Lemma~\ref{L:RelHomog}.
\end{proof}

\subsubsection{Deformations of nodal curves}
We now discuss the deformation theory of nodal curves.  

\begin{dfn}[Nodal curve]
	Let $S$ be a scheme.  A \emph{nodal curve} over $S$ is an     algebraic space $C$ and a projection $\pi : C \rightarrow S$ that is \'etale-locally isomorphic to $\operatorname{Spec} \mathcal{O}_S[x,y] / (xy - t)$ for a local section $t$ of $\mathcal O_S$.  That is, there is an \'etale cover of $S$ by affines $U = \operatorname{Spec} A$
	 and an \'etale cover of $\pi^{-1} U$ by schemes $V$, each of which admits an \'etale map to $\operatorname{Spec} A[x,y] / (xy - t)$ for some $t \in A$.

	We shall write $\mathcal N$ for the stack over schemes whose $S$-points are the families of nodal curves over $S$.
\end{dfn}

\begin{rem} \label{R:NC}
	A more conventional definition of a nodal curve over $S$ is as a flat family $\pi : C \rightarrow S$ whose fibers are $1$-dimensional, reduced schemes whose only singularities are ordinary double points.  It follows from \cite[Prop.~III.2.8]{FreitagKiehl} that the two notions are equivalent.
\end{rem}

We will show that nodal curves form a homogeneous stack (Lemma \ref{L:NChmg}), compute the tangent space of this stack (Lemmma \ref{L:nodal-tangent}), and show that it is formally smooth (Corollary  \ref{C:NCsm}).  This whole section could easily be adapted to curves with locally planar singularities (or to more general schemes of finite type with hypersurface singularities), but for concreteness, we stick to nodal curves.

The method used in the last section to study deformations of vector bundles works quite well for deformations of \emph{smooth}  curves, and even families of smooth schemes $\pi : X \rightarrow S$.  If $S'$ is a square-zero extension of $S$ with ideal $J$, one discovers that extensions of $X$ to $S'$ always exist locally, that any two deformations are locally isomorphic, and that any two deformations are related by a section of $T_{X/S} \otimes \pi^*J$, from which it follows by a \v{C}ech calculation (as in the last section) that (see \cite[Thm.~8.5.9~(b)]{FGAe}):

\begin{enumerate}[(i)]
\item there is obstruction to the existence of a flat extension $X'$ over $S$ extending $X$ lying in $H^2(X, T_{X/S} \otimes \pi^\ast J)$;
\item should a deformation exist, the isomorphism classes of flat extensions form a torsor under $H^1(X, T_{X/S} \otimes \pi^\ast J)$;
\item automorphisms of a fixed flat extension are canonically $H^0(X, T_{X/S} \otimes \pi^\ast J)$.
\end{enumerate} 
The deformation theory of nodal curves introduces a new complication:  It is still the case that deformations exist locally, and one can easily compute that automorphisms of deformations can be identified with sections of $T_{X/S} \otimes J$.  However, it is no longer the case any two deformations are locally isomorphic, since the node $xy = 0$ has a $1$-parameter family of first order deformations $xy = \lambda \epsilon$.  A naive \v{C}ech calculation will therefore not suffice, but with a little more effort, we will see that the deformations and obstructions for nodal curves can be classified in much the same way as for smooth curves.

Suppose that $C$ is an $S$-point of $\mathcal N$.  
An $S$-point of $T_{\mathcal N}$ lying above $C$ is a cartesian diagram~\eqref{E:TMg}, in which $C'$ is an $S[\epsilon]$-point (see Definition~\ref{D:RDN} for notation) of $\mathcal N$:

\begin{equation} \label{E:TMg} \vcenter{\xymatrix{
C \ar[r] \ar[d]_\pi & C' \ar[d] \\
S \ar[r] & S[\epsilon].
}} \end{equation}
We consider the more general problem of extending an $S$-point $C$ of $\mathcal N$ to a $S[\epsilon J]$-point:

\begin{equation*}  \xymatrix{
C \ar[r] \ar[d]_\pi & C' \ar[d] \\
S \ar[r] &  \ar@{-->}@/_1pc/[l] S[\epsilon J].
} \end{equation*}
Here $J$ is a quasicoherent sheaf on $S$ and $C'$ is an $S[\epsilon J]$-point of $\mathcal N$.  Since there is a canonical retraction $S[\epsilon J] \rightarrow S$, we can consider $C'$ as an $S$-scheme.  As $C$ is the closed sub-scheme of $C'$ determined by $\pi^\ast J$ (with $\pi^*J^2=0$), we may  form the exact sequence~\eqref{E:omega-extension}:

\begin{equation} \label{E:omega-extension}
0 \rightarrow \pi^\ast J \rightarrow \mathcal{O}_C \mathop{\otimes}_{\mathcal{O}_{C'}} \Omega_{C'/S} \rightarrow \Omega_{C/S} \rightarrow 0
\end{equation}
The right exactness of this sequence follows from general principles and does not depend on the fact that $C$ is a nodal curve.  To see the left exactness, we may work \'etale-locally and assume:
\begin{align*}
S & = \operatorname{Spec} A & S' & = \operatorname{Spec} A' \\
C & = \operatorname{Spec} B  & C' & = \operatorname{Spec} B' \\
B & = A[x,y] / I & B' & = A'[x,y] / I' \\
I & = (xy - t) & I' & = (xy - t')
\end{align*}
In fact, the local structure of nodes implies we can arrange for $C$ and $C'$ to be \emph{\'etale} over $A[x,y] / I$ and $A'[x,y] / I'$, respectively.  
Then the sequence~\eqref{E:omega-extension} on $C$ is pulled back via the \'etale (and in particular flat) map $C \rightarrow \operatorname{Spec} B$ from the corresponding sequence over $\operatorname{Spec} B$.  For the sake of proving~\eqref{E:omega-extension} is exact, we can therefore replace $C$ with $\operatorname{Spec} B$ and $C'$ with  $\operatorname{Spec} B'$ without loss of generality.

Our diagram of rings is
$$
\xymatrix{
B=A[x,y]/(xy-t) & B'=A'[x,y]/(xy-t') \ar[l] \\
A \ar[u] \ar@{-->}@/^1pc/[r]&  \ar[l] \ar[u] A'=A[\epsilon J].
}
$$
We are trying to show the exactness of
\begin{equation} \label{E:omega-extension-ring}
0  \to B\otimes_A J \rightarrow B \mathop{\otimes}_{B'} \Omega_{B'/A} \rightarrow \Omega_{B/A} \rightarrow 0.
\end{equation} 
The left exactness of~\eqref{E:omega-extension-ring} is equivalent to showing that, for every injective  
\footnote{Recall that a morphism $M'\to M$ of modules over a ring $A$ is injective if and only if for every injective $A$-module $I$ , 
$\operatorname{Hom}_A(M,I)\to \operatorname{Hom}_A(M',I)$ is surjective.  To see this, choose an inclusion  $M'\hookrightarrow I$ for some injective module $I$.  Then we get a map $M \rightarrow I$ extending the inclusion $M' \subseteq I$, so $M' \rightarrow M$ must be injective.}
$B$-module $M$, 
the map  $\operatorname{Hom}_B(B\otimes_{B'}\Omega_{B'/A},M)\to \operatorname{Hom}_B(B\otimes_A J,M)$ is surjective; i.e., 
every $B$-module homomorphism $B\mathop{\otimes}_A J  \rightarrow M$ extends to a morphism $B\otimes_{B'}\Omega_{B'/A}\to M$.  Using tensor-hom adjunction, and the universal property of the module of K\"ahler differentials, this translates into the statement that, every $A$-module homomorphism $J \rightarrow M$ should extend to an $A$-derivation $B' \rightarrow M$:
$$
\xymatrix{
B'=\frac{A[\epsilon J][x,y]}{(xy-t')} \ar@{-->}[rd]^{\partial}& \\
J \ar[r]^\varphi \ar@{^(->}[u]^{\cdot \epsilon}& M.
}
$$

First we will extend  $\varphi$ to a derivation $$\delta' : A[\epsilon J][x,y] \rightarrow M.$$   To do this, we will invoke the identifications
$$
\operatorname{Hom}_A(J,M)=\operatorname{Hom}_{A\operatorname{-alg},\operatorname{Id}_A}(A[\epsilon J],A[\epsilon M]) = \operatorname{Der}_A(A[\epsilon J],M).
$$
On the left, we take a homomorphism $\varphi:J\to M$ and send it to the ring homomorphism $\operatorname{Id}_A\oplus \varphi:A\oplus \epsilon J\to A\oplus \epsilon M$.  The group in the middle consists of the $A$-algebra homomorphisms that are the identity on the first term.  On the right, we just project to get a derivation.   Similarly, we have indentifications 
$$
\operatorname{Hom}_{A\operatorname{-alg},\operatorname{Id}_A}(A[\epsilon J][x,y],A[\epsilon M]) = \operatorname{Der}_A(A[\epsilon J][x,y],M).
$$
In summary, the homomorphism $\varphi:J\to M$ induces a ring homomorphism $\tilde \varphi:A[\epsilon J]\to A[\epsilon M]$.  We can extend this to  $\hat \varphi:A[\epsilon J][x,y]\to A[\epsilon M]$ by sending $x$ and $y$ to any element of $M$.  This then gives a derivation $\delta'$ as desired.

 The choices of such extensions $\delta'$ are a torsor under $\operatorname{Der}_A(A[x,y], M)$, and we will adjust $\delta'$ by such a derivation so that it descends to a derivation $B' \rightarrow M$.  The obstruction to descending $\delta'$ to $B'$ is the homomorphism $\delta' \big|_{I'}$, so we want to find a derivation $\delta \in \operatorname{Der}(A[x,y], M)$ such that the restriction of $\delta$ to $I'$ (via the composition $I' \subseteq A'[x,y] \rightarrow A[x,y]$) agrees with $\delta'$.

To execute this, let us take a fixed extension $\delta'$.   Then  $\delta'(I'J) \subseteq I' \delta(J) + J \delta(I') = 0$ because $M$ is a $B$-module, and therefore $I' M = JM = 0$.  This implies $\delta \big|_{I'}$ descends to $I'/JI' = I$.  We also have $\delta'({I'}^2) = I' \delta'(I') = 0$, for the same reason, so $\delta'$ gives us a homomorphism $u : I/I^2 \rightarrow M$.

Now, we have an exact sequence
\begin{equation*}
0 \rightarrow I/I^2 \rightarrow B \mathop{\otimes}_{A[x,y]} \Omega_{A[x,y] / A} \rightarrow \Omega_{B/A} \rightarrow 0,
\end{equation*}
either by \cite[Thm.~D.2.7, p.~310]{sernesi} or a direct verification.  Since $M$ is injective, the homomorphism $u : I/I^2 \rightarrow M$ extends to a derivation $\delta : A[x,y] \rightarrow M$ (again we are using tensor-hom adjunction and the universal property of the sheaf of K\"ahler differentials).  Trivially extending  $\delta$ to $A'[x,y]$ we obtain a derivation  $\delta'-\delta$ on $A'[x,y]$.   We have that $\delta'-\delta =0$ on $I'$ essentially by construction.  Moreover, $\delta' - \delta $ agrees with $\delta' \big|_J = \varphi$ on $J$ because $\delta $ vanishes there.  Thus $\delta' - \delta $ descends to a derivation in $\operatorname{Der}_A(B,M)$ that agrees with $\varphi$ on $J$, as required.  We conclude that \eqref{E:omega-extension} is exact.

\vskip .2cm 
Conversely, given any extension of $\mathcal O_C$-modules
\begin{equation*}
0 \rightarrow \pi^\ast J \rightarrow \Omega' \rightarrow \Omega_{C/S} \rightarrow 0
\end{equation*}
we can define an extension of $C/S$ to $S[\epsilon J]$.  Namely, 
one may define $$\mathcal{O}_{C'} = \mathcal O_{C} \mathbin\times_{\Omega_{C/S}} \Omega',$$ 
where the map from $\mathcal O_C$ to $\Omega_{C/S}$ is the universal differential, $d$.  
This is clearly an $\mathcal O_C$-module, but 
 it is not a priori obvious that $\mathcal{O}_{C'}$ is equipped with a ring structure.  The ring structure can be constructed easily  by factoring $d$ as
\begin{equation*}
\mathcal{O}_C \xrightarrow{\mathrm{id} + \epsilon d} \mathcal{O}_C \!+\! \epsilon \Omega_{C/S} \longrightarrow \Omega_{C/S}
\end{equation*}
and recognizing that $$\mathcal{O}_{C'} \simeq \mathcal O_C \mathbin\times_{\mathcal O_C\!+\!\epsilon \Omega_{C/S}} (\mathcal O_C \!+\! \epsilon \Omega').$$  
Here the map on the right is
$$
\mathcal O_C+\epsilon \Omega'\to \mathcal O_C+\epsilon \Omega_{C/S}
$$
$$
a+\omega'\mapsto a+\tilde \omega'
$$
where we are using the given map $\Omega'\to \Omega_{C/S}$, $\omega'\mapsto \tilde \omega'$; this is the standard map taking $$\operatorname{Hom}_{\mathcal O_C}(\Omega',\Omega_{C/S})\to \operatorname{Hom}_{\mathcal O_C\text{-alg},\operatorname{Id}}(\mathcal O_C[\epsilon \Omega'],\mathcal O_C[\epsilon \Omega_{C/S}]).$$ 
This above description of $\mathcal O_{C'}$ is as a  fiber product of rings, hence is naturally equipped with a ring structure.

 The map 
$$C'\to S[\epsilon J]$$ is given topologically by $C\to S$, and at the level of sheaves by a map
$$
 \mathcal O_{S[\epsilon]}=\mathcal O_S\oplus \epsilon J \to \mathcal O_{C'}= \mathcal O_C \mathbin\times_{\mathcal O_C\!\oplus\!\epsilon \Omega_{C/S}} (\mathcal O_C \!\oplus\! \epsilon \Omega').
$$
We are given maps $\mathcal O_S\to \mathcal O_C$, and $J\to \Omega'$.
The map $\mathcal O_S \oplus \epsilon J\to \mathcal O_C$ is defined  by the given map on the first term, plus the zero map on the second term.  
The map $\mathcal O_{S}[\epsilon J]\to \mathcal O_C[\epsilon \Omega']$ is given by the map $\mathcal O_S[\epsilon J]\to \mathcal O_C[\epsilon \pi^*J]$, and then the natural map $\mathcal O_C[\epsilon \pi^*J]\to \mathcal O_C[\epsilon \Omega']$.  These maps agree on composition to $\mathcal O_C[\epsilon \Omega_{C/S}]$, and so define a morphism to the fibered product.  

One can check that these processes are inverses of one another, and so 
this yields an equivalence of categories between the tangent space of $\mathcal N$  at $C/S$ and the category of extensions of $\Omega_{C/S}$ by $\pi^*J$.

\begin{lem} \label{L:nodal-tangent}
Let $\pi : C \rightarrow S$ be an $S$-point of the stack $\mathcal N$ of all nodal curves.  Then

\begin{equation*}
T_{\mathcal{N}}(C/S) = \mathbf{Ext}(\Omega_{C/S}, \mathcal{O}_C) ,
\end{equation*}
where we write $\mathbf{Ext}(A,B)$ for the \emph{groupoid} of extensions of $B$ by $A$.
In particular,
\begin{align*}
T^{-1}_{\mathcal{N}}(C/S) &= \operatorname{Hom}_{\mathcal O_C}(\Omega_{C/S},\mathcal O_C)=\Gamma(C, T_{C/S}) \\
T^0_{\mathcal{N}}(C/S) &= \operatorname{Ext}^1_{\mathcal O_C}(\Omega_{C/S}, \mathcal{O}_C) .
\end{align*}
\end{lem}

Observe that the discussion above never made use of the assumption that $C$ be proper over $S$.  Thus $\operatorname{\underline{Ext}}^1(\Omega_{C/S}, \mathcal{O}_C)$ may be interpreted as the associated sheaf of the presheaf of isomorphism classes of local deformations of $C$.  This allows us to interpret the $5$-term exact sequence of the local-to-global spectral sequence for $\operatorname{Ext}(\Omega_{C/S}, \mathcal O_C)$ deformation theoretically:
\begin{equation*}
0 \rightarrow H^1(C, T_{C/S}) \rightarrow \operatorname{Ext}^1(\Omega_{C/S}, \mathcal{O}_C) \rightarrow \Gamma(C, \operatorname{\underline{Ext}}^1(\Omega_{C/S}, \mathcal{O}_C)) \rightarrow H^2(C, T_{C/S}).
\end{equation*}
If $S$ is affine then $H^2(C, T_{C/S}) = 0$ because $C$ has relative dimension $1$ over $S$.  The first term of the sequence may be interpreted, via a \v{C}ech calculation, as the set of isomorphism classes of locally trivial deformations of $C$.  Somewhat imprecisely, we have:
\begin{equation*}
0 \rightarrow (\text{loc.\ triv.\ defs.\ of $C$}) \rightarrow (\text{defs.\ of $C$}) \rightarrow (\text{defs.\ of nodes of $C$}) \rightarrow 0.
\end{equation*}
When $S$ is not affine, the final term $H^2(C, T_{C/S})$ may be interpreted as the obstruction to finding a deformation of the curve realizing specified deformations of the nodes.

\vskip .2 cm 
Now we turn to the more general problem of extending a curve over a general square-zero extension $S \subseteq S'$ with ideal $J$:
\begin{equation} \label{E:curve-def} \vcenter{\xymatrix{
C \ar[d]_\pi \ar@{-->}[r]  & C' \ar@{-->}[d] \\
S \ar[r] & S'
}} \end{equation}
We will make use of homogeneity:

\begin{lem}\label{L:NChmg}
	The CFG of all nodal curves (not necessarily proper!) is homogeneous, as are the substack of all proper nodal curves, the substack of curves of fixed genus, and the substack of canonically polarized curves.
\end{lem}

\begin{proof} 
	Let  $Q \subseteq Q'$ be an infinitesimal extension of schemes and $Q \rightarrow R$ an affine morphism.  Form the pushout $R'$ of $Q'$ and $R$ under $Q$ using Theorem \ref{T:Homog}.  Letting $\mathcal N$ denote the stack of all nodal curves, we construct the inverse to the map~\eqref{E:nodal-homog}.
	\begin{equation} \label{E:nodal-homog}
	\mathcal N(R') \rightarrow \mathcal N(Q') \mathop{\times}_{\mathcal N(Q)} \mathcal N(R)
	\end{equation}
	Suppose that $C$ is a nodal curve over $Q$, that $C'$ is an extension of $C$ to $Q'$, and that $C = D \mathop{\times}_R Q$ for nodal curve $D$ over $R$   and observe that $C \subseteq C'$ is an infinitesimal extension and $C \rightarrow D$ is affine, so we can also form the pushout $D'$ of $C'$ and $D$ under $C$ (Theorem \ref{T:Homog}).

	By Theorem \ref{T:flat-homog}, $D'$ is flat over $R'$ and $D' \mathop{\times}_{R'} R = D$, $D' \mathop{\times}_{R'} Q' = C'$, and $D' \mathop{\times}_{R'} Q = C$.  We argue that $D'$ is a nodal curve over $R'$.  Since the fibers of $D'$ over $R'$ are the same as the fibers of $D$ over $R$, the fibers of $D'$ over $R'$ are nodal curves.  We have already seen that $D'$ is flat over $R'$, so we therefore only need to check that $D'$ is locally of finite presentation over $R'$.
	
 One can verify easily that local generators and relations for $\mathcal{O}_D$ as an algebra over $\mathcal{O}_R$ lift to local generators and relations for $\mathcal{O}_{D'}$ as an algebra over $\mathcal{O}_{R'}$, implying it is locally of finite presentation as well.

	This proves the homogeneity of the stack of all nodal curves.  For the remaining statements, note that they are stable under infinitesimal deformation.
\end{proof}

Now we return to the problem of completing~\eqref{E:curve-def}:
\begin{enumerate}[(i)]
\item The problem of completing~\eqref{E:curve-def} can be solved locally in $C$.  Indeed, one may find a cover of $S$ by affine open subschemes $U = \operatorname{Spec} A$ and elements $t \in A$ such that $\pi^{-1} U$ has a cover by open affines $V$, each of which is \'etale over $\operatorname{Spec} B$, with $B = A[x,y] / (xy - t)$.  If $U' = \operatorname{Spec} A'$ is the open subset of $S'$ whose preimage in $S$ is $U$ then one may form $B' = A'[x,y] / (xy - t')$ where $t'$ is any lift of $t$ to $A'$.  Then $\operatorname{Spec} B'$ is a deformation of $\operatorname{Spec} B$ to $\operatorname{Spec} A'$ and this lifts uniquely to a deformation of $V$ by \cite[Thm.~(18.1.2)]{EGAIV4}.

We may thus select local deformations $V'$ for each $V$ in an open cover of $C$ and attempt to glue.

\item \label{B:coprodExamp}Observe that 
\begin{equation*}
\mathcal{O}_{S'} \mathop{\times}_{\mathcal{O}_S} \mathcal{O}_{S'} \simeq \mathcal{O}_{S'} + \epsilon J,
\end{equation*}
via the map $(f,g) \mapsto f + \epsilon(g-f)$,
so that $S' \mathop{\amalg}_{S} S' \simeq S' \mathop{\amalg}_S S[\epsilon J]$ and therefore by homogeneity, we have
\begin{equation*}
\mathcal{N}(S') \mathop{\times}_{\mathcal{N}(S)} \mathcal{N}(S') = \mathcal{N}(S' \mathop{\amalg}_S S') \simeq \mathcal{N}(S' \mathop{\amalg}_S S[\epsilon J]) = \mathcal{N}(S') \mathop{\times}_{\mathcal{N}(S)} \mathcal{N}(S[\epsilon J]) .
\end{equation*}
In other words, any two extensions of an open subset of $C$ to a nodal curve over $S'$ differ by a uniquely determined element of $\mathcal{N}(S[\epsilon J])$.
\item We have already seen that $\mathcal{N}(S[\epsilon J])$ may be identified with the category of extensions $\mathbf{Ext}(\Omega_{C/S}, \pi^\ast J)$.  Thus the first obstruction to gluing the deformations $U'$ comes from making sure that the isomorphism classes of the deformations $U'$ agree on overlaps.  Since the isomorphism classes form a torsor under $\underline{\operatorname{Ext}}^1(\Omega_{C/S}, \pi^\ast J)$, this obstruction lies in $H^1\bigl(C, \underline{\operatorname{Ext}}^1(\Omega_{C/S}, \pi^\ast J)\bigr)$.
\item One can arrange very easily for this obstruction to vanish by working locally in $S$.  Indeed, $\underline{\operatorname{Ext}}^1(\Omega_{C/S}, \pi^\ast J)$ is quasicoherent and is supported on the nodes of $C$, which is finite (and in particular affine) over $S$.  Therefore when $S$ is affine the first obstruction vanishes.
\item Now assuming that the first obstruction vanishes we can fix a compatible system of isomorphism classes of local deformations.  We must determine whether one can find genuine local deformations within those isomorphism classes in a compatible way.  For each open $U \subseteq C$ in a suitable cover we select a deformation $U'$ in the specified isomorphism class.  If $U_i$ and $U_j$ are two such open sets then write $U'_{ij}$ for the restriction of $U'_i$ to $U_i \cap U_j$.  Because we have chosen the isomorphism classes compatibly, $U'_{ij} \simeq U'_{ji}$, and we may select such an isomorphism $\varphi_{ij}$.  Over triple overlaps the cocycle condition $\varphi_{ki} \circ \varphi_{jk} \circ \varphi_{ij} = \mathrm{id}$ obstructs the gluing.  As the element $\varphi_{ki} \circ \varphi_{jk} \circ \varphi_{ij}$ lies in the automorphism group of $U'_{ijk}$, which is canonically identified with $\Gamma\bigl(U_{ijk}, \underline{\operatorname{Hom}}(\Omega_{C/S},\pi^\ast J)\bigr)$ we obtain a \v{C}ech $2$-cocycle in $H^2\bigl(C, \underline{\operatorname{Hom}}(\Omega_{C/S}, \pi^\ast J)\bigr)$.  Once again, this obstruction vanishes when $S$ is affine, this time because $C$ has relative dimension $1$ over $S$ and the coefficients are taken in a quasicoherent sheaf.
\end{enumerate}

\begin{cor} \label{C:NCsm}
	The stack of all nodal curves satisfies the formal criterion for smoothness.
\end{cor}
\begin{proof}
	We have just seen that all obstructions to infinitesimal deformation vanish over an affine base, which is precisely the formal criterion for smoothness.
\end{proof}

\begin{cor}
	Any algebraic substack of the stack of all nodal curves that is locally of finite presentation and stable under infinitesimal extensions is smooth.  In particular, the stack of all proper nodal curves is smooth and the open substack of all canonically polarized nodal curves is smooth.
\end{cor}

\begin{rem} \label{R:cc}
An observant reader will have noticed that the obstructions to~\eqref{E:curve-def} constructed above lie in the graded pieces of $\operatorname{Ext}^2(\Omega_{C/S}, \pi^\ast J)$ induced by the local-to-global spectral sequence.  This is not an accident:  There is a single obstruction in $\operatorname{Ext}^2(\Omega_{C/S}, \pi^\ast J)$ whose vanishing is equivalent to the existence of a solution to~\eqref{E:curve-def}.  We will discuss one way to obtain this obstruction in \S\ref{S:obstructions}.

In fact, one can dispense with the assumption that $C$ be a curve over $S$ and obtain an obstruction to deformation in $\operatorname{Ext}^2(\mathbf{L}_{C/S}, \pi^\ast J)$ where $\mathbf{L}_{C/S}$ is the relative cotangent complex of $C$ over $S$, constructed by Illusie~\cite[Cor.~2.1.3.3, Thm.~2.1.7]{Illusie-1}.  For a nodal curve, one has $\mathbf{L}_{C/S} = \Omega_{C/S}$.  In general, $\mathbf{L}_{C/S}$ is concentrated in nonpositive degrees and $\Omega_{C/S}$ is the $0$-th homology group.
\end{rem}

\subsubsection{Simultaneous deformation of curves and vector bundles}

Let $S$ be a scheme, $C$ a curve over $S$, and $E$ a vector bundle on $C$.  We consider the problem of extending $C$ and $E$ to the scheme $S[\epsilon J]$ where $J$ is a quasicoherent sheaf of $\mathcal{O}_S$-modules.  We begin by defining a sheaf of modules $\Upsilon_C(E)$ on $C$ to play a role analogous to the one played by the module of differentials when we studied deformations of curves.

For each $\mathcal{O}_C$-module $F$, define $\Phi(F)$ to be the set of pairs $(\delta, \varphi)$ where $\delta : \mathcal{O}_C \rightarrow F$ is an $\mathcal O_S$-derivation and $\varphi : E \rightarrow F \otimes E$ is what we will call a \emph{$\delta$-connection}.  That is, for any local sections $f \in \mathcal{O}_C$ and $x \in E$, we have
\begin{equation*}
\varphi(fx) = \delta(f) \mathop{\otimes} x + f \varphi(x) .
\end{equation*}
Then $\Phi(F)$ is naturally a covariant functor of $F$.
 There is a natural $\Gamma(C, \mathcal O_C)$-module structure on $\Phi(F)$, in which $\lambda \in \Gamma(C, \mathcal O_C)$ acts by
\begin{equation*}
\lambda . (\delta, \varphi) = (\lambda \delta, \lambda \varphi) .
\end{equation*}
There is also an evident exact sequence~\eqref{E:upsilon-seq}.
\begin{equation} \label{E:upsilon-seq}
0 \rightarrow \operatorname{Hom}(E, F \otimes E) \rightarrow \Phi(F) \rightarrow \operatorname{Der}_{\mathcal O_S}(\mathcal O_C, F)
\end{equation}
Here the map $\operatorname{Hom}(E, F \otimes E) \rightarrow \Phi(F)$ sends $\varphi$ to $(0, \varphi)$ and the map $\Phi(F) \rightarrow \operatorname{Der}_{\mathcal O_S}(\mathcal O_C, F)$ sends $(\delta, \varphi)$ to $\delta$.

For each open set $U\subseteq C$,  setting   $\underline{\Phi}(F)(U):=\Gamma(U, \underline{\Phi}(F)) = \Phi(F \big|_U)$, we obtain a covariant functor
$$
\underline{\Phi}:(\mathcal O_C\textsf {-mod})\to (\mathcal O_C\textsf {-mod})
$$
with $\Phi(F)=\underline{\Phi}(F)(C)$, and an exact sequence 
$$
0\to \underline {\operatorname{Hom}}_{\mathcal O_C}(E^\vee\otimes E,-)\to \underline {\Phi}(-)\to \operatorname{Der}_{\mathcal O_S}(\mathcal O_C,-)=\operatorname{Hom}_{\mathcal O_C}(\Omega_{C/S},-),
$$
in the sense that it is exact when applied to any $\mathcal O_C$-module.  In fact, we claim that this is surjective on the right, and that $\underline \Phi$ (and hence $\Phi$) is representable; i.e.,   there is a  sheaf of $\mathcal{O}_C$-modules $\Upsilon_{C/S}(E)$  such that $\underline \Phi(-)=\underline{\operatorname{Hom}}_{\mathcal O_C}(\Upsilon_{C/S}(E),-)$, and $  \Phi(-)= \operatorname{Hom}_{\mathcal O_C}(\Upsilon_{C/S}(E),-)$.

\begin{lem} \label{L:Upsilon-extension}
The functors $\Phi$ and $\underline \Phi$ defined above are both representable by the same sheaf of $\mathcal{O}_C$-modules $\Upsilon_{C/S}(E)$ fitting into a short exact sequence~\eqref{E:Upsilon-extension} inducing~\eqref{E:upsilon-seq}:
\begin{equation} \label{E:Upsilon-extension}
0 \rightarrow \Omega_{C/S} \rightarrow \Upsilon_{C/S}(E) \rightarrow \underline{\operatorname{End}}(E) \rightarrow 0
\end{equation}
\end{lem}

\begin{proof}
	For the representability we give an explicit construction, although it is possible to obtain the same result more quickly by an application of  the adjoint functor theorem.  
As mentioned above, to prove that $\Phi$ is representable it is equivalent to prove that $\underline{\Phi}$ is representable, for which we can work locally in $C$.  We can therefore assume $E = \mathcal{O}_C^{\oplus n}$.  But then if $\delta : \mathcal{O}_C \rightarrow F$ is any derivation, it itself gives a $\delta$-connection by $\delta^{\times n}(x_1, \ldots, x_n) = (\delta(x_1) \otimes 1, \ldots, \delta(x_n) \otimes 1)$.  This gives a natural bijection between $\Phi(F)$ and $\operatorname{Der}_{\mathcal O_S}(\mathcal{O}_C, F) \times \operatorname{Hom}(E, F \otimes E)$ sending $(\delta, \varphi)$ to $(\delta, \varphi - \delta^{\times n})$.
  But $\operatorname{Der}_{\mathcal O_S}(\mathcal{O}_C, -)$ is representable by $\Omega_{C/S}$ and $\operatorname{Hom}(E, (-) \otimes E)$ is representable by $E \otimes E^\vee$.  Thus $\underline \Phi$ is representable.  To check that the sequence  \eqref{E:Upsilon-extension} in the lemma is exact, it is enough to check locally, which we have just done.
\end{proof}

\begin{rem} \label{R:inf-aut}
	We can view $\operatorname{Der}_{\mathcal O_S}(\mathcal{O}_C, \pi^*J) = \operatorname{Hom}(\Omega_{C/S}, \pi^*J)$ as the group of automorphisms of $C[\epsilon \pi^*J]$ that act as the identity on $C$ and on $\pi^*J$ (the group of `infinitesimal automorphisms').  Similarly, we can view the group  $\Phi(\pi^*J) = \operatorname{Hom}(\Upsilon_{C/S}(E), \pi^*J)$ as the group of automorphisms of the pair $(C[\epsilon \pi^*J], E + \epsilon J \otimes \pi^*E)$, the trivial square-zero extension of $(C,E)$ by $J$, that act trivially on $C$, $J$, and $E$.  

	In many deformation problems, deformations and obstructions are classified by analyzing the ways to glue, and the obstruction to gluing, deformations along infinitesimal automorphisms.  Thus computing the infinitesimal automorphism group of the objects under consideration goes a long way toward understand how the object can deform.  Indeed, it was by calculating the infinitesimal automorphisms of a deformation of $(C,E)$ that we arrived at the definition of $\Phi$ in the first place.
\end{rem}

Now consider an extension $(C', E')$ of $(C, E)$ to $S[\epsilon J]$.  Let $\pi : C \rightarrow S$ be the projection.  We have a canonical projection
\begin{equation} \label{E:upsilon-projection}
\Upsilon_{C'/S}(E') \rightarrow \Upsilon_{C/S}(E) 
\end{equation}
and it is easy to see that this is surjective.  This induces an exact sequence
\begin{equation} \label{E:upsilon-extension}
0 \rightarrow \pi^\ast J \rightarrow \mathcal{O}_C \mathop{\otimes}_{\mathcal{O}_{C'}} \Upsilon_{C'/S}(E') \rightarrow \Upsilon_{C/S}(E) \rightarrow 0,
\end{equation}
where the morphism $\pi^\ast J \rightarrow \mathcal{O}_C \mathop{\otimes}_{\mathcal{O}_{C'}} \Upsilon_{C'/S}(E')$ is the composition
\begin{equation} \label{E:ideal-inclusion}
\pi^\ast J \rightarrow \mathcal{O}_C \mathop{\otimes}_{\mathcal{O}_{C'}} \Omega_{C'/S} \rightarrow \mathcal{O}_C \mathop{\otimes}_{\mathcal{O}_{C'}} \Upsilon_{C'/S}(E').
\end{equation}
As in the proof of the exactness of~\eqref{E:omega-extension}, the right exactness and exactness in the middle of~\eqref{E:upsilon-extension} is formal using the universal properties.  The left exactness follows easily from the exactness of \eqref{E:omega-extension} and \eqref{E:Upsilon-extension}.  Indeed, $\underline{\operatorname{End}}(E')$ is a flat $\mathcal{O}_{C'}$-module, so tensoring the short exact  sequence \eqref{E:Upsilon-extension} for $E'$ on $C'/S$ with $\mathcal O_{C}\otimes_{\mathcal O_C'}(-)$, we obtain the  short exact sequence
\begin{equation*}
0 \rightarrow \mathcal{O}_C \mathop{\otimes}_{\mathcal{O}_{C'}} \Omega_{C'/S} \rightarrow \mathcal{O}_C \mathop{\otimes}_{\mathcal{O}_{C'}} \Upsilon_{C'/S}(E') \rightarrow \mathcal{O}_C \mathop{\otimes}_{\mathcal{O}_{C'}} \underline{\operatorname{End}}(E') \rightarrow 0.
\end{equation*}
 Therefore using also the exactness of \eqref{E:omega-extension},   the morphisms in~\eqref{E:ideal-inclusion} are both injections.
The discussion shows that each extension $(C',E')$ of $(C,E)$ induces an extension of $\Upsilon_{C/S}(E)$ by $\pi^\ast J$.  

\vskip .2cm 
Conversely, given such an extension,
\begin{equation} \label{E:upsilon-extension-2}
0 \rightarrow \pi^\ast J \rightarrow \Upsilon' \rightarrow \Upsilon_{C/S}(E) \rightarrow 0
\end{equation}
we can recover $\mathcal{O}_{C'}$ as $\delta^{-1} \Upsilon' = \mathcal O_{C'} \mathbin\times_{\Upsilon} \Upsilon'$, where $\delta : \mathcal{O}_C \rightarrow \Upsilon_{C/S}(E)$ is the universal derivation; i.e., derivation $\delta$ in the universal pair $(\delta,\varphi)$ obtained from the identity map  under the identification $\operatorname{Hom}(\Upsilon_{C/S}(E),\Upsilon_{C/S}(E))=\Phi(\Upsilon_{C/S}(E))$.  
 Likewise, we can recover $E'$ by first tensoring~\eqref{E:upsilon-extension-2} by $E$ and then pulling back via the universal $\delta$-connection $\varphi$:
\begin{equation*} \xymatrix@R=15pt{
0 \ar[r] & \pi^\ast J \otimes E \ar@{-->}[r] \ar@{=}[d] & E' \ar@{-->}[r] \ar@{-->}[d] & E \ar[r] \ar[d]^{\varphi} & 0 \\
0 \ar[r] & \pi^\ast J \otimes E \ar[r] & \Upsilon' \otimes E \ar[r] & \Upsilon_{C/S}(E) \otimes E \ar[r] & 0
} \end{equation*}

Summarizing the discussion above:
\begin{lem} \label{L:def-curve-vb}
Let $\mathcal{G}$  be the stack of pairs $(C,E)$ where $C$ is a nodal curve and $E$ is a vector bundle on $C$.  Suppose that $(C,E)$ is an $S$-point of $\mathcal{F}$.  Then
\begin{equation*}
T_{\mathcal{F}}(C,E) = \mathbf{Ext}\bigl(\Upsilon_{C/S}(E), \mathcal{O}_C\bigr) .
\end{equation*}
In particular,
\begin{align*}
T^{-1}_{\mathcal{F}}(C,E) &= \operatorname{Hom}\bigl(\Upsilon_{C/S}(E), \mathcal{O}_C\bigr) \\
T^0_{\mathcal{F}}(C,E) &= \operatorname{Ext}^1\bigl(\Upsilon_{C/S}(E), \mathcal{O}_C\bigr) .
\end{align*}
\end{lem}

Note that for a nodal curve $C/S$ (or more generally a curve with locally planar singularities), applying $\operatorname{Hom}(-,\mathcal O_C)$ to Lemma \ref{L:Upsilon-extension}, and utilizing Lemmas \ref{L:def-curve-vb}, \ref{L:nodal-tangent}, and Corollary \ref{C:TFib}, we  obtain a long exact sequence
$$
\xymatrix@R=4em @C=1em {
    0 \ar[r] & T^{-1}_{\mathsf {Fib}_{C/S}}(E) \ar[r] & T^{-1}_{\mathcal F}(C,E) \ar[r] & T^{-1}_{\mathcal N}(C/S)
               \ar@{->} `r/8pt[d] `/10pt[l] `^dl[ll]|{} `^r/3pt[dll] [dll] \\
                & T^{0}_{\mathsf {Fib}_{C/S}}(E) \ar[r] & T^{0}_{\mathcal F}(C,E) \ar[r] & T^{0}_{\mathcal N}(C/S)
                \ar@{->} `r/8pt[d] `/10pt[l] `^dl[ll]|{} `^r/3pt[dll] [dll] \\
                                & \operatorname{Ext}^2(\underline{\operatorname{End}}(E),\mathcal O_C) \ar[r] & \operatorname{Ext}^2(\Upsilon_{C/S}(E),\mathcal O_C) \ar[r] & \operatorname{Ext}^2(\Omega_{C/S},\mathcal O_C)
               \ar@{->}[r]&\cdots \\
    }
    $$
Note that if $S$ is affine, then the last row above is $0$.  
\begin{rem}
	There is an explicit treatment of infinitesimal deformations of pairs $(C,E)$ where $C$ is a curve and $E$ is a \emph{line bundle} in \cite[Thms.~3.1, 4.6]{Wang}, with infinitesimal automorphisms, deformations, and obstructions lying respectively  in $\operatorname{Ext}^i(\mathscr{P}^1_{X/S}(E), E)$, $i=0,1,2$, where $\mathscr{P}^1(E)$ is the sheaf of \emph{principal parts} of $E$, and fits in an exact sequence $0\to \Omega_{C/S}(E)\to  \mathscr{P}^1_{X/S}(E) \to E\to 0$.  
	  The same formulas do not give a deformation--obstruction theory for the pair $(C,E)$ when the rank of $E$ is larger than~$1$.
\end{rem}

\begin{rem}
	Illusie provides a $2$-step obstruction theory for deformations of the pair $(C,E)$.  The primary obstruction is to deforming $C$, which lies in this case in $\operatorname{Ext}^1(\Omega_{C/S}, \pi^\ast J)$, and in general in $\operatorname{Ext}^2(\mathbf{L}_{C/S}, \pi^\ast J)$, as we have discussed (see Remark \ref{R:cc}).  A secondary obstruction in $\operatorname{Ext}^2(\mathscr{P}^1(E), E)$ obstructs the existence of a deformation of $E$ that is compatible with a fixed deformation $C'$ of $C$ \cite[Prop.~IV.3.1.5]{Illusie-1}.  This latter obstruction may be constructed as the cup product of the class $[C'] \in \operatorname{Ext}^1(\mathbf{L}_{C/S}, \pi^\ast J)$ and the \emph{Atiyah class} \cite[\S\S{}IV.2.3.6--7]{Illusie-1}.

	One can arrive at this $2$-step obstruction theory\footnote{We have only verified the obstruction groups coincide, not the obstruction classes.} from the extension~\eqref{E:Upsilon-extension}, which induces an exact sequence:
	\begin{equation*}
		\operatorname{Ext}^2(\underline{\operatorname{End}}(E), \pi^\ast J) \rightarrow \operatorname{Ext}^2(\Upsilon_{C/S}(E), \pi^\ast J) \rightarrow \operatorname{Ext}^2(\Omega_{C/S}, \pi^\ast J)
	\end{equation*}
	The first term may of course be identified canonically with $\operatorname{Ext}^2(E, E \otimes \pi^\ast J)$ since $E$ is flat.
\end{rem}

For the sake of completeness, we also observe homogeneity:

\begin{lem} \label{L:CVBhmg}
	The stack of pairs $(C,E)$ where $C$ is a curve and $E$ is a vector bundle on $E$ is homogeneous.
\end{lem}
\begin{proof}
	Let $\mathcal X$ denote the stack of pairs $(C,E)$ as in the statement of the lemma and let $\mathcal N$ be the stack of nodal curves.  The forgetful map $\mathcal X \rightarrow \mathcal N$ is relatively homogeneous by Lemma \ref{L:Fib-hmg} and $\mathcal N$ is homogeneous by Lemma \ref{L:NChmg} so $\mathcal X$ is homogeneous by Lemma \ref{L:RelHomog}.
\end{proof}

\subsubsection{Simultaneous deformation of curves, vector bundles, and morphisms of vector bundles}
\label{S:CVBm}

We will consider a curve $C$ over $S$, vector bundles $E$ and $F$ on $C$, and a homomorphism $\sigma : E \rightarrow F$.  We ask in how many ways these data can be extended to $C'$, $E'$, $F'$, and $\sigma'$ over $S[\epsilon J]$ where $J$ is a quasicoherent sheaf on $S$.  This time, deformations will be controlled by a complex rather than by a module.

Imitating the last section we can construct a quasicoherent sheaf $\Upsilon_{C/S}(E,F)$
 controlling simultaneous deformations of the two vector bundles $E$ and $F$.  This will be the universal example of a quasicoherent $\mathcal O_C$-module, equipped with a derivation
\begin{gather*}
\delta : \mathcal{O}_C \rightarrow \Upsilon_{C/S}(E,F) \\
\end{gather*}
and $\delta$-connections,
\begin{gather*}
\varphi_E : E \rightarrow \Upsilon_{C/S}(E,F) \otimes E \\
\varphi_F : F \rightarrow \Upsilon_{C/S}(E,F) \otimes F.
\end{gather*}
As it would be similar to the proof of Lemma \ref{L:Upsilon-extension}, we will omit an explicit verification of the existence of such a universal object, as well as the construction of the natural exact sequence:
\begin{equation*}
0 \rightarrow \Omega_{C/S} \rightarrow \Upsilon_{C/S}(E, F) \rightarrow \underline{\operatorname{End}}(E) \times \underline{\operatorname{End}}(F) \rightarrow 0
\end{equation*}

Before studying the question of deforming a homomorphism $\sigma : E \rightarrow F$, we analyze how we can recover the deformation $\underline{\operatorname{Hom}}(E',F')$ of $\underline{\operatorname{Hom}}(E,F)$ from deformations of $E$ and $F$, encoded as an extension of $\Upsilon_{C/S}(E,F)$ by $\pi^\ast J$. 

If $\sigma : E \rightarrow F$ is an homomorphism of vector bundles we obtain an element
\begin{gather*}
[\sigma, \varphi] \in \Upsilon_{C/S}(E,F) \otimes \underline{\operatorname{Hom}}(E,F) \\
[\sigma, \varphi] := (\mathrm{id}_{\Upsilon_{C/S}(E,F)} \otimes \sigma) \circ \varphi_E - \varphi_F \circ \sigma : E \rightarrow \Upsilon_{C/S}(E,F) \otimes F .
\end{gather*}
As $[\sigma, \varphi]$ depends linearly on $\sigma$, this induces a homomorphism of $\mathcal{O}_C$-modules:
\begin{equation*}
\varphi_{\underline{\operatorname{Hom}}(E,F)} : \underline{\operatorname{Hom}}(E,F) \rightarrow \Upsilon_{C/S}(E,F) \otimes \underline{\operatorname{Hom}}(E,F)
\end{equation*}

\begin{lem}
Suppose that $\Upsilon' \in \mathbf{Ext}\bigl(\Upsilon_{C/S}(E,F), \pi^\ast J)$ corresponds to extensions $C'$ of $C$, $E'$ of $E$, and $F'$ of $F$.  Then there is a cartesian square of sheaves of abelian groups: 
\begin{equation*} \xymatrix{
		\underline{\operatorname{Hom}}(E',F') \ar[r] \ar[d]_{\varphi_{\underline{\operatorname{Hom}}(E',F')}} & \underline{\operatorname{Hom}}(E,F) \ar[d]^{\varphi_{\underline{\operatorname{Hom}}(E,F)}} \\
		\Upsilon' \otimes \underline{\operatorname{Hom}}(E,F) \ar[r] & \Upsilon \otimes \underline{\operatorname{Hom}}(E,F)
} \end{equation*}
\end{lem}
\begin{proof}
	For the sake of readability, we will write 
	\begin{align*}
		\Upsilon & = \Upsilon_{C/S}(E,F) & H & = \underline{\operatorname{Hom}}_{\mathcal{O}_C}(E,F) & H' & = \underline{\operatorname{Hom}}_{\mathcal{O}_{C'}}(E',F') & \varphi & = \varphi_H 
	\end{align*}
	We also observe that we can identify $\Upsilon' = \mathcal{O}_{C} \mathop{\otimes}_{\mathcal{O}_{C'}} \Upsilon_{C'/S}(E',F')$
	 by Lemma \ref{L:def-curve-vb} (which allows us to identify the exact sequences~\eqref{E:upsilon-extension} and~\eqref{E:upsilon-extension-2}) and then write $\varphi'$ for the composition of $\varphi_{H'}$ and reduction modulo $\pi^\ast J$:
	\begin{equation*}
		 \varphi' : H' \xrightarrow{\varphi_{H'}} \Upsilon_{C'/S}(E',F') \otimes_{\mathcal{O}_{C'}} H' \rightarrow \Upsilon' \otimes_{\mathcal{O}_C} H
	\end{equation*}
	Consider the following commutative squares:
	\begin{equation*} \xymatrix{
			\mathcal{O}_{C'} \ar[r] \ar[d]_{\rho + \epsilon \delta'} & \mathcal{O}_C \ar[d]^{\mathrm{id} + \epsilon \delta} \\
			\mathcal{O}_C + \epsilon \Upsilon' \ar[r] & \mathcal{O}_C + \epsilon \Upsilon
	} \hskip1cm \xymatrix{
		H' \ar[r] \ar[d]_{\rho + \epsilon \varphi} & H \ar[d]^{\mathrm{id} + \epsilon \varphi} \\
		H + \epsilon (\Upsilon' \otimes H) \ar[r] & H + \epsilon (\Upsilon \otimes H)
	} \end{equation*}
We have written $\rho$ for reduction modulo $\pi^\ast J$ and $\delta'$ for the composition
\begin{equation*}
\mathcal O_{C'} \rightarrow \Upsilon_{C'/S}(E',F') \rightarrow \mathcal O_C \otimes_{\mathcal O_{C'}} \Upsilon_{C'/S}(E',F') = \Upsilon' .
\end{equation*}
We have already seen that the square on the left is cartesian in  the discussion of \eqref{E:upsilon-extension-2} and Lemma \ref{L:def-curve-vb}.  The square on the right is cartesian as well, since its horizontal arrows are surjective and we can identify the kernels of both horizontal arrows with $\pi^\ast J \otimes H$:  on the top we have applied $\mathop{\otimes}_{\mathcal{O}_{C'}} H'$ to the exact sequence
	\begin{equation*}
		0 \rightarrow \pi^\ast J \rightarrow \mathcal{O}_{C'} \rightarrow \mathcal{O}_C \rightarrow 0
	\end{equation*}
	noting that $H'$ is flat over $\mathcal{O}_{C'}$, and on the bottom we have applied $\mathop{\otimes}_{\mathcal{O}_C} H$ to the exact sequence~\eqref{E:upsilon-extension-2}, this time noting that $H$ is flat over $\mathcal{O}_C$.  Each entry in the square on the right is a module under the ring in the corresponding entry of the square on the left and the arrows in the square on the right are homomorphisms with respect to the arrows in the square on the left.  It follows that the module structure on $H'$ is induced from the fiber product.  The lemma now follows, since the square
	\begin{equation*} \xymatrix{
			H + \epsilon(\Upsilon' \otimes H) \ar[r] \ar[d] & H + \epsilon (\Upsilon \otimes H) \ar[d] \\
			\Upsilon' \otimes H \ar[r] & \Upsilon \otimes H
	} \end{equation*}
	is cartesian.
\end{proof}

If we fix $\sigma : E \rightarrow F$, then $\varphi_{\underline{\operatorname{Hom}}(E,F)}(\sigma) \in \Upsilon_{C/S}(E,F) \otimes \underline{\operatorname{Hom}}(E,F)$ induces a linear map:
\begin{equation} \label{E:def-curve-vb-sect-complex}
\underline{\operatorname{Hom}}(E,F)^\vee \rightarrow \Upsilon_{C/S}(E,F) \ \ \ \ \ \ \ \ \ :\Upsilon^\bullet_{C/S}(E,F,\sigma)
\end{equation}
via evaluation against $\underline{\operatorname{Hom}}(E,F)$.  
As indicated  write $\Upsilon^\bullet_{C/S}(E,F,\sigma)$ for the complex~\eqref{E:def-curve-vb-sect-complex}, concentrated in degrees $[-1,0]$.

\begin{rem}
As with $\Upsilon_{C/S}(E)$, one may arrive at the definition of $\Upsilon^\bullet_{C/S}(E,F,\sigma)$ by contemplating, with a bit of care, the automorphism group of the trivial extension of $(C,E,F,\sigma)$ to $S[\epsilon J]$ (cf.\ Remark \ref{R:inf-aut}).
\end{rem}

In order to state the next lemma, we recall the category  $\mathbf{Ext}(K^\bullet, L)$ of extensions of a complex $K^{-1} \rightarrow K^0$ by a module $L$ is, by definition, the category of commutative diagrams
\begin{equation*} \xymatrix{
& & & K^{-1} \ar[dl] \ar[d] \\
0 \ar[r] & L \ar[r] & M \ar[r] & K^0  \ar[r] & 0
} \end{equation*}
in which the bottom row is exact~\cite[\S{}VII.3]{sga7-1}.

\begin{lem}\label{L:C-EF-TanAut}
Let $\mathcal F$ be the stack of quadruples $(C,E,F,\sigma)$ where $C$ is a nodal curve, $E$ and $F$ are vector bundles on $C$, and $\sigma : E \rightarrow F$ is a morphism of vector bundles.  At an $S$-point $(C,E,F,\sigma)$, we have
\begin{equation*}
T_{\mathcal{F}}(C,E,F,\sigma) = \mathbf{Ext}\bigl(\Upsilon^\bullet_{C/S}(E,F,\sigma), \mathcal{O}_C\bigr)
\end{equation*}
and therefore
\begin{gather*}
T_{\mathcal{F}}^{-1}(C,E,F,\sigma) = \operatorname{Ext}^0\bigl(\Upsilon^\bullet_{C/S}(E,F,\sigma), \mathcal{O}_C\bigr) \\
T_{\mathcal{F}}^0(C,E,F,\sigma) = \operatorname{Ext}^1\bigl(\Upsilon^\bullet_{C/S}(E,F,\sigma), \mathcal{O}_C\bigr) .
\end{gather*}
\end{lem}
\begin{proof}
As in the previous sections we will actually prove the analogous statement about extensions to $S[\epsilon J]$.  Suppose that $(C',E',F',\sigma')$ is an extension of $(C,E,F,\sigma)$ to $S[\epsilon J]$.  We have a commutative diagram:
\begin{equation*} \xymatrix{
& & \mathcal{O}_C \mathop{\otimes}_{\mathcal{O}_{C'}} \underline{\operatorname{Hom}}(E',F')^\vee \ar[r]^<>(0.5){\sim} \ar[d] & \underline{\operatorname{Hom}}(E,F)^\vee \ar[d] \\
0 \ar[r] & \pi^\ast J \ar[r] & \mathcal{O}_C \mathop{\otimes}_{\mathcal{O}_{C'}} \Upsilon_{C'/S}(E',F') \ar[r] & \Upsilon_{C/S}(E,F) \ar[r] & 0
} \end{equation*}
The bottom row is exact, as in the last section.  As the upper horizontal arrow is an isomorphism, we obtain an element of $\mathbf{Ext}\bigl(\Upsilon^\bullet_{C/S}(E,F,\sigma), \pi^\ast J\bigr)$. 

Conversely, suppose we are given an extension:
\begin{equation} \label{E:complex-extension} \vcenter{\xymatrix{
& & & \underline{\operatorname{Hom}}(E,F)^\vee \ar[dl]_\beta \ar[d] \\
0 \ar[r] & \pi^\ast J \ar[r] & \Upsilon' \ar[r] & \Upsilon_{C/S}(E,F) \ar[r] & 0 
}} \end{equation}
We may construct 
\begin{align*}
\mathcal{O}_{C'} & = \delta^{-1} \Upsilon' = \mathcal O_C \mathbin\times_{\Upsilon} \Upsilon' \\
E' & = \varphi_E^{-1} (\Upsilon' \otimes E) = E \mathbin\times_{\Upsilon \otimes E} (\Upsilon' \otimes E) \\
F' & = \varphi_F^{-1} (\Upsilon' \otimes F)  = F \mathbin\times_{\Upsilon \otimes F} (\Upsilon' \otimes F) \\
\end{align*}
To get $\sigma' \in \operatorname{Hom}(E',F')$, regard $\beta$ as an element of $\Upsilon' \otimes \underline{\operatorname{Hom}}(E,F)$.  Then the image of $\beta$ in $\Upsilon_{C/S}(E,F) \otimes \underline{\operatorname{Hom}}(E,F)$ is $[\sigma, \varphi]$, by the commutativity of~\eqref{E:complex-extension}.  Therefore 
$\sigma' = (\sigma, \beta)$ defines an element of
\begin{equation*}
\underline{\operatorname{Hom}}(E',F') = \underline{\operatorname{Hom}}(E,F) \mathop\times_{\Upsilon \otimes \underline{\operatorname{Hom}}(E,F)} \Upsilon' \otimes \underline{\operatorname{Hom}}(E,F).
\end{equation*}
We leave the verification that these constructions are mutually inverse to the reader.
\end{proof}

\begin{lem} \label{L:CVBMhmg}
	Let $\mathcal F$ be the stack of quadruples $(C,E,F,\sigma)$ where $C$ is a nodal curve, $E$ and $F$ are vector bundles on $C$, and $\sigma : E \rightarrow F$ is a homomorphism of vector bundles.  Then $\mathcal F$ is homogeneous.
\end{lem}
\begin{proof}
	Let $\mathcal N$ be the stack of nodal curves.  The forgetful map $\mathcal F \rightarrow \mathcal N$ is relatively homogeneous by Corollary \ref{C:VBMapHomog} and $\mathcal N$ is homogeneous by Lemma \ref{L:NChmg}.  Therefore $\mathcal F$ is homogeneous by Lemma \ref{L:RelHomog}.
\end{proof}

\subsubsection{Deformations of Higgs bundles}

We may apply the method of the previous section to study deformations of a Higgs bundle $\phi : E \rightarrow E \otimes \omega_{C/S}$ on a nodal curve $C$ over $S$.  We find that the deformation theory of $(C,E,\phi)$ is controlled by the complex $\Upsilon_{C/S}(E,\sigma)$, in degrees $[-1,0]$:
\begin{equation*}
\underline{\operatorname{Hom}}(E, E \otimes \omega_{C/S})^\vee \rightarrow \Upsilon_{C/S}(E)
\end{equation*}

\begin{lem}
Let $(C,E,\phi)$ be a nodal curve of genus $g$ over $S$ equipped with a Higgs bundle.  Then
\begin{equation*}
T_{\mathcal{HS}h_{\overline{M}_g}}(C,E,\phi) = \mathbf{Ext}\bigl(\Upsilon_{C/S}(E,\sigma), \mathcal{O}_C\bigr) 
\end{equation*}
and
\begin{gather*}
T^{-1}_{\mathcal{HS}h_{\overline{M}_g}}(C,E,\phi) = \operatorname{Ext}^0\bigl(\Upsilon_{C/S}(E,\sigma), \mathcal{O}_C\bigr)  \\
T^0_{\mathcal{HS}h_{\overline{M}_g}}(C,E,\phi) = \operatorname{Ext}^1\bigl(\Upsilon_{C/S}(E,\sigma), \mathcal{O}_C\bigr) 
\end{gather*}
\end{lem}

\begin{lem} \label{L:Higgs-hmg}
	The stack of Higgs bundles on nodal (resp.\ stable) curves is homogeneous.
\end{lem}
\begin{proof}
	Let $\mathcal F$ be the stack of quadruples $(C,E,F,\sigma)$ where $C$ is a nodal curve over a tacit base $S$, $E$ and $F$ are vector bundles on $C$, and $\sigma : E \rightarrow F$ is a homomorphism.  Let $\mathcal E$ be the stack of pairs $(C,E)$ where $C$ is a nodal curve and $E$ is a vector bundle on $C$.  Let $\mathcal H$ be the stack of Higgs bundles on nodal curves.
	
	We have seen in Lemma \ref{L:CVBhmg} that $\mathcal E$ is homogeneous and in Lemma \ref{L:CVBMhmg} that $\mathcal F$ is homogeneous.  Then we have two projections $\mathcal F \rightarrow \mathcal E$, one sending $(C,E,F,\sigma)$ to $F$ and the latter sending $(C,E,F,\sigma)$ to $\omega_{C/S} \otimes E$.  We can identify $\mathcal H$ with the equalizer of these two maps, that is, with the fiber product $\mathcal F \mathbin\times_{\mathcal E \times \mathcal E} \mathcal E$, 
	hence is homogeneous by Lemma \ref{L:RelHomog}.

	Again by Lemma \ref{L:RelHomog}, we deduce that $\mathcal H \rightarrow \mathcal N$ is relatively homogeneous, where $\mathcal N$ denotes the stack of nodal curves.  By base change, the stack $\mathcal{HS}h_{\overline{\mathcal M}_g}$ of Higgs bundles on stable curves of genus $g$ is homogeneous over $\overline{\mathcal M}_g$ (again by Lemma \ref{L:RelHomog}).  But we know $\overline{\mathcal M}_g$ is homogeneous (Lemma \ref{L:NChmg}) so we deduce that $\mathcal{HS}h_{\overline{\mathcal M}_g}$ is homogeneous as well.
\end{proof}

\subsubsection{Deformations of principal bundles}
\label{S:def-tors}

Suppose that $X \subseteq X'$ is a square-zero extension of schemes by the ideal $J$ and $G'$ is a smooth algebraic group over $X'$.  Denote by $G$ the preimage of $X$ in $G'$.  Let $P$ be an \emph{\'etale} $G$-torsor over $X$, meaning that the map $G \mathbin\times_X P \rightarrow P \mathbin\times_X P$ is an isomorphism, and that $P$ covers $X$ in the \'etale topology (Definition~\ref{D:top}).  We would like to classify the extensions of $P$ to $G'$-torsors over $X'$.

\begin{rem}
An algebraic group can have \'etale torsors that are not torsors in the Zariski topology (see \cite{Serre-EFA}), so it is important that we have specified \'etale torsors here.  Since $G$ is smooth, working in the fppf or fpqc topology would not yield any more torsors:  flat descent implies that any $G$-torsor would be representable by a smooth algebraic space over $X$, hence would have a section \'etale-locally.
\end{rem}

\begin{rem}
This discussion of $G$-torsors generalizes the discussion of \S \ref{S:def-vect}:  given a vector bundle $E$ over $X$, we may take $G = \mathrm{GL}_n$ and let $P$ to be the $G$-torsor of isomorphisms between $E$ and $\mathcal O_X^n$.
\end{rem}

Before analyzing the deformation problem, we introduce the (underlying vector bundle of the) Lie algebra of $G$.  Let $e : X \rightarrow G$ and $e' : X' \rightarrow G'$ denote the identity sections.  
We set  $$\mathfrak g = e^\ast T_{G/X}$$
to be  the Lie algebra of $G$.
  We will also be interested in a closely related object.  To introduce it, 
note that the small \'etale sites of $X$ and $X'$ are equivalent (Lemma~\ref{L:et-inf-lift}), and that if $U$ is an \'etale scheme over $X$ we denote by $U'$ the unique (up to unique isomorphism) extension of $U$ to an \'etale scheme over $X'$.  

We write $G'_{\mathrm{\acute{e}t}}$ and $G_{\mathrm{\acute{e}t}}$ for the restrictions of $G'$ and $G$ to this common \'etale site, and we introduce   $\mathfrak g'$ for the kernel of the projection from $G'_{\mathrm{\acute{e}t}}$ to $G_{\mathrm{\acute{e}t}}$:
\begin{equation*}
0\to \mathfrak g'\to G'_{\mathrm{\acute{e}t}} \rightarrow G_{\mathrm{\acute{e}t}} \rightarrow 0
\end{equation*}
  In general $\mathfrak g'$ does not coincide with $\mathfrak g$, but the following lemma relates them:

\begin{lem}
Let $\mathfrak g$ and $\mathfrak g'$ be as above. Then $\mathfrak g' \simeq \mathfrak g \otimes J$ as sheaves of groups on the \'etale site of $X$.  In particular, $\mathfrak g'$ is a sheaf of commutative groups.
\end{lem}
\begin{proof}
We observe that $\mathfrak g'$ is canonically a torsor under $\mathfrak g \otimes J$.  Indeed, for an open subset $U\subseteq X$, denoting by $G'_e(U)$ the morphisms $U' \to G'$ whose restriction to $U$ are the identity section, then we have:
\begin{align*}
\mathfrak g'(U) \times \mathfrak g'(U) 
& = G'_e(U') \times G'_e(U') \\
& = G'_e(U' \mathbin\amalg_U U') &&  \text{by homogeneity of $G$} \\
& = G'_e(U' \mathbin\amalg_U U[\epsilon J_U]) & & \text{$\mathcal O_{U'} \mathbin\times_{\mathcal O_U} \mathcal O_{U'} \simeq \mathcal O_{U'} \mathbin\times_{\mathcal O_U} \mathcal O_U[\epsilon J]$, p.~\pageref{B:coprodExamp}}  \\
& = G_e(U') \times G_e(U[\epsilon J])  && \text{by homogeneity again}\\
& = G_e(U') \times (\mathfrak g \otimes J)(U) 
\end{align*}
since $\mathfrak g$ is the tangent space to $G$ at the origin.  But $\mathfrak g'$ also has a canonical section over $X$ coming from $e'$.  Therefore $\mathfrak g' \simeq \mathfrak g \otimes J$ as a $\mathfrak g \otimes J$-torsor.  In particular, $\mathfrak g'$ inherits a group structure from $\mathfrak g \otimes J$.

There is another group structure on $\mathfrak g'$ by virtue of its construction as a kernel.  These two group structures commute with one another, in the sense that the multiplication map of either group structure is a homomorphism with respect to the other, and the identity elements are the same.  By a standard argument, this means the two group structures coincide (and that both are commutative).
\end{proof}

As in the earlier sections, we make some observations about the local triviality of this problem:
\begin{enumerate}[(i)]
\item There is an \emph{\'etale} cover of $X'$ by maps $U' \rightarrow X'$ such that, setting $U = U' \mathop\times_{X'} X$, each $G_U$-torsor $P_U$ extends to a $G'_{U'}$-torsor on $U'$.
\item Any two extensions of $P$ to $X'$ are isomorphic on a suitable \'etale cover of $X'$.
\item \label{B:PrinGbund-iii} If $P'$ and $P''$ are two extensions of $P$ and $u, v : P' \rightarrow P''$ are two isomorphisms between them, then there is a unique map $h : P' \rightarrow G'$ over $X'$ such that $h.u = v$ (we are denoting the group  action map for the torsors with a dot, and the map $h.u$ is the composition of $h\times u : P' \times P' \rightarrow G' \times P''$ with the action map $G' \times P'' \rightarrow P''$).

As $u$ and $v$ both reduce to the identity map on $P$, the map $h$ must reduce to the constant map $P \rightarrow e$.  Therefore $h$ factors uniquely through $\mathfrak g'$:
$$
\xymatrix@R=1em@C=1em{
P' \ar[rr]^h \ar@{-->}[rd]&&G'\\
&\mathfrak g'  \ar@^{^(->}[ru]
}
$$

Furthermore, the induced map $h:P'\to \mathfrak g'$ is equivariant with respect to the natural action of $G'$ on $P'$ and on $\mathfrak g'$ (by conjugation).  Indeed, we have
\begin{align*}
h(g.x) . u(g.x) & = v(g.x) \\
h(g.x) g . u(x) & = g . v(x) = g h(x) . u(x)
\end{align*}
for all $x$ in $P'$.  It follows that $h(g.x) = g h(x) g^{-1}$ for all $g \in G'$.

Finally, $h:P'\to \mathfrak g'$ factors through the projection $P' \rightarrow P=P'/\mathfrak g'$.  This is because $\mathfrak g'$ is commutative, so $h(g.x) = g h(x) g^{-1} = h(x)$ whenever $g \in \mathfrak g'$.  Therefore $h$ factors through $P' / \mathfrak g' = P$.  Thus $h$ factors uniquely as a \emph{$G$-equivariant map} $$P \longrightarrow \mathfrak g'=\mathfrak g\otimes J.$$
\end{enumerate}

Putting all of this together using the same \v{C}ech or gerbe argument from Section~\ref{S:def-vect}, we obtain a theory of deformations and obstructions:
\begin{teo}
Suppose that $G$ is a smooth algebraic group over a scheme $X$, that $X'$ is a square-zero extension of $X$, and that $G'$ is an extension of $G$ to $X'$ by the ideal $J$.  Let $P$ be a $G$-torsor on $X$.  There is an obstruction to extending $P$ to a $G'$-torsor on $X'$ lying in $H^2\bigl(X, \underline{\operatorname{Hom}}_G(P, \mathfrak g \otimes J)\bigr)$.  Should this obstruction vanish, deformations are canonically a torsor on $X$ under $H^1\bigl(X, \underline{\operatorname{Hom}}_G(P, \mathfrak g \otimes J)\bigr)$, and automorphisms of any given deformation are canonically isomorphic to $H^0\bigl(X, \underline{\operatorname{Hom}}_G(P, \mathfrak g \otimes J) \bigr)$.
\end{teo}

In the case of the trivial extension $X[\epsilon]$, so that $J=\mathcal O_X$, and $\mathfrak g=\mathfrak g\otimes J=\mathfrak g'$, the bundle $$\mathfrak p:= \underline{\operatorname{Hom}}_G(P, \mathfrak g)$$ is known as the \emph{adjoint bundle} of $P$.

\begin{rem}
The adjoint bundle can also be constructed as the quotient of $P \times \mathfrak g$ by the diagonal action of $G$.  To see that these are equivalent, note that a section of $P \times_G \mathfrak g$ corresponds to a $G$-orbit in $P \times \mathfrak g$, which is the graph of an equivariant map $P \rightarrow \mathfrak g$.  Thus the definition of the adjoint bundle given above is equivalent to the one given in Section~\ref{S:G-Higgs-1}.
\end{rem}  

\begin{lem}
The adjoint bundle $\mathfrak p$ of a $G$-torsor $P$ over $X$ is isomorphic to the space of $G$-invariant vector fields on the fibers of $P$ over $X$.
\end{lem}
\begin{proof}
Recall that a vector field on the fibers of $P$ over $X$ is an $X$-morphism $V : P[\epsilon] \rightarrow P$ that lifts the identity on $P \subset P[\epsilon]$.  To be invariant means that the diagram
\begin{equation*} \xymatrix{
G \times P[\epsilon] \ar[r]^-{\mathrm{id} \times V} \ar[d]_{p_2} & G \times P \ar[d]^\alpha \\
P[\epsilon] \ar[r]^V & P
} \end{equation*}
commutes, $\alpha$ being the action map and $p_2$ the second projection.  Now, the zero vector field also gives a map $P[\epsilon] \rightarrow P$, and the pair $(0, V)$ gives an $X$-morphism
\begin{equation}\label{E:LemGInvVect}
P[\epsilon] \stackrel{(0,V)}{\longrightarrow} P \mathbin\times_X P \stackrel{\sim}{\longleftarrow} G \mathbin\times_X P .
\end{equation}
The first map is equivariant with respect to the diagonal $G$-action on $P \mathbin\times_X P$.  Since the map $G \mathbin\times_X P \rightarrow P \mathbin\times_X P$ sends $(g,y)$ to $(gy,y)$, the second arrow is also equivariant, provided we let $G$ act by conjugation on itself.  The original condition that $P[\epsilon] \rightarrow P$ restrict to the identity on $P \subset P[\epsilon]$ means that $P \subset P[\epsilon] \rightarrow P \mathbin\times_X P$ factors through the diagonal, and therefore that $P \subset P[\epsilon] \rightarrow G \mathbin\times_X P \rightarrow G$ factors through the identity section of $G$.

Now composing  \eqref{E:LemGInvVect} with the projection $G \mathbin\times_X P \rightarrow G$, we get an equivariant map
\begin{equation*}
P[\epsilon] = P \times X[\epsilon] \rightarrow G, 
\end{equation*}
that factors through the identity section of $G$ when restricted to $P$.  This is the same as to give an equivariant map
\begin{equation*}
P \rightarrow \underline{\operatorname{Hom}}_X(X[\epsilon], G) \mathbin\times_{\underline{\operatorname{Hom}}_X(X, G)} \{ e \} = T_{G/X} \mathbin\times_{G} \{ e \} = \mathfrak g ,
\end{equation*}
which, by definition is the same as to give a section of the adjoint bundle of $P$.
\end{proof}

\begin{cor} \label{C:adj-dual}
The adjoint bundle $\mathfrak p$ of a $G$-torsor $P$ has the structure of a sheaf of Lie algebras.  If $G$ is semisimple then the Killing form furnishes an isomorphism between $\mathfrak p$ and its dual.
\end{cor}
\begin{proof}
The point is to verify that the bracket of equivariant vector fields on $P$ is equivariant.  This can be done with a diagram chase, using the definition of the Lie bracket in Section~\ref{S:tangent-bundle}, but we omit it.  Then $P$ is locally isomorphic to $G$ so $\mathfrak p$ is locally isomorphic to $\mathfrak g$, and this isomorphism preserves the Killing form.  If $G$ is semisimple then $\mathfrak g$ is self dual with respect to the Killing form.
\end{proof}

For the next statement, let $X$ be an $S$-scheme and let $G$ be a smooth group scheme over $X$.  Let $\mathsf{Fib}^G_{X/S}$ denote the $S$-stack whose sections over $T$ are the $G$-torsors on $X_T$.

\begin{teo}
We have $T_{\mathsf{Fib}^G_{X/S}}(P) = \mathrm{B}\underline{\operatorname{Hom}}_G(P, \mathfrak g) = \mathrm{B}\kern1pt\mathfrak p$ where $\mathfrak g$ denotes the Lie algebra of $G$.  In particular,
\begin{align*}
T^{-1}_{\mathsf{Fib}^G_{X/S}}(P) &= H^0\bigl(X, \underline{\operatorname{Hom}}_G(P, \mathfrak g)\bigr) = H^0(X, \mathfrak p) \\
T^0_{\mathsf{Fib}^G_{X/S}}(P) &= H^1\bigl(X, \underline{\operatorname{Hom}}_G(P, \mathfrak g)\bigr) = H^1(X, \mathfrak p)  .
\end{align*}
\end{teo}
\begin{proof}
Let $X' = X[\epsilon]$.  Let $P'$ be the trivial extension of $P$ to $X'$.  Suppose that $P''$ is another extension.  Then $\underline{\operatorname{Isom}}_{G'}^P(P'', P')$ (isomorphisms respecting the $G'$-action and the maps to $P$) is a torsor on $X'$ under $\underline{\operatorname{Hom}}_G(P, \mathfrak g)$ as explained in \eqref{B:PrinGbund-iii} above ($X'$ has the same \'etale site as $X$).  The same argument shows that $\underline{\operatorname{Aut}}_{G'}^P(P') = \mathfrak p$, so that, given any torsor $Q$ under $\mathfrak p=\underline{\operatorname{Hom}}_G(P, \mathfrak g)$, we can form the sheaf of $\mathfrak p$-equivariant maps from $Q$ to $P'$
\begin{equation*}
P'' = \underline{\operatorname{Hom}}_{\mathfrak p}(Q, P') .
\end{equation*}
These constructions are inverse isomorphisms between $T_{\mathsf{Fib}^G_{X/S}}(P)$ and the stack of torsors under $\mathfrak p = \underline{\operatorname{Hom}}_G(P, \mathfrak g)$.
\end{proof}

Now suppose that $X$ is a reduced (and therefore Cohen--Macaulay) projective curve over a field $k$ (i.e., $S=\operatorname{Spec}k$).  Then by Serre duality, we have
\begin{equation*}
T^0_{\mathsf{Fib}^G_{X/S}}(P)=H^1(X, \mathfrak p) \simeq H^0\bigl(X, \underline{\operatorname{Hom}}(\mathfrak p, \omega_X)\bigr)^\vee .
\end{equation*}

Generalizing Definition \ref{D:GHiggsCC}:  

\begin{dfn}[$G$-Higgs bundle] \label{D:GHiggsk}
Let $X$ be a  reduced projective curve over a field $k$, and let  $G$ be a smooth algebraic group over $k$. 
A \emph{$G$-Higgs bundle} on $X$ is a pair $(P,\Phi)$ where $P$ is a $G$-torsor  over $X$  and $\Phi\in \operatorname{Hom}(\mathfrak p,\omega_X)$, where $\mathfrak p:= \underline{\operatorname{Hom}}_G(P, \mathfrak g)$.
\end{dfn}

\begin{rem}
If $G$ is semisimple then $\mathfrak p$ is self-dual (Corollary~\ref{C:adj-dual}), so a $G$-Higgs bundle on $X$ can also be viewed as an element of $H^0(X, \mathfrak p \otimes \omega_X)$, or, in another popular notation, of $H^0(X, \operatorname{ad} P \otimes K_X)$.  In other words, Definition \ref{D:GHiggsCC} agrees with  Definition \ref{D:GHiggsk}.
\end{rem}

In the case where $P$ is the $\mathrm{GL}_n$-torsor associated to a vector bundle $E$, we have $\mathfrak p = \underline{\operatorname{End}}(E)$, so that
\begin{equation*}
\underline{\operatorname{Hom}}(\mathfrak p, \omega_X) = \underline{\operatorname{Hom}}(E, E \otimes \omega_X),
\end{equation*}
and $T^0_{\mathsf{Fib}^G_{X/k}}(P)= H^1(X, \mathfrak p)\simeq \operatorname{Hom}(\mathfrak p,\omega_X)^\vee=\operatorname{Hom}(E,E\otimes \omega_X)^\vee$.
More generally, extending the definition of $G$-Higgs bundles to families of curves, and using Grothendieck duality,  the discussion above  proves the following corollary:

\begin{cor}
Let $X$ be a family of smooth curves over $S$.  Higgs bundles on $X/S$ correspond to isomorphism classes of relative cotangent vectors for $\mathsf{Fib}_{X/S}$ over $S$.  Likewise, $G$-Higgs bundles on $X/S$ correspond to relative cotangent vectors over $S$ for the stack $\mathsf{Fib}^G_{X/S}$ of principal $G$-bundles on $X$.
\end{cor}

\subsection{Obstruction theory}
\label{S:obstructions}

In this section, we will consider a stack $\mathcal{F}$  (not a priori algebraic) and a lifting problem
\begin{equation} \label{E:lifting-problem} \xymatrix{
		S \ar[r]^{\xi} \ar[d] & \mathcal{F} \\
		S' \ar@{-->}[ur]
} \end{equation}
in which $S'$ is a square-zero extension of $S'$ with ideal $J$.  We are looking for an \emph{obstruction theory} that can detect whether this lifting problem has a solution.  The obstruction theory will be an $\mathcal{O}_S$-module $T^1_{\mathcal{F}}(\xi, J)$ that depends only on $J$ and $\xi \in \mathcal{F}(S)$, \emph{not on the particular extension $S'$}.  The above lifting problem will then produce a natural obstruction $\omega \in T^1_{\mathcal{F}}(\xi, J)$ whose vanishing is equivalent to the existence of a lift.

We will see that for moduli problems that are locally unobstructed in a suitable sense (that includes all of the examples considered here) there is a natural choice for $T^1_\mathcal{F}(\xi,J)$ arising from the infinitesimal deformation theory.  This phenomenon could be better explained in the language of torsors under abelian group stacks, but to introduce such objects would take us too far afield.  We will rely instead on injective resolutions and \v{C}ech methods.

Various definitions of obstruction theories have appeared in the literature~\cite[\S2.6]{Artin-versal}, \cite[Def.~4.4]{BF}, \cite[Def.~1.2]{LT}, \cite[Tag 07YG]{stacks}, \cite[Def.~6.6]{Hall}, \cite[Def.~3.2]{obs}.  From the perspective of derived algebraic geometry, an obstruction theory arises from the promotion of a moduli problem to a derived moduli problem~\cite{Lurie-thesis}.  The definition we adopt here is closest to~\cite{Hall}.

\begin{dfn} \label{D:obs}
	Let $\mathcal F$ be a stack on the \'etale site of schemes.  An \emph{obstruction theory} for $\mathcal{F}$ is a system of abelian groups $T^1(\xi, J)$ depending contravariantly on a scheme $S$ 
	 and an element $\xi \in \mathcal{F}(S)$, and covariantly on a quasicoherent $\mathcal{O}_S$-module $J$, together with an obstruction map 
\begin{equation*}
\omega : \operatorname{Exal}(\mathcal{O}_S, J) \rightarrow T^1(\xi, J)
\end{equation*}
that is natural in $S$ and $J$ and has the following property:  $\omega(S') = 0$ if and only if diagram~\eqref{E:lifting-problem} admits a lift.
\end{dfn}

Above we use the notation $\operatorname{Exal}(\mathcal{O}_S, J)$ for the $\mathcal O_S$-module of algebras that are extensions of $\mathcal O_S$ by $J$  with $J^2=0$ (see e.g., \cite[\S 1.1]{sernesi} for more details).

We spell out precisely the meaning of naturality in Definition \ref{D:obs}.  Suppose there is a commutative (not necessarily cartesian) diagram 
of square-zero extensions~\eqref{E:sqz} where $R \subseteq R'$ has ideal $I$ and $S \subseteq S'$ has ideal $J$.
\begin{equation} \label{E:sqz} \vcenter{ \xymatrix{
		R \ar[r] \ar[d]_f & R' \ar[d] \\
		S \ar[r] & S'
}} \end{equation}
Then we get a homomorphism $f^\ast J \rightarrow I$ and therefore a morphism:
\begin{equation} \label{E:obs-nat}
	T^1(\xi, J) \rightarrow T^1(f^\ast \xi, f^\ast J) \rightarrow T^1(f^\ast \xi, I).
\end{equation}
We also have obstructions $\omega(S') \in T^1(\xi, J)$ and $\omega(R') \in T^1(f^\ast \xi, I)$.  The naturality alluded to in Definition \ref{D:obs} requires that $\omega(S')$ is carried under the morphism~\eqref{E:obs-nat} to $\omega(R')$.

Our goal in this section will be to illustrate a technique for constructing an obstruction theory for a stack.  For concreteness, we will consider the case where $\mathcal{F}$ is the stack of quadruples $(C,E,F,\sigma)$ in which $\pi:C\to S$ is a nodal curve, $E$ and $F$ are vector bundles on $C$, and $\sigma : E \rightarrow F$ is a morphism of vector bundles.  However, apart from the following observations, analogues of which frequently hold for other stacks, the particular choice of $\mathcal{F}$ will not enter into the rest of the discussion:
\setcounter{enumi}{0}
\begin{enumerate}[(1)]
	\item \label{obs:1} Fix a scheme $S$, an $S$-point $\xi = (C,E,F,\sigma)$ of $\mathcal{F}$, and a square-zero extension $S \subseteq S'$ with ideal $J$.  We obtain a square-zero extension of sheaves of commutative rings on the Zariski site, $\pi^{-1} \mathcal{O}_{S'} \rightarrow \pi^{-1} \mathcal{O}_S$, with ideal $\pi^{-1} J$.  Our problem can be phrased as the search for a square-zero extension $\mathcal{O}_{C'}$ of $\mathcal{O}_C$ by $\pi^\ast J$ compatible with the extension $\pi^{-1} \mathcal{O}_{S'} \rightarrow \pi^{-1} \mathcal{O}_S$ and the homomorphism $\pi^{-1} J \rightarrow \pi^\ast J$, together with locally free $\mathcal{O}_{C'}$-modules $E'$ and $F'$ extending $E$ and $F$ and a morphism $\sigma' : E' \rightarrow F'$ extending $\sigma$.

	We could replace $\pi^{-1} \mathcal{O}_{S'}$ by the extension $\mathscr{A}$ of $\pi^{-1} \mathcal{O}_S$ 
	by $\pi^\ast J$ obtained by pushout:
\begin{equation*} \xymatrix{
		0 \ar[r] & \pi^{-1} J \ar[r] \ar[d] & \pi^{-1} \mathcal{O}_{S'} \ar[r] \ar[d] & \pi^{-1} \mathcal{O}_S \ar[r] \ar@{=}[d] & 0 \\
		0 \ar[r] & \pi^\ast J \ar[r] & \mathscr{A} \ar[r] & \pi^{-1} \mathcal{O}_S \ar[r] & 0
} \end{equation*}
To find an extension of $\mathcal{O}_{C}$ by $\pi^\ast J$ compatible with $\mathscr{A}$ is the same as to find a $\pi^{-1} \mathcal O_{S'}$-algebra extension compatible with the homomorphism $\pi^{-1} J \rightarrow \pi^\ast J$.
Indeed, an $\mathscr A$-algebra extension induces a $\pi^{-1} \mathcal{O}_{S'}$-algebra extension by composition with $\pi^{-1} \mathcal O_{S'} \rightarrow \mathscr{A}$; conversely, if $\mathscr B$ is a $\pi^{-1} \mathcal O_{S'}$-algebra extension of $\mathcal O_C$ by $\pi^\ast J$ then the map $\pi^{-1} J \rightarrow \mathscr B$ factors through $\pi^\ast J$ so the map $\pi^{-1} \mathcal O_{S'} \rightarrow \mathscr B$ factors through $\mathscr A$.

Moreover, we may consider an arbitrary square-zero extension $\mathscr{A}$ of $\pi^{-1} \mathcal{O}_S$ by an $\mathcal{O}_C$-module $\mathscr{J}$ and define $\mathcal{F}_\xi(\mathscr{A})$ to be the category of quadruples $(\mathcal{O}_{C'}, E', F', \sigma')$ where $\mathcal{O}_{C'}$ is an extension of $\mathcal{O}_C$ by $\mathscr{J}$, compatible with $\mathscr{A}$; both $E'$ and $F'$ are locally free $\mathcal{O}_{C'}$-modules extending $E$ and $F$, respectively; and $\sigma' : E' \rightarrow F'$ is a morphism of $\mathcal{O}_{C'}$-modules extending $\sigma$.

We observe that $\mathcal{F}_\xi$ is a homogeneous functor.  This allows us to define $T^{-1}_{\mathcal{F}}(\xi, \mathscr{J})$ and $T^0_{\mathcal{F}}(\xi, \mathscr{J})$.

\item \label{obs:2} There is a complex $\Upsilon^\bullet$
 of $\mathcal{O}_C$-modules with quasicoherent cohomology such that 
\begin{gather*}
T_{\mathcal{F}}(\xi, \mathscr{J}) = \mathbf{Ext}(\Upsilon^\bullet, \mathscr{J}) .
\end{gather*}
In other words, the infinitesimal deformations and automorphisms of $\xi$ should be representable by a complex.  In the case of interest, $\Upsilon^\bullet$ is the complex $\Upsilon^\bullet_{C/S}(E, F, \sigma)$, constructed in \S\ref{S:CVBm}.

For the representability of the obstruction theory by a complex on $S$, we will also need the observation that $\Upsilon^\bullet$ is perfect as an object of the derived category.  That is, locally in $C$ it can be represented by a bounded complex of locally free modules.  This is even true globally for $\Upsilon_{C/S}(E, F, \sigma)$. 

\item \label{obs:3} Whenever $\mathscr{J}$ is an \emph{injective} $\mathcal{O}_C$-module and $\mathscr{A}$ is a square-zero extension of $\pi^{-1} \mathcal{O}_S$ by $\mathscr{J}$, there is some $\xi' \in \mathcal{F}_\xi(\mathscr{A})$.  In our situation of interest, we can see that this is the case by considering the successive obstructions introduced in earlier sections.  There is no local obstruction to deforming curves (Corollary \ref{C:NCsm}).  The first global obstruction (gluing isomorphism classes of local deformations) lies in 
\begin{equation*}
H^1\bigl(C, \underline{\operatorname{Ext}}^1(\Omega_{C/S}, \mathscr{J})\bigr) = 0 .
\end{equation*}
The obstruction to finding a global deformation inside compatible local isomorphism classes lies in
\begin{equation*}
H^2\bigl(C, \underline{\operatorname{Hom}}(\Omega_{C/S}, \mathscr{J})\bigr) = 0 .
\end{equation*}
Therefore we can find $\mathcal{O}_{C'}$.  Now the obstructions to extending $E$ and $F$ lie in
\begin{gather*}
\operatorname{Ext}^2\bigl(\underline{\operatorname{End}}(E)^\vee, \mathscr{J}\bigr) = 0 \\
\operatorname{Ext}^2\bigl(\underline{\operatorname{End}}(F)^\vee, \mathscr{J}\bigr) = 0 .
\end{gather*}
Finally, the obstruction to extending $\sigma$, once $\mathcal{O}_{C'}$, $E'$, and $F'$ have been chosen, lies in
\begin{equation*}
\operatorname{Ext}^1\bigl(\underline{\operatorname{Hom}}(E,F)^\vee, \mathscr{J}\bigr) = 0.
\end{equation*}
One could also verify the existence of $\xi'$ directly by choosing local deformations arbitrarily and using the fact that $\mathscr{J}$ is a flasque sheaf to assemble them into a global deformation.  Thus, if one wants to axiomatize the discussion of this section, one should require that $\mathcal{F}_\xi(\mathscr{A})$ be locally nonempty, when regarded as a CFG over the Zariski site of $C$.
\end{enumerate}

Now we will build a global obstruction to the existence of $\xi' \in \mathcal{F}(S')$ lifting $\xi \in \mathcal{F}(S)$.  The obstruction will lie in $\operatorname{Ext}^2(\Upsilon^\bullet, \pi^\ast J)$, which we will take as the definition of $T^1_{\mathcal{F}}(\xi, J)$.

We begin by choosing a resolution 
\begin{equation*}
0 \rightarrow \pi^\ast J \rightarrow \mathscr{J}^0 \rightarrow \mathscr{J}^1 \rightarrow 0
\end{equation*}
where $\mathscr{J}^0$ and $\mathscr{J}^1$ are sheaves of $\mathcal{O}_C$-modules and $\mathscr{J}^0$ is injective.  We leave it to the reader to verify that the construction of the obstruction given here is independent of the choice of resolution.

We can push out the extension:
\begin{equation*} \xymatrix{
0 \ar[r] & \pi^{-1} J \ar[r] \ar[d] & \pi^{-1} \mathcal{O}_{S'} \ar[r] \ar[d] & \pi^{-1} \mathcal{O}_S \ar[r] \ar@{=}[d] & 0 \\
0 \ar[r] & \mathscr{J}^0  \ar[r] \ar[d] & \mathscr{A} \ar[r] \ar[d] & \pi^{-1} \mathcal{O}_S \ar[r] \ar@{=}[d] & 0 \\
0 \ar[r] & \mathscr{J}^1 \ar[r] & \mathscr{B}  \ar[r] & \pi^{-1} \mathcal{O}_S \ar[r] & 0
} \end{equation*}
Note $\mathscr{B}$ is canonically isomorphic to $\pi^{-1} \mathcal{O}_S + \epsilon \mathscr{J}^1$ because the map $\pi^{-1} J \rightarrow \mathscr{J}^1$ is zero.

The commutative diagram above induces a map:
\begin{equation} \label{E:extension-morphism}
\mathcal{F}_\xi(\mathscr{A}) \rightarrow \mathcal{F}_\xi(\mathscr{B}) = T_{\mathcal{F}}(\xi, \mathscr{J}^1).
\end{equation}
The identification with $T_{\mathcal F}(\xi, \mathscr J^1)$ comes by virtue of the canonical isomorphism $\mathscr B \simeq \pi^{-1} \mathcal O_S + \epsilon \mathscr J^1$.

Note that $\mathcal{F}_\xi(\mathscr{B})$ contains a canonical zero element (corresponding to the zero element of $T_{\mathcal F}(\xi, \mathscr J^1)$), namely the image of $\xi$ under the map $\mathcal{F}_\xi(\pi^{-1} \mathcal{O}_S) \rightarrow \mathcal{F}_\xi(\mathscr{B})$ induced from the homomorphism $\pi^{-1} \mathcal O_S \rightarrow \pi^{-1} \mathcal{O}_S + \epsilon \mathscr{J}^1 = \mathscr{B}$, so we can speak of the kernel of~\eqref{E:extension-morphism}.  By definition, this is the fiber product of groupoids $0 \mathop{\times}_{\mathcal{F}_\xi(\mathscr{B})} \mathcal{F}_\xi(\mathscr{A})$.

\begin{lem}
There is a canonical identification between $\mathcal{F}_\xi(\pi^{-1} \mathcal{O}_{S'})$ and the kernel of~\eqref{E:extension-morphism}.
\end{lem}
\begin{proof}
The diagram
\begin{equation*} \xymatrix{
\pi^{-1} \mathcal{O}_{S'} \ar[r] \ar[d] & \mathscr{A} \ar[d] \\
\pi^{-1} \mathcal{O}_S \ar[r] & \mathscr{B}
} \end{equation*}
is cartesain (the bottom arrow is the canonical splitting of the projection $\mathscr{B} \rightarrow \pi^{-1} \mathcal{O}_S$).  Therefore the homogeneity of $\mathcal{F}_\xi$ implies that
\begin{equation*} \xymatrix{
\mathcal{F}_\xi(\pi^{-1} \mathcal{O}_{S'}) \ar[r] \ar[d] & \mathcal{F}_\xi(\mathscr{A}) \ar[d] \\
\mathcal{F}_\xi(\pi^{-1} \mathcal{O}_S) \ar[r] & \mathcal{F}_\xi(\mathscr{B})
} \end{equation*}
is also cartesian.  But $\mathcal{F}_\xi(\pi^{-1} \mathcal{O}_S) = 0$, by definition, so this identifies $\mathcal{F}_\xi(\pi^{-1} \mathcal{O}_{S'})$ with the kernel of~\eqref{E:extension-morphism}.
\end{proof}

As $\mathscr{J}^0$ is injective, $\mathcal{F}_\xi(\mathscr{A}) \neq \varnothing$ (observation~\ref{obs:3} on p.~\pageref{obs:3}), 
so that $\operatorname{Ext}^1(\Upsilon^\bullet, \mathscr{J}^0)$ acts simply transitively on the set $\bigl| \mathcal{F}_\xi(\mathscr{A}) \bigr|$ of isomorphism classes in $\mathcal{F}_\xi(\mathscr{A})$ (by observation~(\ref{obs:2}) on p.~\pageref{obs:2}).  As $\mathscr{B}$ is a (canonically) split extension of $\pi^{-1} \mathcal{O}_{S}$, we have a (canonical) identification $\operatorname{Ext}^1(\Upsilon^\bullet, \mathscr{J}^1) = \bigl| \mathcal{F}_\xi(\mathscr{B}) \bigr|$.  Now we have an exact sequence in the top row below, with compatible actions illustrated in the bottom row:
\begin{equation*}  \xymatrix@R=10pt{
\operatorname{Ext}^1(\Upsilon^\bullet, \pi^{-1} J) \ar[r] & \operatorname{Ext}^1(\Upsilon^\bullet, \mathscr{J}^0) \ar[r] & \operatorname{Ext}^1(\Upsilon^\bullet, \mathscr{J}^1) \ar[r] & \operatorname{Ext}^2(\Upsilon^\bullet, \pi^{-1} J)  \\
\bigl|\mathcal{F}_\xi(\pi^{-1} \mathcal{O}_{S'})\bigr| \ar[r] \ar@(ul,ur) & \bigl|\mathcal{F}_\xi(\mathscr{A})\bigr| \ar[r] \ar@(ul,ur) & \bigl|\mathcal{F}_\xi(\mathscr{B})\bigr| \ar@{=}[u]
} \end{equation*}
Note that $\mathcal{F}_\xi(\pi^{-1} \mathcal{O}_{S'})$ may be empty.  As the action of $\operatorname{Ext}^1(\Upsilon^\bullet, \mathscr{J}^0)$ on $\bigl| \mathcal{F}_\xi(\mathscr{A}) \bigr|$ is faithful and transitive, we have that  the image of $\bigl| \mathcal{F}_\xi(\mathscr{A}) \bigr|$ is an $\operatorname{Ext}^1(\Upsilon^\bullet, \mathscr{J}^0)$ coset in $\operatorname{Ext}^1(\Upsilon^\bullet, \mathscr{J}^1)$.  In other words, it gives a well-defined element of $\operatorname{Ext}^2(\Upsilon^\bullet, \pi^{-1} J)$ obstructing the existence of an element of $\mathcal{F}_\xi(\pi^{-1} \mathscr{O}_{S'})$.  We therefore define $$T^1_{\mathcal{F}}(S,J) := \operatorname{Ext}^2(\Upsilon^\bullet, \pi^{-1} J).$$  By construction, $T^1_{\mathcal{F}}(S,J)$ is functorial with respect to $S$ (contravariant) and $J$ (covariant) and the obstruction class is natural.

\begin{dfn}[Representable deformation-obstruction theory] \label{D:rep-obs}
	Let $\mathcal F$ be a stack in the \'etale topology on schemes.  We will say that a deformation-obstruction theory $T^i_{\mathcal{F}}$, $i = -1, 0, 1$, is \emph{representable} at an $S$-point $\xi$ if there is a complex of locally free sheaves $\mathbf{E}^\bullet$ on $S$, such that for any $f : T \rightarrow S$ and any quasicoherent sheaf $J$ on $T$, we have a natural (in $T$ and in $J$) identification
\begin{equation*}
T^i_{\mathcal{F}}(f^\ast \xi, J) = \operatorname{Ext}^i(f^\ast \mathbf{E}^\bullet, J)
\end{equation*}
for $i = -1, 0, 1$.  We will say it is \emph{finitely presentable} if the vector bundles in the complex $\mathbf E^\bullet$ may be chosen to have finite rank.

We say that the deformation theory is \emph{locally representable} if it is representable at all points valued in affine schemes. We say it is \emph{locally finitely presentable} if it is finitely presentable at all points valued in affine schemes.
\end{dfn}

\begin{lem}\label{L:C-EF-Obs-lfp}
Let $\mathcal{F}$ be the stack of quadruples $(C,E,F,\sigma)$ where $C$ is a nodal curve, $E$ and $F$ are vector bundles over $C$, and $\sigma : E \rightarrow F$ is a morphism of vector bundles.  Then the obstruction theory for $\mathcal{F}$ introduced above is locally finitely presentable.
\end{lem}
\begin{proof}
Let $A$ be a commutative ring and let $\xi = (C,E,F,\sigma)$ be an $A$-point of $\mathcal{F}$.  We want to show that there is a complex of locally free $A$-modules $\mathbf{E}^\bullet$ representing $T^i_{\mathcal{F}}(\xi, J)$ for $i = -1, 0, 1$ and all $A$-modules $J$.  We assume first that $A$ is noetherian.

Let $\Xi^\bullet$ be a bounded above complex 
of locally free sheaves on $C$ that is dual to $\Upsilon^\bullet$.  Then
\begin{equation*}
T^i_{\mathcal{F}}(\xi, J) = \operatorname{Ext}^i(\Upsilon^\bullet, \pi^\ast J) = H^i(C, \Xi^\bullet \otimes \pi^\ast J) .
\end{equation*}
Therefore, by \cite[III.12.2]{Hartshorne}, there is a complex $L^\bullet$ of finite rank vector bundles on $S$ such that\footnote{Loc.\ cit.\ requires that $\Xi^\bullet$ be a quasicoherent sheaf, but the proof works for a complex as long as one takes the total complex of the \v{C}ech double complex at the bottom of p.\ 182.}
\begin{equation*}
T^i_{\mathcal{F}}(\xi, J) = h^i(L^\bullet \otimes_A J) .
\end{equation*}
But now we may take $\mathbf{E}^\bullet$ to be a complex dual to $L^\bullet$ and obtain
\begin{equation*}
\operatorname{Ext}^i(\mathbf{E}^\bullet, J) = h^i(L^\bullet \otimes_A J)
\end{equation*}
as required.

Now we pass to the general case.  We can always write $A$ as a filtered colimit of commutative rings of finite type $A = \varinjlim A_i$.  As $\mathcal{F}$ is locally of finite presentation (Corollary~\ref{C:curve-vb-morph-lfp}) there is some index $j$ and some $\xi_j = (C_j,E_j,F_j,\sigma_j) \in \mathcal{F}(A_j)$ inducing $(C,E,F,\sigma)$.  Now, $T^i_{\mathcal{F}}(\xi, J) = T^i_{\mathcal{F}}(\xi_j, J)$, with $J$ regarded as an $A_j$-module via the map $A_j \rightarrow A$.  Therefore
\begin{equation*}
T^i_{\mathcal{F}}(\xi, J) = T^i_{\mathcal{F}}(\xi_j, J) = \operatorname{Ext}^i_{A_j}(\mathbf{E}^\bullet, J) = \operatorname{Ext}^i_{A}(A \mathop{\otimes}_{A_j} \mathbf{E}^\bullet, J)
\end{equation*}
so $T_{\mathcal{F}}$ is locally representable at $\xi$.
\end{proof}

\begin{lem} \label{L:aut-def-obs-dim}
Suppose that $\mathcal X$ is a stack with a locally finitely presentable obstruction theory.  Then, for any field-valued point $\xi$ of $\mathcal X$, the vector spaces $T^{-1}_{\mathcal X}(\xi)$, $T^0_{\mathcal X}(\xi)$, and $T^1_{\mathcal X}(\xi)$ are finite dimensional.
\end{lem}
\begin{proof}
This is immediate, because the vector spaces in question may be identified with cohomology groups of the dual complex of a complex of finite rank vector bundles.
\end{proof}

\section{Artin's criterion for algebraicity}
\label{S:Artin}

Intuition from analytic moduli spaces suggests that moduli spaces should locally be embedded as closed subspaces of their tangent spaces.  For this to apply in an algebraic context, locally must be interpreted to mean \emph{\'etale-locally} for schemes and algebraic spaces, and \emph{smooth-locally} for algebraic stacks.

Artin gives criteria under which a stack is locally cut out by polynomial equations inside its tangent space, thereby ensuring the stack is algebraic.
  Since Artin's original formulation, there have been a number of improvements~\cite{Flenner, Lurie-DAG14, Pridham, Hall, Hall-Rydh}.  The statement we give here is close to the form given by Hall~\cite{Hall}, but with some hypotheses strengthened for the sake of transparency:

\begin{teo}[{Artin's criterion \cite[Thm.~A]{Hall}}] \label{T:Artin-criteria}
Let $S$ be an excellent scheme and let $\mathcal X$ be a CFG over the category $\mathsf S/S$ of $S$-schemes.  Then $\mathcal X$ is an  algebraic stack over $(\mathsf S/S)_{\operatorname{et}}$ that is locally of finite presentation over $S$ if and only if it has the following properties:
\begin{enumerate}[(1)]
\item \label{T:ACstack} $\mathcal X$ is a stack in the \'etale topology (Definition~\ref{D:Stack}).
\item \label{T:AChomog} $\mathcal X$ is homogeneous (Definition~\ref{D:Homog}).
\item \label{T:ACTan} $\mathcal X$ has finite dimensional tangent and automorphism spaces (Section~\ref{S:homogeneity}).
\item \label{T:ACint} $\mathcal X$ is integrable (Definition~\ref{D:FormEl}).
\item \label{T:AClfp} $\mathcal X$ is locally of finite presentation (Sections~\ref{S:lfp2} and~\ref{S:lfp}). 
\item \label{T:ACobs} $\mathcal X$ has a locally finitely presentable obstruction theory (Definition~\ref{D:rep-obs}).
\end{enumerate}
\end{teo}

\begin{rem}
The assumption \eqref{T:ACTan} actually follows from \eqref{T:AClfp} (see Lemma~\ref{L:aut-def-obs-dim}); however, we include the hypothesis   \eqref{T:ACTan}  since it is useful from an expository perspective.  In particular, a theorem of Schlessinger--Rim (see Theorem \ref{T:Rim}) uses the hypotheses \eqref{T:AChomog} and \eqref{T:ACTan} on a CFG.  
\end{rem}

\begin{rem}
Hall only requires the existence of a multistep obstruction theory, which is an \emph{a priori} weaker hypothesis than~\eqref{T:ACobs}.  \emph{A posteriori}, every algebraic stack has a cotangent complex, whence a single step obstruction theory. In our case, we actually constructed the obstruction theory for Higgs bundles in pieces in Section~\ref{S:def}, but we were able to assemble it into a single-step obstruction theory in Section~\ref{S:obstructions}. 
\end{rem}

\begin{rem}
Property~\eqref{T:AClfp} is sometimes phrased, `$\mathcal X$ is limit preserving'.
\end{rem}

\begin{rem}  For Property~\eqref{T:ACint},
some would say `$\mathcal X$ is effective' or `formal objects of $\mathcal X$ may be effectivized' or `formal objects of $\mathcal X$ can be algebraized' (in the literature, the term `algebraized' is sometimes reserved for algebraization over a scheme of finite type; see Definition \ref{D:Alg-Fin-Type}).  We picked up the term `integrable' from \cite{Bhatt-Halpern-Leinster}, the intuition being that infinitesimal arcs can be integrated into formal arcs, in the manner that tangent vectors are integrated to curves on smooth manifolds. 
\end{rem}

The rest of this section will be devoted to explaining these properties and verifying them for the stack of Higgs bundles and related moduli problems; at the end, we give a brief explanation of how these properties combine to imply a stack is algebraic.

\subsection{The Schlessinger--Rim criterion}
\label{S:Schlessinger}

We have seen that homogeneity is a necessary condition for representability by an algebraic stack.  The Schlessinger--Rim criterion implies that it is sufficient for \emph{prorepresentability}.

By abstract nonsense, a covariant functor is prorepresentable if and only if it preserves finite limits.  Generally, it is not very practical to check that functors on Artin rings preserve arbitrary finite limits, ultimately due to the restrictions on descent for flat modules along non-flat morphisms in Theorem~\ref{T:flat-homog}.  Fortunately, Schlessinger was able to prove that homogeneity, that is, respect for a restricted class of limits, along with a finite dimensional tangent space, are sufficient to imply a functor is prorepresentable \cite{Schlessinger}.  In this section, we will discuss Rim's generalization of Schlessinger's result to groupoids~\cite[Exp.~VI]{sga7-1}.

\vskip .2 cm 
Suppose that ${\mathcal X}$ is a CFG  on $\Lambda$-schemes, where $\Lambda$ is a complete noetherian  local ring with residue field $k$, and $\xi:\operatorname{Spec}k\to \mathcal X$ is a $k$-point.  Let $\mathscr{C}_\Lambda$ be the category of local artinian   $\Lambda$-algebras with residue field $k$.  One may interpret $\mathscr{C}_\Lambda^{\mathsf{op}}$ as the category of infinitesimal extensions of $\operatorname{Spec} k$.  Let ${\mathcal X}_\xi$ be the fibered category on $\mathscr{C}_\Lambda^{\mathsf{op}}$ whose fiber over   a ring $A$ in $\mathscr{C}_\Lambda^{\mathsf{op}}$ consists of all $\eta \in {\mathcal X}(A)$ whose image via the projection $A \rightarrow k$ is $\xi$.  Th CFG $\mathcal X_\xi$  gives a formal picture of ${\mathcal X}$ near the point $\xi$.

\begin{rem}
Note that the restriction of the \'etale topology to $\mathscr C^{\mathsf{op}}_\Lambda$ is trivial, as every morphism in $\mathscr C^{\mathsf{op}}_\Lambda$ has a section over the residue field, and sections of \'etale maps extend infintisimally.  Therefore all covers in $\mathscr C^{\mathsf{op}}_{\Lambda}$ have sections, every presheaf is a sheaf, and every CFG is a stack.  Therefore we can use the terms `stack' and `CFG' interchangeably over $\mathscr C^{\mathsf{op}}_\Lambda$.
\end{rem}

\begin{rem}
Typically, we will start with a CFG $\mathcal X$ over $\mathsf S$.  This induces a CFG $\mathcal X_{\mathscr{C}_\Lambda^{\mathsf{op}}}$ over $\mathscr{C}_\Lambda^{\mathsf{op}}$ by restricting to the full subcategory obtained by taking  objects over the spectrum of such a ring.   Given an object $\xi\in \mathcal X(\operatorname{Spec}k)$, the category $\mathcal X_\xi$ over $\mathscr{C}_\Lambda^{\mathsf{op}}$ has objects the pairs $(\eta, \phi)$ where $\phi : \xi \rightarrow \eta$ is a morphism of $\mathcal X$ (necessarily cartesian) lying above an infintiesimal extension $\operatorname{Spec} k \rightarrow \operatorname{Spec} A$, with $A$ in $\mathscr{C}_\Lambda^{\mathsf{op}}$:
\begin{equation*}
\xymatrix{
 \xi \ar@{->}[r]^{ }  \ar@{|->}[d]& \eta \ar@{|->}[d]\\
 \operatorname{Spec}k \ar@{->}[r]^{}& \operatorname{Spec} A\\
}
\end{equation*}
  Morphisms in $\mathcal X_\xi$ are defined in the obvious way.
\end{rem}

We write $\widehat{\mathscr{C}}_\Lambda$ for the category of complete local $\Lambda$-algebras with residue field $k$ that are formally of finite type over $\Lambda$; i.e., completions of rings of finite type over $\Lambda$.

Recall the notion of a groupoid object of $\widehat{\mathscr C}^{\mathsf{op}}_{\Lambda}$ from Definition~\ref{D:gpd-obj} and its associated CFG~\ref{E:gpd-cfg}.  Note that this coincides with the associated stack from Definition~\ref{D:gpd-stk}, since the topology of $\widehat{\mathscr C}^{\mathsf{op}}_{\Lambda}$ is trivial.

\begin{dfn}[{\cite[Def.~VI.2.11]{sga7-1}}] 

We say  a category fibered in groupoids over $\mathscr{C}_\Lambda^{\mathsf{op}}$ is \emph{prorepresentable} if it is representable by (i.e., equivalent to) the CFG associated to a groupoid object $\xymatrix@C=10pt{U_1 \ar@<1.5pt>[r]^s \ar@<-1.5pt>[r]_t & U_0}$  in $\widehat{\mathscr{C}}_\Lambda^{\mathsf{op}}$ (Example \ref{E:gpd-proj}).  The groupoid is said to be \emph{smooth} if the morphisms $s$ and $t$ satisfy  the formal criterion for smoothness (Definition \ref{D:formal-sm-et-nr}).
\end{dfn}

\begin{rem}
In the above discussion, it is important to note that we are considering stacks over $\mathscr C^{\mathsf {op}}_\Lambda$, not $\widehat{\mathscr C}^{\mathsf {op}}_\Lambda$.  In particular, if $\mathcal X$ is representable by a groupoid object $\xymatrix@C=10pt{U_1 \ar@<1.5pt>[r]^s \ar@<-1.5pt>[r]_t & U_0}$  in $\widehat{\mathscr{C}}_\Lambda^{\mathsf{op}}$, the morphism $U_0\to \mathcal X$ (Example~\ref{E:gpd-proj})  is not in general defined by an object over $U_0$ (as $U_0$ is not in general in $\mathscr C^{\mathsf {op}}_\Lambda$).  However, if $\mathfrak m$ is the maximal ideal of $U_0$, and $V_k$ is the vanishing locus of $\mathfrak m^j$ in $U_0$, we can write $U_0=\varprojlim_k V_k$ with each $V_k$ in $\mathscr C^{\mathsf{op}}_{\Lambda}$ and there is a family of compatible morphisms in $V_k \to \mathcal X$ defined by objects of $\mathcal X$ over $V_k$---in other words, an element of $\varprojlim_k \mathcal X(U_0/\mathfrak m^j)$.   Such an element is  called a formal element of $\mathcal X$ over $U_0$ (see Definition \ref{D:FormEl}). 
 We will return to this topic  again in \S \ref{S:algebraization}.  
Note finally that if $\xymatrix@C=10pt{U_1 \ar@<1.5pt>[r]^s \ar@<-1.5pt>[r]_t & U_0}$ is a smooth groupoid object, then the morphism $U_0\to \mathcal X$ is formally smooth (Proposition \ref{P:Gr-Stack-Pres}).
 \end{rem}

The following theorem says, essentially, that a stack looks like an algebraic stack in a formal neighborhood of a point if and only if it is homogeneous:

\begin{teo}[{Schlessinger~\cite[Thm.~2.11]{Schlessinger}}, {Rim~\cite[Thm.~VI.2.17]{sga7-1}}] \label{T:Rim}
\
Let $\mathcal X$ be a fibered category over $\mathscr{C}_{\Lambda}^{\mathsf{op}}$ such that $\mathcal X(k)$ is a single point.  Then $\mathcal X$ is prorepresentable by a smooth groupoid object of $\widehat{\mathscr C}_\Lambda^{\mathsf{op}}$ if and only if $\mathcal X$ is homogeneous and $T^{-1}_{\mathcal X}$ and $T^0_{\mathcal X}$ are finite dimensional $k$-vector spaces.
\end{teo}

One direction of the implications in the theorem is clear: it is quite easy to see that the stack associated to a
 a smooth groupoid object $\xymatrix@C=10pt{U_1 \ar@<1.5pt>[r]^s \ar@<-1.5pt>[r]_t & U_0}$ of $\widehat{\mathscr C}_\Lambda^{\mathsf{op}}$
 is homogeneous, since  the stacks associated to the $U_i$  are homogeneous.  The finite dimensionality of $T^{-1}_{\mathcal X}(\xi)$ and $T^0_{\mathcal X}(\xi)$ follows from the finite dimensionality of the tangent spaces of stacks represented by objects of $\widehat{\mathscr{C}}_\Lambda^{\mathsf{op}}$, which consists of objects formally of finite type.

\subsection{Local finite presentation}
\label{S:lfp}

The definition of morphisms locally of finite presentation was given in \S \ref{S:lfp2}.  The proof of the following lemma is formal:

\begin{lem} \label{L:lfp-comp}
	\begin{enumerate}[(i)]
		\item Suppose that $g : \mathcal Y \rightarrow \mathcal Z$ is locally of finite presentation.  Then $f : \mathcal X \rightarrow \mathcal Y$ is locally of finite presentation if and only if $gf : \mathcal X \rightarrow \mathcal Z$ is locally of finite presentation.
		\item The base change of a morphism that is locally of finite presentation is also locally of finite presentation.
	\end{enumerate}
\end{lem}

There is a repertoire of techniques for proving moduli problems are locally of finite presentation to be found in \cite{EGAIV3}.  We combine these with Lemma~\ref{L:lfp-comp} to prove that the stack of Higgs bundles is locally of finite presentation.

\begin{lem} \label{L:curve-lfp}
The stack of proper nodal curves is locally of finite presentation.
\end{lem}
\begin{proof}
	Let $\mathcal N$ denote the stack of proper nodal curves.  Suppose that a commutative ring $A$ is the filtered colimit of commutative rings $A_i$.  Put $S = \operatorname{Spec} A$ and $S_i = \operatorname{Spec} A_i$.  Let $C$ be an element of $\mathcal{N}(A)$.  That is, $C$ is a flat family of nodal curves over $A$.  We want to show that $C$ is induced by base change from a nodal curve $C_i$ over $A_i$ for some $i$, and that (up to increasing the index $i$) this curve is unique up unique isomorphism.

First of all, $C$ is of finite presentation over $S$.  By \cite[Thm.~(8.8.2)~(ii)]{EGAIV3}, there is an index $i$ and a scheme $C_i$ of finite presentation over $A_i$ such that $C = C_i \mathop{\times}_{S_i} S$.  It follows from~\cite[Thm.~(8.8.2)~(i)]{EGAIV3} that $C$ is unique up to unique isomorphism and increase of the index $i$.  By \cite[Thm.~(8.10.5)~(xii)]{EGAIV3}, we can arrange for $C_i$ to be proper over $S_i$ by replacing $i$ with a larger index.

Now $C$ has a cover by open subsets $U$, each of which admits an \'etale map $U \rightarrow V$, where $V = \operatorname{Spec} A[x,y] / (xy - t_U)$ for some $t_U \in A$.  Refining the cover, we can assume that the open subsets $U$ are affine, and hence of finite presentation over $A$.  As $C$ is quasicompact, this cover can be assumed finite, so by increasing $i$, we can assume that $t_U$ appears in $A_i$ for all $i$.  By increasing $i$ still further, we can assume each $U$ is the preimage in $C$ of an open subset $U_i \subseteq C_i$~\cite[Prop.~(8.6.3)]{EGAIV3} and that these open subsets cover $C_i$~\cite[Thm.~(8.10.5)~(vi)]{EGAIV3}.  Now $V = V_i \mathop{\times}_{S_i} S$ so by~\cite[Thm.~(8.8.2)~(i)]{EGAIV3}, the map $U \rightarrow V$ is induced from a map $U_i \rightarrow V_i$ over $S_i$, at least after increasing $i$ still further.  By \cite[Prop.~(17.7.8)~(ii)]{EGAIV4}, we can ensure that the map $U_i \rightarrow V_i$ is \'etale, at least after increasing $i$.  Then $C_i$ is a family of nodal curves over $S_i$, and the proof is complete.
\end{proof}

\begin{lem} \label{L:vb-lfp}
Let $\mathcal E$ be the stack of pairs $(C, E)$ where $C$ is a proper nodal curve and $E$ is a vector bundle on $C$ and let $\mathcal N$ be the stack of proper nodal curves.  The projection $\mathcal E \rightarrow \mathcal N$ is locally of finite presentation.
\end{lem}
\begin{proof}
As before, $A$ is the filtered colimit of commutative rings $A_i$.  We suppose that $(C, E)$ is an $A$-point of $\mathcal E$ and that $C$ is induced from a nodal curve $C_i$ over $A_i$.  We want to show that, up to increasing $i$, the vector bundle $E$ is induced from a unique (up to unique isomorphism) vector bundle $E_i$ over $C_i$.  As $E$ is of finite presentation, we may increase $i$ to obtain a quasicoherent sheaf of finite presentation $E_i$ over $C_i$ inducing $E$ by pullback \cite[Thm.~(8.5.2)~(ii)]{EGAIV3}.  The uniqueness of $E_i$ follows from \cite[Thm.~(8.5.2)~(i)]{EGAIV3}.  Increasing $i$ still further, we can ensure that $E_i$ is a vector bundle \cite[Prop.~(8.5.5)]{EGAIV3}.
\end{proof}

\begin{lem} \label{L:map-lfp}
Let $\mathcal F$ be the stack of tuples $(C, E, F, \sigma)$ where $C$ is a nodal curve, $E$ and $F$ are vector bundles on $C$, and $\sigma : E \rightarrow F$ is a morphism of vector bundles.  Let $\mathcal G$ be the stack of triples $(C,E,F)$ as above, and $\mathcal F \rightarrow \mathcal G$ the projection forgetting $\sigma$.  Then $\mathcal F$ is locally of finite presentation over $\mathcal G$.
\end{lem}
\begin{proof}
We assume $A = \varinjlim A_i$ is a filtered colimit of commutative rings and  that $(C,E,F,\sigma) \in \mathcal F(A)$ and $(C_i,E_i,F_i) \in \mathcal G(A_i)$ induces $(C,E,F) \in \mathcal G(A)$.  By an immediate application of \cite[Thm.~(8.5.2)~(i)]{EGAIV3}, we discover that, after increasing $i$, we can find $\sigma_i : E_i \rightarrow F_i$ inducing $\sigma$ and that $\sigma_i$ is unique up to further increasing~$i$.
\end{proof}

Combining Lemmas~\ref{L:lfp-comp},~\ref{L:curve-lfp},~\ref{L:vb-lfp}, and~\ref{L:map-lfp}, we obtain

\begin{cor} \label{C:curve-vb-morph-lfp}
The stack of tuples $(C,E,F,\sigma)$ where $C$ is a nodal curve, $E$ and $F$ are vector bundles on $C$, and $\sigma : E \rightarrow F$ is a morphism of vector bundles is locally of finite presentation.
\end{cor}

\begin{cor}
The stack of Higgs bundles is locally of finite presentation.
\end{cor}
\begin{proof}
	This is deduced from the previous corollary by the same argument as in Lemma \ref{L:Higgs-hmg}.
\end{proof}

\subsection{Integration of formal objects}
\label{S:algebraization}

Let $A$ be a complete noetherian  local ring with maximal ideal $\mathfrak{m}$.  A formal $A$-point of $X$ is an object of the inverse limit $\varprojlim_k X(A/\mathfrak{m}^j)$.  It can be checked easily that for any scheme $X$, the function 
\begin{equation} \label{E:eff}
X(A) \rightarrow \varprojlim X(A/\mathfrak{m}^j)
\end{equation}
is a bijection.  It is only slightly more difficult to verify this gives an equivalence   when $X$ is an algebraic stack, provided that one interprets the limit of groupoids correctly.  One efficient way of describing the limit is as the category of sections of $\mathcal X$ over the subcategory
\begin{equation*}
\operatorname{Spec}(A/\mathfrak m) \rightarrow \operatorname{Spec}(A/\mathfrak m^2) \rightarrow \operatorname{Spec}(A/\mathfrak m^3) \rightarrow \cdots
\end{equation*}
of $\mathscr C^{\mathsf{op}}_{\Lambda}$.

Lifting a formal $A$-point of $X$ to an $A$-point of $X$ may be seen as an analogue of integrating a tangent vector to a curve.  We must require a formal point, as opposed to merely a tangent vector, because not every tangent vector can be integrated on a singular space.\footnote{Many formulations of Artin's criterion only require~\eqref{E:eff}, when applied to a stack $\mathcal X$,  to have dense image.  This strengthens the analogy to integrating tangent vectors, since there is not a unique curve with a given tangent vector.  However, it is generally no more difficult to prove~\eqref{E:eff} is an equivalence than it is to prove it has dense image.}

\begin{dfn}[Integrating formal points] \label{D:FormEl}
Let $\mathcal X$ be a fibered category over the category of schemes and let $A$ be a complete noetherian  local ring with maximal ideal $\mathfrak{m}$.  By a \emph{formal $A$-point} of $\mathcal X$ we mean an object of the category
\begin{equation*}
\varprojlim_j \mathcal X(A/ \mathfrak{m}^j) .
\end{equation*}
We say that a formal $A$-point of $\mathcal X$ can be \emph{algebraized}, or that it can be \emph{effectivized}, or that it is \emph{integrable} if it lies in the essential image of the functor
\begin{equation*}
\mathcal X(A) \rightarrow \varprojlim_j \mathcal X(A/\mathfrak{m}^j) .
\end{equation*}
If every formal $A$-point  of $\mathcal X$ can be algebraized, for every complete noetherian  local ring $A$, then we say \emph{formal objects of $\mathcal X$ can be algebraized}, or \emph{can be effectivized}, or \emph{are integrable}.
\end{dfn}

For a long time, the main algebraization theorem was Grothendieck's existence theorem, which asserts that formal objects of the stack of coherent sheaves on a proper scheme can be algebraized:

\begin{teo}[{Groth.~existence \cite[Thm.~8.4.2]{FGAe}}, {\cite[Thm.~(5.1.4)]{EGAIII1}}] \label{T:GET}
Let $X$ be a proper scheme over $S = \operatorname{Spec} A$ with $A$ a complete noetherian  local ring with maximal ideal $\mathfrak{m}$.  For each $j$, let $S_j = \operatorname{Spec} A / \mathfrak{m}^j$ and let $X_j = X \mathop{\times}_S S_j$.  Then
\begin{equation*}
\operatorname{Coh}(X) \rightarrow \varprojlim_j \operatorname{Coh}(X_j)
\end{equation*}
is an equivalence of categories.
\end{teo}

Very recently, Bhatt~\cite{Bhatt} and Hall and Rydh \cite{Hall-Rydh-TD} have proved strong new integration theorems extending Grothendieck's.  Since Grothendieck's existence theorem will suffice for the stack of Higgs bundles on curves, we will not need to state these new results.

\begin{lem} \label{L:int-curves}
Formal families of proper nodal curves can be algebraized.
\end{lem}
\begin{proof}
	Let $\mathcal N$ denote the stack of proper nodal curves.
	Let $A$ be a complete noetherian  local ring with maximal ideal $\mathfrak{m}$.  Set $A_j = A/ \mathfrak{m}^{j+1}$ and $S_j = \operatorname{Spec} A_j$ and suppose that $C_j \in \mathcal{N}(S_j)$ are the components of a formal $A$-point of $\mathcal{N}$.  Then $C_0$ is a curve over the field $A_0$.  Pick a very ample line bundle $L_0$ on $C_0$ with $H^1(C_0, L_0) = 0$.  Extend $L_0$ inductively to a compatible system of line bundles $L_j$ on each $C_j$:  By Lemma \ref{L:VBdef}, the obstruction to extending $L_j$ to $L_{k+1}$ lies in
\begin{equation*}
H^2 \bigl(C, \pi^\ast (\mathfrak{m}^j / \mathfrak{m}^{j+1}) \bigr) = \mathfrak{m}^j /\mathfrak{m}^{j+1} \otimes H^2(C_0, \mathcal{O}_{C_0}) = 0.
\end{equation*}
Then $V_j = \pi_\ast L_j$ is a (trivial) vector bundle on $S_j$.  Moreover, by Lemma \ref{L:Sdef}, the obstruction to extending a section of $L_j$ to a section of $L_{k+1}$ lies in
\begin{equation*}
H^1 \bigl(C, \pi^\ast (\mathfrak{m}^j / \mathfrak{m}^{j+1}) \otimes L_0 \bigr) = \mathfrak{m}^j / \mathfrak{m}^{j+1} \otimes H^1(C_0, L_0) = 0
\end{equation*}
so $V_j \big|_{S_\ell} = V_\ell$ for $\ell \leq k$.  There is therefore a (trivial) vector bundle $V$ on $S$ whose restriction to $S_j$ is $V_j$ for all $k$.  The complete linear series of the $L_j$ give a system of closed embeddings $C_j \rightarrow \mathbf{P}(V_j)$.  We may regard the structure sheaves $\mathcal{O}_{C_j}$ as a compatible system of quotients of the structure sheaf of $\mathbf{P}_{A_j}(V_j)$, so by Grothendieck's existence theorem they can be algebraized to a quotient $\mathcal{O}_C$ of the structure sheaf of $\mathbf{P}_A(V)$.  Let $C$ be the corresponding closed subscheme of $\mathbf{P}_A(V)$.

By construction $C$ is proper over $\operatorname{Spec} A$.  It is also flat by the infinitesimal criterion for flatness~\cite[Ex.~6.5]{Eisenbud}.  Therefore it is a family of nodal curves (see Rem.~\ref{R:NC}).
\end{proof}

\begin{rem}
The above argument can be used more generally to show that formal families of proper schemes can integrated if the central fiber $X$ has $H^2(X, \mathcal O_X) = 0$.  See \cite[Thm.~2.5.13]{sernesi} or \cite[Thm.4]{Grothendieck-GAGA} for different ways of organizing the ideas.
\end{rem}

\begin{lem}\label{L:C-EF-Int}
Let $\mathcal{F}$ be the stack of quadruples $(C,E,F,\sigma)$ where $C$ is a nodal curve, $E$ and $F$ are vector bundles on $C$, and $\sigma : E \rightarrow F$ is a morphism of vector bundles.  Then formal objects of $\mathcal{F}$ can be integrated.
\end{lem}
\begin{proof}
Suppose $A$ is a complete noetherian  local ring with maximal ideal $\mathfrak{m}$, set $A_j = A / \mathfrak{m}^{j+1}$.  Given a formal family $(C_j,E_j,F_j,\sigma_j) \in \mathcal{F}(A_j)$, we seek an element  $(C,E,F,\sigma) \in \mathcal{F}(A)$ inducing it.  We may find $C$ by the Lemma~\ref{L:int-curves}.  Note that $C$ is projective over $A$ and the $E_j$ and $F_j$ are each formal families of coherent sheaves over $C$, so by Grothendieck's existence theorem (Theorem~\ref{T:GET}) they can be algebraized to coherent sheaves $E$ and $F$ over $C$.  Moreover, both $E$ and $F$ are flat, by the infinitesimal criterion for flatness, so they are vector bundles.  One more application of Grothendieck's existence theorem algebraizes the family of homomorphisms of coherent sheaves $\sigma_j : E_j \rightarrow F_j$ to a homomorphism $\sigma : E \rightarrow F$ and the lemma is complete.
\end{proof}

\subsection{Artin's theorems on algebraization and approximation}

The question of integrability concerns objects of  a stack $\mathcal X$ lying over the spectrum of a complete local algebra over a field.  
With the Schlessinger--Rim theorem and integration, we can factor any morphism $\operatorname{Spec} k \rightarrow \mathcal X$ through a map $\operatorname{Spec} A \rightarrow \mathcal X$ such that $A$ is a complete noetherian local ring and is \emph{formally smooth} over $\mathcal X$.  This is tantalizingly close to showing $\mathcal X$ is an algebraic stack:  we need too find a factorization that is genuinely smooth over $\mathcal X$.

The distinction between smoothness and formal smoothness is finiteness of presentation, so we need to find a finite type ring $B$ \emph{that is still formally smooth} over $\mathcal X$ at the $k$-point, and induces $A$ by completion at a point (this will give smoothness of the map $B\to \mathcal X$ at the given $k$-point of $B$; to extend to smoothness on an open neighborhood of the $k$-point, see \S \ref{S:ArtCritPf}).  

This may be the subtlest part of the proof of Artin's criterion.  It is resolved by Artin's \emph{approximation theorem}, proved originally by Artin~\cite[Thm.~1.12]{Artin-approximation} with some technical hypotheses, and in its current form by B.\ Conrad and J.\ de Jong~\cite[Thm.~1.5]{CdJ} using a spectacular theorem of Popescu~\cite[Thm.~1.3]{Popescu}.

The algebraization  theorem (Theorem \ref{T:Artin-Alg}) asserts that under our assumption that $\mathcal X$ is locally of finite presentation (Theorem \ref{T:Artin-criteria}\eqref{T:AClfp}), 
given $\xi \in \mathcal{X}(A)$ in an appropriate  complete local ring with residue field $k$, one can find a finite type $k$-algebra  $B$ with a marked point (maximal ideal $\mathfrak n$), whose completion  at  the marked point is $A$, and an element $\eta \in \mathcal{X}(B)$ that agrees to a specifiable finite order with $\xi$.  That is one may select $j$ beforehand and then find $\eta$ such that the restriction of $\eta$ to $B / \mathfrak{n}^j \simeq A / \mathfrak{m}^j$ agrees with the restriction $\xi_j$ of $\xi$.  Even though $\eta \in \mathcal{X}(B)$ does not necessarily restrict to ${\xi} \in \mathcal{X}(A)$, it will differ from ${\xi}$ only up to an automorphism of $A$ that is the identity modulo $\mathfrak{m}^j$.  In particular, it will still be formally smooth at the closed point (see the proof of Theorem \ref{T:Artin-Alg}).

\vskip .2 cm 
We begin by introducing notation.
Let $k$ be a field and let $\Lambda$ be a complete noetherian local ring with residue field $k$.    Recall, we denote by $\mathscr C_\Lambda$ the category of local artinian $\Lambda$-algebras with residue field $k$, and by  $\widehat{\mathscr{C}}_\Lambda$  the category of complete local $\Lambda$-algebras with residue field $k$ that are formally of finite type over $\Lambda$.

\begin{dfn}[Algebraization over a scheme of finite type]\label{D:Alg-Fin-Type} Let $\mathcal X$ be a CFG over $\mathsf S$ and 
let $(A,\mathfrak m) \in   \widehat{\mathscr{C}}_\Lambda$.  We say a formal element $\hat \eta $ of $\mathcal X$ over $\operatorname{Spec}A$, i.e., an object  of $\varprojlim_j\mathcal X(A/\mathfrak m^j)$, can be \emph{algebraized over a scheme of finite type} if there exist:
\begin{itemize}
\item a finitely generated $\Lambda$-algebra $B$,   

\item a $k$-point $s:\operatorname{Spec}k\to \operatorname{Spec}B$ corresponding to a maximal ideal $\mathfrak n\subseteq B$, and,

\item  an object $\theta\in \mathcal X(B)$,
\end{itemize}
 such that 
 \begin{enumerate}
\item   $A=\widehat B_{\mathfrak n}$, the completion of $B$ at $\mathfrak n$, and,
\item the formal element $\hat \eta$ is isomorphic to the image of $\theta$ under the map 
$$
\mathcal X(B)\to \varprojlim_j\mathcal X(A/\mathfrak m^j)
$$ 
induced by the morphism $B\to B_{\mathfrak n}\to B_{\mathfrak n}/\mathfrak n^j=A/\mathfrak m^j$.  
\end{enumerate}

\end{dfn}

Artin's algebraization theorem asserts the following:

\begin{teo}[{Artin's algebraization theorem}] \label{T:Artin-Alg} \label{T:alg}
Let $\mathcal X$ be a CFG that is locally of finite presentation over an  excellent DVR or field $  \Lambda$.
Let $(A,\mathfrak m)$ be a complete local $\Lambda$-algebra with residue field $k$ that is formally of finite type over $\Lambda$. 
  If $\bar \eta \in \mathcal X(A)$ then there is a finite type $\Lambda$-algebra $B$ with a maximal ideal $\mathfrak n$ at whose completion $B$ is isomorphic to $A$, along with $\eta \in \mathcal X(B)$ inducing $\bar \eta \in \mathcal X(A)$.
\end{teo}

The proof is almost an immediate consequence of Artin's approximation theorem:

\begin{teo}[{Artin's approx.~\cite[Thm.~1.12]{Artin-approximation}, \cite[Thm.~1.5]{CdJ}}] \label{T:approx}
Let $\mathcal X$ be a CFG that is locally of finite presentation over an  excellent DVR or field $  \Lambda$.
Let $(A,\mathfrak m)$ be a complete local $\Lambda$-algebra with residue field $k$ that is formally of finite type over $\Lambda$. 
 Given $\bar \eta\in \mathcal X(A)$, there is a finitely generated $\Lambda$-algebra $B$ with $k$-point $s:\operatorname{Spec} k\to \operatorname{Spec} B$ corresponding to a maximal ideal $\mathfrak n$ and $\widehat B_{\mathfrak n}=A$, such that for any positive integer $j$ there is an element $\eta\in \mathcal X(B)$ (depending on $j$) such that the images of $\eta$ and $\bar \eta$ are isomorphic in $\mathcal X(A/\mathfrak m^j)$. 
\end{teo}

\begin{proof}[Proof of Theorem~\ref{T:alg} from Theorem~\ref{T:approx}]
Take $B$ and $\eta$ as in the Approximation theorem, with $j=2$.   In other words, $\eta_2$ and $\hat \eta_2$ are isomorphic in $\mathcal X(A/\mathfrak m^2)$. Now, inductively, using formal smoothness, we can use diagram~\eqref{E:alg}
\begin{equation} \label{E:alg} \vcenter{ \xymatrix{
\operatorname{Spec} A/\mathfrak m^j \ar[d] \ar[r]^{\psi_{j-1}} & \operatorname{Spec} A \ar[dd]^{\bar \eta} \\
\operatorname{Spec} A/\mathfrak m^{j+1} \ar@{-->}[ur]_{\psi_j} \ar[d] \\
\operatorname{Spec} B \ar[r]^\eta & \mathcal X
}} \end{equation}
to construct for all $j$ morphisms $\psi_j:\operatorname{Spec}A/\mathfrak m^{j+1}\to \operatorname{Spec} A$ so that $\bar \eta \psi_j\cong \eta_j$.  This induces  a morphism  $\psi:\operatorname{Spec}A \to \operatorname{Spec} A$ so that $\bar \eta \psi \cong \eta$.  But $\psi$ is an isomorphism modulo $\mathfrak m^2$, so it must be surjective, and a surjective endomorphism of a noetherian ring is an isomorphism (e.g., \cite[Lem.~C.5]{sernesi}).
\end{proof}

\subsection{Algebraicity of the stack of Higgs bundles}
\label{S:rep-proof}

In 
Lemma \ref{L:CVBMhmg}, 
Lemma \ref{L:C-EF-TanAut},
Lemma \ref{L:C-EF-Int},
Corollary \ref{C:curve-vb-morph-lfp}, and 
Lemma \ref{L:C-EF-Obs-lfp}, 
  we have verified the conditions of Theorem~\ref{T:Artin-criteria} for the stack $\mathcal F$ that parameterizes proper nodal curves equipped with a homomorphism of vector bundles: 

\begin{teo} \label{T:cvbm-stack}
Let $\mathcal{F}$ be the stack of quadruples $(C,E,F,\sigma)$ where $C$ is a nodal curve, $E$ and $F$ are vector bundles, and $\sigma : E \rightarrow F$ is a morphism of vector bundles.  Then $\mathcal{F}$ is an     algebraic stack.
\end{teo}

Analogous arguments show that the stack of Higgs bundles satisfies the axioms, hence is algebraic.  Alternately, one may consider the stacks $\mathcal{E}_1$ of pairs $(C,E)$ where $C$ is a nodal curve and $E$ is a vector bundle on $C$ and $\mathcal{E}_2$ of triples $(C,E,F)$ where $C$ is a nodal curve and $E$ and $F$ are vector bundles on $C$.  (Note that $\mathcal{E}_2 = \mathcal{E}_1 \mathop{\times}_{\mathcal{M}} \mathcal{E}_1$.)  Then we can construct the stack of Higgs bundles as the fiber product
$\mathcal{E}_1 \mathop{\times}_{\mathcal{E}_2} \mathcal{F}$
where the map $\mathcal{E}_1 \rightarrow \mathcal{E}_2$ sends $(C,E)$ to $(C,E,E \otimes \omega_C)$.  We therefore have

\begin{cor}
The stack of Higgs bundles on proper nodal curves is an     algebraic stack.
\end{cor}

\subsection{Outline of the proof of Artin's criterion}\label{S:ArtCritPf}

We will briefly summarize the proof of Theorem~\ref{T:Artin-criteria}.   Since assumption Theorem~\ref{T:Artin-criteria}\eqref{T:ACstack}  is that $\mathcal X$ is  a stack, the key point is to establish the existence of a smooth  representable covering of $\mathcal X$ by a scheme. 
The basic idea of the proof is to begin with an arbitrary point $\xi_0 \in \mathcal{X}(k)$, valued in a field $k$, and find a smooth neighborhood $U \rightarrow \mathcal X$ of this point by enlarging $\operatorname{Spec} k$ until it is smooth over $\mathcal X$.  In more concrete terms, $U$ will be a versal deformation of $\xi_0$.  Repeating this for every point of $\mathcal{X}$ and taking a disjoint union of the different $U$ gives a smooth cover of $\mathcal{X}$ by a scheme. 

We now explain in more detail how the arguments in the previous sections imply the existence of the schemes $U$.  

\vskip .2 cm 
\noindent \emph{Versality at a point}
\vskip .2 cm 

\noindent By the Schlessinger--Rim theorem (see Theorem \ref{T:Rim}), the homogeneity of $\mathcal{X}$ (Theorem \ref{T:Artin-criteria}\eqref{T:AChomog}) and the finite dimensionality of $T_{\mathcal X}^{-1}(\xi_0)$ and $T_{\mathcal X}^{0}(\xi_0)$ (Theorem \ref{T:Artin-criteria}\eqref{T:ACTan})
guarantee that $\mathcal{X}$ is prorepresentable at $\xi_0$.  That is, there is a formal groupoid $\widehat{V}_1 \rightrightarrows \widehat{V}_0$, with $\widehat{V}_i = \operatorname{Spec} \widehat{R}_i$ for complete noetherian  local rings $\widehat{R}_i$, whose associated CFG (Definition \ref{D:CFG-GrOb}) agrees  with $\mathcal{X}$ on infinitesimal extensions of $\xi_0$.  

This gives a \emph{formal} morphism $\widehat{V}_0 \rightarrow \mathcal{X}$ that is formally smooth (Example \ref{E:gpd-proj}).
  In other words, we have compatible elements $\xi_j \in \mathcal{X}(\operatorname{Spec}\widehat{R}_0 / \mathfrak{m}^{j+1})$, where $\mathfrak{m}$ is the maximal ideal of $\widehat{R}_0$.  The assumption that formal objects of $\mathcal{X}$ integrate  uniquely (Theorem \ref{T:Artin-criteria}\eqref{T:ACint}) guarantees that this formal morphism comes from a genuine morphism $\widehat{V}_0 \rightarrow \mathcal{X}$, i.e., from an element of $\xi \in \mathcal{X}(\operatorname{Spec}\widehat{R}_0)$.

Now the map $\widehat{V}_0 \rightarrow \mathcal{X}$ is formally smooth at $\xi_0$, but it is not of finite type.  To remedy this we can use the local finite presentation of $\mathcal{X}$ (Theorem \ref{T:Artin-criteria}\eqref{T:AClfp}), which ensures that $\widehat{V}_0 \rightarrow \mathcal{X}$ must factor through some scheme $V$ of finite type.  Unfortunately, it is not clear we can exert any control over $V$, even formally, to guarantee it is formally smooth over $\mathcal{X}$.  Fortunately, we may rely on Artin's algebraization theorem (Theorem~\ref{T:alg}) to ensure that $V$ is still formally smooth over $\mathcal X$ at the central point.

This does not yet guarantee that the map $V \rightarrow \mathcal X$ is smooth:  we only have formal smoothness at one point.  The next step will be to show formal smoothness at a point implies formal smoothness nearby.

\vskip .2 cm 
\noindent \emph{Versality in a neighborhood}
\vskip .2 cm 

\noindent  Write $V = \operatorname{Spec} R$.  We know that $R$ is of finite type and we now have a map $\eta : V \rightarrow \mathcal{X}$ that we know to be formally smooth at a point lifting $\xi$.  All that is left is to find an open neighborhood of this point at which the map is actually smooth.  For this we use the obstruction theory $T^1_\mathcal{X}$ (Theorem \ref{T:Artin-criteria}\eqref{T:ACobs}).  

Hall shows that the existence of a locally presentable obstruction theory representing the automorphisms, deformations, and obstructions of $\mathcal X$ implies that there is a \emph{relative} obstruction theory $T^1_{V/\mathcal X}$ for $V$ over $\mathcal X$ and that $T^1_{V/\mathcal X}$ is a \emph{coherent functor}.

As $V$ is formally smooth over $\mathcal{X}$ at $\xi_0$, we know that $T^1_{V/\mathcal{X}}(\eta, J) = 0$ for any quasicoherent sheaf $J$ on $V$ supported at $\xi_0$.  By a theorem of Ogus and Bergman~\cite[Thm.~2.1]{OB}, coherent functors over noetherian  rings satisfy an analogue of Nakayama's lemma, which guarantees that $T^1_{W/\mathcal{X}}(\eta \big|_{W}, J) = 0$ for all quasicoherent sheaves on $W$, where $W$ is the localization of $V$
 at the point $\xi_0$.  Making use of the local finite presentation of $\mathcal{X}$ (Theorem \ref{T:Artin-criteria}\eqref{T:AClfp}), Hall shows that this implies $T^1_{U/\mathcal{X}}(\eta \big|_U, J) = 0$ for all quasicoherent $J$ on an open subset $U \subseteq V$ containing $W$.  Then $U \rightarrow \mathcal{X}$ is locally of finite presentation and satisfies the formal criterion for smoothness.

\vskip.2cm
\noindent \emph{Bootstrapping to representability}
\vskip.2cm

To conclude that $U \rightarrow \mathcal X$ is smooth, we only need to show that it is representable by algebraic spaces.  This is proved by observing that the hypotheses of Theorem \ref{T:Artin-criteria} on $\mathcal X$ imply relative versions of themselves for the map $U \rightarrow \mathcal X$ (cf.\ Lemma \ref{L:RelHomog}, Lemma \ref{L:lfp-comp}, and the relative obstruction theory $T^1_{U/\mathcal X}$ mentioned above).  The same hypotheses then hold for any base change $U_Z \rightarrow Z$ via any map $Z \rightarrow \mathcal X$.  Taking $Z$ to be a scheme and viewing $U_Z$ as a sheaf in the \'etale topology on schemes over $Z$, this reduces the problem to showing that $U_Z$ is an algebraic space.  Now we can try to prove the theorem for $\mathcal X = U_Z$:  In effect, Theorem \ref{T:Artin-criteria} is reduced to the case where $\mathcal X$ is a \emph{sheaf} as opposed to a stack.

Now we have a scheme $U$ and a map $U \rightarrow \mathcal X$ that is formally smooth and locally of finite presentation, and we are faced with the same problem:  to show $U \rightarrow \mathcal X$ is representable by algebraic spaces.  But this time, the relative diagonal of $U$ over $\mathcal X$ is injective, so that if we iterate the process one more time, we discover once again that our task is to prove $U_Z$ is an algebraic space.  But $U_Z = U \mathbin\times_{\mathcal X} Z$ is a subsheaf of $U \times Z$, since the diagonal of $\mathcal X$ is injective.  Therefore taking $U \times Z$ to be the new base scheme and replacing $\mathcal X$ with $U \times Z$, we discover we may assume further that $\mathcal X$ is even a \emph{subsheaf} of the base scheme $S$.

But now, if $W \rightarrow \mathcal X$ is any map, we have 
\begin{equation*}
W \mathbin\times_{\mathcal X} U = W \mathbin\times_S U
\end{equation*}
since $\mathcal X$ is a subsheaf of $S$.  Since $U \rightarrow S$ is schematic---it is a morphism of schemes, after all---so must be $U \rightarrow \mathcal X$.


\makeatletter

\setcounter{section}{0}
\setcounter{subsection}{0}
\gdef\thesection{\@Alph\c@section}

\makeatother

\section{Sheaves, topologies, and descent} \label{S:more-isotriv}

We collect a few more technical topics surrounding the subject of descent.  In particular, we show that the presence of nontrivial automorphisms \emph{always} prevents an algebraic stack from having a representable sheaf of isomorphism classes (Corollary \ref{C:IfAutNotRep}), we give a few more technically efficient ways of thinking about descent (\S \ref{S:MoreDesc}), we discuss a bit more about saturations of pretopologies (\S \ref{S:top}) and we recollect and example of Raynaud showing that genus~$1$ curves do not form a stack unless one admits algebraic spaces into the moduli problem (\S \ref{S:Raynaud}).

\subsection{Torsors and twists}\label{S:TorsTwist}

The point of this section is to explain how nontrivial automorphisms give rise to nontrivial families, and prevent a stack from having a representable sheaf of isomorphism classes.  In other words, moduli   functors parameterizing isomorphism classes of  objects with nontrivial automorphism groups are never representable.

\subsubsection{Locally trivial families and cohomology}\label{S:LocTrivH}
In general, one can build a locally trivial family with fiber $X$ over a base $S$ from cohomology classes in $H^1\bigl(S, \operatorname{Aut}(X)\bigr)$ (\cite[Tag 02FQ]{stacks}).  If one represents the cohomology class with a \v{C}ech cocycle then it is a recipe for assembling the family from trivialized families on open subsets.  The cocycle condition is precisely the one necessary to ensure that the family can be glued together, while the coboundaries act via isomorphisms.  The cocyle condition appears in the definition of the gluing condition for a stack in \S\ref{S:Descent} for precisely this reason.

\subsubsection{Locally trivial families from torsors over the base} \label{S:LocTrivPi}
Example \ref{E:isotriv1} describes a special case of a standard  method for constructing locally trivial families is via automorphisms, and covers of the base.  Namely, given  schemes $X_1$ and $S$, and an \'etale principal $\Gamma$-bundle $\widetilde S\to S$ for some discrete group $\Gamma$, then for each   $$\phi:\Gamma \to \operatorname{Aut}(X_1)$$ one can construct a (\'etale) locally trivial family $X/S$ with fibers isomorphic to $X_1$ by setting $X=(X_1\times \widetilde S)/\Gamma $, where the quotient is via the (free)  diagonal action induced by $\phi$.  Alternately, this can be constructed as the space of equivariant morphisms from $\widetilde S$ to $X_1$.  Using the previous remark, this  defines a map $$\operatorname{Hom}(\Gamma,\operatorname{Aut}(X_1))\to H^1(S,\operatorname{Aut}(X_1))$$ (of course, the cohomology should be taken in a topology in which $\widetilde S$ is a torsor).    
When $\widetilde S\to S$ is a universal cover, we obtain  a map $\operatorname{Hom}(\pi_1(S),\operatorname{Aut}(X_1))\to H^1(S,\operatorname{Aut}(X_1))$.

\begin{rem}\label{R:CompHPi} 
In the discussion above (\S \ref{S:LocTrivPi}),  it is possible for the isotrivial family $X$ to be trivial, even if the homomorphism $\phi$ is nontrivial.  We will highlight one situation in which the family $X$ can be assured to be nontrivial.
Let $X_1$ and $\widetilde S\to S$ be as in the discussion above (\S \ref{S:LocTrivPi}).
Let  $G=\operatorname{Aut}(X_1)$, and let $G'\le G$ be a discrete subgroup.  From \S  \ref{S:LocTrivH} and \ref{S:LocTrivPi} above, we obtain    
 a commutative diagram
\begin{equation}\label{E:torsormono} \vcenter{
\xymatrix{
\operatorname{Hom}(\pi_1(S),G') \ar@{^(->}[r] \ar@{=}[d] &\operatorname{Hom}(\pi_1(S),G)   \ar[r] \ar[d] &\operatorname{Hom}(\pi_1(S),\pi_0(G)) \ar@{=}[d] \\
H^1(S,G') \ar[r] & H^1(S,G) \ar[r] & H^1(S,\pi_0(G))  
}}
\end{equation}
where the horizontal maps are the natural maps, and 
the vertical  equalities on the left and right come from the fact that $G'$ and $\pi_0(G)$ are totally disconnected.
In summary, we can conclude that a locally trivial  family $X/S$ obtained from a homomorphism $\phi:\pi_1(S)\to G'$ will be a nontrivial family so long as the image of $\phi$ in $\operatorname{Hom}(\pi_1(S),\pi_0(G))$ is nontrivial.  Note that the kernel of the  natural map  $H^1(S,G') \to  H^1(S,G)$ corresponds to 
principal $G'$ bundles that can be equivariantly embedded in $G\times S$; i.e., the kernel is given by 
sections over $S$ of the coset space $(G/G')\times S$.  
\end{rem}

\begin{exa}[Disconnected automorphism groups] \label{E:IsoTriv}
If $\operatorname{Aut}(X_1)$ is disconnected, then there exist nontrivial locally trivial families $X/S$ with fibers isomorphic to $X_1$.   This follows immediately from the previous remark, so long as one can find a space $S$ with $\pi_1(S)=\mathbb Z$, or at least that has a nontrivial principal $\mathbb Z$-bundle. Indeed, for any  $\alpha\in \operatorname{Aut}(X_1)$ not  in the connected component of the identity, one would take $G'=\langle \alpha \rangle$ and set $X/S$ to be the locally trivial family corresponding to the map $\mathbb Z\to  \langle \alpha \rangle$ by $1\mapsto \alpha$.     In the complex analytic setting, we can simply take $S=\mathbb C^*$.  
  However, we can arrange for this even in the category of schemes by joining a pair of rational curves at two points;  this has fundamental group $\mathbb Z$ in the sense that it has a simply connected covering space (even in the Zariski topology!) with a simply transitive action action of $\mathbb Z$.
\end{exa}

\begin{exa}[Finite automorphism groups]\label{E:IsoTrivGen}
If $\operatorname{Aut}(X_1)$ is finite and nontrivial (or more generally, has a finite subgroup not contained in the connected component of the identity), one can easily  construct similar examples using finite covers.  Indeed, then $G=\operatorname{Aut}(X_1)$ contains a nontrivial finite cyclic subgroup $G'=\mu_n$ (not contained in the connected component of the identity).  Let $S=\mathbb C^*$ and $\widetilde S\to S$ be the cyclic cover $z\mapsto z^n$, which is an \'etale principal $\mu_n$-bundle.  Then, as in  \eqref{E:torsormono} of Remark \ref{R:CompHPi} above,  we obtain a diagram 
\begin{equation*}
\vcenter{
\xymatrix@C=1em@R=1em{
\operatorname{Hom}(\mu_n,G') \ar@{=}[r] \ar[rd]&\operatorname{Hom}(\pi_1(S),G') \ar@{^(->}[r] \ar@{=}[d] &\operatorname{Hom}(\pi_1(S),G)   \ar@{->}[r] \ar[d] &\operatorname{Hom}(\pi_1(S),\pi_0(G)) \ar@{=}[d] \\
&H^1(S,G') \ar[r] & H^1(S,G) \ar[r] & H^1(S,\pi_0(G)).  
}}
\end{equation*}
Therefore, a generator of $\mu_n=\operatorname{Hom}(\mu_n,G')$  determines a nontrivial locally trivial family given explicitly by $X=(X_1\times \mathbb C^*)/\mu_n\to S=\mathbb C^*/\mu_n=\mathbb C^*$, where the quotient is by the diagonal action under the identifications $\mu_n\le \operatorname{Aut}(X_1)$, and $\mu_n\le \operatorname{Aut}(\mathbb C^*)$ acting by a primitive $n$-th root of unity.
\end{exa}

\begin{exa}[Isotrivial families of curves]  \label{E:isotriv}
For every $g$ there exists a relative curve $\pi:X\to S$ of genus $g$ that is isotrivial, but not isomorphic to a trivial family.  For $g=0$, any nontrivial ruled surface $X\to S$ provides an example.  The previous Example \ref{E:IsoTrivGen}, and Example \ref{E:isotriv1}, provide nontrivial isotrivial families for   $g\ge 1$. 
 For instance, for $g\ge 2$, one could begin with a hyperelliptic curve $\{ y^2 = f(x) \}$ 
   carrying the nontrivial action of $\mu_2$ sending $y$ to $-y$.  The construction in Example \ref{E:IsoTrivGen}  yields the family $\{ t y^2 = f(x) \}$, where $t$ is the coordinate on $S=\mathbb C^*$. 
   \end{exa}

\subsubsection{Twisting by a torsor}

All of the previous examples are special cases of the general process of twisting by a torsor. 

\begin{exa} [Twist by a torsor] \label{E:twist}
 Suppose that $Z$ is an $S$-point of a stack $\mathcal X$ and that the automorphisms group of $Z$ is a smooth group scheme $G$ over $S$ (in fact, a flat group scheme is enough if $\mathcal X$ is an algebraic stack).  Let $P$ be a $G$-torsor.  

As $G$ is smooth, $P$ is smooth over $S$, and hence has sections over some \'etale cover of $S$.  Therefore $P$ is covering in the \'etale topology.  (Note that $P\to S$ it is not necessarily an \'etale morphism,
 but this is not an obstacle to using it for descent, since we may use smooth descent; alternatively,   see Appendix~\ref{S:top}.)  

By descent, $\mathcal X(S)$ may be identified with the $G$-equivariant objects of $\mathcal X(P)$.  In particular, let $Z_P$ in $\mathcal X(P)$ be the pullback of $Z$ to $P$.  Then $Z \in \mathcal X(S)$ corresponds to $Z_P$ with the trivial action of $G$.  However, we can also ask $G$ to act on $Z_P=P\times_SZ$ by the given action of $G$ on $P$ and by automorphisms on $Z$, giving an object $Y:=P \mathbin\times_G Z:=(P\times_S Z)/G$
 in  $\mathcal X(S)$ by descent. We call $Y$ the twist of $Z$ by the torsor $P$.  

We note that the torsor $P$ can be recovered from a twist $Y$ of $Z$ as the sheaf $\mathscr Isom_{\mathcal X}(Z, Y)$.  Thus the twist is nontrivial if and only if the torsor $P$ was.  Moreover, one has an equivalence of categories between the full subcategory of $\mathcal X(S)$ consisting of twists of $Z$ and the category $\mathrm{B}G(S)$ consisting of $G$-torsors on $S$. 
\end{exa}

The following theorem was communicated to us by Jason Starr~\cite{Starr-MO}.
  It implies that if $\mathcal X$ is a stack in the fppf topology in which there are objects with nontrivial automorphisms, and $X$ is its associated sheaf of isomorphism classes, then $X$ cannot be representable by a scheme (Corollary \ref{C:IfAutNotRep}).

\begin{teo}\label{T:IfAutTors}
Let $G$ be an algebraic group of finite type over an algebraically closed field $k$.  Then there is a scheme $S$ of finite type  over $k$ and a nontrivial $G$-torsor over $S$.
\end{teo}
\begin{proof}
We suppose that $G$ is an algebraic group of finite type and that $H^1(S, G) = 0$ for every scheme $S$.  We wish to show that $G=0$.   We will break the proof into several steps.

\vskip .2 cm \noindent \textbf{Step 1:} \emph{$G$ is connected}.
  We will show that $G$ is connected by showing that any homomorphism $\mathbb Z \rightarrow G$ has image in the connected component of the identity.  To this end, as in Example \ref{E:IsoTriv}, 
  let $S$ be a $k$-scheme with a nontrivial $\mathbb Z$-torsor $P$.  For concreteness, let us take $S$ to be an irreducible rational curve with a single node (otherwise smooth) and take $P$ to be an infinite chain of copies of the normalization of $S$, attached at nodes.  
Then using the right hand side of the commutative diagram \eqref{E:torsormono}, we may conclude every homomorphism $\mathbb Z\to G$  composes to a morphism $\mathbb Z\to G\to \pi_0(G)$, with trivial image, and we are done.

\vskip .2 cm
We now assume always that $G$ is connected,  
and  proceed to consider  $G_{\text{red}} \subseteq G$, the maximal reduced closed subscheme (necessarily a subgroup since we work over an algebraically closed field, e.g.,  
 \cite[Tag 047R]{stacks}).
By 
Chevalley's theorem (see for instance \cite[Thm.~1.1]{conradchev} for a modern treatment),    
there is an exact sequence
\begin{equation*}
1 \rightarrow G_{\text{aff}} \rightarrow G_{\text{red}} \rightarrow A \rightarrow 0
\end{equation*}
where $G_{\text{aff}}$ is smooth connected and affine, and $A$ is an abelian variety.   Our next goal will be to show that $G_{\text{red}}=0$; we will do this in several steps.

\vskip .2 cm \noindent \textbf{Step 2:} $H^1(S,G_{\text{red}})=0$ \emph{for all $S$}.
  For any scheme $S$, we have an exact sequence:
\begin{equation*}
\operatorname{Hom}(S, G) \rightarrow \operatorname{Hom}(S, G/G_{\text{red}}) \rightarrow H^1(S, G_{\text{red}}) \rightarrow H^1(S, G)
\end{equation*}
We have assumed that $H^1(S, G) = 0$.  Furthermore, $G \rightarrow G/G_{\text{red}}$ is smooth (since $G_{\text{red}}$ is) and $G/G_{\text{red}}$ is artinian, so $G$ admits a (not necessarily homomorphic) section over $G/G_{\text{red}}$.  This implies $\operatorname{Hom}(S,G) \rightarrow \operatorname{Hom}(S, G/G_{\text{red}})$ is surjective, so $H^1(S, G_{\text{red}}) = 0$.

\vskip .2 cm \noindent \textbf{Step 3:} $G_{\text{aff}}$   \emph{is unipotent}.
Consider the exact sequence
\begin{equation*}
\operatorname{Hom}(\mathbb P^1, G_{\text{red}}) \rightarrow \operatorname{Hom}(\mathbb P^1, A) \rightarrow H^1(\mathbb P^1, G_{\text{aff}}) \rightarrow \operatorname{Hom}(\mathbb P^1, G_{\text{red}}) = 0.
\end{equation*}
As $A$ is an abelian variety, every map from $\mathbb P^1$ to $A$ is constant and therefore lifts to $G_{\text{red}}$.  It follows that $H^1(\mathbb P^1, G_{\text{aff}}) = 0$.
 Let $U \subseteq G_{\text{aff}}$ be the unipotent radical \cite[11.21, p.157]{Borel-AG}.  Consider the exact sequence:
 \begin{equation*}
H^1(\mathbb P^1, G_{\text{aff}}) \rightarrow H^1(\mathbb P^1, G_{\text{aff}}/U) \rightarrow H^2(\mathbb P^1, U)
\end{equation*}	
Since $U$ is an iterated extension of $\mathbb G_a$s \cite[Chap.~V, Cor.~15.5~(ii), p.~205]{Borel-AG}, we have that $H^2(\mathbb P^1, U) = 0$, whence $H^1(\mathbb P^1, G_{\text{aff}}/U) = 0$. 

 Let $T \subseteq G_{\text{aff}}/U$ be a maximal torus, and let $W$ be its Weyl group.  As $G_{\text{aff}}/U$ is reductive by definition \cite[Chap.~IV, 11.21, p.158]{Borel-AG}, there is an injection (in fact a bijection!)
\begin{equation*}
H^1(\mathbb P^1, T) / W \rightarrow H^1(\mathbb P^1, G_{\text{aff}}/U)=0 
\end{equation*}
by a theorem of Grothendieck \cite[Thm.~1.1]{Grothendieck} (note that Grothendieck works analytically, but his proof of the injectivity is valid algebraically \cite[Thm.~0.3]{MT}).  Thus $H^1(\mathbb P^1,T)$ is finite; although we do not need it, note that  since the Weyl group coinvariants of $H^1(\mathbb P^1, T)$ are trivial if and only if $H^1(\mathbb P^1, T)$ is, we can actually conclude immediately that $H^1(\mathbb P^1, T) = 0$.  But $T \simeq \mathbb G_m^r$ for some $r$, and $H^1(\mathbb P^1, \mathbb G_m) = \mathbb Z$.  Therefore $r = 0$.  That is, the maximal torus $T$ of $G_{\text{aff}}/U$ is trivial, so that $G_{\text{aff}}$ is unipotent \cite[Chap.~IV, Cor.~11.5, p.~148]{Borel-AG}. 

\vskip .2 cm \noindent \textbf{Step 4:} $G_{\text{red}}$   \emph{is affine}.
Now that we know $G_{\text{aff}}$ is unipotent, we argue that $A = 0$.  Let $S = \mathbb A^1 \smallsetminus \{ 0 \}$.  Consider the exact sequence
\begin{equation*}
H^1(S, G_{\text{red}}) \rightarrow H^1(S, A) \rightarrow H^2(S, G_{\text{aff}}) .
\end{equation*}
Since $S$ is affine, and $G_{\text{aff}}$ is unipotent and therefore  an iterated extension of  quasicoherent sheaves (associated to $\mathbb G_a$s), we have  $H^2(S, G_{\text{aff}}) = 0$.  This implies $H^1(S, A) = 0$.  Choose an integer $n$ relatively prime to the characteristic of $k$.  Consider the exact sequence
\begin{equation*}
\operatorname{Hom}(S, A) \xrightarrow{[n]} \operatorname{Hom}(S, A) \rightarrow H^1(S, A[n]) \rightarrow H^1(S, A) 
\end{equation*}
where $[n]$ denotes multiplication by $n$ and $A[n]$ is the $n$-torsion subgroup.  Since $S$ is rational, every map $S \rightarrow A$ is constant.  In particular, $\operatorname{Hom}(S,A) \xrightarrow{[n]} \operatorname{Hom}(S,A)$ is surjective.  Therefore $H^1(S, A[n])$ injects into $H^1(S,A) = 0$, and so is also $0$.  Now, $A[n] \simeq (\mathbb Z / n \mathbb Z)^{2g}$ where $g$ is the dimension of $A$.  We know that $H^1(S, \mathbb Z/n \mathbb Z) = \mathbb Z/ n \mathbb Z$, so we deduce that $g = 0$, and therefore $A = 0$.

\vskip .2 cm \noindent \textbf{Step 4:} $G_{\text{red}}=0$.
Now we know that $G_{\text{red}} = G_{\text{aff}}$ is affine and unipotent.  We can choose an injective homomorphism $G_{\text{red}} \subseteq G'$ where $G'$ is smooth affine and reductive (i.e., embed it in an appropriate $\operatorname{GL}_n$).  Consider the $G_{\text{red}}$-torsor $G'$ over $G'/G_{\text{red}}$.  This must be trivial, since $H^1(G'/G_{\text{red}}, G_{\text{red}}) = 0$ from Step 2, so $G' \simeq G'/G_{\text{red}} \times G_{\text{red}}$ as a scheme.  But $G'$ is affine, so this implies $G'/G_{\text{red}}$ is affine.  Therefore $G_{\text{red}}$ is reductive, by Matsushima's criterion \cite[Thm.~A]{richardson77}.  As it is also unipotent, this means $G_{\text{red}} = 0$.

\vskip .2 cm \noindent \textbf{Step 5:} $G=0$.  
Since $G_{\text{red}}=0$, this  means that $G = G/G_{\text{red}}$ is the spectrum of an artinian local ring.  But we can once again choose a closed embedding $G \subseteq G'$ where $G'$ is smooth and affine (for example, let $G$ act on its ring of regular functions).  The quotient $G'/G$ is also reduced and affine.  But $H^1(G'/G, G) = 0$ by assumption, so that $G' \simeq G'/G \times G$ as a scheme.  As $G'$ is reduced, this means $G$ is reduced, and therefore $G = 0$.
\end{proof}

\begin{cor}\label{C:IfAutNotRep}
Suppose that $X$ is the presheaf of isomorphism classes in an algebraic stack $\mathcal X$ such that all objects over algebraically closed fields have automorphism groups that are algebraic groups of finite type over the field.  If $X$ is a sheaf in the fppf topology then $X \simeq \mathcal X$.  In particular, if $\mathcal X$ admits an object over an algebraically closed field $k$  with a nontrivial algebraic automorphism group of finite type over $k$, then $X$ is not representable by a scheme.  
\end{cor}
\begin{proof}
Let $x$ be a $k$-point of $\mathcal X$, where $k$ is algebraically closed.  Let $G$ be the stabilizer group of $x$.  We obtain a monomorphism $\mathrm BG \rightarrow \mathcal X_k$.  If $P$ is any $G$-torsor over a scheme  $S$ then we obtain a twist $S \rightarrow \mathrm BG \rightarrow \mathcal X$ (Example \ref{E:twist}).  By definition, the twists agree locally, so that the induced maps to $X$ coincide locally.  But we have assumed $X$ is a sheaf, so that the maps agree globally as well.  But the torsor can be recovered, up to isomorphism, from the twist, so all $G$-torsors over all $k$-schemes are trivial.  Therefore by Theorem \ref{T:IfAutTors}, $G$ is trivial, so no point of $\mathcal X$ has a nontrivial stabilizer group and $\mathcal X\simeq X$.
\end{proof}

\subsection{More on descent} \label{S:MoreDesc}

We highlight a few alternate formulations of the descent properties (Section~\ref{S:descent}).  The definitions we present here are more efficient than those given Section~\ref{S:descent}, and often lead to more streamlined proofs, but they come at a cost of abstraction.

\subsubsection{Descent using gluing data} \label{S:descent-gluing}

In Definition~\ref{D:ConcDS} we needed a cleavage in order to be able to write things like $X_i \big|_{S_{ij}}$.  To do this properly requires keeping track of a number of canonical isomorphisms, which were intentionally elided in Definition~\ref{D:ConcDS}.  Reliance on a cleavage can be avoided by explicitly choosing a restriction at each step, instead of insisting on a canonical choice from the beginning.  Although this definition avoids the technical deficiencies of Definition~\ref{D:ConcDS}, it only exacerbates the proliferation of indices.  Nevertheless, we will see in \S\ref{S:descent-sieves} that it points the way towards a definition that is both technically correct and pleasantly efficient.

\begin{dfn}[Descent datum via gluing data] \label{D:descent-gluing}
Let $(\mathsf S, P \mathscr T)$ be a presite and let $\{ S_i \rightarrow S \}$ be a cover of $S$ in $\mathsf S$.  A \emph{descent datum} with respect to this cover consists of the following data:
\begin{enumerate}[(i)]
\item objects $X_i \in \mathcal M(S_i)$, $X_{ij} \in \mathcal M(S_{ij})$, and $X_{ijk} \in \mathcal M(S_{ijk})$ for all indices $i,j,k$ of the cover;
\item morphisms
\begin{gather*}
X_{ij} \rightarrow X_i \qquad
X_{ij} \rightarrow X_j \qquad
X_{ijk} \rightarrow X_{ij} \qquad
X_{ijk} \rightarrow X_{ik} \qquad
X_{ijk} \rightarrow X_{jk}
\end{gather*}
in $\mathcal M$ respectively covering the canonical projections
\begin{gather*}
S_{ij} \rightarrow S_i \qquad
S_{ij} \rightarrow S_j \qquad
S_{ijk} \rightarrow S_{ij} \qquad
S_{ijk} \rightarrow S_{ik} \qquad
S_{ijk} \rightarrow S_{jk} .
\end{gather*}
\end{enumerate}
The category of descent data with respect to $\{ S_i \rightarrow S \}$, denoted $\mathcal M(S_\bullet)$, has as objects the descent data as defined above.  A morphism $(X_\bullet) \rightarrow (Y_\bullet)$ consists of morphisms 
\begin{gather*}
X_i \rightarrow Y_i \qquad 
X_{ij} \rightarrow Y_{ij} \qquad
X_{ijk} \rightarrow X_{ijk}
\end{gather*}
for all indices $i,j,k$, commuting with the sructural morphisms (a more precise formulation of this condition will appear below in Remark~\ref{R:gluing-reform}).
\end{dfn}

\begin{rem}
An observant reader may be wondering where the cocycle condition~\eqref{E:cc} is hiding in Definition~\ref{D:descent-gluing}.  It is built into the commutativity diagram~\eqref{E:barycentric},
\begin{equation} \label{E:barycentric} \xymatrix@C=10pt{
& & X_i \\
& X_{ij} \ar[ur] \ar[dl] & \ar[u] \ar[l] \ar[r] X_{ijk} \ar[dll] \ar[d] \ar[drr] & X_{ik} \ar[ul] \ar[dr] \\
X_j  & & \ar[ll] X_{jk} \ar[rr] & & X_k
} \end{equation}
which itself is forced by the commutativity of the corresponding diagram with each $X$ replaced by an $S$, and the fact that every arrow in $\mathcal M$ is cartesian.  The morphism $\alpha_{ij}$ of Definition~\ref{D:dd} is the composition
\begin{equation*}
X_i \big|_{S_{ij}} \xleftarrow{\sim} X_{ij} \xrightarrow{\sim} X_j \big|_{S_{ij}}
\end{equation*}
with $\alpha_{ik}$ and $\alpha_{jk}$ defined similarly.  The cocycle condition is the identity of the compositions
\begin{gather*}
X_i \big|_{S_{ijk}} \xleftarrow{\sim} X_{ij} \big|_{S_{ijk}} \xrightarrow{\sim} X_j \big|_{S_{ijk}} \xleftarrow{\sim} X_{jk} \big|_{S_{ijk}} \xrightarrow{\sim} X_k \big|_{S_{ijk}} \\
X_i \big|_{S_{ijk}} \xrightarrow{\sim} X_{jk} \big|_{S_{ijk}} \xrightarrow{\sim} X_k \big|_{S_{ijk}}
\end{gather*}
by virtue of the restriction of Diagram~\eqref{E:barycentric} to $S_{ijk}$.
\end{rem}

If $X$ is an object of $\mathcal M(S)$ then using Axiom~(i) of a category fibered in groupoids (Definition~\ref{D:CFG}) we can find induced objects $X_i \in \mathcal{M}(S_i)$, $X_{ij} \in \mathcal{M}(S_{ij})$, and $X_{ijk} \in \mathcal M(S_{ijk})$ for all indices $i,j,k$.  We also obtain the required morphisms among these by many applications of Axiom~(ii) of a category fibered in groupoids.  By an application of the axiom of choice, we can do this for all objects of $\mathcal M(S_\bullet)$ and obtain a functor $\mathcal M(S) \rightarrow \mathcal M(S_\bullet)$.

\begin{dfn}[Effective descent datum via gluing data]
A descent datum $X \in \mathcal M(S_\bullet)$ for $\mathcal M$ with respect to a cover $\{ S_i \rightarrow S \}$ is said to be effective if it lies in the essential image of $\mathcal M(S) \rightarrow \mathcal M(S_\bullet)$.
\end{dfn}

\begin{rem} \label{R:gluing-reform}
We give a technical reformulation of Definition~\ref{D:descent-gluing}, motivated by the idea that the data we are keeping track of can be organized concisely  by the nerve of subcovers  consisting of three open sets; i.e., $2$-simplices.  
Suppose that the cover $\{ S_i \rightarrow S \}$ is indexed by a set $I$.  Let $\boldsymbol{\Delta}$ be the category of nonempty, totally ordered, finite sets of cardinality $\leq 3$, equipped with a morphism to $I$.  That is, an object of $\boldsymbol{\Delta}$ is a pair $(T, f)$ where $T$ is a totally ordered finite set with $1 \leq |T| \leq 3$ (i.e., $T$ is either $\{ 0 \}$, $\{ 0 < 1 \}$, or $\{ 0 < 1 < 2 \}$ up to isomorphism) and $f : T \rightarrow I$ is a function.  A morphism $(T, f) \rightarrow (T', f')$ is an order preserving function $g : T \rightarrow T'$ such that $f' \circ g = f$.  

We introduce abbreviations for certain objects of $\boldsymbol{\Delta}$.  Every object of $\boldsymbol{\Delta}$ is isomorphic to one of the following:
\begin{enumerate}[(i)]
\item for $i \in I$, we write $i$ for the object $(\{ 0 \}, (0 \mapsto i))$ of $\boldsymbol{\Delta}$;
\item for $i,j \in I$, we write $ij$ for the object $(\{ 0 < 1 \}, (0 \mapsto i, 1 \mapsto j))$;
\item for $i,j,k \in I$, we write $ijk$ for the object $(\{ 0 < 1 < 2 \}, (0 \mapsto i, 1 \mapsto j, 2 \mapsto k))$.
\end{enumerate}

The cover $\{ S_i \rightarrow S \}$ induces a functor $S_\bullet : \boldsymbol{\Delta}^{\mathsf {op}} \rightarrow \mathsf S$ sending $x$ to $S_x$.  In other words, it sends $i$ to $S_i$ and $ij$ to $S_{ij}$ and $ijk$ to $S_{ijk}$.  The morphisms are all the canonical projections among fiber products.

Now a descent datum, in the sense of Definition~\ref{D:descent-gluing} is a functor $X_\bullet : \boldsymbol{\Delta}^{\mathsf {op}}  \rightarrow \mathcal M$ such that $\pi X_\bullet = S_\bullet$:
\begin{equation*} \xymatrix{
& \mathcal M \ar[d]^\pi \\
\boldsymbol{\Delta} \ar[r]_{S_\bullet} \ar[ur]^{X_\bullet} & \mathsf S
} \end{equation*}
A morphism of descent data $X_\bullet \rightarrow Y_\bullet$ is a natural transformation of functors that projects to the identity natural transformation of $S_\bullet$.
\end{rem}

\subsubsection{Descent using sieves} \label{S:descent-sieves}

\paragraph{Sieves on topological spaces}

In fact, there is an even more efficient formulation of the definition of a sheaf.  Recall that the if $V$ is an object of $\mathsf O_X$ then $h_V$ is the functor represented by $V$.  It is convenient to work with the related functor, which we again denote by $h_V$, defined as   $ h_V(W) = \{ \iota_{W,V} \}$ if $W \subseteq V$ and $h_V(W) = \varnothing$ otherwise. 

Suppose that $U \in \mathsf O_X$.  For any open cover $\mathcal U$ of $U$, define a presheaf $h_{\mathcal U}$:
\begin{equation*}
h_{\mathcal U}(W) = \bigcup_{V \in \mathcal U} h_V(W) = \begin{cases} \{ \iota_{W,X} \} & \exists \ V \in \mathcal U \text{ such that } W \subseteq V   \\ \varnothing & \text{else} \end{cases}
\end{equation*}
If $\mathscr F : \mathsf O_X^{\mathsf{op}} \rightarrow (\mathsf{Set})$ is a presheaf then, as $h_{\mathcal U} \subseteq h_X$ (e.g., if $W\subseteq U$ for $U\in \mathcal U$, then we use  $W\subseteq U\subseteq X$), we get a morphism
\begin{equation}  \label{E:ShSieve}
\mathscr F(X) \xleftarrow{\sim} \operatorname{Hom}(h_X, \mathscr F) \rightarrow \operatorname{Hom}(h_{\mathcal U}, \mathscr F) .
\end{equation}
The bijectivity of the arrow on the left is Yoneda's lemma.

The interested reader may prove the following proposition:
\begin{pro}
Let $\mathscr F$ be a presheaf.
\begin{enumerate}[(i)]
\item $\mathscr F$ is separated if and only if~\eqref{E:ShSieve} is an injection for all open covers $\mathcal U$ of all $U \in \mathsf O_X$.
\item $\mathscr F$ is a sheaf if and only if~\eqref{E:ShSieve} is a bijection for all open covers $\mathcal U$ of all $U \in \mathsf O_X$.
\end{enumerate}
\end{pro}

\paragraph{Sieves in general}

\begin{dfn}
Let $\mathsf S$ be a category and $S$ an object of $\mathsf S$.  A sieve of $\mathsf S$ is a subcategory of $\mathsf S/S$ that is fibered in groupoids over $\mathsf S/S$.
\end{dfn}

In other words, a sieve of $S$ is a subcategory $\mathcal R \subseteq \mathsf S/S$ such that whenever $S'' \rightarrow S'$ is a morphism in $\mathsf S/S$ and $S'$ is in $\mathcal R$ then $S''$ is also in $\mathcal R$.

In Example~\ref{E:CFGSiS} we saw how to associate a sieve $\mathcal R$ to any family of maps $\{ S_i \rightarrow S \}$.  By a \emph{covering sieve} we mean a sieve associated to a covering family in the pretopology.  In fact, all of the axioms of a Grothendieck topology can be formulated purely in terms of  sieves \cite[Exp.~II, D\'ef.~1.1]{sga4-1}, but that will not concern us here.

\begin{dfn}[Descent datum via sieves]
Let $\mathsf S$ be a presite, let $\pi : \mathcal M \rightarrow \mathsf S$ be a category fibered in groupoids over $\mathsf S$, and let $\mathcal R$ be a covering sieve of $S \in \mathsf S$.  A \emph{descent datum} for $\mathcal M$ with respect to $\mathcal R$ is a functor $X : \mathcal R \rightarrow \mathcal M$ lifting the canonical projection $\mathcal R \subseteq \mathsf S/S \rightarrow \mathsf S$ (sending an object $S' \rightarrow S$ of $\mathsf S/S$ to $S'$):
\begin{equation*} \xymatrix{
& & \mathcal M \ar[d] \\
\mathcal R \: \ar@{^(->}@<-1pt>[r] \ar@/^10pt/[urr]^X & \mathsf S/S \ar[r] & \mathsf S.
} \end{equation*}
Descent data over the sieve $\mathcal R$ are the objects of a category $\mathcal M(\mathcal R)$ where a morphism $X \rightarrow Y$ is a natural transformation projecting to the identity natural transformation of the projection $\mathcal R \rightarrow \mathsf S$.
\end{dfn}

To construct the functor $\mathcal M(S) \rightarrow \mathcal M(\mathcal R)$ for a sieve $\mathcal R$ of $S$, observe that $\mathcal M(S) \simeq \mathcal M(\mathsf S/S)$ by the $2$-Yoneda lemma (on the left we mean the fiber of $\mathcal M$ over $S$ and on the right we mean the category of morphisms from $\mathsf S/S$ to $\mathcal M$).  Composing this equivalence with the restriction $\mathcal M(\mathsf S/S) \rightarrow \mathcal M(\mathcal R)$ induced from the inclusion $\mathcal R \subseteq \mathsf S/S$ induces
\begin{equation*} \xymatrix{
\mathcal M(S) \simeq \mathcal M(\mathsf S/S) \rightarrow \mathcal M(\mathcal R)
} \end{equation*}

\begin{dfn}[Effective descent datum via sieves] \label{D:descent-sieves}
A descent datum $X \in \mathcal M(\mathcal R)$ for $\mathcal M$ with respect to a sieve $\mathcal R$ is said to be \emph{effective} if it lies in the essential image of $\mathcal M(S) \rightarrow \mathcal M(\mathcal R)$.
\end{dfn}

\subsection{Grothendieck topologies}
\label{S:top}

Grothendieck pretopologies seem quite natural from the definition of a sheaf, but have the deficiency that many different pretopologies can give rise to the same category of sheaves.  This is not unlike the way different bases of a topological space should be considered equivalent.  The topology associated to a pretopology is the finest pretopology that gives the same category of sheaves (see \cite[Rem.~2.25, Def.~2.47, Prop.~2.49]{FGAe}).  In this section, we give an idea of how this works, providing a review of \cite[\S 2.3.5]{FGAe}, although here for brevity we define the topology directly, without a discussion of refinements of pretopologies.

Note that typically a Grothendieck topology is defined in terms of sieves, not coverings; the point is that the sieves associated to a pretopology are the same as the sieves associated  to the   topology obtained from the pretopology \cite[Prop.~2.48]{FGAe}, and therefore both induce the same Grothendieck topology in the sense of sieves.

\vskip .2 cm  For clarity of the discussion in this section, we will refer to coverings $\{S_\alpha\to S\}$ in a pretopology $\mathscr T$ on a category $\mathsf S$ as \emph{basic coverings}.  We start with a definition (cf.~\cite[Def.~2.45]{FGAe}):

\begin{dfn}[Coverings with respect to $\mathscr T$] \label{D:top}
Let $\mathscr T$ be a pretopology on a category $\mathsf S$.  We call a family of morphisms $\{S_\alpha \rightarrow S\}$ \emph{covering} (with respect to $\mathscr T$) if there is a basic covering family $\{T_\beta \rightarrow S\}$ such that, for each $\beta$, there is an $\alpha$ such that the morphism $T_\beta \rightarrow S$ factors through $S_\alpha$.  The covering family $\{ T_\beta \rightarrow S \}$ is called a basic covering \emph{refinement} of $\{ S_\alpha \rightarrow S \}$.
\end{dfn}

\begin{rem}
Note that the covering  $\{S_\alpha \rightarrow S\}$ in Definition \ref{D:top}  is a covering in the sense of  Definition \ref{D:CovStack}.
\end{rem}

\begin{lem} \label{L:top}
Assume that $\mathsf S$ has all fiber products.  Let $\mathscr T$ be a pretopology on $\mathsf S$ and let $\mathscr T'$ be the collection of all covering families with respect to $\mathscr T$.  Then $\mathscr T'$ is a pretopology on $\mathsf S$.
\end{lem}
\begin{proof}
Property (PT0) is automatic, since $\mathsf S$ has all fiber products.  Likewise (PT3) is immediate.

If $\{ S_\alpha \rightarrow S \}$ has a basic covering refinement $\{ T_\beta \rightarrow S \}$ and $S' \rightarrow S$ is any morphism then $\{ S_\alpha \mathbin\times_S S' \rightarrow S' \}$ has the basic covering refinement $\{ T_\beta \mathbin\times_{S} S' \rightarrow S' \}$, hence is in $\mathscr T'$.  This proves (PT1).

Suppose that $\{ S_\alpha \rightarrow S \}$ has a basic covering refinement $\{ T_\beta \rightarrow S \}$, and each $S_\alpha$ has a family $\{ S_{\alpha \gamma} \rightarrow S_\alpha \}$ with a basic covering refinement $\{ T_{\alpha \delta} \rightarrow S_\alpha \}$.  For each $\beta$, choose an $\alpha(\beta)$ and a factorization of $T_\beta \rightarrow S$ through $S_{\alpha(\beta)}$.  Then the maps $T_{\alpha(\beta) \delta} \mathbin\times_{S_\alpha(\beta)} T_\beta \rightarrow T_\beta$ are covering.  As the $T_\beta$ cover $S$, the property (PT2) implies that the $T_{\alpha(\beta) \delta}$ cover $S$.  Therefore the $T_{\alpha(\beta) \delta}$ give a basic covering refinement of the family $S_{\alpha \gamma} \rightarrow S$.
\end{proof}

\begin{dfn}[Grothendieck topology]
A pretopology $\mathscr T$ is called  \emph{a (Grothendieck) topology} if $\mathscr T' = \mathscr T$, in the notation of Lemma~\ref{L:top}.  If $\mathscr T$ is a pretopology then $\mathscr T'$ is called the \emph{associated topology} to $\mathscr T$.
\end{dfn}

\begin{rem}
In \cite[Def.~2.52]{FGAe} what we call a  topology is called  a saturated pretopology.  
\end{rem}

\begin{exa}
Let $\mathscr T$ be the pretopology on topological spaces where the basic covering families are open covers.  Then every surjective local isomorphism has a section over a suitable open cover, so surjective local isomorphisms are covering in the associated saturated topology.
\end{exa}

\begin{exa} \label{E:smooth-etale-cover}
Consider the \'etale pretopology on schemes, defined in Example~\ref{E:etale}.  Every smooth surjection admits a section over some \'etale cover, so every smooth surjection is covering in the associated topology to the \'etale pretopology.
\end{exa}

The following lemma shows that a pretopology and its associated topology have the same sheaves.  Passage to the associated topology may therefore significantly expand the class of morphisms with respect to which one can use descent.

\begin{lem}[{\cite[Prop.~2.49, Prop.~2.53(iii)]{FGAe}}]
If $\mathscr T$ is a pretopology on $\mathsf S$ then a presheaf on $\mathsf S$ is a sheaf with respect to $\mathscr T$ if and only if it is a sheaf with respect to $\mathscr T'$.
\end{lem}
\begin{proof}
Since $\mathscr T \subseteq \mathscr T'$, it is immediate that sheaves in $\mathscr T'$ are sheaves in $\mathscr T$.  For the converse, suppose that $\mathscr F$  is a sheaf with respect to $\mathscr T$ and let $\{ S_\alpha \rightarrow S \}$ be a covering family (with respect to $\mathscr T'$).  

Construct a presheaf $\mathscr F'$ over $S$ by the following formula:
\begin{equation*}
\mathscr F'(R) = \operatorname{eq} \Bigl( \prod_\alpha \mathscr F(U_\alpha \mathbin\times_S R) \rightrightarrows \prod_{\alpha,\beta} \mathscr F(U_\alpha \mathbin\times_S U_\beta \mathbin\times_S R) \Bigr).
\end{equation*}
There is a natural map $\varphi : \mathscr F \rightarrow \mathscr F'$, which we would like to show is an isomorphism.  Both $\mathscr F$ and $\mathscr F'$ are sheaves in the pretopology $\mathscr T$, so this is a local problem in $\mathscr T$.  We can therefore replace $S$ by a basic cover from $\mathscr T$.  Since $\{ U_\alpha \rightarrow U \}$ has a refinement by a basic cover, we can assume that there is a section $\sigma : S \rightarrow U_\alpha$ for some $\alpha$.

We can use $\sigma$ to construct an inverse $\psi$ to the map $\varphi : \mathscr F \rightarrow \mathscr F'$.  Indeed, if $\xi \in \mathscr F'(S)$ then let $\xi_\alpha$ be its projection on the $\alpha$ component.  Then $\sigma^\ast(\xi_\alpha) \in \mathscr F(S)$ and we set $\psi(\xi) = \sigma^\ast(\xi_\alpha)$.  It is immediate that $\psi \varphi(\eta) = \eta$ for all $\eta \in \mathscr F(S)$.  

We check that $\varphi \psi(\xi) = \xi$.  What we need to check is that, for all indices $\beta$,
\begin{equation*}
\xi_\beta = \sigma^\ast(\xi_\alpha) \big|_{U_\beta} .
\end{equation*}
By assumption, $\xi$ is equalized by the maps to $\prod_{\alpha,\beta} \mathscr F(U_{\alpha \beta})$, where $U_{\alpha \beta} = U_\alpha \mathbin\times_S U_\beta$.  Therefore $\xi_\alpha \big|_{U_{\alpha \beta}} = \xi_\beta \big|_{U_{\alpha \beta}}$.  But $(\sigma, \mathrm{id}_{U_\beta})$ determines a section of $U_{\alpha \beta}$ over $U_\beta$, so we determine that 
\begin{equation*}
\xi_\beta = \sigma^\ast (\xi_\beta \big|_{U_{\alpha \beta}}) = \sigma^\ast (\xi_\alpha \big|_{U_\alpha \beta}) = \sigma^\ast(\xi_\alpha) \big|_{U_\beta},
\end{equation*}
as required.
\end{proof}

\subsection{An example of ineffective descent}
\label{S:Raynaud}

The category fibered in groupoids $\mathcal{M}_1$ is not an \'etale stack!  As we will see, it is possible to create a descent datum for genus~$1$ curves that is not effective.
 This is really a deficiency of our definition of $\mathcal{M}_1$ in \S\ref{S:CFG-curves}, by which $\mathcal M_1$ parameterizes \emph{schematic} families of smooth, proper curves of genus~$1$.  The proper thing to do would be to include in our moduli problem smooth proper families of \emph{algebraic spaces} (Definition~\ref{D:AlgStack}) whose fibers are curves of genus~$1$.

To construct an ineffective descent datum for $\mathcal M_1$, it will help to notice that every descent datum for $S$ can at least be descended to a \emph{sheaf} on the big \'etale site of $S$ that is locally on $S$ representable by a family of genus~$1$ curves.  This sheaf is precisely the algebraic space we should have admitted into the moduli problem for $\mathcal M_1$.  Our task in this section is to construct such a sheaf $X$ that is locally in $S$ representable by genus~$1$ curves, but is not globally representable by a scheme.

To begin, note that if $X\to S$ is any family of smooth curves of genus~$1$ over a base $S$, then the relative  Jacobian $J\to S$ of $X\to S$ is a family of abelian schemes of dimension~$1$ over $S$.  This construction is local on $S$, so that even if $X$ is merely a sheaf over $S$ that is \emph{locally} representable by a family of smooth curves of genus~$1$, one obtains a descent datum for a family of elliptic curves over $S$.  But the descent datum comes with a compatible family of ample line bundles (coming from the origin of the group structure) so by Theorem~\ref{T:polarized-descent}, it can be descended to a family of elliptic curves over $S$.

Furthermore, $J$ acts on $X$ making $X$ into a $J$-torsor.  Therefore $X$ is classified up to isomorphism by an element $[X]\in H^1(S, J)$ (this can be \'etale or flat cohomology).  Raynaud shows that, provided $S$ is quasicompact, this element $[X]$ is torsion if and only if $X$ is projective over $S$~\cite[Cor.~XIII~2.4~ii)]{Raynaud}, and that, provided $S$ is normal, $X$ is projective over $S$ if and only if it is representable by a scheme~\cite[Prop.~XIII~2.6]{Raynaud}.  In fact, Raynaud proves these statements more generally about torsors under abelian varieties: 

\begin{teo}[{\cite[Cor.~XIII~2.4~ii)]{Raynaud}}] \label{T:proj-tors}
Assume that $S$ is a quasicompact scheme and let $J$ be a projective abelian variety over $S$.  Then a $J$-torsor $X$ is projective if and only if its class in $H^1(S, J)$ is torsion.
\end{teo}

\begin{teo}[{\cite[Prop.~XIII~2.6~i)]{Raynaud}}] \label{T:proj-rep}
Let $S$ be a quasicompact, normal scheme and $J$ an abelian variety.  Then a $J$-torsor $X$ is representable by a scheme if and only if it is projective.
\end{teo}

\begin{proof}[Sketch of a proof of Theorem~\ref{T:proj-tors}]
First, suppose the class represented by $X$ is torsion, say $n [X] = 0$.  As $n[X]$ is represented by $\bigl[X / J[n]\bigr]$, where $J[n]$ is the $n$-torsion of $J$, we have a finite map
\begin{equation*}
X \rightarrow X / J[n] \simeq J.
\end{equation*}
This implies $X$ is projective over $S$, as $J$ is.

Conversely, if $X$ is projective  over $S$, then a relatively  ample line bundle $L$ on $X$ over $S$ induces a relatively  ample line bundle $L'$ on $J$~\cite[Lem.~XI~1.6]{Raynaud}.  To see this, note that the relative N\'eron--Severi group $\operatorname{NS}_{X/S}$ of $X$ over $S$ is isomorphic to $\operatorname{NS}_{J/S}$, so that $L$ determines a class in $\operatorname{NS}_{J/S}$.  Any line bundle $M$ on $J$ determines a line bundle 
\begin{equation} \label{E:bilinear}
\mu^\ast M \otimes p_1^\ast M^\vee \otimes p_2^\ast M^\vee \otimes e^\ast M
\end{equation}
(where $\mu : J \times J \rightarrow J$ is the addition map, $p_i$ are the projections, and $e$ is the composition of the projection $J \rightarrow S$ and the zero section) on $J \mathop\times_S J$ that depends only on the N\'eron--Severi class of $M$.  Restricting this to the diagonal of $J \times J$ recovers a line bundle on $J$ whose image in the N\'eron--Severi group is twice that of $M$.  Altogether, this gives a map:
\begin{equation*}
\operatorname{NS}_{X/S} \simeq \operatorname{NS}_{J/S} \rightarrow \operatorname{Pic}_{J/S}.
\end{equation*}
Applying this to $L$
 yields a line bundle $L'$ on $J$.  Moreover, under a local isomorphism between $J$ and $X$, the N\'eron--Severi class of $L'$ is double that of $L$.  Thus $L'$ is relatively ample on $J$ over $S$.

In general, a relatively ample line bundle $M$ on $J$ induces an isogeny $J \rightarrow \hat{J}$ of abelian schemes over $S$ (where $\hat{J}$ is the dual abelian variety) by way of~\eqref{E:bilinear}, and hence a morphism $H^1(S,J) \rightarrow H^1(S,\hat{J})$.  The image of $[X]$ is the class $[\hat{X}]$, where $\hat{X}$ is the $\hat{J}$-torsor consisting of line bundles on $X$ that lie in 
 the same N\'eron--Severi class as $L'$ \cite[Cor.~XIII~1.2~ii)]{Raynaud}.  By assumption this torsor is trivial (the line bundle $L^{\otimes 2}$ provides a section), so $[\hat{X}] = 0$.  On the other hand, we have an exact sequence
\begin{equation*}
H^1(S,K) \rightarrow H^1(S,J) \rightarrow H^1(S,\hat{J}) ,
\end{equation*}
where $K$ is the kernel of $J \rightarrow \hat{J}$, so $[X]$ lies in the image of $H^1(S,K)$.  Since $L$ is ample, $K$ is finite, so $H^1(S,K)$ is torsion, and therefore so is $[X]$.  
\end{proof}

\begin{proof}[Sketch of a proof of a special case of Theorem~\ref{T:proj-rep}]
We will prove a special case of Theorem~\ref{T:proj-rep}, following the proof of \cite[Thm.~V~3.10]{Raynaud}, that will suffice to construct our example.  We assume that $S$ is the spectrum of a noetherian local ring with closed point $s$ and generic point $\eta$.  Let $X$ be a $J$-torsor over $S$.  Choose an effective Cartier divisor $D \subseteq X$ whose complement $U \subseteq X$ is quasi-affine and meets the closed fiber of $X$.  We will argue that the line bundle $L = \mathcal O_X(D)$ must be ample.

We make a few observations:
\begin{enumerate}[(a)]
\item \label{obs:a} \emph{The $J$-orbit of $U$ is all of $X$.}  This is because $U$ meets $X_s$ and $J_s$ acts transitively on $X_s$, so $JU$ contains $X_s$.  But every point of $X$ specializes to a point of $X_s$, since $X$ is proper over $S$, and $JU$ is open, hence contains all of $X$.
\item \emph{Let $M$ be the line bundle
\begin{equation} \label{E:cube}
M = (p_1 + p_2 + p_3)^\ast L \otimes (p_1 + p_3)^\ast L^\vee \otimes (p_2 + p_3)^\ast L^\vee \otimes p_3^\ast L
\end{equation}
on $J \mathbin\times_S J \mathbin\times_S X$.  There is some positive integer $n$ such that $M^{\otimes n}$ is the pullback of a line bundle on $J \mathbin\times_S J$.}  As $J \mathbin\times_S J$ is normal, it is sufficient to verify this over the generic point of $S$ \cite[Cor.~(21.4.13), Erratum 4.53, p.361]{EGAIV4}.
  We may therefore assume that $S = \eta$ is the spectrum of a field.  If $X(\eta) \neq \varnothing$ then $X \simeq J$ and this is a version of the theorem of the cube (see~\cite[\S5, Cor.~6 and \S6, Thm., Cor.~2 p.58]{Mumford-AV}); in this case $n = 1$.  If $X$ does not have a section over $\eta$ then it will certainly have one over some finite extension $p : \eta' \rightarrow \eta$, so $p^\ast M$ is the pullback of a line bundle $M'$ on $J_{\eta'} \mathbin\times_{\eta'} J_{\eta'}$.  If $\eta'$ has degree $n$ over $\eta$ then the norm of $p^\ast M$ is $M^{\otimes n}$ and this is the pullback of the line bundle $\operatorname{Norm}_{\eta'/\eta}(M')$ on $J_{\eta} \mathbin\times_{\eta} J_{\eta}$.

We replace $D$ with $nD$ so that~\eqref{E:cube} is the pullback of a line bundle on $J \mathbin\times_S J$ without passing to a tensor power.
\item \label{obs:c} Restricting~\eqref{E:cube} to a point $(g,-g)$ of $J$ yields an isomorphism of line bundles on $X$:
\begin{equation*}
T_g^\ast L \otimes T_{-g}^\ast L \simeq L^{\otimes 2}
\end{equation*}
where $T_g$ denotes translation by $g$.
\end{enumerate}

To prove the ampleness of $L$, we must show that, for any point $x$ of $X$, and any open neighborhood $V$ of $x$ in $X$, there is some index $n$ and some $f \in \Gamma(X, L^{\otimes n})$ such that the open set $X_f \subseteq X$ defined by the nonvanishing of $f$ is contained in $V$ \cite[Thm.~(4.5.2)]{EGAII}.  Suppose first that $x \in U$.  Let $z$ be the tautological section of $\mathcal O_X(D)$ that vanishes exactly along $D$.  As $U$ is quasi-affine, there is some affine open neighborhood $W$ of $x$ in $V \cap U$, defined by the nonvanishing of a function $h$ on $U$.  Then there is some $n \geq 1$ such that $z^n h$ extends to all of $X$.  That is, we may regard $z^n h$ as a section without poles of $\mathcal O_X(nD)$ and $W_{h} = W$.

If $x$ is not in $U$, we can at least find a $g \in J$ such that $T_g(x) \in U$ (by observation~\eqref{obs:a}, above).  Applying the argument above to $T_g(x)$ and $T_g(V)$, we can find a section $f \in \Gamma(X, L^{\otimes n})$ such that $x \in W_f \subseteq T_{g}(V)$ by the argument above.  We have some freedom in the choice of $g$, and by avoiding a closed set of possibilities, we can ensure that $T_{-g}(x) \in U$ as well.  Then we can consider $f' = T_g^\ast f \otimes T_{-g}^\ast z^n$ as a section of $T_g^\ast L^{\otimes n} \otimes T_{-g}^\ast L^{\otimes n} \simeq L^{\otimes 2n}$ (by~\eqref{obs:c}, above).  Now, $W_{f'} = T_g^{-1}(W_f) \cap T_{-g}^{-1}(W_z) = T_g^{-1}(W_f) \cap T_{-g}^{-1}(U)$.  By our choice of $g$, we have $x \in W_{f'}$ and $W_{f'} \subseteq T_g^{-1} T_g(V) = V$, as required.
\end{proof}

Finally, \cite[XIII~3.2]{Raynaud} proves that there is a genus~$1$ curve over a normal noetherian local ring of dimension~$2$ whose image in $H^1(S,J)$ is non-torsion, establishing the existence of the desired family $X\to S$.  Another version of the construction (and a slightly stronger conclusion) can be found in~\cite{Zomervrucht}.  We will summarize Raynaud's construction.

We begin with a discrete valuation ring $R$ with algebraically closed residue field.  Let $T = \operatorname{Spec} R$ and assume that $T$ has a connected, \'etale double cover $T' \rightarrow T$.  Let $\pi$ be a uniformizer for $T$.  For any object $Y$ over $T$, we will write $Y'$ for its base change to $T'$.  We write $\tau$ for the generic point of $T$ and $t$ for the special point; $\tau'$ is the generic point of $T'$ and $t_1$ and $t_2$ are the two points in the fiber of $T'$ over $t$.

Let $E$ be an elliptic curve over $T$ and let $V \subseteq E$ be the complement of the zero section in $E$.  Let $W$ be the quotient of $V$ by the inversion in the group law of $E$, so $W \simeq \mathbb A^1_T$.  Let $\gamma : W \rightarrow W$ be multiplication by $\pi$ and let $Z$ be the normalization in $V$ of the composition
\begin{equation*}
V \rightarrow W \xrightarrow{\gamma} W .
\end{equation*}

All we will use about $Z$ is the following lemma, whose proof is straightforward:
\begin{lem}
The map $f : V \rightarrow Z$ is an isomorphism over $\tau \in T$ and constant over $t \in T$.  Furthermore, $Z$ is normal.
\end{lem}

Let $s$ be the unique closed point of $Z$ and let $S = \operatorname{Spec} \mathcal O_{Z,s}$.  Let $\eta'$ be the generic point of $S'$ and let $s_1$ and $s_2$ the two closed points lying above $s$.  Let $U_i$ be the complement in $S'$ of $s_i$, and let $U_{12}$ be their intersection.  Then $U_1 \cup U_2 = S'$, so we use the Mayer--Vietoris sequence to find a class in $H^1(S', E')$:
\begin{equation*}
H^0(U_1, E') \times H^0(U_2, E') \rightarrow H^0(U_{12}, E') \rightarrow H^1(S', E')
\end{equation*}
Now, $U_{12} = \{ \eta' \}$ so $H^0(U_{12}, E') = E'(\eta')$.  Recall that we have a map $E' \rightarrow Z'$ that is an isomorphism over the generic points.  Therefore we have a canonical element $\xi$ of $E'(\eta')$ corresponding to the inclusion of the generic point.

An element of $H^0(U_i, E')$ can be seen as a rational map from $E'$ to itself over $T'$ that restricts to a constant map over $t_i$.  Any such map must in fact factor through a section of $E'$ over $T'$.  Therefore we have
\begin{equation*}
H^0(U_i, E') = E'(T') .
\end{equation*}

Consider the image of $\xi$ in $H^1(S', E')$.  This cannot possibly be torsion, for if it were then it would have a multiple in the image of $H^0(U_1, E') \times H^0(U_2, E')$.  That is impossible, because no multiple of the identity map on an elliptic curve is a difference of constant maps.  Therefore we have found a non-torsion element in $H^1(S', E')$.  Since $S'$ is \'etale over the normal scheme $S$, it is normal, and therefore by Theorem~\ref{T:proj-tors}, the corresponding sheaf is a descent datum for an elliptic curve over $S'$ that is not projective, hence not effective by Theorem~\ref{T:proj-rep}.

Raynaud goes a bit further and shows that the base for the descent may be chosen to be local.  Indeed, letting $q$ denote the projection from $S'$ to $S$, we have
\begin{equation*}
H^1(S', E') = H^1(S, q_\ast E') .
\end{equation*}
Now, $q_\ast E' = q_\ast q^\ast E'$ is the Weil restriction of scalars of $E'$ via the finite, \'etale map $q$, hence is an abelian scheme of dimension~$2$ over $S$.  It comes with a canonical inclusion $E \subseteq q_\ast E'$ whose quotient is another elliptic curve $F$ (in fact a quadratic twist of $E$).  Then we have an exact sequence
\begin{equation*}
H^1(S, E) \rightarrow H^1(S, q_\ast E') \rightarrow H^1(S, F)
\end{equation*}
so the non-torsion class $\xi \in H^1(S', E') = H^1(S, q_\ast E)$ determines a non-torsion class either in $H^1(S, F)$ or in $H^1(S, E)$.  Either way, we obtain a non-effective descent datum for genus~$1$ curves over $S$.

\section{The many meanings of algebraicity}
\label{S:AlgStk}

We work  over the presite $\mathsf S$ of \'etale covers of schemes (over some fixed base scheme).  Many authors have given different definitions of algebraicity.  Deligne and Mumford required a schematic diagonal and an \'etale cover by a scheme~\cite[Def.~(4.5)]{DM}, but insisted their definition was not the right one except for quasiseparated stacks.  Knutson required all algebraic spaces to be quasiseparated~\cite[Ch.~2, Def.~1.1]{Knutson}.  Artin gave his definition only for stacks that are locally of finite presentation and required a diagonal representable by algebraic spaces (in the sense of Knutson) and a smooth cover by a scheme~\cite[Def.~(5.1)]{Artin-versal}.  Laumon and Moret-Bailly defined an `algebraic stack (understood quasiseparated)' by adding quasiseparation to Artin's conditions~\cite[Def.~(4.1)]{LMB}.  The Stacks Project uses Artin's conditions, but without requiring the algebraic spaces to be quasiseparated~\cite[Tag 026N]{stacks}.

Recall that in Definition~\ref{D:alg-stack} we defined algebraic stacks as the smallest class of stacks on the category of schemes that includes all schemes and includes all stacks that admit smooth covers by stacks in the class.  Our definition is equivalent to the Stacks Project's, but usage in the literature varies widely.

\subsection{Stacks with flat covers by schemes and stacks over the fppf site}
\label{S:flat}

The \'etale topology is not the only natural topology on schemes.  From the perspective of descent, at least for quasicoherent sheaves, the flat topologies (fppf and fpqc) might be even more natural.  The fpqc topology presents certain technical issues, owing to the absence of a sheafification functor, so we will not discuss it.

There are two ways one might try to replace the \'etale topology with the fppf topology in our discussion of algebraic stacks.  We might try to limit the class of stacks by insisting they be stacks in the fppf, as opposed to just the \'etale, topology.  Or, we might enlarge the class of stacks under consideration by permitting them to have fppf, as opposed to necessarily smooth, covers by schemes.  It turns out that either modification yields the same class of algebraic stacks.  We will sketch the main ideas behind this result.

We start by stating the following lemma whose proof we omit  as it is well-known and not difficult.

\begin{lem}
The class of flat morphisms of schemes is stable under composition and base change and is local to the source and target in the fppf topology.
\end{lem}

The lemma implies that one could develop a theory of flat-adapted algebraic stacks in the \'etale topology.  The following theorem of Artin explains that to do so would yield nothing new:

\begin{teo}[{\cite[Thm.~(6.1)]{Artin-versal}, \cite[Thm.~10.1]{LMB}, \cite[Tag 06DB]{stacks}}]  \label{T:ArtinPres} Let $\mathcal X$ be a stack in the fppf topology on the category of schemes.
If there exists morphism
$$
\begin{CD}
U@>P>> \mathcal X
\end{CD}
$$
from an  algebraic space $U$ that is representable by     algebraic spaces, faithfully flat, and of finite presentation, then there is such a morphism $P$ that is smooth and surjective.  In particular, 
$\mathcal X$ is an  SP algebraic stack in the sense of Definition~\ref{D:alg-stack}.
\end{teo}
\begin{proof} 
We give a rough sketch of the proof, following Artin.  By an `induction on stackiness', it is sufficient to assume that the diagonal of $\mathcal X$ is representable by algebraic spaces.

We consider the following moduli problem $\mathcal V$.  Choose a cover of $U$ by a disjoint union $U_0$ of affine schemes (which is certainly possible, since $U$ is a scheme).  Since $U_0 \mathbin\times_{\mathcal X} U_0$ is an algebraic space, we can choose a disjoint union of affine schemes $U_1$ and a smooth cover $U_1 \rightarrow U_0 \mathbin\times_{\mathcal X} U_0$.  Let $s$ and $t$ denote the two projections from $U_1$ to $U_0$.  Let $U_2$ be a smooth cover of the space of triples $(\alpha, \beta, \gamma) \in U_1 \times U_1 \times U_1$ such that $s(\alpha) = s(\gamma)$, $t(\alpha) = s(\beta)$, and $t(\beta) = t(\gamma)$ and $\beta \circ \alpha = \gamma$ as isomorphisms between objects of $\mathcal X$.

For any scheme $S$, we define an $S$-point of $\mathcal W$ to be
\begin{enumerate}[(i)]
\item the choice of a finite union of components $V_i \subseteq U_i$ for each $i$ such that $V_\bullet$ forms a subgroupoid of $U_\bullet$,
\item a finite, locally free, surjective $S$-scheme $Z$ with a distinguished basis $\mathcal O_Z \simeq \mathcal O_S^d$, and 
\item a morphism of groupoids $Z_\bullet \rightarrow V_\bullet$:
\begin{equation*} \xymatrix{
Z_2 \ar[d] \ar@<2pt>[r] \ar[r] \ar@<-2pt>[r] & Z_1 \ar[d] \ar@<1.5pt>[r] \ar@<-1.5pt>[r] & Z_0 \ar[d] \ar[r] & S \\
V_2 \ar@<2pt>[r] \ar[r] \ar@<-2pt>[r] & V_1 \ar@<1.5pt>[r] \ar@<-1.5pt>[r] & V_0 \ar[r] & \mathcal X 
} \end{equation*}
\item where we have set $Z_0 = Z$ and $Z_i = Z_{i-1} \mathbin\times_S Z$ for $i \geq 1$.
\end{enumerate}

We argue that $\mathcal V$ is representable by a disjoint union of affine schemes, indexed by the choice of $V_\bullet$.  The algebra structure on $\mathcal O_Z$ is determined by its structure constants and various identities among them, hence is parameterized by an affine scheme.  The maps $Z_i \rightarrow V_i$ are determined by various elements of $\mathcal O_Z$ and relations among them (since the $V_i$ are affine schemes).  For each commutativity condition we have a pair of maps $Z_i \rightarrow V_{i-1}$ that we wish to coincide.  That is a closed condition (since affine schemes are separated).

Any $S$-point of $\mathcal V$ determines a descent datum for a morphism $S  \rightarrow \mathcal X$ in the fppf topology.  Since $\mathcal X$ is a stack in the fppf topology, this descends to a morphism $S \rightarrow \mathcal X$ and we obtain a morphism of groupoids $\mathcal V \rightarrow \mathcal X$.  We have just seen that $\mathcal V$ is representable by a disjoint union of affine schemes, so it remains to verify this map is smooth.

Now let $\mathcal W \subseteq \mathcal V$ be the open substack where $Z_0 \rightarrow V_0 \mathbin\times_{\mathcal X} S$ is a local complete intersection morphism.  To see that $\mathcal W$ is indeed open in $\mathcal V$, note that there are open subsets $W_i \subseteq Z_i$ where the maps $W_i \rightarrow V_i$ are local complete intersection morphisms.  Since the $Z_i$ are proper over $S$, the image in $S$ of the complement of $W_i$ is closed, so the condition that the fiber of $Z_i \rightarrow V_i$ be a closed immersion and a local complete intersection morphisms is open on $S$.

Now we verify that $\mathcal W$ is locally of finite presentation, formally smooth, and surjective over $\mathcal X$.  

To see that $\mathcal W \rightarrow \mathcal X$ is locally of finite presentation, one observes that once a morphism $S \rightarrow \mathcal X$ is specified, a lift to $\mathcal W$ involves only a finite amount of additional data.  

Next we verify the smoothness, for which we can use the infinitesimal criterion.  Given a lifting problem
\begin{equation*} \xymatrix{
S \ar[r] \ar[d] & \mathcal W \ar[d] \\
S' \ar@{-->}[ur] \ar[r] & \mathcal X
} \end{equation*}
in which $S$ is affine and $S'$ is an infinitesimal extension of $S$, we have, by definition a morphism of groupoids $Z_\bullet \rightarrow V_\bullet$, with the $Z_i$ finite and locally free over $S$, that we would like to extend to $Z'_\bullet \rightarrow V_\bullet$ with $Z'_i$ finite and locally free over $S'$.

In this case, the map $Z_0 \rightarrow U_0$ is a local complete intersection morphism, and $Z_0$ is affine, so there is no obstruction to extending it to a morphism $Z'_0 \rightarrow U_0$ with $Z'_0$ finite and locally free over $S'$.  This induces a pair of morphisms $Z'_1 \rightrightarrows V_0$, hence a map $Z'_1 \rightarrow V_0 \mathbin\times_{\mathcal X} V_0$.  Now, $V_1 \rightarrow V_0 \mathbin\times_{\mathcal X} V_0$ is smooth, so the map $Z_1 \rightarrow V_0 \mathbin\times_{\mathcal X} V_0$ lifts to $V_1$.  Now let $R \subseteq V_1 \times V_1 \times V_1$ be the set of triples $(\alpha, \beta, \gamma)$ such that the equation $\beta \circ \alpha = \gamma$ makes sense and holds in $\mathcal X$.  Then we obtain $Z'_2 \rightarrow R$ and $V_2$ is smooth over $R$, so $Z'_2 \rightarrow R$ lifts to $V_2$.  An infinitesimal deformation of a local complete intersection morphism is still a local complete intersection morphism, and an infinitesimal deformation of a closed embedding is still a closed embedding, so we have produced the required extension.

Finally, we have to check $\mathcal W \rightarrow \mathcal X$ is surjective.  Let $k$ be the spectrum of an algebraically closed field and let $S = \operatorname{Spec} k$.  Let $\xi$ be a $k$-point of $\mathcal X$.  The fiber of $U_0$ over $S$ is a nonempty algebraic space $T$.  Choose a smooth cover $P$ of $T$ by a scheme.  This scheme is certainly flat over $k$, so it has a dense open subset where it is Cohen--Macaulay \cite[Tag 045U]{stacks}.  Pick a point $p$ of $P$ where $P$ is Cohen--Macaulay, and let $Z_0$ be the vanishing locus of a regular sequence at $p$.  Now let $V_0$ be a component of $U_0$ that contains the image of $p$ under the composition $Z_0 \rightarrow P \rightarrow T \rightarrow U_0$.  Then $Z_0 \rightarrow V_0 \mathbin\times_{\mathcal X} S$ is a local complete intersection morphism by construction.  Furthermore, we obtain a map $Z_1 \rightarrow V_0 \mathbin\times_{\mathcal X} V_0 \subseteq U_1$.  But recall that $Z_1$ has just one point, by construction, so we choose a component $V_1$ of $U_1$ whose image in $U_0 \mathbin\times_{\mathcal X} U_0$ contains the image of $Z_0$.  Since $Z_1$ is artinian, there is a lift of $Z_1 \rightarrow V_0 \mathbin\times_{\mathcal X} V_0$ to $V_1$.  Then we repeat the same process to get $Z_2 \rightarrow V_2$ and we conclude.
\end{proof}

\begin{rem}
Note that the condition that $\mathcal X$  have separated and quasicompact diagonal does not depend on the presentation.
\end{rem}

\begin{cor} \label{C:flat-quotient}
Suppose that $G$ is a \emph{flat} group scheme over $S$, acting on an $S$-scheme $X$.  Then the stack $[X/G]$ (see \S\ref{S:TorsG}) is  algebraic.  If $G$ is quasiseparated over $S$  (e.g., quasiprojective) then $[X/G]$ is a quasiseparated algebraic stack.
\end{cor}
\begin{proof}
The cover $X \rightarrow [X/G]$ is a $G$-torsor, hence is an fppf cover.  Therefore from the theorem, $[X/G]$  is an algebraic stack.  We may identify $X \mathbin\times_{[X/G]} X$ with $X \times_S G$, under which identification the diagonal map becomes the inclusion $(\mathrm{id}_X, e) : X \rightarrow X \times_S G$, with $e$ denoting the identity section of $G$ over $S$.  This morphism is certainly representable and separated (it is an injective morphism of schemes).  It is quasicompact if $G$ is quasiseparated over $S$, since a section of a quasiseparated morphism is quasicompact \cite[Tag 03KP]{stacks}.
\end{proof}

Next we consider the question of stacks in the fppf topology.  Note first that an algebraic stack in the fppf topology is clearly an algebraic stack in the \'etale topology, by restriction.  Now we show the converse:

\begin{cor}
Algebraic stacks are stacks in the fppf toplogy.
\end{cor}
\begin{proof}
Suppose that $\mathcal X$ is an algebraic stack.  Let $\mathcal X'$ be the fppf stackification.  By induction, we can assume that we have already shown the diagonal of $\mathcal X$ is a relative fppf sheaf, which means that $\mathcal X \rightarrow \mathcal X'$ is injective.

Now let $U \rightarrow \mathcal X$ be a smooth cover.  We argue that $U \rightarrow \mathcal X'$ is representable by algebraic spaces.  Indeed, if $S \rightarrow \mathcal X'$ is any morphism, we can find an fppf cover $T \rightarrow S$ such that $T \rightarrow \mathcal X'$ lifts to $\mathcal X$.  Then $T \mathbin\times_{\mathcal X'} U = T \mathbin\times_{\mathcal X} U$ since $\mathcal X \subseteq \mathcal X'$, and $T \mathbin\times_{\mathcal X} U$ is an algebraic space.  But the map
\begin{equation*}
T \mathbin\times_{\mathcal X'} U \rightarrow S \mathbin\times_{\mathcal X} U
\end{equation*}
is the base change of the fppf cover $T \rightarrow U$, so $S \mathbin\times_{\mathcal X} U$ has an fppf cover by an algebraic space.  It is therefore an algebraic space, as required.

Furthermore, $S \mathbin\times_{\mathcal X'} U \rightarrow S$ is a smooth cover, since smoothness can be verified locally in the fppf topology and $T \mathbin\times_{\mathcal X'} U = T \mathbin\times_{\mathcal X} U \rightarrow T$ is a smooth cover.  Therefore $U \rightarrow \mathcal X'$ is a smooth cover, and as this factors through $\mathcal X$, the map $\mathcal X \rightarrow \mathcal X'$ is a smooth cover.  Now both $\mathcal X$ and $\mathcal X'$ are stacks in the \'etale topology, and $\mathcal X \rightarrow \mathcal X'$ was already seen to be injective, so $\mathcal X \rightarrow \mathcal X'$ is an isomorphism.
\end{proof}

\subsection{Other definitions of algebraicity}

The following definition collects some of the most common meanings attributed to algebraicity of a stack on the \'etale site of schemes, in roughly chronological order.  After giving the definition, we analyze the relationships among them, as well as to our Definition  \ref{D:alg-stack}.

\begin{dfn}[Algebraic stack] \label{D:AlgStack}
	Let $\mathcal X$ be a category fibered in groupoids over 
	$\mathsf S=\mathsf S_{\operatorname{et}}$.\footnote{Much of the literature works over the \'etale site of affine schemes.  Since every scheme has an \'etale (even Zariski) cover by affine schemes, the notions of stacks agree.}   We define various notions of algebraic stack using the table below.  
Namely, we call $\mathcal X$ a \emph{Deligne--Mumford algebraic stack (resp.~Knutson algebraic space, resp.~Artin algebraic stack, etc.})  if the diagonal $\Delta:\mathcal X\longrightarrow \mathcal X\times \mathcal X$ satisfies the condition specified in the second column of the table on page \pageref{F:AlgStackDef}, and there is a scheme $U$ and a surjection $p:U\longrightarrow \mathcal X$ that satisfies the 	conditions in the third column.
			The  morphism $p : U \rightarrow \mathcal X$  is called a \emph{presentation} of $\mathcal X$.
	A \emph{morphism} between any such stacks is a morphism of the underlying CFGs.
	\end{dfn}

\newcounter{tempfoot}
\setcounter{tempfoot}{\value{footnote}}

\renewcommand{\thefootnote}{\alph{footnote}}
\setcounter{footnote}{0}

\begin{figure}\label{F:AlgStackDef}
	\begin{center}
		\begin{longtable}{>{\raggedright}m{.3\textwidth}>{\raggedright}m{.4\textwidth}>{\raggedright}m{.3\textwidth}l}
			We call $\mathcal X$\ldots & if the diagonal  $\Delta : \mathcal X \rightarrow \mathcal X \times \mathcal X$ \newline is\ldots  \setcounter{footnote}{0}\footnotemark & and there is a scheme $U$ and a surjective morphism 
			\ $p : U \longrightarrow \mathcal X$  that is \ldots \footnotemark &\\[20pt]
			\hline \\
			a \emph{Deligne--Mumford algebraic  (DM algebraic) stack \cite[Def.~(4.6)]{DM}}\footnotemark & schematic & \'etale. & \\[20pt]
			a \emph{Knutson algebraic space \cite[Def.~II.1.1]{Knutson}} & injective and quasicompact \footnotemark & schematic  and \'etale. & \\[20pt]
			an \emph{Artin algebraic stack~\cite[Def.~(5.1)]{Artin-versal}}\footnotemark & representable by Knutson algebraic spaces &  smooth. & \\[20pt]
			a \emph{Laumon--Moret-Bailly (LMB) algebraic space \cite[Def.~(1.1)]{LMB}} & injective, schematic, and quasicompact &  \'etale. & \\[20pt]
			a \emph{Laumon--Moret-Bailly Deligne--Mumford (LMB DM) stack \cite[Def.~(4.1)]{LMB}} & representable by Laumon--Moret-Bailly algebraic spaces, separated, and quasicompact &   \'etale.   \footnotemark& \\[30pt]
			a \emph{Laumon--Moret-Bailly (LMB) algebraic stack \cite[Def.~(4.1)]{LMB}} & representable by Laumon--Moret-Bailly algebraic spaces, separated, and quasicompact &  smooth.  \setcounter{footnote}{5}\footnotemark & \\[25pt]
			a \emph{Fantechi Deligne--Mumford (F DM) stack \cite[Def.~5.2]{fantechi}} & & schematic  and \'etale. &\\[25pt]
			a \emph{Fantechi (F) algebraic stack \cite[Def.~5.2]{fantechi}}  & & schematic  and smooth. &\\[20pt]
			a \emph{stacks project (SP) algebraic space \cite[Tag~025Y]{stacks}} \setcounter{footnote}{6} \footnotemark & injective and schematic &   \'etale. & \\[20pt]
			a \emph{stacks project Deligne--Mumford (SP DM) stack \cite[Tag 03YO]{stacks}}  \setcounter{footnote}{6} \footnotemark & representable by SP algebraic spaces &  \'etale. &\\[25pt]
			a \emph{stacks project (SP) algebraic stack \cite[Tag~026O]{stacks}} \setcounter{footnote}{6} \footnotemark & representable by SP algebraic spaces &  smooth. &\\[20pt]
			\end{longtable}
	\end{center}
	\end{figure}

	\setcounter{footnote}{1}\footnotetext{Note  that our definition of representability is different, but equivalent to some of the notions used in the literature (see Lemma \ref{L:RepAff}).  
	The notions of quasicompact and separated morphisms of algebraic spaces defined in \S\ref{S:AlgStack2} carry over directly to all of the notions of algebraic spaces discussed here.
	Also, in the sources, the various notions of algebraic spaces are defined for sheaves, rather than stacks. Here, for uniformity,  we have simply added the injectivity hypothesis on the diagonal (see Lemma \ref{L:InjDiag}). }

	\setcounter{footnote}{2}\footnotetext{
The morphism $p:U\to \mathcal X$ is either assumed to be schematic, or is  representable by the same class of algebraic spaces for which the diagonal is representable by virtue of Lemma \ref{L:DiagIsomRep}.  In \S\ref{S:AlgStack2}, we defined  surjective, \'etale, and  smooth morphisms of algebraic spaces; these definitions carry over directly to all of the notions of algebraic spaces discussed here.}

	\setcounter{footnote}{3}\footnotetext{Deligne and Mumford caution that their definition is correct only for quasiseparated stacks \cite[Footnote~(1), p.~98]{DM}.}
	
	\setcounter{footnote}{4}\footnotetext{Quasicompactness of the  diagonal immediately implies quasicompactness of the map in \cite[Def.~1.1(c)]{Knutson}.  We leave the converse to the reader.  See also \cite[Tech.~Detail 1.p]{Knutson}, cf.~Stacks Project Algebraic Spaces.}
	
	\setcounter{footnote}{5}\footnotetext{Artin gives his definition only under an additional assumption of local finite presentation.}

	\setcounter{footnote}{6}\footnotetext{Laumon and Moret-Bailly ask only for a Laumon--Moret-Bailly algebraic space $U$ and a smooth surjection onto $\mathcal X$~\cite[Def.~(4.1)]{LMB}, but then $U$ has an \'etale cover by a scheme, so the definition is equivalent.}

	\setcounter{footnote}{7}\footnotetext{The Stacks Project requires its algebraic spaces and algebraic stacks to be sheaves in the fppf topology \cite[Tag 025Y, 026O]{stacks}.  This yields an equivalent definition by \cite[Tag 076M]{stacks} in the case of algebraic spaces and by \cite[Tag~076U]{stacks} in the case of algebraic stacks (see \S\ref{S:flat}).}

\renewcommand{\thefootnote}{\arabic{footnote}}
\setcounter{footnote}{\value{tempfoot}}

\begin{wrn}
There is an unfortunate, confusing point in the nomenclature introduced in Definition \ref{D:AlgStack}.  Deligne and Mumford defined an algebraic stack to be what we have, for the sake of historical verisimilitude, called a `Deligne--Mumford algebraic stack' above.  Artin defined algebraic stacks more inclusively, and the modern terminology is more inclusive still.  Meanwhile, the term Deligne--Mumford stack has come to refer to algebraic stacks with unramified diagonal.  As the term `algebraic' has become ever more inclusive, so has `Deligne--Mumford', so that now the class of `Deligne--Mumford stacks', while  contained in the class of algebraic stacks,  unfortunately includes some stacks that are not  `Deligne--Mumford algebraic stacks' in the sense we defined them here.   The relationship among the definitions is clarified in Figure \ref{E:kafka}.
\end{wrn}

\begin{rem}\label{R:CovSurj}
A smooth (resp.~\'etale) morphism of algebraic stacks $\pi : \mathcal Y \rightarrow \mathcal X$ that is representable by  algebraic spaces is surjective if and only if it is covering in the \'etale topology.  Indeed, $\pi$ is surjective if and only if its base change to any scheme is surjective, if and only if its base change to any scheme is covering, if and only if it is covering.
 In short, for   Definition \ref{D:AlgStack}, in the table on page \pageref{F:AlgStackDef}  we could replace the 
heading 
\emph{`... and there is a scheme $U$ and a surjective morphism $p:U\to \mathcal X$ ...'}
with the heading 
\emph{`... and there is a scheme $U$ and a cover $p:U\to \mathcal X$ ...'}.  
  Indeed, in the definition, we have either stipulated that $p$ is at least  representable by algebraic spaces, or we obtain this from the condition on the diagonal (see Lemma \ref{L:DiagIsomRep}).
\end{rem}

\subsection{Remarks on representability}

Note that in order to be able to speak about morphisms representable by the classes of algebraic stacks in  Definition \ref{D:AlgStack}, we need to know that these categories admit fiber products. 

\begin{lem}[{cf.\ \cite[Tags~02X2 and~04T2]{stacks}}] \label{L:FibProd}
	All of the classes of stacks in Definition \ref{D:AlgStack} admit fiber products and these coincide with fiber products taken on the underlying CFGs.
\end{lem}
\begin{proof}
	To deal with the entries in Definition~\ref{D:AlgStack} that involve a schematic cover (DM stack, K algebraic space, LMB algebraic space, FDM stack, F algebraic stack, SP algebraic space), suppose that $\mathcal X \rightarrow \mathcal Z$ and $\mathcal Y \rightarrow \mathcal Z$ are morphisms of stacks of the appropriate type.  Choose smooth schematic covers (or \'etale schematic covers, as the case warrants) $X \rightarrow \mathcal X$, $Y \rightarrow \mathcal Y$, and $Z \rightarrow \mathcal Z$ by schemes $X$, $Y$, and $Z$.  Set $X_Z = X \mathop{\times}_{\mathcal Z} Z$ and $Y_Z = Y \mathop{\times}_{\mathcal Z} Z$, as in the diagram below:
$$
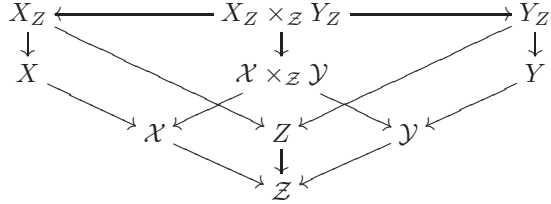
\xymatrix@!C=2pc@R=2ex{
X_Z \ar[d] \ar[rrdd]&&X_Z\times_{\mathcal Z}Y_Z \ar[rr] \ar[ll] \ar[d]&&Y_Z \ar[lldd] \ar[d]\\
X \ar[rd] &&\mathcal X\times_{\mathcal Z}\mathcal Y \ar[dl] \ar[dr]&&Y \ar[ld]\\
&\mathcal X \ar[rd]&Z\ar[d]&\mathcal Y \ar[ld]&\\
&&\mathcal Z&&\\
}
$$
The projections $X_Z \rightarrow \mathcal X$ and $Y_Z \rightarrow \mathcal Y$ are both smooth (or \'etale, according to the case) and covering (by base change and composition; Lemma \ref{L:FPCompCov}).  Now,
	\begin{equation*}
		X_Z \mathop{\times}_{\mathcal Z} Y_Z = (X_Z \times Y_Z) \mathop{\times}_{Z \mathop{\times} Z} (Z \mathop{\times}_{\mathcal Z} Z)
	\end{equation*}
	so $X_Z \mathop{\times}_{\mathcal Z} Y_Z$ is a fiber product of schemes, hence is a scheme. 
	Moreover, the map $X_Z \mathop{\times}_{\mathcal Z} Y_Z \rightarrow \mathcal X \mathop{\times}_{\mathcal Z} \mathcal Y$ is a fibered product of smooth (or \'etale, as the case warrants) coverings hence is a smooth (or \'etale) covering.  Finally,  via composition and base change, Lemma \ref{L:Rep} implies the maps $X_Z \rightarrow \mathcal X$ and $Y_Z \rightarrow \mathcal Y$ are schematic, so the same applies to $X_Z \mathop{\times}_{\mathcal Z} Y_Z \rightarrow \mathcal X \mathop{\times}_{\mathcal Z} \mathcal Y$ from Lemma \ref{L:Rep}\eqref{I:rep-prod}.

	Now this implies that fiber products of SP algebraic spaces are SP algebraic spaces so that the same argument can be repeated, with `representable by algebraic spaces' substituted for `schematic' and `algebraic space' substituted for `scheme'.  This proves the lemma for SP algebraic stacks and SP DM stacks.

	All of the remaining classes in Definition~\ref{D:AlgStack} can be characterized as algebraic stacks with additional conditions on the diagonal.  The conditions imposed on the diagonal are all properties that are preserved by fibered products of morphisms (note that properties preserved by composition and fibered product are preserved by fibered products of morphisms; see the proof of Lemma \ref{L:Rep}\eqref{I:rep-prod}), and the diagonal morphism of a fibered product is the fibered product of the diagonal morphisms.   
\end{proof}

\begin{rem}\label{R:MRAS}
Using Definition~\ref{D:Rep}, we now may speak of morphisms representable by any of the classes of stacks in Definition \ref{D:AlgStack}.
\end{rem}

Although we will not need to do so, it is sometimes convenient to test representability by various classes of stacks using smaller classes.  The following lemma shows that the class of affine schemes suffices to address most representability questions.

\begin{lem}\label{L:RepAff}
A morphism of stacks $f:\mathcal X\to \mathcal Y$ is representable by  a class of algebraic spaces in Definition \ref{D:AlgStack} (resp.~schematic) if and only if for every affine scheme $S=\operatorname{Spec}A$, the fibered product $\mathcal X\times_{\mathcal Y}S$ is representably by an algebraic space in that class (resp.~a scheme).
\end{lem}

\begin{proof}
All of the properties involved in the various definitions of algebraic space satisfy \'etale descent, so they can be verified \'etale locally.  Since algebraic spaces have \'etale covers by affine schemes (Theorem~\ref{T:more-alg-stack}) we can check if a morphism is representable by algebraic spaces in any sense by testing with affine schemes.  The same argument with Zariski descent and schemes takes care of the respected case.
\end{proof}

If $\mathcal X$ is a CFG then conditions on the diagonal correspond by base change to conditions on the fiber product $S \mathop{\times}_{\mathcal X \times \mathcal X} S$ for all schemes $S$ and all pairs of morphisms $x,y : S \rightarrow \mathcal X$.  Recall from Example~\ref{E:diagonal} that this fiber product may be identified with $\mathscr I\!\!\mathit{som}_{\mathcal X}(x,y)$ and from Lemma~\ref{L:InjDiag} that a CFG has injective diagonal if and only if it is equivalent to a sheaf.

\begin{lem}[{\cite[Cor.~3.13]{LMB}, \cite[Prop.~5.12]{DMstacks}, \cite[Tag 045G]{stacks}}]\label{L:DiagIsomRep}
Let $\mathcal M$ be a CFG over over $\mathsf S$.  The following conditions are equivalent:
\begin{enumerate}

\item The diagonal morphism $\mathcal M\stackrel{\Delta}{\longrightarrow} \mathcal M \times \mathcal M$ is representable by SP algebraic spaces (resp.~representable by K algebraic spaces, resp.~schematic);

\item For all $S$ in $\mathsf S$, and all $x,y$ in $\mathcal M(S)$, the presheaf $\mathscr I\!\!\mathit{som}(x,y)$ on $\mathsf S$ is representable by an SP algebraic $S$-space (resp.~K algebraic $S$-space, resp.~$S$-scheme);

\item For all $S$ in $\mathsf S$ and all $x$ in $\mathcal M(S)$, the morphism $x : S \rightarrow \mathcal M$ (guaranteed by the Yoneda lemma) is representable by SP algebraic spaces (resp.~representable by K algebraic spaces, resp.~schematic);

\item For every SP algebraic space (resp.~K algebraic space, resp.~scheme) $S$, we have that every morphism $S\to \mathcal M$ is SP-representable (resp.~K representable, resp.~schematic).   \label{L:DiagIsomRep:ItemT}
\end{enumerate}
\end{lem}

\subsection{Overview of the relationships among the definitions of algebraicity}

Our definition of an algebraic stack can also be characterized in a similar manner to Definition \ref{D:AlgStack}.
	
\begin{teo} \label{T:more-alg-stack}
	Each type of stack $\mathcal X$ in the first column below is characterized by the condition on the diagonal in the second column and the existence of a surjection $p : U \rightarrow \mathcal X$ satisfying the condition in the third column.
	\begin{center}
		\begin{tabular}{>{\raggedright}m{.3\textwidth}>{\raggedright}m{.4\textwidth}>{\raggedright}m{.3\textwidth}l}
			A stack $\mathcal X$ is... & if the diagonal  $\Delta : \mathcal X \rightarrow \mathcal X \times \mathcal X$ \newline is\ldots  & and there is a scheme $U$ and a surjective morphism 
			\ $p : U \longrightarrow \mathcal X$  that is \ldots   &\\[20pt]
			\hline \\
			an algebraic space  (Definition \ref{D:AlgDMAlgSp})    & injective   & schematic and  \'etale. & \\[20pt]
			a Deligne--Mumford stack (Definition \ref{D:AlgDMAlgSp})	&	unramified & representable by algebraic spaces and smooth. &\\[20pt]
			a Deligne--Mumford stack (Definition \ref{D:AlgDMAlgSp})    &   & representable by algebraic spaces and \'etale. &\\[20pt]
			an algebraic stack  (Definition \ref{D:AlgDMAlgSp}) &  & representable by algebraic spaces and smooth. &\\[20pt]
		\end{tabular}
	\end{center}		
\end{teo}

\begin{proof}[Sketch.]

Suppose  $\mathcal X$ is an algebraic stack.   Considering the iterative nature of Definition \ref{D:alg-stack}, it is clear that $\mathcal X$ has a smooth cover $P:U\to \mathcal X$ by a scheme $U$. We need to show that $P$ is representable by algebraic spaces.   This morphism must have injective relative diagonal (since $U \mathbin\times_{\mathcal X} U$ is a sheaf of sets and $U \rightarrow U \mathbin\times_{\mathcal X} U \rightarrow U \times U$ is injective), 
so it is representable by algebraic spaces, as required.

This argument applies also to show that Deligne--Mumford stacks have \'etale covers by algebraic spaces. 

	Suppose now that $\mathcal X$ is an algebraic space.  Then the diagonal of $\mathcal X$ is injective and representable by algebraic spaces.  In particular, it is locally quasifinite and separated, so by separated, locally quasifinite descent \cite[Tag 02W8]{stacks}, it is schematic.  

	To complete the proof, we need to show that algebraic spaces and algebraic stacks with unramified diagonals have \emph{\'etale} covers by schemes.  We will show these statements simultaneously (using an `induction on stackiness'), mostly following \cite[Tag 06N3]{stacks}.

	We argue that every finite-type point of $\mathcal X$ has an \'etale neighborhood that is a scheme.  This will suffice, since if $x$ is a geometric point of $\mathcal X$, we can find a smooth map $U \rightarrow \mathcal X$ where $U$ is an affine scheme and $x$ lifts to $U$.  Then this lift has a specialization to a closed point of $U$, which induces a point $y$ of $\mathcal X$ of finite type.  Any \'etale neighborhood of $y$ will also contain $x$.

	Let $x = \operatorname{Spec} k$ be a point of $\mathcal X$ of finite type.  We argue first that there is a factorization $x \rightarrow y \rightarrow \mathcal X$ where $y$ is \emph{unramified} over $\mathcal X$.  Let $R = x \mathbin\times_{\mathcal X} x$.  Then $R$ is flat over $x$ (since $x$ is the spectrum of a field) and of finite type (since the composition $u \mathbin\times_{\mathcal X} u \rightarrow u \mathbin\times_{\mathcal X} \mathcal X$ is the base change of a morphism $u \rightarrow \mathcal X$ of finite type). 

	Since $R$ is of finite type over $x$ via the first projection, its geometric fiber has finitely many components.  Replacing $k$ with a finite extension, we can therefore assume that the connected components of $R$ are geometrically connected.  Let $R_0 \subset R$ be the connected component of the diagonal section.  Note that $R_0 \rightarrow x \times x$ is unramified and its fiber over the diagonal $x \rightarrow x \times x$ is injective on geometric points.  But every geometric fiber is either empty or is a torsor under the fiber over the diagonal, so that every fiber is injective on geometric points.  It follows that $R_0 \rightarrow x \times x$ is a injection, so that it defines an equivalence relation on $x$.

	Let $y = x / R_0$ be the quotient of $x$ by this equivalence relation.  Since $R_0 \subset R$ is open, $R_0$ is flat over $x$ and therefore this is an algebraic space (Theorem~\ref{T:ArtinPres}) equipped with a map $y \rightarrow \mathcal Z$.  In fact, $y$ must be the spectrum of a field:  choose a smooth cover $\operatorname{Spec} A \rightarrow R_0$ by a scheme, and let $\ell$ be the equalizer of the two maps $k \rightarrow A$.  Then $\ell$ is a field (since an element of $k$ is equalized by two homomorphisms if and only if its inverse is).  Moreover, we obtain a map $y \rightarrow \operatorname{Spec} \ell$ that is covering and injective in the fppf topology.  Therefore it is an isomorphism (see \cite[Tag 0B8A]{stacks}).

	Now we argue that $y \rightarrow \mathcal X$ is unramified.  Indeed, the diagonal map $y \rightarrow y \mathbin\times_{\mathcal X} y$ pulls back via the fppf cover $R_0 = x \mathbin\times_{\mathcal X} x \rightarrow y \mathbin\times_{\mathcal X} y$ to $x \mathbin\times_{y} x = R_0$, which is open in $R$.  Therefore $y \rightarrow y \mathbin\times_{\mathcal X} y$ is an open embedding, so $y$ is unramified over $\mathcal X$.

	Now choose a smooth morphism $U \rightarrow \mathcal X$ containing $y$ in its image, with $U = \operatorname{Spec} A$ affine.  Let $V = U \mathbin\times_{\mathcal X} y$.  Then $V$ is a smooth algebraic space over $y$ and $V \rightarrow U$ is unramified.  If $\mathcal X$ is an algebraic space then we have seen that $U \rightarrow \mathcal X$ is schematic, so that $V$ is a scheme; in that case we write $W = V$.  In general, then we know the diagonal of $\mathcal X$ is at least representable by algebraic spaces, so $V$ is an algebraic space.  By `induction on stackiness' we may assume that there is a scheme $W$ and an \'etale map $W \rightarrow V$ whose image in $\mathcal X$ contains the image of $y$.

	Up to an \'etale extension of $\ell$, we can assume that $y$ is the image of an $\ell$-point of $W$, which we denote $w$.  We take $u$ to be its image in $U$.  The map $W \rightarrow U$ is unramified, so the maximal ideal $\mathfrak m$ of $w$ in $\mathcal O_W$ is generated by the image of the maximal ideal $\mathfrak n$ of $u$ in $U$.  We can therefore choose functions $f_1, \ldots, f_d \in \mathcal O_{U,u}$ whose images in $\mathcal O_{W,w}$ form a basis for $\mathfrak m/\mathfrak m^2$.  Replacing $W$ by an open neighborhood of $w$, the vanishing locus of $f_1, \ldots, f_d$ in $W$ will be $\{ w \}$.  In particular, the vanishing locus is unramified over $y$.  But the locus where $U \rightarrow \mathcal X$ is unramified is open, so that there is an open neighborhood $U'$ of $u \in U$ where $U' \rightarrow \mathcal X$ is unramified.  Since it is also smooth, it is \'etale, as required.
	\end{proof}

The implications among all of the definitions are described in Figure \ref{E:kafka}.  They can all essentially be explained by putting various conditions on the diagonal of an algebraic stack.  In this sense, from an expository perspective,  the definition of an algebraic stack is the basic definition, and the rest can easily be obtained from this. In practice this  is somewhat misleading, however, since one must first define an algebraic space to define an algebraic stack.

\begin{figure}
\begin{equation*} \vcenter{
\xymatrix@C=.9cm @R=.5cm{
&&\text{DM~Alg.} \ar@{=>}[dd] \ar@{=>}[r]&\text{SP~Alg.}~\Delta \text{ sch.} \ar@{=>}[dd]  \ar@{<=>}@/_1 pc/[l]_<>(0.5){\ \Delta \text{ n.r.}}\\
&&&\\
& & \text{F DM} \ar@{=>}[r]  \ar@{=>}[dd]  &  \text{F Alg.}  \ar@{=>}[dd]  \ar@{<=>}@/_1 pc/[l]_<>(0.5){\ \Delta \text{ n.r.}} \\
&&&\\
	\text{Sch.} \ar@{=>}[r]  \ar@{<=>}@/_.8 pc/[dd]_<>(0.5){\ \Delta \text{ q.c.}}&\text{SP~Alg.~Sp.} \ar@/^1pc/@{=>}[ruuuu]  \ar@{=>}[r]  \ar@{<=>}@/_.8 pc/[dd]_{\text{\scriptsize $\Delta$ q.c.}}& \text{SP~DM}  \ar@{<=>}@/^1.6 pc/[uuuu]^<>(0.4){\ \Delta \text{ sch.}}  \ar@{=>}[r]   \ar@{<=>}@/_.8 pc/[dd]_<>(0.5){\ \Delta \text{ q.c.+sep.}}   \ar@{<=>}@/_1 pc/[l]_<>(0.5){\ \Delta \text{ inj}} & \text{SP~Alg.} \ar@{<=>}@/_1.6 pc/[uuuu]_<>(0.4){\ \Delta \text{ sch.}} \ar@{<=>}@/_.8 pc/[dd]_<>(0.5){\ \Delta \text{ q.c.+sep.}}  \ar@{<=>}@/_1 pc/[l]_<>(0.5){\ \Delta \text{ n.r.}}\\
&&&\\
\text{q.s.~Sch.} \ar@{=>}[r] \ar@{=>}[uu]&\text{LMB~Alg.~Sp.} \ar@{=>}[r] \ar@{=>}[uu]&  \text{LMB~DM}  \ar@{=>}[r]      \ar@{=>}[uu] \ar@{<=>}@/^1 pc/[l]^<>(0.5){\ \Delta \text{ inj}}& \text{LMB~Alg.} \ar@{=>}[uu]  \ar@{<=>}@/^1 pc/[l]^<>(0.5){\ \Delta \text{ n.r.}}\\
&\text{K~Alg.~Sp.} \ar@{<=>}[u] &&\\
}}
\end{equation*}
\caption[Diagram of definitions of algebraic stacks]{
\label{E:kafka}
An arrow from one entry to another signifies that the class of objects at the arrow's tail are also of the type at its head.  An arrow with a label means that an object of the type at the tail satisfying the additional condition named in the label is also of the type at the head.  A double-headed arrow should be interpreted as a pair of arrows pointing in both directions with the same label; in other words, the condition in the label makes the conditions at its ends equivalent.  
}
\end{figure}

\subsection{Relationships among definitions of algebraic spaces}
\label{S:rel-alg-sp}
The following implications hold for algebraic spaces:

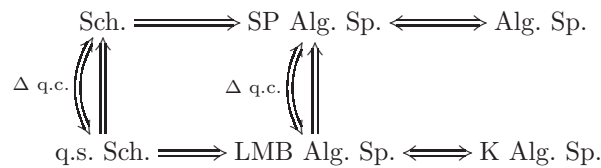
\begin{equation*} \vcenter{
\xymatrix@C=.9cm @R=.5cm{
	\text{Sch.} \ar@{=>}[r]  \ar@{<=>}@/_.8 pc/[dd]_<>(0.5){\ \Delta \text{ q.c.}}&\text{SP~Alg.~Sp.}    \ar@{<=>}@/_.8 pc/[dd]_{\text{\scriptsize $\Delta$ q.c.}} &\text{Alg.~Sp.} \ar@{<=>}[l]\\
	&\\
\text{q.s.~Sch.} \ar@{=>}[r] \ar@{=>}[uu]&\text{LMB~Alg.~Sp.} \ar@{=>}[uu]&\text{K~Alg.~Sp.} \ar@{<=>}[l] \\
}}
\end{equation*}
An arrow with a label indicates that the implication holds under the additional assumption indicated on the diagonal.  A two headed arrow implies that the definitions are equivalent under the given assumption on the diagonal.

The only arrows that require justification are the equivalence between LMB algebraic spaces and K algebraic spaces, and the equivalence between  algebraic spaces and SP algebraic spaces.

\begin{lem}\label{L:AlgSpaceEquiv}
LMB algebraic spaces are the same as K algebraic spaces and  SP algebraic spaces are the same as  algebraic spaces.
\end{lem}

\begin{proof}
  It is clear that LMB algebraic spaces are K algebraic spaces and that SP algebraic spaces are algebraic spaces.  For the converse, we only need to show that the diagonal is schematic.  This is \cite[Tag 046K]{stacks}.  In fact,   the diagonal of an algebraic space is injective, hence separated and locally quasifinite.  Therefore the diagonal is schematic, by separated, locally quasifinite descent~\cite[Tag~02W8]{stacks}.
\end{proof}

\subsection{Relationships among the definitions of algebraic stacks}
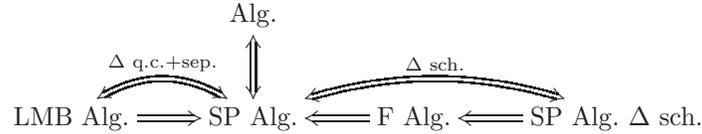
\label{S:rel-alg}

\begin{equation*} 
\xymatrix{
&\text{Alg.} \ar@{<=>}[d]&&\\
\text{LMB~Alg.} \ar@{=>}[r]  \ar@{<=>}@/^15pt/[r]^{\text{$\Delta$ q.c.+sep.}}& \text{SP~Alg.}   \ar@{<=}[r]   \ar@{<=>}@/^15pt/[rr]^{\text{$\Delta$  sch.}} & \text{F~Alg.} \ar@{<=}[r]   & \text{SP~Alg.}~\Delta \text{ sch.}\\
} 
\end{equation*}

Most of the implications are immediate, so we only make a few comments.

\begin{lem}\label{L:DiagAlgStack}
The diagonal morphism of an algebraic stack  is representable by algebraic spaces.  
\end{lem}

\begin{proof}
 One can check this using Theorem \ref{T:more-alg-stack} and a slight modification of the proof of  \cite[Tag 04XS]{stacks}.
\end{proof}

\begin{lem}
Algebraic stacks are the same as  SP algebraic stacks.
\end{lem}
\begin{proof} From Theorem \ref{T:more-alg-stack}, we only need to show that the diagonal morphism is representable by SP algebraic spaces.  Since we have shown already that algebraic spaces are SP algebraic spaces,  we conclude using the previous lemma.
\end{proof}

\begin{lem}
An SP algebraic stack with quasicompact and separated diagonal is an LMB algebraic stack.
\end{lem}
\begin{proof}
Suppose that $\mathcal X$ is  SP algebraic, with quasicompact and separated diagonal.  By Lemma \ref{L:DiagAlgStack}, the  diagonal morphism for $\mathcal X$ is representable by algebraic spaces.  An algebraic space is an LMB algebraic space if and only if its diagonal is quasicompact (see Section~\ref{S:rel-alg-sp}).  Since the diagonal of $\mathcal X$ is separated, the double diagonal is a closed embedding and, a fortiori, quasicompact.  Hence the diagonal of $\mathcal X$ is representable by LMB algebraic spaces, as required.
\end{proof}

\begin{rem}  
We expect  there are F algebraic stacks that do not have schematic diagonal, but we do not know of any example.  We also expect that there are SP algebraic stacks that do not admit a smooth schematic morphism from a scheme (and therefore are not F algebraic stacks), but we do not know of any example.  See Example \ref{E:DMnSchD} for  an SP algebraic stack that does not have schematic diagonal.
\end{rem}

\subsection{Relationships among the definitions of Deligne--Mumford stacks}
\label{S:rel-dm}

For Deligne--Mumford stacks, we have implications:

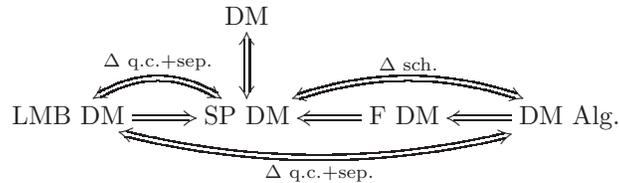
\begin{equation*} 
\xymatrix{
&\text{DM} \ar@{<=>}[d]&&\\
\text{LMB~DM} \ar@{=>}[r]  \ar@{<=>}@/^15pt/[r]^{\text{$\Delta$ q.c.$+$sep.}}   \ar@{<=>}@/_15pt/[rrr]_{\text{$\Delta$ q.c.$+$sep.}} & \text{SP~DM}   \ar@{<=}[r]   \ar@{<=>}@/^15pt/[rr]^{\text{$\Delta$  sch.}}& \text{F~DM} \ar@{<=}[r]   & \text{DM Alg.}\\
} 
\end{equation*}
Most of the implications are immediate, but we focus on  the main points:

\begin{lem}
DM stacks are the same as SP DM stacks.  
\end{lem}

\begin{proof}
From Theorem \ref{T:more-alg-stack} we only need to show that a DM stack has diagonal representable by algebraic spaces (as these are the same as SP algebraic spaces).   However, by Theorem \ref{T:more-alg-stack} it is immediate that DM stacks are algebraic stacks, and we have seen in Lemma \ref{L:DiagAlgStack} that the diagonal of an algebraic stack is representable by algebraic spaces. 
\end{proof}

\begin{lem}[{\cite[Lem.~4.2]{LMB}}] \label{L:LMB-DM-sch-diag}
All LMB DM stacks have schematic diagonal.
\end{lem}
\begin{proof}
Any locally quasifinite, separated morphism that is representable by algebraic spaces is schematic~\cite[Tag 03XX]{stacks}.
\end{proof}

The rest of the implications are obvious from Theorem \ref{T:more-alg-stack} and the two lemmas above.

\begin{exa}\label{E:DMnSchD}
We will construct an SP Deligne--Mumford stack without  schematic diagonal (i.e., an SP DM stack that is not a DM algebraic stack).  Let $G$ be $\mathbb A^1$ with a doubled origin.  This can be regarded as a group scheme over $\mathbb A^1$ by distinguishing one of the two origins as the identity element.  If $X \rightarrow \mathbb A^1$ is a morphism of schemes, with $X_0$ the fiber over the origin, then a $G$-torsor on $X$ is a $\mathbb Z/2\mathbb Z$-torsor on $X_0$.

We will show that the map $p : \mathbb A^1 \rightarrow \mathrm BG$ is not schematic.  Indeed, if $Z \rightarrow \mathrm BG$ is any morphism then the base change of $p$ is the total space of the corresponding torsor.  Therefore we have to find a scheme $Z$ over $\mathbb A^1$ and a $G$-torsor over $Z$ that is not representable by a scheme.

To find such a torsor, choose a scheme $W$ and an \'etale double cover $W' \rightarrow W$ that does not have a section Zariski-locally.  Let $Z = \mathbb A^1_W$ (so that $Z_0 = W$) and let $P$ be the $G$-torsor over $Z$ corresponding to $W' \rightarrow Z_0$.
\end{exa}

\begin{rem}  
In regards to the example above, we expect that there are also F DM stacks that do not have schematic diagonal (and are therefore not DM algebraic stacks), but we do not know of any example.  We also expect that there are SP DM stacks that do not admit a smooth schematic morphism from a scheme (and therefore are not F DM stacks), but we do not know of any example.   Example~\ref{E:DMnSchD} shows there are SP algebraic stacks that do not have schematic diagonal, and are therefore not DM algebraic stacks.
\end{rem}

\subsection{Stacks with unramified diagonal}

The implications in Figure \ref{E:kafka} regarding unramified diagonal all follow immediately from the following lemma:

\begin{lem}\label{L:UnRamDiagDM}
An algebraic stack with unramified diagonal is a DM stack.
\end{lem}
\begin{proof} Now that we have the identification between algebraic stacks and SP algebraic stacks, and DM stacks and SP DM stacks, this is  \cite[Tag 06N3]{stacks}.
\end{proof}

\subsection{The adapted perspective} \label{S:AdaptPersp}

We show how all of the stacks that have appeared in this paper can be described as  stacks adapted to a given presite.

\subsubsection{Stacks adapted to the \'etale presite of schemes}
We have already seen the following.   
A stack adapted to the \'etale presite with injective diagonal is the same thing as an algebraic space.  
A stack adapted to the \'etale presite with injective and quasicompact diagonal is the same thing as an LMB  algebraic space.
A stack adapted to the \'etale presite is the same thing as an F DM stack.  
A stack adapted to the \'etale presite with quasicompact and separated diagonal is the same thing as an LMB DM stack (Lemma~\ref{L:LMB-DM-sch-diag}).

\subsubsection{Stacks adapted to the \'etale presite of algebraic spaces}

Stacks adapted to the \'etale presite of algebraic spaces  induce  Deligne--Mumford stacks on the \'etale presite of schemes.  Indeed, given a  stack in the \'etale topology on algebraic spaces one obtains  a stack on the \'etale presite of schemes by restriction.  Conversely, a stack on the \'etale presite of schemes extends uniquely to the \'etale presite of algebraic spaces since every algebraic space has an \'etale cover by schemes.  Using again that every algebraic space has an \'etale cover by schemes,  the definition of stack adapted to the \'etale presite of algebraic spaces agrees with the characterization of Deligne--Mumford stacks in Theorem~\ref{T:more-alg-stack}.

\subsubsection{Stacks adapted to the smooth presite of algebraic spaces}
Stacks adapted to the smooth presite of algebraic spaces induce algebraic stacks by restricting from the category of algebraic spaces to schemes.  

\subsubsection{Stacks adapted to the fppf presite of algebraic spaces}
A stack adapted to the fppf presite of  algebraic spaces induces an   algebraic  stack on the \'etale presite of schemes, by restricting from the category of algebraic spaces to the category of schemes (see Theorem~\ref{T:ArtinPres}).

\subsubsection{An example of a stack that is not adapted to a presite}

\begin{exa}\label{E:LogAb}
Consider the logarithmic abelian varieties of Kajiwara, Kato, and Nakayama \cite{KKN-I,KKN-II,KKN-III}, which are sheaves in the \'etale topology and possess logarithmically \'etale covers by logarithmic schemes but have no such \'etale covers (logarithmically \'etale maps are a more general class of morphisms including all \'etale maps but also some blowups and other non-flat morphisms).  
However, the logarithmic \'etale topology is not subcanonical and there is no subcanonical topology in which logarithmic abelian varieties are covered by logarithmic schemes. 
\end{exa}

\subsection{Conditions on the relative diagonal of a stack and bootstrapping}
As we have seen, most of the variations on the definition of an  algebraic stacks outlined in Definition \ref{D:AlgStack}  can be obtained from the definition of an algebraic stack by imposing conditions on the diagonal.  Here we aim to extend the bootstrapping result of Proposition \ref{P:AdaptRel} to these other cases.  
In other words, our goal here is to show that a stack that is \emph{relatively algebraic} over an algebraic stack (according to any of the definitions in Definition~\ref{D:AlgStack}) is itself algebraic.

For the next lemma, recall that for a morphism   $f:\mathcal X\to \mathcal Y$ of algebraic stacks,  the diagonal $\Delta_f$ is representable by  algebraic  spaces (\cite[Tag 04XS]{stacks}).

\begin{lem}\label{L:DcondL} Let $\mathbf P$ be a property of morphisms of  algebraic spaces that is stable under composition and base change.  
Let $\mathcal X,\mathcal Y,\mathcal Z$ be  algebraic stacks over $\mathsf S$.
\begin{enumerate}

\item If $\mathcal X\stackrel{f}{\to} \mathcal Y$  has property $\mathbf P$ for $\Delta_f$   and $\mathcal Y'\to  \mathcal Y$ is any morphism, then the morphism $\mathcal X\times_{\mathcal Y}\mathcal Y'\stackrel{f'}{\to}\mathcal Y'$ obtained from the fibered product has property $\mathbf P$ for $\Delta_{f'}$.

\item If $\mathcal X\stackrel{f}{\to}\mathcal Y$ and $\mathcal Y\stackrel{g}{\to}\mathcal Z$ 
are morphisms such that $\Delta_f$ and $\Delta_{g}$ have property $\mathbf P$, then for the composition $\mathcal X\stackrel{g\circ f}{\to}\mathcal Y$, the diagonal morphism $\Delta_{g\circ f}$ has property $\mathbf P$.

\item If $\mathcal X\stackrel{f}{\to} \mathcal X'$ and $\mathcal Y\stackrel{g}{\to} \mathcal Y'$ are morphisms over a stack $\mathcal Z$ such that $\Delta_f$ and $\Delta_g$ have property $\mathbf P$, 
 then for the fibered product morphism 
$\mathcal X\times_{\mathcal Z} \mathcal Y\stackrel{f\times_{\operatorname{id}_{\mathcal Z}} g}{\longrightarrow} \mathcal X'\times_{\mathcal Z}\mathcal Y'$ the diagonal  morphism $\Delta_{f\times_{\operatorname{id}_{\mathcal Z}} g}$ has property $\mathbf P$.

\item  A morphism $f : \mathcal X \rightarrow \mathcal Y$ has property $\mathbf P$ for $\Delta_f$ if and only if 
 for every scheme $S$ and every morphism $S \rightarrow \mathcal Y$, the base change $f':\mathcal X \mathop{\times}_{\mathcal Y} S \rightarrow S$ has property $\mathbf P$ for $\Delta_{f'}$.
 
\item Suppose that property $\mathbf P$ satisfies the following condition:  if If $ X\stackrel{f}{\to} Y$ and $ Y\stackrel{g}{\to} Z$ are morphisms of algebraic spaces such that $g$ and $f\circ g$ satisfy property $\mathbf P$, then $f$ satisfies property $\mathbf P$.    Then  If $\mathcal X\stackrel{f}{\to}\mathcal Y$ and $\mathcal Y\stackrel{g}{\to}\mathcal Z$ 
are morphisms such that $\Delta_g$ and $\Delta_{g\circ f}$ have property $\mathbf P$, then   $\Delta_{f}$ has property $\mathbf P$.
\end{enumerate}
\end{lem}

\begin{proof}  This is essentially  contained in \cite[Tag 04YV]{stacks}.    
(1)  follows from the $2$-cartesian diagrams:
\begin{equation}\label{E:DiagBC}
\xymatrix{
\mathcal X' \ar@{->}[r]^{f'}\ar@{->}[d]& \mathcal Y' \ar@{->}[d]\\
\mathcal X \ar@{->}[r]^f& \mathcal Y\\
}
\hskip .5 in 
\xymatrix{
\mathcal X' \ar@{->}[r]^{\Delta_{f'}\ \ \ }\ar@{->}[d]& \mathcal X'\times_{\mathcal Y'}\mathcal X' \ar@{->}[d]\\
\mathcal X \ar@{->}[r]^{\Delta_f\ \ \ }&\mathcal X\times_{ \mathcal Y}\mathcal X\\
}
\end{equation}
The diagram on the left induces the diagram on the right, and it then follows from base change that $\Delta_{f'}$ has property $\mathbf P$. 

(2) follows from (1) using the $2$-commutative diagram below for  the morphisms $f:\mathcal X\to \mathcal Y$ and $g:\mathcal Y\to \mathcal Z$.   The square is a $2$-fibered product.
\begin{equation}\label{E:CompDiag}
\xymatrix{
\mathcal X \ar@{->}[r]^<>(0.7){\Delta_f} \ar@{->}[rd]  \ar@{->}@/^2pc/[rr]^{\Delta_{g\circ f}}& \mathcal X\times_{ \mathcal Y}\mathcal X \ar@{->}[r]  \ar@{->}[d]& \mathcal X\times_{\mathcal Z} \mathcal X \ar@{->}[d]\\
& \mathcal Y\ar@{->}[r]^<>(0.5){\Delta_{g}}&\mathcal Y\times_{\mathcal Z}\mathcal Y. 
}
\end{equation}

(3) is essentially \cite[Rem.~(1.3.9) p.33]{EGAI}, which observes that the conclusion follows from (1) and (2), together with  the fact that given morphisms $f:\mathcal X\to \mathcal X'$ and $g:\mathcal Y\to \mathcal Y'$ over a stack $\mathcal Z$, the product   $\mathcal X\times_{\mathcal Z}\mathcal Y\stackrel{f\times_{\operatorname{id}_{\mathcal Z}}g}{\longrightarrow} \mathcal X'\times_{\mathcal Z}\mathcal Y'$ is given by the composition of morphisms obtained from fibered product diagrams:
$$
\begin{CD}
\mathcal X\times_{\mathcal Z}\mathcal Y @>f\times_{\operatorname{id}_{\mathcal Z}} \operatorname{id}_{\mathcal Y}>> \mathcal X'\times_{\mathcal Z}\mathcal Y @>\operatorname{id}_{\mathcal X'}\times_{\operatorname{id}_{\mathcal Z}} g>> \mathcal X'\times_{\mathcal Z}\mathcal Y'.
\end{CD}
$$

(4)  If $\Delta_f$ has property $\mathbf P$, this follows from (1).  Conversely, assume that 
 for every scheme $S$ and every morphism $S \rightarrow \mathcal Y$, the base change $\mathcal X \mathop{\times}_{\mathcal Y} S \rightarrow S$ has property $\mathbf P$ for its diagonal.
By definition, for 
 $\Delta_f : \mathcal X \rightarrow \mathcal X \mathop{\times}_{\mathcal Y} \mathcal X$ to have  property $\mathbf P$ means that the base change $S \mathop{\times}_{\mathcal X \mathop{\times}_{\mathcal Y} \mathcal X} \mathcal X \rightarrow S$ has property $\mathbf P$.  But
	\begin{equation*}
		S \mathop{\times}_{\mathcal X \mathop{\times}_{\mathcal Y} \mathcal X} \mathcal X = S \mathop{\times}_{\mathcal X_S \mathop{\times}_S \mathcal X_S} \mathcal X_S
	\end{equation*}
	where $\mathcal X_S = \mathcal X \mathop{\times}_{\mathcal Y} S$ and by assumption, the diagonal of $\mathcal X_S \rightarrow S$ has property $\mathbf P$. 
	
(5) This follows from diagram \eqref{E:CompDiag}, and the previous parts of the lemma.
\end{proof}

\begin{cor} \label{C:AlgRel}
	For any of the classes $\mathsf C$ of objects introduced in Definition \ref{D:AlgStack}, if $\mathcal Y$ is of class $\mathsf C$ and $f : \mathcal X \rightarrow \mathcal Y$ is representable by objects of $\mathsf C$ then $\mathcal X$ is of class $\mathsf C$.
\end{cor}

\begin{proof}
Lemma \ref{L:AlgRel} covers the case of algebraic stacks.
For  F algebraic stacks one can easily adapt the proof of Lemma \ref{L:AlgRel}.  The remaining  classes $\mathsf C$ of stacks introduced in Definition \ref{D:AlgStack} can can be obtained by imposing various conditions  on the diagonal of an  algebraic stack, all of which are stable under composition and base change.  
Fix a class of such stacks, and call the necessary conditions on the diagonal condition $\mathbf P$.  
In particular,  $\mathcal Y$ is of class $\mathsf C$ means that the diagonal $\Delta_\pi$ of the structure map $\pi:\mathcal Y\to \mathsf S$ has property $\mathbf P$.   If $f : \mathcal X \rightarrow \mathcal Y$ is representable by objects of $\mathsf C$, then from Lemma \ref{L:DcondL}(4), we have that $\Delta_f$ has property $\mathbf P$.  Then by Lemma \ref{L:DcondL}(2), we have that $\Delta_{\pi\circ f}$ has property $\mathbf P$.  In other words, $\mathcal X$ is of class $\mathsf C$. 
\end{proof}

\begin{rem}  The arguments above show the following, as well (\cite[Tag 04YV]{stacks}). 
Let   $f:\mathcal X\to \mathcal Y$ be a morphism of algebraic stacks.
The morphism  $f$ is representable by LMB algebraic stacks  (resp.~LMB DM stacks, resp.~LMB algebraic spaces)  if and only if  $\Delta_f$ is quasicompact and separated (resp.~quasicompact, separated, and unramified, resp.~quasicompact, separated, and injective).
\end{rem}

\section{Groupoids and stacks}\label{S:2-cat-2}

Stacks are often studied via groupoid objects.  In this section we discuss torsors, groupoid objects, and stacks arising as quotients of groupoid objects.  In the end we  show that groupoid objects adapted to a presite, i.e., those where the source and target maps are coverings in the presite, are essentially the same thing as stacks adapted to the presite.

\subsection{Torsors and group quotients} \label{S:TorsG}
Let $X$ be in $\mathsf S/S$, and let $G$ be a sheaf of groups over $S$ acting on the right  on $X$:
$$
\begin{CD}
X\times_SG @>\sigma >> X
\end{CD}
$$
 We define a CFG over $\mathsf S/S$, $[X/G]$ 
in the following way.  The objects over an $S$-scheme $f:S'\to S$ are  diagrams
$$
\xymatrix@C=1.5em@R=1.5em{
P'\ar@{->}[d]\ar@{->}[r]&X_{S'}\\
S'&\\
}
$$
where $P'$ is a $G_{S'}$-torsor (principal bundle) over $S'$, and $P'\to X_{S'}$ is a $G_{S'}$-equivariant morphism.    Morphisms are defined by pullback.  There is a morphism $[X/G]\to S$ given by forgetting everything except the $S$-scheme $f:S'\to S$, and there is an $S$-morphism $q:X\to [X/G]$ given by the trivial $G_X$-bundle 
$$
\xymatrix{
X\times_SG\ar@{->}[d]_{pr_1}\ar@{->}[r]^{\operatorname{id}\times\sigma}&X\times_SX \\
X.&\\
}
$$
This induces a $2$-cartesian diagram
\begin{equation}\label{E:X/G}
\begin{gathered}
\xymatrix{
X\times_SG\ar@{->}[d]_{pr_1}\ar@{->}[r]^{ \sigma}& X \ar@{->}[d]^q \\
X\ar@{->}[r]^{q \ \  \ }&[X/G]\\
}
\end{gathered}
\end{equation}
that is a co-equalizer for  $
\xymatrix{
X\times_SG \ar@{->}@<-2pt>[r]_<>(0.5){pr_1} \ar@{->}@<2pt>[r]^<>(0.5){\sigma}&  X 
}
$ (i.e., initial in the category of stacks for the diagram \eqref{E:X/G}); in other words,    $[X/G]$ is a quotient in the category of stacks  for the action of $G$ on $X$, in the sense that any $G$-equivariant map out of $X$ factors through it.
     Note that if there exists a scheme  $X/G$ that is a quotient in the category of schemes  for the action of $G$ on $X$ (i.e., a co-equalizer in the category of \emph{schemes}), then there is a morphism
$$
[X/G]\to X/G.
$$

Of particular importance is the trivial action of $G$ on $X = S$.  The quotient $[S/G]$ is denoted $\mathrm{B}G$.  As a CFG, $\mathrm{B}G$ consists of the pairs $(S', P')$ where $S' \in \mathsf S/S$ and $P'$ is a $G$-torsor over $S$.

\begin{rem}
There is  a similar construction for left group actions; in the notation above, we would have the stack $[G\backslash X]$. 
\end{rem}

\subsection{Groupoid objects and groupoid quotients}

\subsubsection{Groupoid objects}

Let $\mathsf{Gpd}$ denote the \emph{category} (not $2$-category) of small groupoids.  There are functors $\mathsf{Obj}$ and $\mathsf{Mor}$ from $(\mathsf{Gpd})$ to $(\mathsf{Set})$ sending a groupoid $X$, respectively, to its set $X_0$ of objects and its set $X_1$ of morphisms.  There are two canonical morphisms $s, t : X_1 \rightarrow X_0$ sending a morphism to its source and target.  More data are required to specify a groupoid, but these are often left tacit and the groupoid is usually denoted by a pair of morphisms of sets 
$
\xymatrix@C=1.5em{
X_1 \ar@{->}@<-2pt>[r]_<>(0.5){s} \ar@{->}@<2pt>[r]^<>(0.5){t}&  X_0
}
$.

 A groupoid object in a category $\mathsf S$ is defined by taking $X_1$, $X_0$, and all of the morphisms involved in specifying the groupoid to lie in $\mathsf S$ (as opposed to $\mathsf {Set}$).  
 If this is the case, then  we obtain  functors $\operatorname{Hom}_{\mathsf S}(-,X_i)$ from $\mathsf S^{\operatorname{op}}$ to $(\mathsf {Set})$.  
This can all be said more  concisely as follows:

\begin{dfn}[Groupoid object] \label{D:gpd-obj}
A \emph{groupoid object} of a category $\mathsf S$ is a functor $\mathscr X : \mathsf S^{\rm op} \rightarrow (\mathsf{Gpd})$, along with objects $X_0$ and $X_1$ in $\mathsf S$, respectively representing the composition of  functors $\mathsf{Obj}\circ \mathscr X$ and $\mathsf{Mor} \circ \mathscr X$ from $\mathsf S^{\rm op}$ to $(\mathsf{Set})$.  The groupoid object $\mathscr X$ is often denoted $
\xymatrix@C=1.5em{
X_1 \ar@{->}@<-2pt>[r]_<>(0.5){s} \ar@{->}@<2pt>[r]^<>(0.5){t}&  X_0
}
$.
A \emph{morphism of groupoids objects} is a morphism (natural transformation) of functors.
\end{dfn}

\begin{rem}
There are morphisms $s,t:X_1\to X_0$   associated to a  the groupoid object $\mathscr X$  in $\mathsf S$ obtained from the source and target maps associated to a groupoid.    Moreover, associated to a morphism $\mathscr X\to \mathscr Y$ of groupoid objects in $\mathsf S$ are morphisms $X_0\to Y_0$ and $X_1\to Y_1$.   
\end{rem}

\begin{exa}[Constant groupoid object]
Let $X$ be any object of $\mathsf S$.  Define $X_0 = X_1 = X$.  Then $
\xymatrix@C=1.5em{
X_1 \ar@{->}@<-2pt>[r]_<>(0.5){\operatorname{Id}} \ar@{->}@<2pt>[r]^<>(0.5){\operatorname{Id}}&  X_0
}
$  is a groupoid object of $\mathsf S$.  We call such a groupoid \emph{constant} and denote it, abusively, by the same letter $X$.
\end{exa}

\begin{exa}[Action groupoid]
Let $G$ be a group object of $\mathsf S$ acting on the left on an object $X_0$.  Define $\mathsf{Obj} \mathscr X(U) = \operatorname{Hom}(U, X_0)$ and let $\mathsf{Mor} \mathscr X(U)$ be the set of all triples $(g, x, y)$ where $x, y \in \mathsf{Obj} X(U)$ and $g \in \operatorname{Hom}(U, G)$ is a $U$-point of $G$ such that $gx = y$.  The composition of $(g, x, y)$ and $(g', y, z)$ is the triple $(g'g, x, z)$.
This is known as the \emph{action groupoid}.   It is also typically denoted by  $\xymatrix{
G\times X_0 \ar@{->}@<-2pt>[r]_<>(0.5){\sigma} \ar@{->}@<2pt>[r]^<>(0.5){pr_2}&  X_0 
}$.
\end{exa}

\begin{exa} \label{E:gpd-cfg}
One can immediately associate a category fibered in groupoids to any groupoid object.  Indeed, suppose that $\mathscr X$ is a groupoid object of $\mathsf S$.  Construct a category $\mathcal X$ whose objects are pairs $(U, \xi)$ where $U$ is an object of $\mathsf S$ and $\xi \in \mathsf{Obj} \mathscr X(U)$.  A morphism $(U, \xi) \rightarrow (V, \eta)$ consists of a morphism $f : U \rightarrow V$ of $\mathsf S$ and a morphism $\phi : \xi \rightarrow f^\ast \eta$ of $\mathscr X(U)$.  The composition of $(f, \phi) : (U, \xi) \rightarrow (V, \eta)$ and $(g, \psi) : (V, \eta) \rightarrow (W, \zeta)$ is $(gf, \phi \circ f^\ast \psi)$.  It is easy to verify that this category is fibered in groupoids over $\mathsf S$ with the projection sending $(U, \xi)$ to $U$.
\end{exa}

When $\mathsf S$ has a topology, this groupoid is rarely a stack, although it is a prestack if the topology is subcanonical.   If $\mathsf S$ is subcanonical and $
\xymatrix@C=1.5em{
X_1 \ar@{->}@<-2pt>[r]_<>(0.5){s} \ar@{->}@<2pt>[r]^<>(0.5){t}&  X_0
}
$ is a groupoid object, then it is common to denote the associated prestack by $
[\xymatrix@C=1.5em{
X_1 \ar@{->}@<-2pt>[r]_<>(0.5){s} \ar@{->}@<2pt>[r]^<>(0.5){t}&  X_0
}]^{\operatorname{pre}}$.
There is an abstract process of \emph{stackification} of prestacks, analogous to sheafification of presheaves, by which a CFG $\mathcal X$ is replaced by the initial stack receiving a map from $\mathcal X$.  In the situation of groupoid objects, this stack is typically denoted by $[\xymatrix@C=1.5em{
X_1 \ar@{->}@<-2pt>[r]_<>(0.5){s} \ar@{->}@<2pt>[r]^<>(0.5){t}&  X_0
}]$ (see e.g., \cite[Def.~3.11, Def.~4.10]{DMstacks} for more details on this approach).    Moreover, there is a morphism $X_0\to [\xymatrix@C=1.5em{
X_1 \ar@{->}@<-2pt>[r]_<>(0.5){s} \ar@{->}@<2pt>[r]^<>(0.5){t}&  X_0
}]$ that makes the 
 following diagram
\begin{equation}\label{E:2cartGrp}
\xymatrix@C=.2em{
X_1 \ar[rrrr]^t  \ar[d]_s&&& & X_0 \ar[d(0.7)]&\\
X_0 \ar[rrr]&&& [X_1 \ar@{->}@<-2pt>[rr]_<>(0.5){s} \ar@{->}@<2pt>[rr]^<>(0.5){t}&&  X_0]
}
\end{equation}
$2$-cartesian (\cite[Tag 04M8]{stacks}) and  
essentially a $2$-coequalizer for  $\xymatrix@C=1.5em{
X_1 \ar@{->}@<-2pt>[r]_<>(0.5){s} \ar@{->}@<2pt>[r]^<>(0.5){t}&  X_0
}$ (see \cite[Tag 04MA]{stacks} for more details on the precise meaning of this).
   In other words, the stack provides a ``quotient'' for the groupoid.

  Rather than undertake an explanation of this construction and the attendant $2$-universal property, we will give a direct construction of the stack associated to a groupoid object.

\begin{rem}
In the case of left group action, the stack $[\xymatrix{
G\times X_0 \ar@{->}@<-2pt>[r]_<>(0.5){\sigma} \ar@{->}@<2pt>[r]^<>(0.5){pr_2}&  X_0 
}]$ is equivalent to the stack $[G\backslash X_0]$.  
\end{rem}

\subsubsection{Augmented groupoids}

\begin{dfn}[Augmented groupoid object]
A groupoid object $\mathscr X$ of a category $\mathsf S$ is said to be \emph{augmented} toward an object $X$ of $\mathsf S$ when it is equipped with a morphism $\mathscr X \rightarrow X$.  A groupoid $\mathscr X$ augmented toward $X$ is often denoted $X_1 \rightrightarrows X_0 \rightarrow X$.

If $\mathscr X \rightarrow X$ and $\mathscr Y \rightarrow Y$ are augmented groupoids, a morphism of augmented groupoids from $\mathscr X \rightarrow X$ to $\mathscr Y \rightarrow Y$ is a commutative diagram of groupoid objects as in~\eqref{E:aug-gpd-mor}:
\begin{equation} \label{E:aug-gpd-mor} \xymatrix{
\mathscr X \ar[r] \ar[d] & \mathscr Y \ar[d] \\
X \ar[r]^f & Y.
} \end{equation}
\end{dfn}

\begin{rem}
To augment a groupoid $\mathscr X$ torwards $X$, it is equivalent to give a morphism $f : X_0 \rightarrow X$ such that $fs = ft$.
\end{rem}

\begin{dfn}[Cartesian morphism of augmented groupoid objects] If $\mathsf S$ admits fibered products, 
a morphism of augmented groupoid objects as in \eqref{E:aug-gpd-mor}   is called \emph{cartesian} if $X_0 \rightarrow X \mathbin\times_Y Y_0$ and $X_1 \rightarrow X \mathbin\times_Y Y_1$ are isomorphisms.
\end{dfn}

\begin{exa}
Suppose that $q : X_0 \rightarrow X$ is a morphism in a category admitting fiber products.  We define a groupoid as follows:  $\mathsf{Obj} \mathscr X(U) = \operatorname{Hom}(U, X_0)$ and $\mathsf{Mor} \mathscr X(U)$ is the set of pairs $(f, g) \in \operatorname{Hom}(U, X_0)$ such that $qf = qg$.  In other words, $\mathsf{Mor} \mathscr X$ is represented by $X_1 = X_0 \mathbin\times_X X_0$.  The composition of the pair $(f,g)$ and $(g,h)$ is, by definition, the pair $(f,h)$, and the identity of $f \in \mathsf{Obj} \mathscr X(U)$ is the pair $(f,f)$.
The groupoid $X_1 \rightrightarrows X_0$ is augmented toward $X$ by construction.
\end{exa}

Let $\mathsf{Gpd}^+_{\mathsf S}$ denote the category of augmented groupoid objects of $\mathsf S$, with cartesian morphisms.  The projection sending an augmented groupoid object $(\mathscr X \rightarrow X)$ to $X$ makes $\mathsf{Gpd}^+_{\mathsf S}$ into a CFG over $\mathsf S$.

\subsubsection{Stacks associated to groupoid objects}

Now suppose that $\mathsf S$ is equipped with a pretopology.  

\begin{dfn}[Presentation of an augmented groupoid obejct] \label{D:gpd-stk}
Let $\mathsf S$ be a presite admitting fibered products.  Let $X_1 \rightrightarrows X_0 \rightarrow X$ be a groupoid of $\mathsf S$ augmented toward $X$.  We call it a \emph{presentation} of $X$ if $X_0 \rightarrow X$ is covering (Definition \ref{D:CovStack}) and the canonical map $X_1 \rightarrow X_0 \mathbin\times_X X_0$ is an isomorphism.
\end{dfn}

\begin{dfn}[Category associated to a groupoid object] \label{D:CFG-GrOb}
Let $\mathsf S$ be a presite admitting fibered products. 
Let $\mathscr X = (X_1 \rightrightarrows X_0)$ be a groupoid object of $\mathsf S$.  We construct a CFG, $\mathcal X$, called the \emph{CFG associated to a groupoid object}.  The objects of $\mathcal X$ are triples $(U, \mathscr U, \xi)$ where $U$ is an object of $\mathsf S$, where $\mathscr U\to U$ is a presentation of $U$, and where $\xi : \mathscr U \rightarrow \mathscr X$ is a morphism of groupoid objects.

A morphism in $\mathcal X$ from $(U, \mathscr U, \xi)$ to $(V, \mathscr V, \eta)$ is a cartesian morphism $(f, \varphi)$ of augmented groupoids from $(\mathscr U \rightarrow U)$ to $(\mathscr V \rightarrow V)$ such that $\eta \circ \varphi = \xi$ as morphisms of groupoids $\mathscr U \rightarrow \mathscr X$.

The morphism $\mathcal X\to \mathsf S$ is given by sending $(U, \mathscr U, \xi)$ to $U$.  
\end{dfn}

The following lemma asserts that the category $\mathcal X$ is a CFG;
the proof is straightforward, so it is omitted.
\begin{lem}
If $\mathsf S$ admits fiber products,
then the category $\mathcal X$ over $\mathsf S$ constructed above is a CFG.
\end{lem}

For the following lemma, let $\mathsf{Cov}$ denote the category whose objects are covering morphisms $X \rightarrow U$ in $\mathsf S$ and whose morphisms are cartesian squares.  If $\mathsf S$ has fiber products then the projection $\mathsf{Cov} \rightarrow \mathsf S$ sending $(X \rightarrow U)$ to $U$ makes $\mathsf{Cov}$ into a CFG over $\mathsf S$ (Example \ref{E:CFG-for-stable-class}).

\begin{lem}
Assume that $\mathsf S$ is a subcanonical presite with fiber products.  Let  $\mathcal X$ be the CFG constructed as above in Definition \ref{D:CFG-GrOb}.  If $\mathsf{Cov}$ is a stack over $\mathsf S$ then so is $\mathcal X$.
\end{lem}
\begin{proof}
We give just a sketch.  
The idea is to use Lemma \ref{L:StCond}. We start with an object $S$ of $\mathsf S$ and the canonical morphism  $\mathsf S/S\to  \mathsf S$.  
Given a cover $\mathcal R = \{ U_i \rightarrow S \}$ of $S$ and a morphism $\mathcal R\to \mathcal X$ we need to show how to obtain the lift $\mathsf S/S\to \mathcal X$.  So, given  groupoid objects $\mathscr U_i$ over each $U_i$, along with compatible data over the double and triple fiber products $U_{ij}$ and $U_{ijk}$, these descend to a groupoid object $\mathscr S$ over $S$ by descending the objects $\mathsf{Obj} \mathscr U_i$ and $\mathsf{Mor} \mathscr U_i$ of $\mathsf{Cov}$ and the morphisms between them (using that $\mathsf{Cov}$ forms a stack and that morphisms between representable objects form sheaves).  Then the maps $\mathscr U_i \rightarrow \mathscr X$ descend to $\mathscr S \rightarrow \mathscr X$ by descending the maps on objects and morphisms, again using the subcanonicity of the site.  This gives the desired morphism $\mathsf S/S\to \mathcal X$.
\end{proof}

\begin{exa} \label{E:gpd-proj}
If $\mathscr X$ is the groupoid object $\xymatrix@C=10pt{U_1 \ar@<1.5pt>[r]^s \ar@<-1.5pt>[r]_t & U_0}$   and $\mathcal X$ is its associated stack,  there is a canonical map $U_0 \rightarrow \mathcal X$, which is covering.  
We construct the triple $(U_0,\mathscr U_0,\mathscr U_0\to \mathscr X)$ giving this morphism as follows.
The presentation $\mathscr U_0\to U_0$ is given by 
$$
\xymatrix@C=10pt{
U_1\times_{U_0}U_1 \ar@<1.5pt>[rr]^<>(0.5){pr_1} \ar@<-1.5pt>[rr]_<>(0.5){comp} && U_1 \ar[rr]^s &&U_0;
}
$$
here the maps for the fibered product are the source and target maps repsectively, and the bottom arrow $comp$ is the composition morphism taking a pair $(\alpha,\beta)$ in $U_1\times_{U_0}U_1$  (over some $S$) to the composition $\beta\circ \alpha$ in $U_1$ (over $S$).  This is a presentation of $U_0$ since  $s$ is covering (the projections $s,t:U_1 \rightarrow U_0$ are always covering in a groupoid object because of the `identity map' section $U_0\to U_1$).

To describe $\mathscr U_0\to \mathscr X$, it is convenient to describe $\mathscr U_0$ as follows.
For any scheme $S$,  $\mathscr U_0(S)$ is the groupoid in which the objects are the morphisms of $\mathscr X(S)$.  The morphisms from $(\xi \rightarrow \zeta)$ to $(\eta \rightarrow \omega)$ in $\mathscr U_0(S)$ are the commutative squares
\begin{equation*} \xymatrix{
\xi \ar[r] \ar[d] & \zeta \ar@{=}[d] \\
\eta \ar[r] & \omega .
} \end{equation*}
That is, there are no morphisms unless $\zeta = \omega$.
Clearly  $\mathsf{Obj} \mathscr U_0(S) = \operatorname{Hom}(S, U_1)$ and $\mathsf{Mor} \mathscr U_0(S) = \operatorname{Hom}(S, U_1\times_{U_0}U_1)$.  There is a canonical map $\mathscr U_0 \rightarrow \mathscr X$ sending an object $(\xi \rightarrow \zeta)$ to $\xi$.   

Finally, we show that $U_0\to \mathcal X$ is covering.  For a scheme $S$, a morphism $S\to \mathcal X$ is a triple $(S,\mathscr S,\xi:\mathscr S\to \mathscr X$), where $\mathscr S=\xymatrix@C=10pt{S_1 \ar@<1.5pt>[r]^{s'} \ar@<-1.5pt>[r]_{t'} &S_0}\to S$ is a presentation.   
The morphism $\mathscr S\to \mathscr X$ induces a morphism $S_0\to U_0$.  The fact that $S_0\to S$ is covering means that there is a cover $\{T_\alpha \to S\}$ of $S$ that lifts to $S_0$, and therefore composing  gives  morphisms $T_\alpha \to U_0$.  This gives the cover $\{T_\alpha \to S\}$ of $S$ whose compositions $T_\alpha\to S\to \mathcal X$ lift to $U_0$.  
\end{exa}

\begin{rem}
In fact, there is a cocartesian diagram in the $2$-category of stacks:
\begin{equation*}
\xymatrix{
U_1 \ar[r]^s  \ar[d]_t& U_0 \ar[d]\\
U_0 \ar[r]& \mathcal X
} \end{equation*}
We will not use this property, so we do not give a proof (see e.g., \cite[Def.~3.11, Def.~4.10]{DMstacks}).
\end{rem}

\subsection{Adapted groupoid objects and adapted stacks}\label{S:GrOb-AS}

Here we show that a stack associated to a groupoid object   having source and target that are  coverings in the presite is the same as a stack adapted to the presite; i.e., it is an algebraic stack.  Moreover, the groupoid object induces a presentation of the stack.

\begin{pro}[{\cite[Prop.~5.21]{DMstacks}}] \label{P:Gr-Stack-Pres}  Let $\mathsf S$ be a subcanonical presite admitting fibered products.  
Let  $\xymatrix@C=1em{U_1 \ar@{->}@<-2pt>[r]_<>(0.5){s} \ar@{->}@<2pt>[r]^<>(0.5){t}& U_0}$ be a groupoid object, and set $\mathcal X=[\xymatrix@C=1em{U_1 \ar@{->}@<-2pt>[r]_<>(0.5){s} \ar@{->}@<2pt>[r]^<>(0.5){t}& U_0}]$. 
If $\mathbf P$ is any property of morphisms in $\mathsf S$ that is stable under base change and local (on the target), 
 then:
\begin{enumerate}
\item  Assuming the diagonal $\mathcal X\to \mathcal X\times \mathcal X$ is $\mathsf S$-representable, it   has property $\mathbf P$ if and only if $(s,t):U_1\to U_0\times U_0$ 
has property $\mathbf P$.

\item  Assuming the  morphism $U_0 \to \mathcal  X$ (corresponding to the identity of $U_0$) is $\mathsf S$-representable, it has property $\mathbf P$ if and only if s (or t) has property $\mathbf P$.
\end{enumerate}
\end{pro}

\begin{proof}
For the first claim, $U_1 \rightarrow U_0 \times U_0$ is the base change of the diagonal via $U_0 \times U_0 \rightarrow \mathcal X \times \mathcal X$, so the former inherits property $\mathbf P$ from the latter.  Conversely, if $S \rightarrow \mathcal X \times \mathcal X$ is any morphism, let
\begin{equation*} \xymatrix{
V \ar[r] \ar[d] & W \ar[d] \\
T \ar[r] & S
} \end{equation*}
be the base change of the cartesian diagram
\begin{equation*} \xymatrix{
U_1 \ar[r] \ar[d] & U_0 \times U_0 \ar[d] \\
\mathcal X \ar[r]^<>(0.5)\Delta & \mathcal X \times \mathcal X .
} \end{equation*}
Then $V \rightarrow W$ has property $\mathbf P$ by base change.  But $W \rightarrow S$ is a cover, since it is the base change of $U_0 \times U_0 \rightarrow \mathcal X \times \mathcal X$, so property $\mathbf P$ descends to $T \rightarrow S$.  This applies to any morphism $S \rightarrow \mathcal X \times \mathcal X$, so $\Delta$ has property $\mathbf P$.

For the second claim, there is a cartesian diagram
\begin{equation} \label{E:groupoid} \xymatrix{
U_1 \ar[r]^t \ar[d]_s & U_0 \ar[d] \\
U_0 \ar[r] & \mathcal X
} \end{equation}
so $s$ and $t$ inherit property $\mathbf P$ from $U_0 \rightarrow \mathcal X$.  Conversely, suppose $S \rightarrow \mathcal X$ is any morphism.  Let 
\begin{equation*} \xymatrix{
V \ar[r] \ar[d] & W \ar[d]  \\
T \ar[r] & S
} \end{equation*}
be the base change of \eqref{E:groupoid}.  Then $V \rightarrow W$ is $\mathbf P$ by base change.  But $W \rightarrow S$ is covering (by base change; see Example  \ref{E:gpd-proj}),
 so $\mathbf P$ descends to $T \rightarrow S$.  This applies to any $S \rightarrow \mathcal X$, so $U_0 \rightarrow \mathcal X$ is $\mathbf P$, as required.
\end{proof}

\begin{pro}[{\cite[Prop.~4.19, 5.19]{DMstacks}}]  
  Let $\mathsf S$ be a subcanonical presite admitting fibered products.  
 Let $\mathcal X$ be a stack, let $P:U\to \mathcal X$ be an  $\mathsf S$-representable morphism from an object $U$ of $\mathsf S$.   Then there is an associated  groupoid object $\xymatrix{U\times_{\mathcal X}U \ar@{->}@<-2pt>[r]_<>(0.5){pr_1} \ar@{->}@<2pt>[r]^<>(0.5){pr_2}& U}$ in $\mathsf S$ (see e.g., \cite[Prop.~3.5]{DMstacks}).   
 If moreover the morphism $P:U\to \mathcal X$ is a cover in the sense of Definition \ref{D:SAdaptStack}, then $\mathcal X$ is equivalent to $[\xymatrix{U\times_{\mathcal X}U \ar@{->}@<-2pt>[r]_<>(0.5){pr_1} \ar@{->}@<2pt>[r]^<>(0.5){pr_2}& U}]$, and the projections $pr_1,pr_2$ are covers in the presite.
\end{pro}

\begin{proof}  
Set $U_0 = U$ and $U_1 = U \mathbin\times_{\mathcal X} U$.  Let $\mathcal U$ be the stack associated to the groupoid object  $\mathscr U = U_1 \rightrightarrows U_0$.  We construct a map $\mathcal U \rightarrow \mathcal X$ by descent.  Let $Z \rightarrow \mathcal U$ be any morphism, where $Z$ is a scheme.  By definition, this corresponds to a groupoid presentation of $\mathscr Z$ of $Z$ and a cartesian morphism $\mathscr Z \rightarrow \mathscr U$.  By composition, this gives a map to the constant groupoid object $\mathcal X$, and this descends uniquely to a map $Z \rightarrow \mathcal X$.  This is easily shown to be functorial in $Z$, hence gives a morphism $\mathcal U \rightarrow \mathcal X$.

Now we argue that $\mathcal U \rightarrow \mathcal X$ is an isomorphism if $U_0$ covers $\mathcal X$.  Indeed, the map $U \rightarrow \mathcal X$ factors through $\mathcal U$, so $\mathcal U \rightarrow \mathcal X$ is surjective.  On the other hand, $U_0 \mathbin\times_{\mathcal U} U_0 \rightarrow U_0 \mathbin\times_{\mathcal X} U_0$ is an isomorphism.  This is the pullback under the cover $U_0 \mathbin\times_{\mathcal X} U_0 \rightarrow \mathcal U \mathbin\times_{\mathcal X} \mathcal U$ of the diagonal $\mathcal U \rightarrow \mathcal U \mathbin\times_{\mathcal X} \mathcal U$.  Therefore the relative diagonal of $\mathcal U \rightarrow \mathcal X$ is an isomorphism, which is to say that $\mathcal U \rightarrow \mathcal X$ is injective.  Combined with the surjectivity, this means $\mathcal U \rightarrow \mathcal X$ is an isomorphism.

The statement that $pr_1$ and $pr_2$ are covers can be obtained   from the $2$-cartesian diagram \eqref{E:2cartGrp}.
\end{proof}

\begin{dfn}[Groupoid object adapted to a presite]
Let $\mathsf S$ be a subcanonical presite admitting fibered products.  We say a groupoid object $\xymatrix@C=1em{U_1 \ar@{->}@<-2pt>[r]_<>(0.5){s} \ar@{->}@<2pt>[r]^<>(0.5){t}& U_0}$ is \emph{adapted to the presite} if $s$ and $t$ are covers in the presite and the natural morphism $U_0\to \mathcal X=[\xymatrix@C=1em{
U_1 \ar@{->}@<-2pt>[r]_<>(0.5){s} \ar@{->}@<2pt>[r]^<>(0.5){t}&  U_0
}]$ is $\mathsf S$-representable.  
\end{dfn}

\begin{cor}\label{C:AdaptStAdaptGr}
The stack associated to a groupoid object adapted to a presite is a stack adapted to the presite; in particular, it is algebraic.  Conversely, a stack adapted to a presite is the stack associated to a groupoid object adapted to the presite.
\end{cor}

\section*{Acknowledgments}
   These notes provide an elaboration on lectures  the first author gave at the summer school \emph{The Geometry, Topology and Physics of Moduli Spaces of Higgs Bundles}  at the Institute for Mathematical Sciences at the National University of Singapore in July of 2014.  
   He would like to thank the organizers   for their invitation, and the IMS at NUS for their hospitality.  He would also like to thank  Tony Pantev for discussions on the moduli  stack of Higgs bundles.  Both authors would like to thank the referees and Sebastian Bozlee for their comments on earlier versions of the paper.

\bibliography{stacksintrobib}

\end{document}